\documentclass{amsart}
\usepackage[USenglish,activeacute]{babel}
\usepackage[T1]{fontenc}
\usepackage{tikz-cd}
\usepackage[utf8]{inputenc}
\usepackage{fancyhdr}
\usepackage{yhmath}
\usepackage{mathrsfs}
\usepackage{graphicx}
\usepackage{wrapfig}
\usepackage{caption}
\usepackage{dsfont}
\usepackage[normalem]{ulem}
\usepackage{enumitem}   
\usepackage{cutwin}
\usepackage{amsfonts}
\usepackage{mathtools}
\usepackage{subcaption}
\usepackage[colorlinks=true]{hyperref}
\usepackage{multirow}
\usepackage{amsmath}
\usepackage{amssymb}
\usepackage{accents}
\usepackage{amsthm}
\usepackage{textcomp}
\usepackage{adjustbox}
\makeatletter
\newtheorem*{not*}{{ Notation}}
\newtheorem{defi}{ Definition}[subsection]
\newtheorem*{defi*}{Definition}
\newtheorem{teo}[defi]{ Theorem}
\newtheorem*{teo*}{{ Theorem}}
\newtheorem{intteo}{Theorem}

\newtheorem{prop}[defi]{ Proposition}
\newtheorem*{prop*}{{ Proposition}}
\newtheorem{obs}[defi]{{Remark}}

\newtheorem*{Lemma*}{Lemma}
\newtheorem{Lemma}[defi]{Lemma}
\newtheorem*{coro*}{Corollary}
\newtheorem{coro}[defi]{Corollary}

\newcommand{\tarc}{\mbox{\large$\frown$}}
\newcommand{\arc}[2][-3ex]{{#2}{\kern #1{\raisebox{1.5ex}{\tarc}}}}

\newcommand{\op}{\operatorname{op}}
\newcommand{\Sp}{\operatorname{Sp}}

\newcommand{\Spf}{\operatorname{Spf}}
\newcommand{\End}{\operatorname{End}}
\newcommand{\Hom}{\operatorname{Hom}}
\newcommand{\Der}{\operatorname{Der}}
\newcommand{\Mod}{\operatorname{Mod}}
\newcommand{\D}{\mathcal{D}}
\newcommand{\OX}{\mathcal{O}}
\newcommand{\Bi}{\operatorname{Bi}}

\newcommand{\gr}{\operatorname{gr}}

\newcommand{\Id}{\operatorname{Id}}
\newcommand{\Indban}{\operatorname{Ind(Ban}_K)}
\newcommand{\EDelta}{\mathbb{L}\Delta^E_*}
\newcommand{\Ifunct}{\mathbb{L}\Delta_*^{\operatorname{S}}}
\newcommand{\RIfunct}{\mathbb{L}\Delta_*^{\operatorname{S}_r}}
\title{Hochschild (Co)-homology of D-cap modules on rigid spaces I}
\author{Fernando Peña Vázquez}

\begin{document}

\begin{abstract}
We introduce a formalism of Hochschild (co)-homology for D-cap modules on a smooth rigid analytic space X, based on the homological tools of Ind-Banach D-cap modules. We introduce several categories of D-cap bimodules for which this theory is well-behaved. Among these, the most important example is the category of diagonal C-complexes. We give an explicit calculation of the Hochschild complex  for diagonal C-complexes, and show that the Hochschild complex of D-cap is canonically isomorphic to the de Rham complex of X. In particular, we obtain a Hodge-de Rham spectral sequence converging to the Hochschild cohomology groups of D-cap. We obtain explicit formulas relating the Hochschild cohomology and homology of a given diagonal C-complex.
\end{abstract}
\maketitle
\tableofcontents
\section{Introduction}
Let $K$ be a complete non-archimedean extension of $\mathbb{Q}_p$, and let $X$ be a smooth and separated rigid analytic space over $K$. We let $\wideparen{\D}_X$ be its sheaf of completed $p$-adic differential operators, as defined by Ardakov-Wadsley in \cite{ardakov2019}.  This is the first in a series of papers \cite{p-adicCheralg},\cite{p-adicCatO},\cite{HochschildDmodules2},\cite{cychomDmod}, in which we convey a systematic study of $\wideparen{\D}_X$ and its relevant categories of modules (\emph{i.e.} co-admissible, $\mathcal{C}$-complexes \emph{etc}) employing techniques from non-commutative geometry and deformation theory.\\
In particular, the main goal of this paper is establishing the foundations of an \emph{analytic} theory of Hochschild (co)-homology for sheaves of $\wideparen{\D}_X$-modules.\bigskip

In recent years, there has been a great push towards developing a theory of \emph{quasi-coherent} modules adapted to analytic geometry. In particular, the aim is obtaining categories of modules which incorporate the analytical side of the picture, while allowing us to work in a completely algebraic setting. Namely,
there is the theory of condensed mathematics and analytic spaces of Clausen-Scholze \cite{analytic}, and the theory of Ben-Bassat, Kremnizer \emph{et al.}, based upon sheaves of Ind-Banach spaces \cite{BBK}. This is the one that this paper is closer to. The new influx of techniques and ideas that these theories have brought in provides us with sufficient homological algebra to adapt some of the classical results of Hochschild (co)-homology to the rigid analytic setting. In particular, our goal is setting up a cohomology theory for sheaves of Ind-Banach $\wideparen{\D}_X$-modules which retains the deformation-theoretic aspects of the classical Hochschild (co)-homology, but also incorporates the $p$-adic analytic aspects of the objects we are working with. Thus, we obtain a  technical framework in which the algebraic theory of deformations  and the analytic theory of perturbations can be studied via the same formalism.\bigskip

Before moving on to the specific contents of the paper, let us give an introduction to some of these notions:
The use of homological techniques in the study of associative algebras can be dated back to  G. Hochschild in the 1940's. In \cite{Hochschild1}, Hochschild gives the first definition of cohomology of an associative algebra. He does so in terms of the \emph{Hochschild cohomology complex}. Let $K$ be a field of characteristic zero, $A$ be a $K$-algebra, and $A^e=A\otimes_KA^{\op}$ be the enveloping algebra of $A$. The bar complex is a 
canonical resolution of $A$ as an $A^e$-module. Namely, it is the following complex of $A^e$-modules:
\begin{equation*}
    B(A)^{\bullet}:=\left( \cdots \rightarrow A\otimes_KA\otimes_KA\xrightarrow[]{d_2} A\otimes_KA\xrightarrow[]{d_1} A\rightarrow 0 \right),
\end{equation*}
where the boundary maps are given by the formula:
\begin{equation*}
d_n(a_0\otimes \cdots \otimes a_{n+1})=\sum_{i=0}^n (-1)^ia_0\otimes \cdots \otimes a_ia_{i+1}\otimes \cdots \otimes a_{n+1}.
\end{equation*}
The Hochschild cohomology complex is obtained by applying $\Hom_{A^e}(-,A)$ to the bar resolution $B(A)^{\bullet}$:
\begin{equation*}
    \operatorname{HH}^{\bullet}(A):=\left( 0 \rightarrow A\rightarrow \Hom_K(A,A)\rightarrow \cdots \rightarrow  \Hom_K(A^{\otimes n},A) \rightarrow \cdots \right).
\end{equation*}
The cohomology groups of this complex are called the Hochschild cohomology groups of $A$. The explicit description of $\operatorname{HH}^{\bullet}(A)$ allowed Hochschild to give tangible interpretations of the Hochschild cohomology groups in low degrees. In particular, $\operatorname{HH}^{0}(A)= Z(A)$, and $\operatorname{HH}^{1}(A)$ is isomorphic to the $K$-vector space of outer derivations of $A$. Even better, M. Gerstenhaber showed in \cite{deformations} that Hochschild cohomology is closely related to the deformation theory of associative algebras:
\begin{defi*}
We define the following concepts:
    \begin{enumerate}[label=(\roman*)]
        \item An infinitesimal deformation of $A$ is a $K[t]/t^2$-algebra structure on:
        \begin{equation*}
            A':=A\otimes_K K[t]/t^2,
        \end{equation*}
        which is isomorphic to $A$ after reducing modulo $t$.
        \item A formal deformation of $A$ is a $K[[t]]$-algebra structure on:
        \begin{equation*}
            \widehat{A}:=\varprojlim_n A\otimes_K K[t]/t^n,
        \end{equation*}
        which is isomorphic to $A$ after reducing modulo $t$.
    \end{enumerate}
\end{defi*}
As shown by Gerstenhaber, $\operatorname{HH}^{2}(A)$ parametrizes the isomorphism classes of infinitesimal deformations of $A$.  Furthermore, the obstructions to lifting an infinitesimal deformation of $A$ to a formal deformation are classified by $\operatorname{HH}^{3}(A)$.\bigskip

Similarly, the \emph{Hochschild homology complex} of $A$, which we will write $\operatorname{HH}_{\bullet}(A)$, can be obtained from the bar resolution by regarding $A$ as a right $A^e$-module, and applying $A\otimes_{A^e}-$ to $B(A)^{\bullet}$. Again, $\operatorname{HH}_{\bullet}(A)$ can be described explicitly, and there are concrete interpretations for the Hochschild homology groups in low degrees. More generally, we can consider Hochschild (co)-homology complexes with coefficients in any $A^e$-module $M$, which we write $\operatorname{HH}^{\bullet}(A,M)$, and $\operatorname{HH}_{\bullet}(A,M)$. These are obtained from the bar resolution of $A$ by applying $\Hom_{A^e}(-,M)$, and $M\otimes_{A^e}-$ respectively. For example, given an $A$-module $M$, we obtain a canonical structure of $A^e$-module on $\End_K(M)$. The Hochschild cohomology groups $\operatorname{HH}^{\bullet}(\End_K(M))$ are specially interesting, as they govern the deformation theory of $M$. Let us introduce the notion of a deformation of a module:
\begin{defi*}
Let $A$ be an associative $K$-algebra, and 
$R$ be a commutative augmented $K$-algebra, with augmentation ideal $R^+$. 
A deformation of an $A$-module $M$ over $R$ is an $A$-module structure on $M_R:=M\otimes_KR$, commuting with the $R$-action, and such that  $M_R/R^+\cong M$
 as $A$-modules. 
\end{defi*}
For example, if $R=K[t]/t^2$, we  call the $R$-deformations of $M$ the infinitesimal deformations of $M$, and if $R=K[t]/t^3$, we  call the $R$-deformations of $M$ the  second order deformations of $M$. Again, using the structure of the Hochschild cohomology complex, it can be shown that the isomorphism classes of infinitesimal deformations are classified by $\operatorname{HH}^{1}(\End_K(M))$, and the obstructions to extending a deformation of $M$ to a second order deformation lie in $\operatorname{HH}^{2}(\End_K(M))$.\bigskip

Another interesting feature of Hochschild cohomology is that it generalizes  many cohomology theories associated to relevant objects in representation theory. For instance, let $\mathfrak{g}$ be a finite-dimensional $K$-Lie algebra, and $M$ be a $U(\mathfrak{g})$-bimodule. Then there is a canonical isomorphism:
\begin{equation*}
    \operatorname{HH}^{\bullet}(U(\mathfrak{g}),M)=\operatorname{H}^{\bullet}(\mathfrak{g}, M_{\operatorname{ad}}),
\end{equation*}
where $\operatorname{H}^{\bullet}(\mathfrak{g},-)$ is the Lie algebra cohomology, and the $\mathfrak{g}$-representation $M_{\operatorname{ad}}$ is defined by the formula:
\begin{equation*}
    \operatorname{ad}(x)m=xm-mx, \textnormal{ for } x\in \mathfrak{g}, \textnormal{ and }m\in M.
\end{equation*}
Let $V$ be a $K$-linear $\mathfrak{g}$-representation. As above, we obtain a $U(\mathfrak{g})$-bimodule structure on $\End_K(V)$, and the previous formula allows us to use Hochschild cohomology to relate Lie algebra cohomology and deformation theory of $\mathfrak{g}$-representations. Furthermore, if $\mathfrak{g}$ is a finite-dimensional reductive $K$-Lie algebra, then we have the following formula:
\begin{equation*}
    \operatorname{HH}^{\bullet}(U(\mathfrak{g}))=Z(U(\mathfrak{g}))\otimes_K\operatorname{H}^{\bullet}(\mathfrak{g}, K),
\end{equation*}
which showcases that Hochschild cohomology and Lie algebra cohomology are intimately related. Finally, let us now give a more geometrically flavored example: Let $G$ be a linear algebraic group over $K$, $\OX(G)$ be the commutative Hopf algebra of regular functions on $G$, and $V$ be a finite-dimensional $K$-linear representation of $G$. In this situation, $V$ has a canonical structure as a co-module over $\OX(G)$. Thus, its dual $V^*$ is a module over $\OX(G)^*$. As before, this endows $\End_K(V)$ with an $\OX(G)^*$-bimodule structure, and we have a canonical isomorphism:
\begin{equation*}
    \operatorname{HH}^{\bullet}(\OX(G)^*, \End_K(V))=\operatorname{H}^{\bullet}(G,\End_K(V)_{\operatorname{ad}}),
\end{equation*}
where $\End_K(V)_{\operatorname{ad}}$ is the adjoint representation of $G$ on $\End_K(V)$, and $\operatorname{H}^{\bullet}(G,-)$ stands for group cohomology (\emph{i.e.} the derived functor of the invariants). Thus, the Hochschild cohomology of $\OX(G)^*$ determines the deformation theory of the $K$-linear $G$-representations, and this is analogous to the deformation theory of vector bundles over the classifying stack $BG$.\bigskip

Ever since their inception, Hochschild cohomology and homology have become central tools in modern mathematics, with applications ranging from deformation theory, algebraic topology, algebraic geometry, functional analysis, and representation theory. The naive definitions in terms of explicit complexes have been replaced by categorical methods, which allow Hochschild (co)-homology to be defined in greater generality. For instance, even if the definition using the bar resolution does not generalize to a lot of settings, a quick glance at the definition leads us to deduce the following identities:
\begin{equation*}
    \operatorname{HH}^{\bullet}(A,M)=R\Hom_{A^e}(A,M), \quad \operatorname{HH}_{\bullet}(A,M)= A\otimes_{A^e}^{\mathbb{L}}M.
\end{equation*}
Thus, we may extend the definition of Hochschild (co)-homology to any closed symmetric monoidal abelian category that behaves well with respect to homological algebra. In particular, this approach allows us to consider Hochschild (co)-homology of sheaves of modules on ringed spaces, and of quasi-coherent modules on an algebraic variety. This theory has been thoroughly studied in recent years, and has proven to be a powerful tool in the study of certain deformation problems in algebraic geometry. See, for instance, the work of M. Kontsevich on deformation-quantization of Poisson structures on $C^{\infty}$-manifolds \cite{KontsevichCinf}, as well as  A. Yukutieli generalization for  algebraic varieties over a field of characteristic zero \cite{yukudeformquant}, and his work on the continuous Hochschild cochain complex of a scheme \cite{contHochscomp}.\bigskip

Up until now, we have only painted the algebraic part of the picture. That is, the part concerning algebraic deformation theory. Let us now describe its analytic counterpart, in the form of perturbation theory. Assume that $K$ is a complete non-archimedean field of characteristic zero, and let $\mathscr{A}$ be a unital associative Banach $K$-algebra. We  let $-\widehat{\otimes}_K-$ denote the complete tensor product of Banach $K$-vector spaces, and $\underline{\Hom}_K(-,-)$ denote the 
space of continuous $K$-linear maps, which we regard as a Banach space with respect to the operator norm. Before introducing perturbation theory, let us first describe the cohomology theory that parametrizes it. This will turn out to be just an analytic version of the Hochschild (co)-homology described above. As  we are now working in a purely analytic framework, we need all of our algebras to be of analytic nature. Hence, the first step is modifying the definition of the enveloping algebra of $\mathscr{A}$:
\begin{equation*}
    \mathscr{A}^e:=\mathscr{A}\widehat{\otimes}_K\mathscr{A}^{\op}.
\end{equation*}
As in the algebraic situation, the category of Banach $\mathscr{A}^e$-modules is equivalent to the category of Banach $\mathscr{A}$-bimodules. In particular, it follows that $\mathscr{A}$ is an $\mathscr{A}^e$-module. Furthermore, there is an analytic version of the bar resolution:
\begin{equation*}
    B(\mathscr{A})^{\bullet}:=\left(\cdots\rightarrow \mathscr{A}^e\widehat{\otimes}_K\mathscr{A}\widehat{\otimes}_K\mathscr{A}\rightarrow \mathscr{A}^e\widehat{\otimes}_K\mathscr{A}\rightarrow\mathscr{A}^e\right),
\end{equation*}
which is an exact complex of Banach $\mathscr{A}^e$-modules satisfying that the augmented complex:
\begin{equation*}
    B(\mathscr{A})^{\bullet}\rightarrow \mathscr{A},
\end{equation*}
is split exact as a complex of Banach $K$-vector spaces. Applying $\underline{\Hom}_{\mathscr{A}^e}(-,\mathscr{A})$ to $B(\mathscr{A})^{\bullet}$, we obtain the following chain complex:
\begin{equation*}
   \mathcal{L}(\mathscr{A})^{\bullet}:= \left(0\rightarrow \mathscr{A}\xrightarrow[]{\delta_0^{\mathscr{A}}} \underline{\Hom}_{K}(\mathscr{A},\mathscr{A})\xrightarrow[]{\delta_1^{\mathscr{A}}} \underline{\Hom}_{K}(\widehat{\otimes}_K^2\mathscr{A},\mathscr{A})\xrightarrow[]{\delta_2^{\mathscr{A}}} \cdots\right).
\end{equation*}
This can be regarded as an analytic version of the Hochschild cohomology complex discussed above. However, the key difference between the algebraic and analytic setups is that $\operatorname{Ban}_K$, the category of Banach $K$-vector spaces, is not an abelian category. Hence, it is not possible to define the 
cohomology groups of $\mathcal{L}(\mathscr{A})^{\bullet}$ in  $\operatorname{Ban}_K$ directly. For instance, given an injective map of Banach spaces $V\rightarrow W$, the quotient space $W/V$ is Banach if and only if the image of $V$ in $W$ is closed. More generally, the cohomology of a complex of Banach spaces is only well-defined if all the morphisms in the complex are strict, that is, if each morphism is closed.\bigskip

In order to overcome this difficulty, one could use the fact that $\operatorname{Ban}_K$ is a quasi-abelian category to compute the cohomology groups of $\mathcal{L}(\mathscr{A})^{\bullet}$ in the left heart $LH(\operatorname{Ban}_K)$. The category $LH(\operatorname{Ban}_K)$ is an abelian category equipped with a functor $I:\operatorname{Ban}_K\rightarrow LH(\operatorname{Ban}_K)$
which is fully faithful, admits a right adjoint, and satisfies that a complex of Banach spaces is strict and exact if and only if its image under $I$ is an exact complex. This perspective is better behaved homologically, since it allows us to translate analytic problems into algebraic problems in an abelian category. In fact, our approach to Hochschild (co)-homology is based upon this construction. However, the cohomology groups that we obtain via this procedure are highly inexplicit. Specifically, they are objects in $LH(\operatorname{Ban}_K)$ which need not be Banach spaces. In fact, the do not have an underlying vector space.\bigskip

Instead of this abstract approach, one can instead define Hochschild cohomology for a Banach algebra in a more naive way. In particular, following the work of J. Taylor in \cite{taylor1972homology}, the cohomology groups of a $\mathscr{A}$ are defined as the following cokernel in the category of semi-normed spaces:
\begin{equation*}
    \operatorname{HH}^n_{\operatorname{T}}(\mathscr{A}):=\operatorname{Coker}(\mathcal{L}(\mathscr{A})^{n-1}\rightarrow \operatorname{Ker}(\delta_n^{\mathscr{A}})).
\end{equation*}
In other words, $\operatorname{HH}^n_{\operatorname{T}}(\mathscr{A})$ is the $n$-th cohomology group of $\mathcal{L}(\mathscr{A})^{\bullet}$ regarded as a complex of $K$-vector spaces, together with the semi-norm given by the surjection $\operatorname{Ker}(\delta_n^{\mathscr{A}})\rightarrow \operatorname{HH}^n_{\operatorname{T}}(\mathscr{A})$. Although the homological properties of this definition are lackluster (especially if one wishes to consider more general topological algebras), its main advantage lies in the fact that the differentials $\delta_n^{\mathscr{A}}$ have an explicit description for all $n\geq 0$. In particular, the cohomology groups $\operatorname{HH}^n_{\operatorname{T}}(\mathscr{A})$ can be given explicit meaning in low degrees. Namely, one can define outer derivations and infinitesimal deformations for Banach algebras, and it can be shown that they are classified by $\operatorname{HH}^1_{\operatorname{T}}(\mathscr{A})$  and $\operatorname{HH}^2_{\operatorname{T}}(\mathscr{A})$ respectively.\\

Let us mention at this point that the cohomology groups $\operatorname{HH}^{\bullet}_{\operatorname{T}}(-)$ can be defined for more general locally convex algebras. One just needs to mind some technical subtleties regarding the behavior of the completed tensor product $\widehat{\otimes}_K$ and the inner homomorphism functor $\underline{\Hom}_K(-,-)$. This is the level of generality considered by J. Taylor in \cite{taylor1972homology}. However, for the purpose of this introduction, it is enough to mention that everything we have discussed thus far carries over verbatim 
if $\mathscr{A}$ is a Fréchet algebra.\bigskip

Now that the board has been set, we can  start discussing perturbation theory: Let $m:\mathscr{A}\times \mathscr{A}\rightarrow \mathscr{A}$ denote the multiplication in $\mathscr{A}$. In order to clarify notation, we let  $(\mathscr{A},m)$ be $\mathscr{A}$ with its Banach algebra structure given by $m$. Notice that $m$ can be regarded as an element in the space:
\begin{equation*}
    \mathcal{L}(\mathscr{A})^{2}=\underline{\Hom}_K(\mathscr{A}\widehat{\otimes}_K \mathscr{A},\mathscr{A}),
\end{equation*}
which is a Banach space under the operator norm. Let $m':\mathscr{A}\times \mathscr{A}\rightarrow \mathscr{A}$ be another multiplication map making $\mathscr{A}$ a Banach algebra. This is also an element in $\mathcal{L}(\mathscr{A})^{2}$, so we can consider the operator norm $\vert m-m' \vert$.  The main idea behind perturbation theory is studying which properties of $(\mathscr{A},m)$ are stable under small perturbations. That is, for which properties of $(\mathscr{A},m)$ there is some $\epsilon > 0$ such that if $(\mathscr{A},m')$ is a Banach algebra with $\vert m-m' \vert < \epsilon$, then $(\mathscr{A},m')$ also satisfies the aforementioned property. For instance, we have the following:
\begin{defi*}
We say a Banach algebra $(\mathscr{A},m)$ is stable if there is some 
$\epsilon > 0$ such that every Banach algebra $(\mathscr{A},m')$  satisfying that the following inequality holds:
\begin{equation*}
    \vert m-m' \vert < \epsilon,
\end{equation*}
admits an isomorphism of Banach algebras $(\mathscr{A},m)\rightarrow (\mathscr{A},m')$.   
\end{defi*}
As it turns out, the perturbation theory of $\mathscr{A}$ is also classified by its Hochschild cohomology. For instance, it can be shown that  if $\operatorname{HH}^2_{\operatorname{T}}(\mathscr{A})=0$ and $\operatorname{HH}^3_{\operatorname{T}}(\mathscr{A})$ is a Hausdorff space, then $\mathscr{A}$ is a stable Banach algebra (\emph{cf.} \cite[I.2.7]{helemskii2012homology}).\bigskip

Furthermore, this flavor of Hochschild cohomology can also be taken with coefficients in any $\mathscr{A}^e$-module and, in analogy with the previous discussion,  Hochschild cohomology with coefficients can be used to study the  perturbation theory of left $\mathscr{A}$-modules. Indeed, let $\mathscr{M}$ be a Banach left $\mathscr{A}$-module. Then the action $a:\mathscr{A}\widehat{\otimes}_K\mathscr{M}\rightarrow \mathscr{M}$ determines a unique representation homomorphism:
\begin{equation*}
   \varphi_a: \mathscr{A}\rightarrow \underline{\Hom}_K(\mathscr{M},\mathscr{M}),
\end{equation*}
which we may regard as an element in $\mathcal{L}(\mathscr{A},\mathscr{M})^1:=\underline{\Hom}_K(\mathscr{A},\underline{\Hom}_K(\mathscr{M},\mathscr{M}))$. As every left $\mathscr{A}$-module structure on $\mathscr{M}$ is determined by a representation homomorphism $\varphi_{b}\in\mathcal{L}(\mathscr{A},\mathscr{M})^1$, we can again make the following definition:
\begin{defi*}
Let $(\mathscr{M},\varphi_a)$ be an $\mathscr{A}$-module. We say $(\mathscr{M},\varphi_a)$ is stable if there is some $\epsilon>0$ such that any other $\mathscr{A}$-module $(\mathscr{M},\varphi_b)$ satisfying the inequality:
\begin{equation*}
    \vert \varphi_a-\varphi_b\vert < \epsilon,
\end{equation*}
admits  an $\mathscr{A}$-linear isomorphism $(\mathscr{M},\varphi_a)\rightarrow (\mathscr{M},\varphi_b)$.  
\end{defi*}
Again, this phenomenon is captured by the Hochschild cohomology groups. For instance, any Banach $\mathscr{A}$-module $\mathscr{M}$ satisfying:
\begin{equation*}
\operatorname{HH}_T^1(\mathscr{A},\underline{\Hom}_K(\mathscr{M},\mathscr{M}))=\operatorname{HH}_T^2(\mathscr{A},\underline{\Hom}_K(\mathscr{M},\mathscr{M}))=0,
\end{equation*}
is a stable Banach $\mathscr{A}$-module (\emph{cf.} \cite[III.4.8.]{helemskii2012homology}).\bigskip

Going back to our arithmetic setup, one can now see how this type of phenomenon could have a profound impact on the theory of $\wideparen{\D}$-modules on smooth rigid analytic spaces, and in particular on the theory of $p$-adic differential equations. The reasoning is that, if one can show that a certain differential module defined by a system of $p$-adic differential equations is stable, then it is possible to replace it by a small perturbation which is easier to approach. This idea is at the core of perturbation theory, and it is the guiding principle of what we want to achieve.\bigskip

Let us now discuss the contents of the paper, and give an overview of the constructions: Let $K$ be a complete non-archimedean extension of $\mathbb{Q}_p$, and let $X$ be a smooth, separated rigid analytic space over $K$. We let $\wideparen{\D}_X$ be its sheaf of infinite order differential operators, as defined by Ardakov-Wadsley in \cite{ardakov2019}.  The goal of this paper is developing a formalism of Hochschild (co)-homology for $\wideparen{\D}_X$-modules. In order to do this, we will rely on the homological machinery of quasi-abelian categories and sheaves of Ind-Banach spaces developed by A. Bode in \cite{bode2021operations}. In order to keep this introduction as short as possible, we will abstain from discussing these concepts here. A straightforward introduction to these topics  can be found in Sections \ref{Section background on sheaves of Ind-Banach spaces} and \ref{Section sheaves of Ind-Banach spaces}.  We make special mention of the tensor product of sheaves of Ind-Banach spaces, which we denote by $-\overrightarrow{\otimes}_K-$, and of the inner homomorphism functor, which  we denote by $\underline{\mathcal{H}om}(-,-)$. These functors endow $\operatorname{Shv}(X,\Indban)$, the category of sheaves on $X$ with values on Ind-Banach spaces over $K$, with a closed symmetric monoidal structure, which shares many properties with the tensor product and sheaves of homomorphisms of sheaves of $K$-vector spaces. In particular, it is possible to define sheaves of Ind-Banach algebras, and categories of modules attached to them. Their construction and basic properties can be found in Proposition \ref{prop closed symmetric monoidal structure on sheaves of Indban} and the discussion above. The derived functors of these operations are thoroughly studied in Section \ref{Section derived tensor-hom 2}.\bigskip

Before we start, let us mention that some of the contents of the paper hold in greater generality than the one stated in this introduction. In particular, many of the constructions below can be adapted for arbitrary Lie algebroids other than the tangent sheaf, and it is the belief of the author that there should be an analog of these results for twisted differential operators. Given the relevance of  twisted differential operators in the theory of locally analytic representations of $p$-adic Lie groups (\emph{cf.} the Beilinson-Bernstein correspondence in \cite{ardakov2021equivariant}), we have decided to work in this more general setting. Again, for the sake of brevity, we will only cover the case of $\wideparen{\D}$-modules in this introduction, but one can check the body of the paper for the concrete statements.\bigskip 

Let us start by clarifying what we mean by Hochschild (co)-homology: We define the sheaf of enveloping algebras of $\wideparen{\D}_X$ as the following sheaf of Ind-Banach algebras:
\begin{equation*}
    \wideparen{\D}_X^e:= \wideparen{\D}_X\overrightarrow{\otimes}_K\wideparen{\D}_X^{\op}.
\end{equation*}
Notice that the category of left Ind-Banach $\wideparen{\D}_X^e$-modules  is canonically isomorphic to the category of $\wideparen{\D}_X$-bimodules. In particular, $\wideparen{\D}_X$ is canonically a left and right  $\wideparen{\D}_X^e$-module.   Choose a chain complex of left $\wideparen{\D}_X^e$-modules $\mathcal{M}\in \operatorname{D}(\wideparen{\D}_X^e)$. In analogy with the previous discussion, one could be tempted to define the Hochschild homology and cohomology of $\wideparen{\D}_X$ with coefficients in $\mathcal{M}$ as the following complexes:
\begin{equation*}   \mathcal{HH}_{\bullet}^{naive}(\wideparen{\D}_X,\mathcal{M})=\wideparen{\D}_X\overrightarrow{\otimes}^{\mathbb{L}}_{\wideparen{\D}_X^e}\mathcal{M}, \quad \mathcal{HH}^{\bullet}_{naive}(\wideparen{\D}_X,\mathcal{M})=R\underline{\mathcal{H}om}_{\wideparen{\D}_X^e}(\wideparen{\D}_X,\mathcal{M}).
\end{equation*}
However, one can easily see that this is not the correct way to go. The first problem we encounter when using this definition is that $\wideparen{\D}_X^e$ is rather poorly behaved. In particular, the sheafification process involved in taking tensor products of Ind-Banach spaces makes it hard to calculate sections of $\wideparen{\D}_X^e$, and  its algebraic properties are  complicated to elucidate. Our solution is introducing the sheaf of bi-enveloping algebras $\wideparen{E}_X$. This is a sheaf of Ind-Banach algebras which is much better behaved, and such that there is a canonical triangulated functor from $\operatorname{D}(\wideparen{\D}_X^e)$ to $\operatorname{D}(\wideparen{E}_X)$.\bigskip

Let $X^2=X\times_{\Sp(K)}X$, and denote the projections by $p_i:X^2\rightarrow X$, for $i=1,2$. We define the sheaf of bi-enveloping algebras of $X$ as the following sheaf of Ind-Banach spaces on the product $X^2$:
\begin{equation*}
    \wideparen{E}_X:=\wideparen{\D}_X \overrightarrow{\boxtimes}\wideparen{\D}_X^{\op}= \OX_{X^2}\overrightarrow{\otimes}_{p_1^{-1}\OX_X\overrightarrow{\otimes}_Kp_2^{-1}\OX_X} \left(p_1^{-1}\wideparen{\D}_X\overrightarrow{\otimes}_Kp_2^{-1}\wideparen{\D}_X^{\op}\right).
\end{equation*}
At first, it may seem like the situation is not much better than before. However, $\wideparen{E}_X$ satisfies surprisingly good properties:
\begin{intteo}\label{teo introduction properties of EX}
The sheaf of bi-enveloping algebras satisfies the following properties:
\begin{enumerate}[label=(\roman*)]
    \item There is a unique Ind-Banach algebra structure on $\wideparen{E}_X$ extending the canonical ones on $\OX_{X^2}$ and $p_1^{-1}\wideparen{\D}_X\overrightarrow{\otimes}_Kp_2^{-1}\wideparen{\D}_X^{\op}$.
    \item  The pushforward along the diagonal $\Delta:X\rightarrow X^2$ induces a functor:
    \begin{equation*}   \Delta_*^E:\Mod_{\Indban}(\wideparen{\D}_X^e)\rightarrow \Mod_{\Indban}(\wideparen{E}_X),
    \end{equation*}
    which is a strongly right exact functor of quasi-abelian categories.
    \item If $\mathcal{T}_{X/K}$ is free, there is an isomorphism of sheaves of  Ind-Banach algebras:
    \begin{equation*}
        \wideparen{\mathbb{T}}:\wideparen{\D}_{X^2}\rightarrow \wideparen{E}_X.
    \end{equation*}
    \item There are mutually inverse equivalences of quasi-abelian categories:
    \begin{equation*}
        \operatorname{S}:\Mod_{\Indban}(\wideparen{\D}_{X^2})\leftrightarrows \Mod_{\Indban}(\wideparen{E}_{X}):\operatorname{S}^{-1},
    \end{equation*}
    which we call the side-switching equivalence for bimodules.  
\end{enumerate}
If the conditions of $(iii)$ hold, then $\operatorname{S}$ agrees with the extension of scalars along $\wideparen{\mathbb{T}}$. 
\end{intteo}
\begin{proof}
Item $(i)$ is Theorem \ref{teo product on E}, $(ii)$ is Proposition \ref{prop extension functor}, $(iii)$ is  Proposition \ref{prop side changing for bimodules}.  The statement at the end of the theorem is shown in Proposition \ref{prop equivalence of pullback and side-changing sheaves}.
\end{proof}
Let us point out that $(ii)$ in the preceding theorem requires that $X$ is a smooth  and separated rigid analytic space. More concretely, the functor can be defined whenever $X$ is smooth, but strong right exactness requires that $X$ is separated. This functor is  called the \emph{extension functor}, and can be roughly described as a pushforward composed with an extension of scalars. It admits a version for left hearts, which we denote by $\Delta_* ^{I(E)}$,  and this version admits a left derived functor:
\begin{equation*}
    \EDelta:\operatorname{D}(\wideparen{\D}_X^e)\rightarrow \operatorname{D}(\wideparen{E}_X),
\end{equation*}
which we call the \emph{derived extension functor}. However, the identity:
\begin{equation*}
    I(\Delta_*^E(-))=\Delta_* ^{I(E)}(I(-)),
\end{equation*}
is generally false, so  we do not have  $\operatorname{H}^0(\EDelta(-))=I(\Delta_*^E(-))$ in general. This is due to the fact that the definition of both $\Delta_*^E$ and $\Delta_* ^{I(E)}$ involves a tensor product, and tensor products do not behave well under $I$ (\emph{cf}. Section  \ref{Section derived tensor-hom 2}). However, we will later define a subcategory of $\operatorname{D}(\wideparen{\D}_X^e)$ for which this does not happen.\\
We can now define Hochschild (co)-homology for sheaves of Ind-Banach $\wideparen{\D}_X$-modules:
\begin{defi*}
For any  $\mathcal{M}^{\bullet}\in\operatorname{D}(\wideparen{\D}_X^e)$ we define the following objects: 
\begin{enumerate}[label=(\roman*)]
    \item The inner Hochschild cohomology complex of $\wideparen{\D}_X$ with coefficients in $\mathcal{M}^{\bullet}$ is the following complex in $\operatorname{D}(\operatorname{Shv}(X,\Indban))$:
\begin{equation*}
  \mathcal{HH}^{\bullet}(\wideparen{\D}_X,\mathcal{M}^{\bullet}):=\Delta^{-1}R\underline{\mathcal{H}om}_{\wideparen{E}_X}(\EDelta\wideparen{\D}_X,\EDelta\mathcal{M}^{\bullet}),
\end{equation*}
\item The Hochschild cohomology complex of $\wideparen{\D}_X$ with coefficients in $\mathcal{M}^{\bullet}$ is the following complex in $\operatorname{D}(\Indban)$:
\begin{equation*}
    \operatorname{HH}^{\bullet}(\wideparen{\D}_X,\mathcal{M}^{\bullet})=R\Gamma(X,\mathcal{HH}^{\bullet}(\wideparen{\D}_X,\mathcal{M}^{\bullet})).
\end{equation*}
\item
The inner Hochschild homology complex of $\wideparen{\D}_X$ with coefficients in $\mathcal{M}^{\bullet}$ is the following complex in $\operatorname{D}(\operatorname{Shv}(X,\Indban))$:
\begin{equation*}
  \mathcal{HH}_{\bullet}(\wideparen{\D}_X,\mathcal{M}^{\bullet}):=\Delta^{-1}\left(\EDelta\wideparen{\D}_X\overrightarrow{\otimes}_{\wideparen{E}_X}^{\mathbb{L}}\EDelta\mathcal{M}^{\bullet}\right),
\end{equation*}
\item The Hochschild homology complex of $\wideparen{\D}_X$ with coefficients in $\mathcal{M}^{\bullet}$ is the following complex in $\operatorname{D}(\Indban)$:
\begin{equation*}
    \operatorname{HH}_{\bullet}(\wideparen{\D}_X,\mathcal{M}^{\bullet})=R\Gamma(X,\mathcal{HH}_{\bullet}(\wideparen{\D}_X,\mathcal{M}^{\bullet})).
\end{equation*}
\end{enumerate}    
\end{defi*}
Most of the material in the paper will be centered around giving an explicit description of these complexes and their cohomology groups. As a first step, we compose the derived extension functor $\EDelta$ with the side-switching operator $\operatorname{S}^{-1}$ from Theorem \ref{teo introduction properties of EX}, arriving at the  \emph{derived immersion functor}:
\begin{equation*}
\Ifunct:=\operatorname{S}^{-1}\circ\EDelta:\operatorname{D}(\wideparen{\D}_X^e)\rightarrow \operatorname{D}(\wideparen{\D}_{X^2}),
\end{equation*}
as well as the immersion functor $\Delta_*^{\operatorname{S}}:=\operatorname{S}^{-1}\circ \Delta_*^E$. As shown in the body of the paper, we may use the immersion functor to obtain a new expression for the \emph{inner Hochschild cohomology complex}. In particular, for $\mathcal{M}\in \operatorname{D}(\wideparen{\D}_X^e)$ we have:
\begin{equation*}
    \mathcal{HH}^{\bullet}(\wideparen{\D}_X,\mathcal{M})=\Delta^{-1}R\underline{\mathcal{H}om}_{\wideparen{\D}_{X^2}}(\Ifunct\wideparen{\D}_X,\Ifunct\mathcal{M}),
\end{equation*}
and the analogous statement holds for the \emph{inner Hochschild homology complex}.\\
Again, it may seem like we are not gaining much with this change. However, as
$\wideparen{\D}_{X^2}$ is the sheaf of infinite order differential operators on $X^2$, this simplification allows us to apply the machinery of \cite{bode2021operations} to our study of Hochschild (co)-homology. In particular, as $\operatorname{S}^{-1}$ does not change the supports of the sheaves, it follows that $\Ifunct\wideparen{\D}_X$, and $\Ifunct\mathcal{M}$ have cohomology supported on the diagonal. Hence, the strategy is using Kashiwara's equivalence along the closed immersion $\Delta:X\rightarrow X^2$ to obtain an explicit calculation of the inner Hochschild (co)-homology complexes. Namely, as showcased in \cite{bode2021operations}, a very general category in which Kashiwara's equivalence works well is the category of $\mathcal{C}$-complexes on $X$ (cf. Definition  \ref{defi C complexes}). Roughly speaking, the category of $\mathcal{C}$ complexes of $X$ is a $p$-adic analog of the category of bounded complexes with coherent cohomology from the classical theory of $\D$-modules.\\
Let $\operatorname{D}_{\mathcal{C}^{\Delta}}(\wideparen{\D}_{X^2})$ denote the category of $\mathcal{C}$-complexes on $X^2$ with cohomology supported on $\Delta(X)$. The next step is defining an adequate category of Ind-Banach $\wideparen{\D}_X^e$-modules which is general enough to contain important examples, and is mapped to $\operatorname{D}_{\mathcal{C}^{\Delta}}(\wideparen{\D}_{X^2})$ under the derived immersion functor. This is the category in which we will be able to give complete calculations of the inner Hochschild (co)-homology.\bigskip

In order to achieve this, we need to be able to understand what conditions on a  $\wideparen{\D}_X^e$-module $\mathcal{M}$ ensure that $\Delta_*^{\operatorname{S}}(\mathcal{M})$ is a co-admissible $\wideparen{\D}_{X^2}$-module. Instead of pursuing this goal directly, we will first take a small detour, and use Theorem 
\ref{teo introduction properties of EX} to dig a bit deeper into the properties of $\wideparen{E}_X$. Namely, the fact that $\wideparen{E}_X$ and $\wideparen{\D}_{X^2}$ are locally isomorphic  suggests that the relevant features of $\wideparen{\D}_{X^2}$ should also be present in $\wideparen{E}_X$. In particular, there should be well-behaved categories of co-admissible $\wideparen{E}_X$-modules and $\mathcal{C}$-complexes. The following theorem confirms this intuition: 
\begin{intteo}
There is a well-defined category of co-admissible $\wideparen{E}_X$-modules, which we denote by $\mathcal{C}(\wideparen{E}_X)$. It is an abelian full subcategory of $\Mod_{\Indban}(\wideparen{E}_X)$, and the side-switching equivalence restricts to an equivalence of abelian categories:
        \begin{equation*}
            \operatorname{S}:\mathcal{C}(\wideparen{\D}_{X^2})\leftrightarrows \mathcal{C}(\wideparen{E}_X):\operatorname{S}^{-1}.
        \end{equation*}
 The analogous statements hold for $\mathcal{C}$-complexes of $\wideparen{E}_X$-modules, and we call the corresponding category $\operatorname{D}_{\mathcal{C}}(\wideparen{E}_X)$. Furthermore assume $X=\Sp(A)$ is an affinoid space which admits an étale map $X\rightarrow \mathbb{A}^n_K$. Then the following holds:
    \begin{enumerate}[label=(\roman*)]
        \item  $\wideparen{E}_X(X^2)=\wideparen{\D}_X(X)^e$.
        \item Let $\mathscr{X}:=\Spf(\mathcal{A})$ be an affine formal model of $X$ with a finite-free $(\mathcal{R},\mathcal{A})$-Lie lattice $\mathcal{T}$ of $\mathcal{T}_{X/K}(X)$. Then the canonical inverse limit:
        \begin{equation*}
            \wideparen{\D}_X(X)^e=\varprojlim_n \widehat{U}(\pi^n\mathcal{T})_K^e,   
        \end{equation*}
        is a Fréchet-Stein presentation of $\wideparen{\D}_X(X)^e$.
        \item There is an equivalence of abelian categories:
        \begin{equation*}
            \operatorname{Loc}(-):\mathcal{C}(\wideparen{\D}_X(X)^e)\leftrightarrows\mathcal{C}(\wideparen{E}_X):\Gamma(X,-).
        \end{equation*}
        \item Let $\mathcal{M}\in \Mod_{\Indban}(\wideparen{\D}_X(X)^e)$ satisfy that $\mathcal{M}$ is co-admissible as a left or right $\wideparen{\D}_X(X)$-module. Then $\mathcal{M}$ is a co-admissible $\wideparen{\D}_X(X)^e$-module.         
    \end{enumerate}  
\end{intteo}
\begin{proof}
The first part of the theorem is a condensed version of the contents of Section \ref{section co-admis over bienvelop 2}. See Theorem \ref{teo equivalence of co-admissible modules under side-changing} and Proposition \ref{prop equivalence of side-changing at the level of C-complexes} for details. For the second part, statements $(i)$ and $(ii)$  are shown in Corollary \ref{coro frechet-stein presentation of bi-enveloping algebra}. Item $(iii)$ is Proposition \ref{prop properties of Loc}, and the last part is a special case of Theorem \ref{teo embedding co-admissible bimodules}.
\end{proof}
This theorem allows us to deal with co-admissibility problems directly at the level of the bi-enveloping algebra $\wideparen{E}_X$, which is a much more natural 
way of dealing with $\wideparen{\D}_X$-bimodules.\bigskip

Now that this final piece of machinery has been put into place, we can finally define the categories of $\wideparen{\D}_X$-bimodules we were looking for: These are the categories of co-admissible diagonal $\wideparen{\D}_X$-bimodules $\mathcal{C}_B^{\operatorname{Bi}}(\wideparen{\D}_X)_{\Delta}$ (cf. Definition \ref{defi co-admissible diagonal bimodule}), and more generally diagonal $\mathcal{C}$-complexes $\operatorname{D}^{\operatorname{Bi}}_{\mathcal{C}}(\wideparen{\D}_X)_{\Delta}$ (cf. Definition \ref{coro immersion Theorem of sheaves of diagonal co-admissible bimodules}). Roughly speaking,  $\mathcal{C}_B^{\operatorname{Bi}}(\wideparen{\D}_X)_{\Delta}$ is the full subcategory of $\Mod_{\Indban}(\wideparen{\D}_X^e)$ given by the modules which are co-admissible as both $\wideparen{\D}_X$-modules and $\wideparen{\D}_X^{\op}$-modules plus an additional compatibility condition. Similarly, $\operatorname{D}^{\operatorname{Bi}}_{\mathcal{C}}(\wideparen{\D}_X)_{\Delta}$  is the full subcategory of $\operatorname{D}(\wideparen{\D}_X^e)$ given by the complexes with cohomology in $\mathcal{C}_B^{\operatorname{Bi}}(\wideparen{\D}_X)_{\Delta}$ plus some finiteness condition. Let us give an overview of the most important features of these categories:
\begin{intteo}
 The category $\mathcal{C}_B^{\operatorname{Bi}}(\wideparen{\D}_X)_{\Delta}$ satisfies the following properties:
 \begin{enumerate}[label=(\roman*)]
     \item $\mathcal{C}_B^{\operatorname{Bi}}(\wideparen{\D}_X)_{\Delta}$ is abelian and $\operatorname{D}^{\operatorname{Bi}}_{\mathcal{C}}(\wideparen{\D}_X)_{\Delta}$ is  triangulated.
     \item  Every $\mathcal{M} \in\mathcal{C}_B^{\operatorname{Bi}}(\wideparen{\D}_X)_{\Delta}$ satisfies that there are canonical quasi-isomorphisms:
     \begin{equation*}
         \EDelta(\mathcal{M})\cong \Delta_*^{I(E)}(\mathcal{M})\cong I(\Delta_*^E(\mathcal{M})).
     \end{equation*}
     In particular, objects in $\mathcal{C}_B^{\operatorname{Bi}}(\wideparen{\D}_X)_{\Delta}$ are $\EDelta$-acyclic and $\Ifunct$-acyclic.
     \item The immersion functor induces an embedding of triangulated categories:
    \begin{equation*}       \Ifunct:\operatorname{D}^{\operatorname{Bi}}_{\mathcal{C}}(\wideparen{\D}_X)_{\Delta}\rightarrow \operatorname{D}_{\mathcal{C}^{\Delta}}(\wideparen{\D}_{X^2}),
    \end{equation*}
    with partial inverse $\Delta^{-1}\circ \operatorname{S}$. This embedding restricts to an exact embedding: $ \Delta_*^{\operatorname{S}}:\mathcal{C}_B^{\operatorname{Bi}}(\wideparen{\D}_X)_{\Delta}\rightarrow \mathcal{C}(\wideparen{\D}_{X^2})_{\Delta}$,  whose image is a Serre subcategory.  
     \item $\wideparen{\D}_X$ is an object of $\mathcal{C}_B^{\operatorname{Bi}}(\wideparen{\D}_X)_{\Delta}$, and we have $\Ifunct(\wideparen{\D}_X)=\Delta_+\OX_X$.     
 \end{enumerate}
\end{intteo}
\begin{proof}
Statement $(i)$ is shown in Propositions \ref{prop co-admissible diag is abelian} and \ref{prop diagonal c comp are triangulated}, and claim $(ii)$  is Lemma \ref{lemma aciclycity of diagonal co-ad for extension fucnt}. Property $(iii)$ is a consequence of Proposition \ref{prop immersion functor and C-complexes} and the previous claim. Finally, $(iv)$ is a combination of Propositions \ref{prop sheaf of completed differential operators is diagonal bimodule} and \ref{prop image of D under side changing}.
\end{proof}
Hence, the category of diagonal $\mathcal{C}$-complexes satisfies the properties we want with respect to the immersion functor, and is general enough to contain $\wideparen{\D}_X$. This theorem, together with the version of Kashiwara's equivalence shown in Proposition  \ref{prop derived internal hom and kashiwara} allows us to show the following theorem: 
\begin{intteo}
Let $X$ be a smooth and separated rigid space. For any diagonal $\mathcal{C}$-complex
$\mathcal{M}^{\bullet}\in \operatorname{D}^{\operatorname{Bi}}_{\mathcal{C}}(\wideparen{\D}_X)_{\Delta}$ we have the following identity in $\operatorname{D}(\operatorname{Shv}(X,\Indban))$:
\begin{equation*}
    \mathcal{HH}^{\bullet}(\wideparen{\D}_X,\mathcal{M}^{\bullet})=\Omega_X\overrightarrow{\otimes}^{\mathbb{L}}_{\wideparen{\D}_X}\left(\Omega_X^{-1}\overrightarrow{\otimes}_{\OX_X}\mathcal{M}^{\bullet}[\mathcal{I}_{\Delta}]\right)[-\operatorname{dim}(X)].
\end{equation*}
In particular, we have the following identification in $\operatorname{D}(\operatorname{Shv}(X,\Indban))$:
\begin{equation*}
    \mathcal{HH}^{\bullet}(\wideparen{\D}_X,\wideparen{\D}_X)=\Omega_{X}^{\bullet}.
\end{equation*}
\end{intteo}
\begin{proof}
This is Theorem \ref{teo explicit computation of hochschild cohomology}.
\end{proof}
Thus, the de Rham complex of $X$ agrees with the \emph{inner Hochschild complex} of $\wideparen{\D}_X$. As a consequence, we may calculate the Hochschild cohomology complex of $\wideparen{\D}_X$ via the following formula:
\begin{equation*}
    \operatorname{HH}^{\bullet}(\wideparen{\D}_X)=R\Gamma(X,\Omega_{X/K}^{\bullet}).
\end{equation*}
A priory, it could seem like this is precisely the complex computing the de Rham cohomology of $X$. However, the situation is not as simple. Indeed, in both cases, the complex we are interested in arises as the hypercohomology of $\Omega_{X/K}^{\bullet}$. However, in the case of the de Rham cohomology, we regard $\Omega_{X/K}^{\bullet}$ as an object in $\operatorname{D}(\operatorname{Shv}(X,\operatorname{Vect}_K))$, and in the case of Hochschild cohomology we regard it as an object in $\operatorname{D}(\operatorname{Shv}(X,\Indban))$. Thus, it is clear that both complexes are not the same. However, the similarity between both constructions hints at the fact that both cohomology theories are somehow related.\bigskip

In order to find this relation, there are some technical difficulties that need to be overcome. First, as $\Indban$ is a quasi-abelian category, the cohomology groups of $\operatorname{HH}^{\bullet}(\wideparen{\D}_X)$ are objects of the left heart $LH(\Indban)$, which are not necessarily contained in $\Indban$. Furthermore, even if  the cohomology groups of $\operatorname{HH}^{\bullet}(\wideparen{\D}_X)$ are Ind-Banach spaces,  there is still not a direct way of relating them to $K$-vector spaces. In particular, unlike locally convex and bornological spaces, Ind-Banach spaces do not have an underlying $K$-vector space. However, by making use of some homological algebra and $p$-adic functional analysis, it is possible to obtain suitable functors at the level of the derived categories. Namely, we have an exact and fully faithful functor:
\begin{equation*}
    J:\widehat{\mathcal{B}}c_K\rightarrow \mathcal{B}c_K,
\end{equation*}
defined in Proposition \ref{prop properties of J}, and an exact and faithful forgetful functor:
\begin{equation*}
    \operatorname{Forget}(-):\mathcal{B}c_K\rightarrow \operatorname{Vect}_K.
\end{equation*}
Furthermore, these functors can be extended to the associated categories of sheaves, and the extensions are still exact (\emph{cf.}  Lemma \ref{lemma exactness of extension to sheaves}). These functors can be used to prove Corollary \ref{coro recovering de Rham cohomology (without topology) from Hochschild cohomology}, where we show the following formula:
\begin{equation*}
    \operatorname{H}_{\operatorname{dR}}^{\bullet}(X)=\operatorname{Forget}(J\operatorname{HH}^{\bullet}(\wideparen{\D}_X)).
\end{equation*}
Along the way, we will develop the formalism of Hochschild homology, which is formally very similar to that of Hochschild cohomology. Our final result in this direction is Theorem \ref{teo expression of Hochschild homology sheaves}, where we give an explicit formula for the inner Hochschild homology complex of any diagonal $\mathcal{C}$-complex on $X$. Specializing to the case of $\wideparen{\D}_X$, we arrive at the following formula:
\begin{equation*}
    \mathcal{HH}_{\bullet}(\wideparen{\D}_X,\wideparen{\D}_X)=\mathcal{HH}^{\bullet}(\wideparen{\D}_X,\wideparen{\D}_X)[2\operatorname{dim}(X)],
\end{equation*}
which yields a $p$-adic version of the classical duality between Hochschild cohomology and homology of the algebra of differential operators on a smooth complex algebraic variety. Furthermore, our calculation of the inner Hochschild cohomology complex of $\wideparen{\D}_X$ yields the following formula:
\begin{equation*}
    \operatorname{HH}_{n}(\wideparen{\D}_X,\wideparen{\D}_X)=\operatorname{H}^{2\operatorname{dim}(X)-n}\left(R\Gamma(X,\Omega_{X}^{\bullet})\right).
\end{equation*}
We point out that similar formulas have been obtained in different contexts. In particular, this result was first obtained for affine algebraic varieties over a field of characteristic zero, Stein complex manifolds, and compact $C^{\infty}$-manifolds by Wodzicki in $\textnormal{\cite[Theorem 2]{wodzicki1987cyclic}}$. The  result for affine algebraic varieties was shown independently by Kassel and Mitschi (see \cite[pp. 19]{pseudodiff}),  and for smooth real manifolds by Brylinski-Getzler in \cite[pp. 18]{pseudodiff}. As a consequence, we obtain an alternative proof of the well-known fact that $\operatorname{H}^n_{\operatorname{dR}}(X)=0$ for $n>2\operatorname{dim}(X)$.
\subsection*{Future and related work}
This is the first part of a series of papers \cite{p-adicCheralg}, \cite{p-adicCatO}, \cite{HochschildDmodules2}, \cite{cychomDmod} in which we convey a comprehensive study of the deformation theory of the algebras $G\ltimes \wideparen{\D}_X(X)$, where $G$ is a finite group acting on a smooth Stein space $X$. In particular, the results of this paper are further extended in \cite{HochschildDmodules2}, where we focus on the case where $X$ is a smooth Stein space. Our interest in Stein spaces stems from the fact that they behave like affinoid spaces with respect to sheaf cohomology, but are better behaved analytically. In particular, if  $X$ is a smooth Stein space, then the de Rham complex:  
\begin{equation*}
    \Omega_{X/K}^{\bullet}(X):= \left( 0\rightarrow \OX_X(X)\rightarrow \Omega_{X/K}^1(X)\rightarrow \cdots \rightarrow \Omega_{X/K}^{\operatorname{dim}(X)}\rightarrow 0 \right),
\end{equation*}
is a strict complex of nuclear Fréchet spaces. It can be shown that this implies that $\operatorname{HH}^{\bullet}(\wideparen{\D}_X(X),\wideparen{\D}_X(X))$ is also a strict complex of nuclear Fréchet spaces, and this property descends to the Hochschild cohomology spaces. In this setting, we can give a more tangible meaning to Hochschild cohomology. In particular, it is shown in 
\cite{HochschildDmodules2} that for every $n\geq 0$ the vector space underlying 
$\operatorname{HH}^{n}(\wideparen{\D}_X)$ classifies the $n$-th degree Yoneda extensions of $I(\wideparen{\D}_X(X))$  by $I(\wideparen{\D}_X(X))$  as $I(\wideparen{\D}_X(X))^e$-modules.\bigskip

Following this line of thought, we then show that we can use the bar resolution to calculate the Hochschild cohomology of $ \wideparen{\D}_X$ in low degrees, and this allows us to compare  our notion of Hochschild cohomology  with J. Taylor's Hochschild cohomology for locally convex algebras (\emph{cf}. \cite{helemskii2012homology}, \cite{taylor1972homology}).  The advantage of this approach is that we obtain explicit interpretations of many Hochschild cohomology spaces in terms of algebraic invariants of $\wideparen{\D}_X(X)$. Furthermore, our results reinforce the heuristic that the spaces $\operatorname{HH}^{\bullet}(\wideparen{\D}_X(X))$ should be regarded as  $p$-adic analytic versions of the classical Hochschild cohomology groups of associative algebras. For instance, the first Hochschild cohomology space $\operatorname{HH}^{1}(\wideparen{\D}_X(X))$ parametrizes the bounded outer derivations of 
$\wideparen{\D}_X(X)$, whereas the classical first Hochschild cohomology group parametrizes arbitrary outer derivations. Finally, we use these results to classify the infinitesimal deformations of $\wideparen{\D}_X(X)$ in certain situations.\bigskip

In \cite{p-adicCheralg} we develop a theory of $p$-adic Cherednik algebras on smooth rigid analytic spaces with a finite group action. These algebras can be regarded as a $p$-adic analytic version of the sheaves of Cherednik algebras constructed by P. Etingof in \cite{etingof2004cherednik}. In particular, let $X$ be a smooth rigid analytic space with an action of a finite group $G$ satisfying some mild technical conditions. The family of $p$-adic Cherednik algebras associated to the action of $G$ on $X$ is a family of sheaves of complete bornological algebras on the quotient $X/G$  which generalizes the skew group algebra $G\ltimes \wideparen{\D}_X$. In particular, given a $p$-adic Cherednik algebra $\mathcal{H}_{t,c,\omega}$, its restriction to the smooth locus of $X/G$ is canonically isomorphic to $G\ltimes \wideparen{\D}_X$. Furthermore,  all $p$-adic Cherednik algebras are locally Fréchet-Stein, and their isomorphism classes are parameterized by geometric invariants associated to the action of $G$ on $X$.\\
In the algebraic setting, this parameter space has a deep deformation-theoretic meaning. Namely, if $X$ is a smooth affine $K$-variety with an action of a finite group $G$, then it can be shown that the family of Cherednik algebras associated to the action of $G$ on $X$  parametrizes all the infinitesimal deformations of the skew-group algebra $G\ltimes \D_X(X)$. We hope that our methods can be extended to show that this phenomena is also present in the rigid analytic setting. The paper \cite{p-adicCheralg} is devoted to laying the foundations of the theory of $p$-adic Cherednik algebras. In particular, we provide a construction of these algebras, study their algebraic properties, and obtain a classification theorem.\bigskip

The study of $p$-adic Cherednik algebras is continued in  \cite{p-adicCatO}, 
where we study the representation theory of $p$-adic Cherednik algebras, focusing on the case where $X$ is the analytification of an affine space, and $G$ acts linearly on $X$. The algebras arising from such an action are called $p$-adic rational Cherednik algebras, and their representation theory shares many features with that of the  $p$-adic Arens-Michael envelope $\wideparen{U}(\mathfrak{g})$ of a finite dimensional reductive Lie algebra $\mathfrak{g}$ studied in \cite{Schmidt2010VermaMO}. In particular, to any $p$-adic rational Cherednik algebra one can attach a category $\wideparen{\OX}$, which is a highest weight category whose simple objects are in bijection with the irreducible representations of $G$.\bigskip

Finally, in our under construction paper \cite{cychomDmod} we aim at employing several techniques from $p$-adic Hodge theory to further our study of the perturbation theory of integrable connections on smooth and separated rigid analytic spaces. Roughly speaking, our strategy here is the following: The homological nature of our theory fits in well with the Riemann-Hilbert correspondence for co-admissible $\wideparen{\D}$-modules constructed by F. Wiersig in \cite{W1},\cite{W2}, and we expect that this correspondence can be used to obtain a well-behaved deformation theory for co-admissible $\wideparen{\D}_X$-modules. In particular, working in the pro-étale topology,  \cite{W2} defines a pair of sheaves of Ind-Banach algebras $\mathbb{B}_{\operatorname{pdR}}^{\dagger}$ and $\OX\mathbb{B}_{\operatorname{pdR}}^{\dagger}$, which can be regarded as decompletions of the almost de Rham period and structure sheaves respectively. These sheaves can be used to define pro-étale versions of the  de Rham and solution functors for $\wideparen{\D}$-modules, which in particular turn co-admissible $\wideparen{\D}$-modules (and more generally $\mathcal{C}$-complexes) into complexes of $\mathbb{B}_{\operatorname{pdR}}^{\dagger}$-modules. In this setting, the deformations we are interested in appear via 
studying the deformations of  sheaves of Ind-Banach modules over $\mathbb{B}_{\operatorname{pdR}}^{\dagger}$,   and then applying the so called reconstruction functor  (a composition of the derived pushforward to the analytic topology  with certain operator). Heuristically, we want to study the deformations of the systems of differential equations by studying their sheaves of solutions.\\

Additionally, our line of work is specially synergetic with the works of  Cortiñas-Meyer-Mukherjee, in which the authors develop a version of analytic cyclic homology for non-archimedean bornological algebras (\emph{cf.} \cite{non-archborn}, \cite{non-archcyclic}, \cite{loccyc}). We expect that our understanding of the Hochschild homology of $\wideparen{\D}_X$ can be used to   calculate the cyclic, periodic cyclic, and analytic cyclic homology of $\wideparen{\D}_X(X)$ in many interesting situations. Namely,  we aim to obtain complete calculations at least whenever $X$ is a $p$-adic Stein space, such as the Drinfeld upper half plane.\\
To motivate this, we recall that the  main reason behind A. Connes' definition of cyclic (co)-homology was its relation with $K$-theory in the context of differential geometry (\emph{cf.} \cite{Connes}). In the non-archimedean setting, there has been a push in recent years to obtain versions of (bivariant) $K$-theory  for non-archimedean analytic algebras. For instance, there are the works of Kerz-Saito-Tamme \cite{K-theory1}, \cite{K-theory2}, and D. Mukherjee \cite{bivK-theo}. Moreover, the non-archimedean bivariant K-theory from \cite{bivK-theo} comes equipped with a canonical map to non-archimedean analytic cyclic homology.  Thus, we expect that our results pave the way to studying the $K$-theory of $\wideparen{\D}_X(X)$, and allow for the introduction of $K$-theoretic techniques in the study of $\wideparen{\D}_X$-modules on rigid analytic spaces. Finally, let us mention that, in the archimedean setup, algebras of differential operators cannot be endowed with submultiplicative seminorms. Thus, this type of phenomenon is unique to the non-archimedean theory.
\subsection*{Structure of the paper}
We will now give an overview of the contents:\\
In Chapter \ref{Section background on sheaves of Ind-Banach spaces} we will provide the reader with the basic notions on quasi-abelian categories, focusing on the cases of complete bornological spaces $\widehat{\mathcal{B}}c_K$, and Ind-Banach spaces $\Indban$. We will describe the closed symmetric monoidal structures on  $\widehat{\mathcal{B}}c_K$, and $\Indban$, compare them via the dissection functor, and discuss their extension to the left heart. Finally, we define algebras in $\widehat{\mathcal{B}}c_K$, and $\Indban$ and study their categories of modules.\\
In Chapter \ref{Section sheaves of Ind-Banach spaces} we will continue with the background material by discussing the category of sheaves of Ind-Banach spaces on a rigid analytic space $X$. We deal with the extension of the closed symmetric monoidal structure on $\Indban$ to the category $\operatorname{Shv}(X,\Indban)$ of sheaves of Ind-Banach spaces on $X$, define pushforward and pullback along morphisms of rigid analytic spaces, and establish a theory of derived functors in this setting. As in the previous section, we finish the chapter by defining sheaves of Ind-Banach algebras, and studying their categories of modules.\\
We finish the preliminaries in Chapter \ref{Section background Lie algebroids}, where we give an overview of the theory of sheaves of Ind-Banach modules over the sheaf of hypercomplete enveloping algebras of a Lie algebroid. In particular,
we will define the category of sheaves of co-admissible modules, the category of $\mathcal{C}$-complexes, and study certain operations between them.\\
We will devote Chapter \ref{Chapter products of Lie algebroids} to the construction of the product of two Lie algebroids and the definition of the bi-enveloping algebra. In particular, if $X$ is a smooth rigid analytic space equipped with a Lie algebroid $\mathscr{L}$, we construct a Lie algebroid $\mathscr{L}^2$, and a sheaf of Ind-Banach algebras $\wideparen{E}(\mathscr{L})$ on $X^2=X\times X$ which are determined by the identities:
\begin{equation*}
    \wideparen{U}(\mathscr{L}^2)=\wideparen{U}(\mathscr{L})\overrightarrow{\boxtimes}\wideparen{U}(\mathscr{L}), \quad \wideparen{E}(\mathscr{L})=\wideparen{U}(\mathscr{L})\overrightarrow{\boxtimes}\wideparen{U}(\mathscr{L})^{\op},
\end{equation*}
and we finish the chapter by studying their properties.\\
We continue our study of the sheaf of complete bi-enveloping algebras in Chapter \ref{Chapter co-admissible modules over the bi-enveloping algebra}, where we study the category of sheaves of Ind-Banach spaces over the sheaf of complete bi-enveloping algebras. In particular, if $\mathscr{L}$ is a Lie algebroid on a smooth rigid analytic space $X$, we build a canonical equivalence of categories:
\begin{equation*}
    \Mod_{\Indban}(\wideparen{U}(\mathscr{L}^2))\leftrightarrows \Mod_{\Indban}(\wideparen{E}(\mathscr{L})),
\end{equation*}
thus obtaining a side-switching formalism for bimodules. We use this equivalence to obtain a category of sheaves of co-admissible $\wideparen{E}(\mathscr{L})$-modules, and $\mathcal{C}$-complexes, and show that the previous equivalence induces equivalences between the corresponding categories.\\
We conclude our study of the category of Ind-Banach $\wideparen{E}(\mathscr{L})$-modules in Chapter \ref{Chapter co-admissibility and the immersion functor}, where we specialize to the case $\mathscr{L}=\mathcal{T}_{X/K}$, and study the properties of the pushforward along the diagonal $\Delta:X\rightarrow X^2$. Namely, we identify certain subcategories of $\Mod_{\Indban}(\wideparen{\D}_X^e)$, which satisfy that for every $\mathcal{M}\in \Mod_{\Indban}(\wideparen{\D}_X^e)$ the sheaf of $\wideparen{E}_X$-modules $\Delta_*\mathcal{M}$ is co-admissible. In particular, we define the category of co-admissible diagonal bimodules, show that it contains $\wideparen{\D}_X$, and then give a derived generalization, which we call the category of diagonal $\mathcal{C}$-complexes on $X$.\\
We finish the paper with Chapter \ref{Chapter Hochschil (co)-homology}, where we calculate the inner Hochschild (co)-homology complex of every diagonal $\mathcal{C}$-complex on $X$. Additionally, we obtain a duality result between Hochschild cohomology and homology of diagonal $\mathcal{C}$-complexes, and a comparison result between the Hochschild cohomology of $\wideparen{\D}_X$ and the de Rham cohomology of $X$. 
\subsection*{Notation and conventions}
For the rest of the text, we let $K$ be be a complete non-archimedean extension of $\mathbb{Q}_p$, denote its valuation ring by $\mathcal{R}$, and choose a pseudo-uniformizer $\pi$. Following the conventions in \cite{ardakov2019}, for a $\mathcal{R}$-module $M$, we will often times write $M_K$ instead of $M\otimes_{\mathcal{R}}K$. In order for our quasi-abelian categories of interest to be elementary, we will work within a fixed Grothendieck universe.
\subsection*{Acknowledgments}
This paper was written as a part of a PhD thesis at the Humboldt-Universität zu Berlin under the supervision of Prof. Dr. Elmar Große-Klönne. I would like to thank Prof. Große-Klönne
for pointing me towards this exciting topic and reading an early draft of the paper.

\bigskip
Funded by the Deutsche Forschungsgemeinschaft (DFG, German Research
Foundation) under Germany´s Excellence Strategy – The Berlin Mathematics
Research Center MATH+ (EXC-2046/1, project ID: 390685689).
\section{Background: Bornological and Ind-Banach spaces}\label{Section background on sheaves of Ind-Banach spaces}
We begin the paper by discussing some of the foundational results needed for the rest of the text. Namely, in this chapter we will deal with the basics of closed symmetric monoidal structures on quasi-abelian categories, focusing on the cases of (complete) bornological and Ind-Banach spaces. We will give the definitions of these two types of spaces, discuss some of their basic properties, and analyze tensor products and inner homomorphism functors in those two categories. In particular, we will construct closed symmetric monoidal structures in the categories of complete bornological spaces, and Ind-Banach spaces, compare these two structures via the dissection functor, study their extensions to the left hearts, and establish several important adjunctions. Lastly, we finish the chapter by obtaining certain derived adjunctions arising from the classical tensor-hom adjunction. The material in this chapter will be systematically used throughout the text.  We would like to point out that we claim no originality in this chapter. In particular, in order to keep the preliminaries contained, we have not included proofs of any of our statements. We have, however, included references in which detailed accounts of the facts discussed below can be found.
\subsection{\texorpdfstring{Bornological spaces}{}}\label{Section bornological spaces and homologicl algebra}
In this section we will recall some notions on bornological spaces and quasi-abelian categories. This will be based on the material in \cite{bode2021operations}, \cite{prosmans2000homological}, and
\cite{schneiders1999quasi}.
\begin{defi}\label{defi bornology}
A  bornology on a $K$-vector space $V$ is a family $\mathcal{B}$ of subsets
of $V$ such that:
\begin{enumerate}[label=(\roman*)]
    \item If $B_1\in \mathcal{B}$ and $B_2\subset B_1$, then $B_2\in\mathcal{B}$.
    \item $\{v\}\in\mathcal{B}$ for all $v\in V$.
    \item $\mathcal{B}$ is closed under taking finite unions.
    \item If $B\in\mathcal{B}$, then $\lambda\cdot B\in\mathcal{B}$ for all $\lambda\in K$.
    \item Let $B\in \mathcal{B}$, and $\mathcal{R}\cdot B$ be the $\mathcal{R}$-module spawned by $B$. Then $\mathcal{R}\cdot B\in\mathcal{B}$.
\end{enumerate}
We call the pair $(V,\mathcal{B})$ a bornological space, and call the elements of $\mathcal{B}$ bounded subsets. We will refer to the pair $(V,\mathcal{B})$ as $V$. A $K$-linear map between bornological spaces $f:V\rightarrow W$ is bounded if it maps bounded subsets to bounded subsets. We let $\mathcal{B}c_K$
be the category of bornological spaces and bounded maps.
\end{defi}

Let $V$ be a bornological space and $W$ be a $K$-vector space. Assume we have a $K$-linear surjection $f:V\rightarrow W$. This map induces a bornology on $W$, with bounded subsets being the images under $f$ of the bounded subsets of $V$. We call this the quotient bornology. Dually, for an injective map $W\rightarrow V$, we get a unique bornology on $W$, which is called the subspace bornology.  In particular, given a map of bornological spaces $f:V\rightarrow W$,
the spaces $\operatorname{CoIm}(f)$ and $\operatorname{Im}(f)$ have canonical bornologies, and we will always regard them as bornological spaces with respect to these bornologies.\\
The category of bornological spaces is closely related to that of locally convex spaces. Indeed, let $LCS_K$ be the category of locally convex topological $K$-vector spaces with continuous maps (see \cite{schneider2013nonarchimedean} for a complete treatment on $LCS_K$). Then we have an adjunction:
\begin{equation*}
    (-)^t:\mathcal{B}c_K\leftrightarrows LCS_K:(-)^b.
\end{equation*}
A precise description of this adjunction may be found in \cite[pp. 93,  Proposition 3]{houzel2006seminaire}. We remark that the functor $(-)^b$ is fully faithful when restricted to the category of metrizable spaces, and exact on semi-normed spaces (cf. \cite[Lemma 4.3]{bode2021operations}, \cite[pp. 102, pp. 109]{houzel2006seminaire}). A bornological space which is in the image of $(-)^b$ will be called the bornologification of a locally convex space.
\begin{obs}
We will say that a bornological space is a normed space if it is the image under $(-)^b$ of a normed space, and we will follow the same convention for Banach spaces, Fréchet spaces and metrizable spaces.
\end{obs}

There are notions of completeness and separability for bornological spaces, akin to those in $LCS_K$. We will focus on completeness.
Let $V$ be a bornological space. By property $(v)$ in Definition \ref{defi bornology}, every bounded subset $B\in\mathcal{B}$ is contained in a bounded $\mathcal{R}$-module. Assume $B$ is a bounded $\mathcal{R}$-module in $V$, and let $V_B=B_K$ be the associated $K$-vector space. Then $V_B$ may be regarded as a normed space with unit ball $V_B^{\circ}=B$, and topology induced by the $\pi$-adic topology on $B$. In this situation, the canonical map $V_B\rightarrow V$ is bounded and we obtain the following identity in $\mathcal{B}c_K$:
\begin{equation*}
    V=\varinjlim V_B,
\end{equation*}
where the colimit is taken over all the bounded $\mathcal{R}$-submodules of $V$. 
\begin{defi}
We say a bornological space $V$ is complete if every bounded subspace $B$ is contained in a bounded $\mathcal{R}$-module $B'$ such that $V_{B'}$ is a Banach space. We define $\widehat{\mathcal{B}}c_K$ as the full subcategory of $\mathcal{B}c_K$ given by the complete bornological spaces.    
\end{defi}
In particular, every complete bornological space is a colimit of Banach spaces. Furthermore, in \cite{prosmans2000homological} F. Prosmans and J-P. Schneiders show that $\widehat{\mathcal{B}}c_K$ can be regarded as a certain full subcategory of the category of Ind-Banach spaces, $\Indban$. Namely, that of essentially monomorphic objects in $\Indban$. We will give a more detailed account of the implications of this fact in the next section.\\
 By \cite[Proposition 5.11]{prosmans2000homological} the inclusion $J:\widehat{\mathcal{B}}c_K\rightarrow \mathcal{B}c_K$ has a left adjoint:
\begin{equation*}
    J:\widehat{\mathcal{B}}c_K\leftrightarrows \mathcal{B}c_K:\widehat{\operatorname{Cpl}
    }.
\end{equation*}
The completion functor is discussed at length in \cite[Section 5]{prosmans2000homological}. The most relevant fact for us that a short exact sequence in  $\widehat{\mathcal{B}}c_K$ is strict if and only if its image under $J$ is strict. Recall the notion of quasi-abelian categories from \cite[Chapter 1]{schneiders1999quasi}:
\begin{defi}
Let $\mathcal{E}$ be an additive category with kernels and cokernels. Any morphism $f:C\rightarrow D$ induces a canonical morphism:
\begin{equation*}
    \operatorname{CoIm}(f)\rightarrow \operatorname{Im}(f).
\end{equation*}
We say $f:C\rightarrow D$ is strict if the above map is an isomorphism. We say $\mathcal{E}$ is a quasi-abelian category if the following conditions hold:
\begin{enumerate}[label=(\roman*)]
    \item For every cartesian square:
    \begin{equation*}
\begin{tikzcd}
C \arrow[r, "f"]             & D            \\
C' \arrow[r, "f'"] \arrow[u] & D' \arrow[u]
\end{tikzcd}
    \end{equation*}
     such that $f$ is a strict epimorphism, $f'$ is also a strict epimorphism.

    \item For every cocartesian square:
    \begin{equation*}
\begin{tikzcd}
C' \arrow[r, "f'"]         & D'          \\
C \arrow[u] \arrow[r, "f"] & D \arrow[u]
\end{tikzcd}
    \end{equation*}
    where $f$ is a strict monomorphism, $f'$ is also a strict monomorphism.
\end{enumerate}
\end{defi}
Notice that any abelian category is quasi-abelian. We say a sequence:
\begin{equation*}
    C\xrightarrow[]{f}D\xrightarrow[]{g} E
\end{equation*}
is strictly exact if $f$ is strict and $\operatorname{Im}(f)=\operatorname{Ker}(g)$.  The class of short strictly exact sequences gives $\mathcal{E}$ the structure of an exact category in the sense of Quillen \cite{Quillen}. Hence, we can define the following:
\begin{defi}
Let $\mathcal{E}$ be a quasi-abelian category. We define the following objects:
\begin{enumerate}[label=(\roman*)]
    \item The homotopy category of $\mathcal{E}$, denoted $\operatorname{K}(\mathcal{E})$, is the localization of the category of chain complexes $\operatorname{Ch}(\mathcal{E})$ at the full triangulated subcategory of complexes homotopic to $0$. 
    \item The derived category of $\mathcal{E}$, denoted $\operatorname{D}(\mathcal{E})$, is the localization of $\operatorname{K}(\mathcal{E})$ at the full triangulated subcategory of strictly exact complexes. 
\end{enumerate}
\end{defi}
Notice that if $\mathcal{E}$ is abelian every morphism is strict, so $\operatorname{D}(\mathcal{E})$ is the usual derived category. For any quasi-abelian category $\mathcal{E}$, we can form a certain abelian category $LH(\mathcal{E})$ called its left heart. It comes equipped with a left exact functor:
\begin{equation*}
    I:\mathcal{E}\rightarrow LH(\mathcal{E}),
\end{equation*}
satisfying certain universal property (cf. \cite[Section 1.2.4.]{schneiders1999quasi}).  The inclusion into the left heart $I:\mathcal{E}\rightarrow LH(\mathcal{E})$ has a left adjoint, is fully faithful, and induces an equivalence of triangulated categories $I:\operatorname{D}(\mathcal{E})\rightarrow \operatorname{D}(LH(\mathcal{E}))$.
The main feature of quasi-abelian categories is the fact that, as $LH(\mathcal{C})$ is an abelian category, we can use the previous equivalence to deduce homological properties of $\operatorname{D}(\mathcal{E})$ via $\operatorname{D}(LH(\mathcal{E}))$. With this equivalence in mind, we will sometimes write $\operatorname{D}(\mathcal{C})$ instead of $\operatorname{D}(LH(\mathcal{C}))$ whenever we dim convenient.\\
There are versions of the usual notions of exactness for functors between quasi-abelian categories:
\begin{defi}
Let $F:\mathcal{E}\rightarrow \mathcal{F}$ be a functor between quasi-abelian categories.
\begin{enumerate}[label=(\roman*)]
    \item We say $F$ is left exact if it maps a strictly exact sequence in $\mathcal{E}$:
\begin{equation*}
    0\rightarrow C \rightarrow D\rightarrow E\rightarrow 0
\end{equation*}
 to a strictly exact sequence in $\mathcal{F}$:
 \begin{equation*}
     0\rightarrow F(C) \rightarrow F(D)\rightarrow F(E).
 \end{equation*}
    \item We say $F$ is strongly left exact if it maps a strictly exact sequence in $\mathcal{E}$:
\begin{equation*}
    0\rightarrow C \rightarrow D\rightarrow E
\end{equation*}
 to a strictly exact sequence in $\mathcal{F}$:
 \begin{equation*}
     0\rightarrow F(C) \rightarrow F(D)\rightarrow F(E).
 \end{equation*}
\end{enumerate}
The corresponding notions of right exactness are defined dually. A functor is said to be (strongly) exact if it is both (strongly) left exact and
(strongly) right exact.
\end{defi}
In some situations, functors between quasi-abelian categories induce functors between their left hearts, and there is a theory of derived functors in this setting. This is studied in \cite[Section 1.3]{schneiders1999quasi}.\bigskip

Many of the constructions of homological algebra depend on the category $\mathcal{E}$ having enough injective or projective objects (\emph{i.e.} existence of derived functors). In particular, when considering sheaves with values on a given category, it is of crucial importance that the category has enough injective objects. Unfortunately, a quasi-abelian category $\mathcal{E}$ will almost never have enough injective objects (cf. \cite[Proposition 1.3.26]{schneiders1999quasi}). However, in many situations $LH(\mathcal{E})$ will have enough injective objects. Thus, in the spirit mentioned above, we will be able to study sheaves with values on $\mathcal{E}$ by considering sheaves with values on $LH(\mathcal{E})$. Unfortunately, not every quasi-abelian category is suited for this kind of study. In particular, certain operations on sheaves such as sheafification, require strong exactness of filtered colimits. Hence, we need to restrict the scope, and impose certain conditions on $\mathcal{E}$ which ensure that filtered colimits are strongly exact. This line of thought leads to the definition of elementary and quasi-elementary quasi-abelian categories, which we now define:
\begin{defi}
Let $\mathcal{E}$ be a quasi-abelian category. We define the following objects:
\begin{enumerate}[label=(\roman*)]
    \item  An object $C$ of $\mathcal{E}$ is called small if the functor $\Hom_{\mathcal{E}}(C,-)$ commutes with direct sums.
    \item  An object $C$ of $\mathcal{E}$ is called tiny if the functor $\Hom_{\mathcal{E}}(C,-)$ commutes with all filtered colimits.
    \item A set $G$ of objects in $\mathcal{E}$ is called a strictly generating set if for any $D\in\mathcal{E}$, there is a family of objects $\{ P_i\}_{i\in I}\subset G$ together with a strict epimorphism:
    \begin{equation*}
        \bigoplus_{i\in I}P_i\rightarrow D.
    \end{equation*}
\end{enumerate}
\end{defi}
\begin{defi}
We say that a quasi-abelian category $\mathcal{E}$ is quasi-elementary (\emph{resp.} elementary)
if it is cocomplete and admits a strictly generating set of small (\emph{resp.} tiny) projective objects.
\end{defi}
The relevance of this notion is made apparent via the following Propositions:
\begin{prop}[{\cite{schneiders1999quasi}}]
Let $\mathcal{E}$ be a quasi-elementary quasi-abelian category. Then $LH(\mathcal{E})$   is a Grothendieck abelian category with enough projective objects. In particular, it has enough injective objects, exact filtered colimits, and a generator. 
\end{prop}
We can relate these notions to bornological spaces via the following proposition:
\begin{prop}\label{prop bornological spaces are quasi-abelian}
The following hold:
\begin{enumerate}[label=(\roman*)]
    \item $\mathcal{B}c_K$  is a complete and co-complete quasi-elementary quasi-abelian category.  
    \item $\widehat{\mathcal{B}}c_K$  is a complete and co-complete quasi-elementary quasi-abelian category .
    \item Let $X$ be a set and $V$ be a Banach space. We define $c_0(X,V)$ as the $K$-vector space of all functions $f:X\rightarrow V$ such that for each $\epsilon > 0$ the following set is finite: 
    \begin{equation*}
        \{ x\in X, \textnormal{ with } \vert f(x)\vert > \epsilon \},
    \end{equation*}
    We regard $c_0(X,V)$ as a Banach space with respect to the supremum norm. Every Banach space of the form $c_0(X,V)$ is a small projective object in $\widehat{\mathcal{B}}c_K$. Furthermore, every object in $\widehat{\mathcal{B}}c_K$ is the cokernel of a strict morphism of direct sums of Banach spaces of the above form.
\end{enumerate}
\end{prop}
\begin{proof}
The first statement is \cite[Proposition 1.8]{prosmans2000homological}, the second is \cite[Proposition 5.6]{prosmans2000homological}, and the third is \cite[Corollary 4.18]{bode2021operations}. 
\end{proof}
Thus, $LH(\widehat{\mathcal{B}}c_K)$ is a Grothendieck abelian category with enough projective objects. Notice that the choice of a Grothendieck universe we made at the beginning of the paper is essential here. Otherwise, there is not a set of small projective generators of $\widehat{\mathcal{B}}c_K$, only a proper class.\\

In light of the previous results, one could wonder why we made the effort of defining elementary quasi-abelian categories. After all, we are mainly interested in homological algebra, and  being quasi-elementary is enough for the left heart to be a Grothendieck abelian category. Unfortunately, being quasi-elementary is not enough for our purposes. In particular, let $\mathcal{E}$ be a quasi-elementary quasi-abelian category. Then direct sums are strongly exact in 
$\mathcal{E}$. However, filtered inductive limits are not strongly exact in $\mathcal{E}$. Furthermore, the canonical morphism:
\begin{equation*}
    I:\mathcal{E}\rightarrow LH(\mathcal{E}),
\end{equation*}
does not commute with filtered colimits. This makes it impossible to define a sheafification functor for presheaves with values in $\mathcal{E}$, so the theory of $\mathcal{E}$-valued sheaves is not well-behaved. For this reason, we need to work with elementary abelian categories. 
\subsection{Ind-Banach spaces}
Unfortunately for us, $\widehat{\mathcal{B}}c_K$ is not an elementary quasi-abelian category. Thus, we need to introduce an auxiliary category that is somehow related to $\widehat{\mathcal{B}}c_K$. This role will be played by the Ind-Banach spaces. Let us start by giving the appropriate definitions:
\begin{defi}
 We define the category $\Indban$  of Ind-Banach spaces as the category with objects given by functors $V:I\rightarrow \operatorname{Ban}_K$, where $I$ is a small filtered category. We call these objects the Ind-Banach spaces. For simplicity, we will write Ind-Banach spaces as formal direct limits:
 \begin{equation*}
      ``\varinjlim"V_i.
 \end{equation*}
Morphisms in $\Indban$ are defined as follows: Given two objects $``\varinjlim"V_i$, $``\varinjlim"W_j$, we set:
\begin{equation*}
    \Hom_{\Indban}(``\varinjlim"V_i,``\varinjlim"W_j)=\varprojlim_i\varinjlim_j \Hom_{\operatorname{Ban}_K}(V_i,W_j).
\end{equation*}
\end{defi}
The category of Ind-Banach spaces is admittedly more inexplicit than $\widehat{\mathcal{B}}c_K$, and it is harder to apply the traditional methods of $p$-adic functional analysis in this setting. On the other hand, the homological properties of $\Indban$ are much better than those of $\widehat{\mathcal{B}}c_K$:
\begin{prop}
The category $\Indban$ is an elementary quasi-abelian category. Furthermore, every Banach space of the form $c_0(X,V)$ is projective in $\Indban$.
\end{prop}
\begin{proof}
This is a consequence of \cite[Proposition 2.1.17]{schneiders1999quasi}. 
\end{proof}
As mentioned above, every complete bornological space $V$ may be expressed as a filtered colimit:
\begin{equation*}
    V=\varinjlim V_B,
\end{equation*}
where the colimit is taken over all bounded $\mathcal{R}$-submodules of $V$. This allows us to define the dissection functor:
\begin{equation*}
    \operatorname{diss}:\widehat{\mathcal{B}}c_K\rightarrow \Indban,
\end{equation*}
defined on objects by taken a complete bornological space $V$ to the formal colimit $``\varinjlim" V_B$, and similarly for morphisms. This functor was thoroughly studied in \cite[Section 5]{prosmans2000homological}. Let us now give an overview of its main features:
\begin{teo}\label{teo first comparison comp born and indban}
The dissection functor $\operatorname{diss}:\widehat{\mathcal{B}}c_K\rightarrow \Indban$ has a left adjoint functor:
\begin{equation*}
   \widehat{L}:\Indban \rightarrow \widehat{\mathcal{B}}c_K, \, ``\varinjlim"V_i\mapsto \varinjlim V_i.
\end{equation*}
These functors satisfy the following identity: $\widehat{L}\circ  \operatorname{diss}=\operatorname{Id}_{\widehat{\mathcal{B}}c_K}.$ Furthermore, the dissection functor is exact, and induces mutually inverse equivalences of triangulated categories:
\begin{equation*}
    \operatorname{diss}:\operatorname{D}(\widehat{\mathcal{B}}c_K)\leftrightarrows \operatorname{D}(\Indban):\mathbb{L}\widehat{L},
\end{equation*}
These functors restrict to mutually inverse equivalences of abelian categories:
\begin{equation*}
   \widetilde{\operatorname{diss}}:LH(\widehat{\mathcal{B}}c_K)\leftrightarrows LH(\Indban): \widetilde{L}.
\end{equation*}
\end{teo}
\begin{proof}
This is shown in \cite[Proposition 5.15]{prosmans2000homological}, and \cite[Proposition 5.16.]{prosmans2000homological}.
\end{proof}
In light of this theorem, it follows that we can regard any complete bornological space as an Ind-Banach space. As shown above, $\Indban$ is an elementary quasi-abelian category. Hence, it presents a much better behavior than $\widehat{\mathcal{B}}c_K$ with respect to sheaf theory and homological algebra. Furthermore, as $\widehat{\mathcal{B}}c_K$ and $\Indban$ have isomorphic left hearts, the techniques of $p$-adic functional analysis are still available to us.
\subsection{\texorpdfstring{The Tensor-Hom adjunction }{}}\label{Section tensor-hom in bornological spaces}
We will now construct closed symmetric monoidal structures on $\widehat{\mathcal{B}}c_K$ and $\Indban$, and discuss their extensions to the left heart. We would like to stress that the constructions also work for $\mathcal{B}c_K$. However, in order to keep the preliminaries to a minimum, we will not discuss this case. A full discussion on closed symmetric monoidal structures on quasi-abelian categories may be found in \cite[Section 3.3]{bode2021operations}, and \cite[Section 1.5]{schneiders1999quasi}.\bigskip

We may endow $\widehat{\mathcal{B}}c_K$ with a closed symmetric monoidal structure via the usual tensor-hom adjunction. Namely, consider a pair of bornological spaces $V,W\in \widehat{\mathcal{B}}c_K$. We may regard the $K$-vector space $V\otimes_KW$ as a bornological space with respect to the bornology induced by the family of subsets:
\begin{equation*}
    \{B_1\otimes_{\mathcal{R}}B_2\textnormal{ } \vert \textnormal{ } B_1\subset V, B_2\subset W \textnormal{ are bounded }\mathcal{R}\textnormal{-modules} \}.
\end{equation*}
This is called the (projective) bornological tensor product of $V$ and $W$. We define $V\widehat{\otimes}_KW$ as the completion of $V\otimes_KW$ with respect to this bornology.\\ 
The $K$-vector space $\Hom_{\widehat{\mathcal{B}}c_K}(V,W)$ has a natural bornology with bounded subsets:
\begin{equation*}
    \{U\subset \Hom_{\widehat{\mathcal{B}}c_K}(V,W) \textnormal{ }\vert \textnormal{ for all bounded }B\subset V,\textnormal{ } \cup_{f\in U}f(B) \textnormal{ is bounded in } W \}.
\end{equation*}
We let $\underline{\Hom}_{\widehat{\mathcal{B}}c_K}(V,W)$ denote $\Hom_{\widehat{\mathcal{B}}c_K}(V,W)$ equipped with this bornology.  We remark that the space $\underline{\Hom}_{\widehat{\mathcal{B}}c_K}(V,W)$ is complete with respect to this bornology whenever $W$ is complete.
\begin{prop}[{\cite[Theorem 4.10]{bode2021operations}}]
The functors $\widehat{\otimes}_K$, and $\underline{\Hom}_{\widehat{\mathcal{B}}c_K}$  make $\widehat{\mathcal{B}}c_K$   into a closed symmetric monoidal category. In particular, for $U,V,W\in \widehat{\mathcal{B}}c_K$, we have an adjunction:
\begin{equation}\label{equation tensor-hom adjunction complete bornological modules}
    \Hom_{\widehat{\mathcal{B}}c_K}(U\widehat{\otimes}_KV,W)=\Hom_{\widehat{\mathcal{B}}c_K}(U,\underline{\Hom}_{\widehat{\mathcal{B}}c_K}(V,W)).
\end{equation}
The unit of the closed symmetric monoidal structure of $\widehat{\mathcal{B}}c_K$ is $K$.
\end{prop}
Notice that, by construction, $\widehat{\otimes}_K$ maps Banach spaces to Banach spaces. Furthermore, for every  set $X$ and every Banach space $V$, the functor $c_0(X,V)\widehat{\otimes}_K-$ is exact. As we saw in Proposition \ref{prop bornological spaces are quasi-abelian}, every object in $\widehat{\mathcal{B}}c_K$ admits a strict resolution by direct sums of Banach spaces of this type. In other words, $\widehat{\mathcal{B}}c_K$ has enough flat projective objects stable under $\widehat{\otimes}_K$. This situation was extensively studied in \cite[Section 3.3]{bode2021operations}. The main consequence is that the functors  $\widehat{\otimes}_K$ and $\underline{\Hom}_{\widehat{\mathcal{B}}c_K}$ extend to a pair of functors $\widetilde{\otimes}_K$ and $\underline{\Hom}_{LH(\widehat{\mathcal{B}}c_K)}$ making $LH(\widehat{\mathcal{B}}c_K)$ a closed symmetric monoidal abelian category with unit $I(K)$. Furthermore, for each $V,W\in \widehat{\mathcal{B}}c_K$ we have the following identities:
\begin{equation*}
I(V)\widetilde{\otimes}_KI(W)=\operatorname{H}^0(V\widehat{\otimes}_K^{\mathbb{L}}W), \quad    \underline{\Hom}_{LH(\widehat{\mathcal{B}}c_K)}(I(V),I(W))=\operatorname{H}^0(R\underline{\Hom}_{\widehat{\mathcal{B}}c_K}(V,W)).
\end{equation*}
We point out, however, that the fact that filtered colimits are not strongly exact in $\widehat{\mathcal{B}}c_K$, implies that it is not clear if every object in  $\widehat{\mathcal{B}}c_K$ is flat. In fact, this is probably false. Hence, it is not known if the functor $I:\widehat{\mathcal{B}}c_K\rightarrow LH(\widehat{\mathcal{B}}c_K)$ is strong symmetric monoidal. In particular, the identity: 
\begin{equation*}
    I(V)\widetilde{\otimes}_KI(W)=I(V\widehat{\otimes}_KW),
\end{equation*}
may not hold in general. In any case, the functor is lax symmetric monoidal (\emph{cf.} \cite[Lemma 4.19]{bode2021operations}), so that we have a natural transformation:
\begin{equation*}
    I(-)\widetilde{\otimes}_KI(-)\rightarrow I(-\widehat{\otimes}_K-).
\end{equation*}
We point out that the left adjoint to the inclusion into the left heart:
\begin{equation*}
    C:LH(\widehat{\mathcal{B}}c_K)\rightarrow \widehat{\mathcal{B}}c_K,
\end{equation*}
is strong symmetric monoidal. As before, we can solve the issues of the closed symmetric monoidal structure on $\widehat{\mathcal{B}}c_K$ by passing onto $\Indban$. However, this is done at the cost of  working with spaces that are less explicit, and worse behaved from the analytical perspective.\\

We will now discuss the closed symmetric monoidal structure in $\Indban$. This is a natural extension of the closed symmetric monoidal structure in the category of $K$-Banach spaces. Namely, let $``\varinjlim"V_i$, and  $``\varinjlim"W_j$ be two Ind-Banach spaces. We define the tensor product of Ind-Banach spaces as follows:
\begin{equation*}
    ``\varinjlim"V_i\overrightarrow{\otimes}_K``\varinjlim"W_j= ``\varinjlim"V_i\widehat{\otimes}_KW_j,
\end{equation*}
where $V_i\widehat{\otimes}_KW_j$ denotes the complete projective tensor product 
of locally convex spaces (\emph{cf}. \cite[Section 17.B]{schneider2013nonarchimedean}). Similarly, we define an inner hom functor via the following formula:
\begin{equation*}
    \underline{\Hom}_{\Indban}(``\varinjlim"V_i,``\varinjlim"W_j)=\varprojlim_i``\varinjlim_j" \underline{\Hom}_{\operatorname{Ban}_K}(V_i,W_j),
\end{equation*}
where $\underline{\Hom}_{\operatorname{Ban}_K}(V_i,W_j)$ represents the Banach space of continuous linear operators with the operator norm (\emph{cf}. \cite[Section 6]{schneider2013nonarchimedean}). Similar to the case of complete bornological spaces, we have the following:
\begin{prop}[{\cite[Theorem 4.10]{bode2021operations}}]
The functors $\overrightarrow{\otimes}_K$, and $\underline{\Hom}_{\Indban}$  make $\Indban$ into a closed symmetric monoidal category. For $U,V,W\in \Indban$, we have an adjunction:
\begin{equation}\label{equation tensor-hom adjunction Ind-Banach spaces}
    \Hom_{\Indban}(U\overrightarrow{\otimes}_KV,W)=\Hom_{\Indban}(U,\underline{\Hom}_{\Indban}(V,W)).
\end{equation}
We point out that the unit of the closed symmetric monoidal structure of $\Indban$ is $K$.
\end{prop}
The fact that Banach spaces are flat with respect to the complete projective tensor product of locally convex spaces implies that Banach spaces are flat with respect to $\overrightarrow{\otimes}_K$. Furthermore, every object in $\Indban$ is, by construction, a filtered colimit of Banach spaces. Furthermore, unlike in $\widehat{\mathcal{B}}c_K$, filtered colimits are strongly exact in $\Indban$. Thus, every object in $\Indban$ is flat with respect to $\overrightarrow{\otimes}_K$, and $\Indban$ has enough flat projectives stable under $\overrightarrow{\otimes}_K$. As before, this implies that the closed symmetric monoidal structure on $\Indban$ extends uniquely to a closed symmetric monoidal structure on $LH(\Indban)$.\\

As we saw in Theorem \ref{teo first comparison comp born and indban}, the dissection functor induces mutually inverse equivalences of abelian categories:
\begin{equation*}
    \widetilde{\operatorname{diss}}:LH(\widehat{\mathcal{B}}c_K)\leftrightarrows LH(\Indban):\widetilde{L}.
\end{equation*}
The next step is using this equivalence to compare the two closed symmetric monoidal structures:
\begin{prop}
The following hold:
\begin{enumerate}[label=(\roman*)]
    \item The functor $\operatorname{diss}:\widehat{\mathcal{B}}c_K\rightarrow \Indban$ is lax symmetric monoidal.
    \item The functor $\widehat{L}:\Indban\rightarrow \widehat{\mathcal{B}}c_K$ is strong symmetric monoidal.
    \item The pair of adjoint functors:
    \begin{equation*}
        I:\Indban\leftrightarrows LH(\Indban):C,
    \end{equation*}
    are strong symmetric monoidal.
    \item The pair mutually inverse equivalences of abelian categories:
\begin{equation*}
    \widetilde{\operatorname{diss}}:LH(\widehat{\mathcal{B}}c_K)\leftrightarrows LH(\Indban):\widetilde{L},
\end{equation*}
are strong symmetric monoidal.
\end{enumerate}
\end{prop}
\begin{proof}
 This is shown in \cite[pp. 36]{bode2021operations}, and  \cite[Proposition 4.22]{bode2021operations}.   
\end{proof}
This proposition reinforces the heuristic that we should forgo of explicit objects in seek of better algebraic properties. In other words, the strategy here is showing that everything we do in $\widehat{\mathcal{B}}c_K$ can be translated to $\Indban$, and that the corresponding extensions to the left hearts agree. After seeing all this, one could wonder why do we care to define the category of complete bornological spaces at all. After all, many of the developments in this section are aimed at showing that $\Indban$ has much better algebraic properties that $\widehat{\mathcal{B}}c_K$. Our interest in $\widehat{\mathcal{B}}c_K$ stems from the fact that, as we will see in Chapter \ref{Section background Lie algebroids}, most of the objects we will be interested in are naturally Fréchet spaces satisfying some nuclearity condition.  These spaces are subject to powerful $p$-adic analytic theorems (\emph{i.e.} Banach's open mapping theorem), so ignoring this features is not the correct approach. Furthermore, all Fréchet spaces are metrizable, and for this class of spaces the dissection functor is well-behaved:
\begin{prop}[{\cite[Proposition 4.25]{bode2021operations}}]\label{prop comparison of tensor products metrizable spaces}
Let $V,W$ be metrizable locally convex  spaces. Then there is a canonical isomorphism in $\Indban$:
\begin{equation*}
  \operatorname{diss}(V)\overrightarrow{\otimes}_K\operatorname{diss}(W)\rightarrow \operatorname{diss}(V\widehat{\otimes}_KW).
\end{equation*}
\end{prop}
This proposition will be a powerful tool later, when we discuss categories of modules over complete bornological algebras and Ind-Banach algebras. To summarize,  the closed symmetric monoidal structure on $\Indban$ given by the functors $\overrightarrow{\otimes}_K$ and $\underline{\Hom}_{\Indban}$ extends to a closed symmetric monoidal structure on $LH(\widehat{\mathcal{B}}c_K)$, given by a pair of functors $\widetilde{\otimes}_K$ and $\underline{\Hom}_{LH(\widehat{\mathcal{B}}c_K)}$. Furthermore, we have natural isomorphisms of functors:
\begin{equation}\label{equation extension of closed symmetric monoidal category structure}
    I(-\overrightarrow{\otimes}_K-)=I(-)\widetilde{\otimes}_KI(-),
\end{equation}
\begin{equation}\label{equation extension of closed symmetric monoidal category structure 2}
    I(\underline{\Hom}_{\Indban}(-,-))=\Hom_{LH(\widehat{\mathcal{B}}c_K)}(I(-),I(-)). 
\end{equation}
\subsection{Categories of modules}
With all these tools at hand, we can start working towards suitable categories of modules in $\widehat{\mathcal{B}}c_K$ and $\Indban$:
\begin{defi}
Let $\mathcal{E}$ be a closed symmetric monoidal category with tensor $\otimes$ and unit $U$. 
\begin{enumerate}[label=(\roman*)]
    \item An object $\mathscr{A}\in \mathcal{E}$ is called a monoid if there is a morphism:
    \begin{equation*}
        \mathscr{A}\otimes\mathscr{A}\rightarrow \mathscr{A},
    \end{equation*}
    called the product, and a unit map $U\rightarrow \mathscr{A}$ satisfying the usual axioms of an algebra.
    \item  Let $\Mod_{\mathcal{E}}(\mathscr{A})$ be the category of $\mathscr{A}$-modules. These are objects $M$ in $\mathcal{E}$ equipped with an action $\mathscr{A}\otimes M\rightarrow M$ satisfying the usual axioms of a module.   
\end{enumerate}    
\end{defi}
The monoids in  $\widehat{\mathcal{B}}c_K$ are called complete bornological algebras, and the monoids in $\Indban$ will be called Ind-Banach algebras.  Most of the properties of $\mathcal{E}$ can be extended to categories of modules in a natural way. In particular, we have the following proposition:
\begin{prop}
Let $\mathcal{E}$ be a closed symmetric monoidal quasi-abelian category, and $\mathscr{A}$ be a monoid in $\mathcal{E}$. The following hold:
\begin{enumerate}[label=(\roman*)]
    \item $\Mod_{\mathcal{E}}(\mathscr{A})$ is a quasi-abelian category.
    \item The forgetful functor: $\Mod_{\mathcal{E}}(\mathscr{A})\rightarrow \mathcal{E}$  commutes with limits and colimits. A morphism in $\Mod_{\mathcal{E}}(\mathscr{A})$ is strict if and only if it is strict in $\mathcal{E}$.
    \item If $\mathcal{E}$ is (co)-complete, then so is $\Mod_{\mathcal{E}}(\mathscr{A})$.
    \item Is $\mathcal{E}$ is quasi-elementary (\emph{resp.} elementary) then so is $\Mod_{\mathcal{E}}(\mathscr{A})$.
\end{enumerate}
\end{prop}
\begin{proof}
    This is \cite[Lemma 3.7]{bode2021operations}.
\end{proof}
In particular, if $\mathscr{A}$ is a complete bornological algebra  then $\Mod_{\widehat{\mathcal{B}}c_K}(\mathscr{A})$ is a quasi-elementary quasi-abelian category which is complete and co-complete. Thus, $LH(\Mod_{\widehat{\mathcal{B}}c_K}(\mathscr{A}))$ is a Grothendieck abelian category. Similarly, if $\mathscr{A}$ is an Ind-Banach algebra, then $\Mod_{\Indban}(\mathscr{A})$ is a complete and cocomplete elementary quasi-abelian category. Furthermore, in the Ind-Banach case, $I$ commutes with the tensor product. In particular $I(\mathscr{A})$ is a monoid in $LH(\widehat{\mathcal{B}}c_K)$. In this situation, it follows by \cite[Proposition 2.6]{bode2021operations} that $I$ induces an equivalence of abelian categories:
\begin{equation*}
    LH(\Mod_{\widehat{\mathcal{B}}c_K}(\mathscr{A}))\rightarrow \Mod_{LH(\widehat{\mathcal{B}}c_K)}(I(\mathscr{A})).
\end{equation*}
In the complete bornological case, $I$ is not strong symmetric monoidal. Thus, the above isomorphism does not hold in general. However, we can once again bridge the gap by passing onto $\Indban$. Indeed, let $\mathscr{A}$ be a complete bornological algebra. The fact that $\operatorname{diss}: \widehat{\mathcal{B}}c_K\rightarrow \Indban$ is lax symmetric monoidal implies that $\operatorname{diss}(\mathscr{A})$ has a canonical structure as an Ind-Banach algebra (\emph{cf.} \cite[Proposition 3.10]{bode2021operations}). The two categories of modules are connected via the following proposition: 
\begin{prop}\label{prop dissection functor and algebras}
Let $\mathscr{A}$ be a complete bornological algebra. The dissection functor \newline $\operatorname{diss}:\widehat{\mathcal{B}}c_K\rightarrow \Indban$ induces an exact and fully faithful functor:
     \begin{equation*}
         \operatorname{diss}_A: \Mod_{\widehat{\mathcal{B}}c_K}(\mathscr{A})\rightarrow \Mod_{\Indban}(\operatorname{diss}(\mathscr{A})).
     \end{equation*}
If $\mathscr{A}$ is the bornologification of a Fréchet K-algebra, this yields an equivalence of
categories:
\begin{equation*}
    LH(\Mod_{\widehat{\mathcal{B}}c_K}(\mathscr{A})) \cong LH(\Mod_{\Indban}(\operatorname{diss}(\mathscr{A}))) \cong \Mod_{LH(\widehat{\mathcal{B}}c_K)}(I(\mathscr
    A)).
\end{equation*}
\end{prop}
\begin{proof}
This is \cite[
Proposition 4.3]{bode2021operations}.
\end{proof}
In all the cases of interest our complete bornological algebra will be a Fréchet (in fact Fréchet-Stein) algebra. Thus, we may use this to work directly in the category of Ind-Banach spaces. 
\begin{obs}
In order to simplify the notation, given a bornologification of a Fréchet $K$-algebra $\mathscr{A}$, we will identify $\mathscr{A}$ with its image under the dissection functor, and write $\mathscr{A}$ instead of $\operatorname{diss}(\mathscr{A})$. Similarly, a complete bornological $\mathscr{A}$-module is an object in $\Mod_{\Indban}(\mathscr{A})$ in the image of $\operatorname{diss}_A$.
\end{obs}
Most of the classical operations of modules over a ring generalize to the setting of quasi-abelian categories with a closed symmetric monoidal structure.  In order to be more concise, we will only develop them in the Ind-Banach setting, but these constructions can be applied to a much more broader context. Let $\mathscr{A}$ be an Ind-Banach algebra. There is an extension of scalars functor:
\begin{equation*}
    -\overrightarrow{\otimes}_K-:\Mod_{\Indban}(\mathscr{A})\times \Indban\rightarrow \Mod_{\Indban}(\mathscr{A}).
\end{equation*}
 Namely, let $\mathcal{M}\in \Mod_{\Indban}(\mathscr{A})$, and $V\in \Indban$. We may endow $\mathcal{M}\overrightarrow{\otimes}_KV$ with an $\mathscr{A}$-module structure given by:
\begin{equation*}
\mathscr{A}\overrightarrow{\otimes}_K\left(\mathcal{M}\overrightarrow{\otimes}_KV\right)=\left(\mathscr{A}\overrightarrow{\otimes}_K\mathcal{M}\right)\overrightarrow{\otimes}_KV\rightarrow \mathcal{M}\overrightarrow{\otimes}_KV, 
\end{equation*}
where the last map is induced by the action of $\mathscr{A}$ on $\mathcal{M}$. Analogously, we get a functor:
\begin{equation*}
    -\widetilde{\otimes}_K-:\Mod_{LH(\widehat{\mathcal{B}}c_K)}(I(\mathscr{A}))\times LH(\widehat{\mathcal{B}}c_K)\rightarrow \Mod_{LH(\widehat{\mathcal{B}}c_K)}(I(\mathscr{A})),
\end{equation*}
and these functors commute with $I$.\\
There is also a co-extension of scalars functor:
\begin{equation*}
    \underline{\Hom}_{\Indban}(-,-):\Mod_{\Indban}(\mathscr{A}^{\op})^{\op}\times \Indban\rightarrow \Mod_{\Indban}(\mathscr{A}).
\end{equation*}
Namely, for $\mathcal{M}\in\Mod_{\Indban}(\mathscr{A}^{\op})$, $V\in \Indban$, we may promote the space $\underline{\Hom}_{\Indban}(\mathcal{M},V)$ to an Ind-Banach $\mathscr{A}$-module via the formula:
\begin{multline}\label{equation map in tensor hom adjunction}   \mathscr{A}\overrightarrow{\otimes}_K\underline{\Hom}_{\Indban}(\mathcal{M},V)\rightarrow \underline{\Hom}_{\Indban}(\mathcal{M}\overrightarrow{\otimes}_K\mathscr{A},V)\\
\rightarrow \underline{\Hom}_{\Indban}(\mathcal{M},V),
\end{multline}
where the second map is induced by the canonical inclusion $\mathcal{M}\rightarrow\mathcal{M}\overrightarrow{\otimes}_K\mathscr{A}$, and the first one is the image under the adjunction $(\ref{equation tensor-hom adjunction complete bornological modules})$ of the map in:
\begin{equation*}
    \Hom_{\Indban}(\mathscr{A},\underline{\Hom}_{\Indban}(\underline{\Hom}_{\Indban}(\mathcal{M},V),\underline{\Hom}_{\Indban}(\mathcal{M}\overrightarrow{\otimes}_K\mathscr{A},V))),
\end{equation*}
defined on each Banach space by the rule $a\mapsto \left(f\mapsto f\circ \cdot a \right)$, where $\cdot a$ denotes the action of $a$ on $\mathcal{M}$.\\ 
This construction makes sense in any closed symmetric monoidal category. In particular, we may repeat this construction on $LH(\widehat{\mathcal{B}}c_K)$, and we get a functor: 
\begin{equation*}
    \underline{\Hom}_{LH(\widehat{\mathcal{B}}c_K)}(-,-):\Mod_{LH(\widehat{\mathcal{B}}c_K)}(I(\mathscr{A}^{\op}))^{\op}\times LH(\widehat{\mathcal{B}}c_K)\rightarrow \Mod_{LH(\widehat{\mathcal{B}}c_K)}(I(\mathscr{A})).
\end{equation*}
As before, the identities in equations (\ref{equation extension of closed symmetric monoidal category structure}), and (\ref{equation extension of closed symmetric monoidal category structure 2}) show that these two functors commute with the embedding into the left heart $I$.\bigskip

There are versions of the complete tensor product and internal hom for $\mathscr{A}$-modules. Namely, the complete tensor product of $\mathscr{A}$-modules:
\begin{equation*}
    -\overrightarrow{\otimes}_{\mathscr{A}}-:\Mod_{\Indban}(\mathscr{A}^{\op})\times \Mod_{\Indban}(\mathscr{A})\rightarrow \Indban,
\end{equation*}
is defined for $\mathcal{M}\in \Mod_{\Indban}(\mathscr{A}^{\op})$, and $\mathcal{N}\in \Mod_{\Indban}(\mathscr{A})$ by the formula:
\begin{equation*}
\mathcal{M}\overrightarrow{\otimes}_{\mathscr{A}}\mathcal{N}:=\operatorname{Coeq}\left( \mathcal{M}\overrightarrow{\otimes}_K\mathscr{A}\overrightarrow{\otimes}_K\mathcal{N}\rightrightarrows \mathcal{M}\overrightarrow{\otimes}_K\mathcal{N}\right),
\end{equation*}
where the maps in the coequalizer are given by the maps:
\begin{equation*}
    T_{\mathcal{M}}:\mathcal{M}\overrightarrow{\otimes}_K\mathscr{A}\rightarrow \mathcal{M}, \quad T_{\mathcal{N}}:\mathscr{A}\overrightarrow{\otimes}_K\mathcal{N}\rightarrow \mathcal{N},
\end{equation*}
induced by the action of $\mathscr{A}$ on $\mathcal{M}$ and $\mathcal{N}$ respectively. Similarly, we have a tensor product for $I(\mathscr{A})$-modules, which we denote by $\widetilde{\otimes}_{I(\mathscr{A})}$, and is defined analogously.\\ 
Notice that the functor $I:\Indban\rightarrow LH(\widehat{\mathcal{B}}c_K)$ does not  preserve arbitrary cokernels. In particular, given a complex $\mathcal{C}^{\bullet}$ of objects in $\Indban$, the complex $I(\mathcal{C}^{\bullet})$ is exact if and only if $\mathcal{C}^{\bullet}$ is strict and exact. Thus, the identity:
\begin{equation*}
I(\mathcal{M}\overrightarrow{\otimes}_{\mathscr{A}}\mathcal{N})=I(\mathcal{M})\widetilde{\otimes}_{I(\mathscr{A})}I(\mathcal{N}),
\end{equation*}
holds if and only if the morphism $T_{\mathcal{M}}-T_{\mathcal{N}}$ is strict.\\

There is also a inner homomorphisms functor:
\begin{equation*}
    \underline{\Hom}_{\mathscr{A}}(-,-):\Mod_{\Indban}(\mathscr{A})^{\op}\times \Mod_{\Indban}(\mathscr{A})\rightarrow \Indban.
\end{equation*}
This is
defined for $\mathcal{M},\mathcal{N}\in \Mod_{\Indban}(\mathscr{A})$ by the expression:
\begin{equation*}
    \underline{\Hom}_{\mathscr{A}}(\mathcal{M},\mathcal{N})=\operatorname{Eq}\left(\underline{\Hom}_{\Indban}(\mathcal{M},\mathcal{N})\rightrightarrows \underline{\Hom}_{\Indban}(\mathscr{A}\overrightarrow{\otimes}_K\mathcal{M},\mathcal{N})\right).
\end{equation*}
The first map in the equalizer is obtained by applying $\underline{\Hom}_{\Indban}(-,\mathcal{N})$ to the action:
\begin{equation*}
    \mathscr{A}\overrightarrow{\otimes}_K\mathcal{M}\rightarrow \mathcal{M},
\end{equation*}
and the second is obtained by applying $\underline{\Hom}_{\Indban}(\mathcal{M},-)$ to the canonical morphism:
\begin{equation*}
    \mathcal{N}\rightarrow \underline{\Hom}_{\Indban}(\mathscr{A},\mathcal{N}).
\end{equation*}
We remark that this construction can be performed in any symmetric monoidal category. In particular, the analogous construction yields the inner homomorphism functor $\underline{\Hom}_{I(\mathscr{A})}$ for $I(\mathscr{A})$-modules. As the functor $I:\Indban\rightarrow LH(\widehat{\mathcal{B}}c_K)$ has a left adjoint, it commutes with equalizers. Hence, by the identities in equations $(\ref{equation extension of closed symmetric monoidal category structure})$, and $(\ref{equation extension of closed symmetric monoidal category structure 2})$ we have:
\begin{align}\label{equation extension of inner hom to left heart}
\begin{split}
    I(\underline{\Hom}_{\mathscr{A}}(\mathcal{M},\mathcal{N}))&\\
 =\operatorname{Eq}\left(I(\underline{\Hom}_{\Indban}(\mathcal{M},\mathcal{N}))\rightrightarrows I(\underline{\Hom}_{\Indban}(\mathscr{A}\overrightarrow{\otimes}_K\mathcal{M},\mathcal{N})) \right)& \\ 
=\operatorname{Eq}\left(\underline{\Hom}_{LH(\widehat{\mathcal{B}}c_K)}(I(\mathcal{M}),I(\mathcal{N}))\rightrightarrows \underline{\Hom}_{LH(\widehat{\mathcal{B}}c_K)}(I(\mathscr{A})\widetilde{\otimes}_KI(\mathcal{M}),I(\mathcal{N})) \right)&\\
=\Hom_{I(\mathscr{A})}(I(\mathcal{M}),I(\mathcal{N}))&.
\end{split}
\end{align}
Notice that this is in stark contrast with the situation of the functor $-\overrightarrow{\otimes}_{\mathscr{A}}-$:
\begin{prop}
The functors $\overrightarrow{\otimes}_{\mathscr{A}}$ and $\underline{\Hom}_{\Indban}$ fit into the following adjunction: 
\begin{equation}\label{equation relative tensor hom adjunction}
    \Hom_{\Indban}(\mathcal{M}\overrightarrow{\otimes}_{\mathscr{A}}\mathcal{N},V)=\Hom_{\mathscr{A}}(\mathcal{N},\underline{\Hom}_{\Indban}(\mathcal{M},V)),
\end{equation}
where $\mathcal{M}\in \Mod_{\Indban}(\mathscr{A}^{\op}),\mathcal{N}\in \Mod_{\Indban}(\mathscr{A})$ and $V\in \Indban$.
\end{prop}
\begin{prop}
The functors $\overrightarrow{\otimes}_{K}$ and $\underline{\Hom}_{\mathscr{A}}$ fit into the following adjunction:
\begin{equation}\label{equation tensor relative hom adjunction}
 \Hom_{\mathscr{A}}(\mathcal{M}\overrightarrow{\otimes}_{K}V,\mathcal{N})=\Hom_{\Indban}(V,\underline{\Hom}_{\mathscr{A}}(\mathcal{M},\mathcal{N})),   
\end{equation} 
where $\mathcal{M},\mathcal{N}\in \Mod_{\Indban}(\mathscr{A})$ and $V\in \Indban$. 
\end{prop}
These adjunctions, follow from general properties of closed symmetric monoidal categories. Hence, we get analogous adjunctions between the left heart versions $LH(\widehat{\mathcal{B}}c_K)$ and $\Mod_{LH(\widehat{\mathcal{B}}c_K)}(I(\mathscr{A}))$.\bigskip

The symmetric monoidal structure on $\Indban$ extends in a natural way to the chain complex category $\operatorname{Ch}(\Indban)$, and to the homotopy category $\operatorname{K}(\Indban)$. Namely, choose chain complexes $V^{\bullet},W^{\bullet}\in \operatorname{Ch}(\Indban)$. We can construct a double complex $V^{\bullet}\overrightarrow{\otimes}^{\bullet,\bullet}_KW^{\bullet}$ by setting:
\begin{equation*}
    V^{\bullet}\overrightarrow{\otimes}_K^{n,m}W^{\bullet} =V^{n}\overrightarrow{\otimes}_KW^{m},
\end{equation*}
and differentials induced by the differentials of $V^{\bullet}$ and $W^{\bullet}$. We then define:
\begin{equation*}
    V^{\bullet}\overrightarrow{\otimes}_KW^{\bullet}=\operatorname{Tot}_{\oplus}\left(V^{\bullet}\overrightarrow{\otimes}^{\bullet,\bullet}_KW^{\bullet} \right).
\end{equation*}
Similarly, we define a double complex $\underline{\Hom}^{\bullet,\bullet}_{\Indban}(V^{-\bullet},W^{\bullet})$ by setting:
\begin{equation*}
    \underline{\Hom}^{n,m}_{\Indban}(V^{-\bullet},W^{\bullet})=\underline{\Hom}_{\Indban}(V^{-n},W^{m}),
\end{equation*}
and differentials induced by those in $V^{\bullet}$, and $W^{\bullet}$. We then define: 
\begin{equation*}
    \underline{\Hom}_{\Indban}(V^{\bullet},W^{\bullet})=\operatorname{Tot}_{\pi}\left(\underline{\Hom}^{\bullet,\bullet}_{\Indban}(V^{-\bullet},W^{\bullet})\right).
\end{equation*}
These constructions yield an adjunction:
\begin{multline}\label{equation thensor-hom adjunction homotopy category}
    \Hom_{\operatorname{K}(\Indban)}(U^{\bullet}\overrightarrow{\otimes}_KV^{\bullet},W^{\bullet})\\
    =\Hom_{\operatorname{K}(\Indban)}(U^{\bullet}
    , \underline{\Hom}_{\Indban}(V^{\bullet},W^{\bullet})),
\end{multline}
and the analogous adjunction holds in $\operatorname{Ch}(\Indban)$. In order to show this, we can use that the functor $I:\operatorname{K}(\Indban)\rightarrow \operatorname{K}(LH(\widehat{\mathcal{B}}c_K))$ is fully faithful, and commutes with taking tensor products and inner homomorphisms. This reduces $(\ref{equation thensor-hom adjunction homotopy category})$ to showing the analogous adjunction in $\operatorname{K}(LH(\widehat{\mathcal{B}}c_K))$. As this is the homotopy category of a Grothendieck abelian category, we can use \cite[Tag 0A8H]{stacks-project}, and in particular equation \cite[Tag 0A5X]{stacks-project}, and Lemma \cite[Tag 0A5Y]{stacks-project}.

\subsection{Derived Tensor-Hom adjunction I}
We will now deal with the theory of derived functors in this setting. The main references are \cite[Section 3.6]{bode2021operations} and \cite[Section 1.3]{schneiders1999quasi}. Recall in particular the notion of explicitly right (left) derivable functor \cite[Definition 1.3.6]{schneiders1999quasi}.\\ 
 As stated above, the tensor product $\overrightarrow{\otimes}_K$ is exact. Hence, it has 
a left derived functor:
\begin{equation*}
    -\overrightarrow{\otimes}_K^{\mathbb{L}}-:\operatorname{D}(\widehat{\mathcal{B}}c_K)\times \operatorname{D}(\widehat{\mathcal{B}}c_K)\rightarrow \operatorname{D}(\widehat{\mathcal{B}}c_K).
\end{equation*}
By construction of the left heart of a quasi-abelian category, for every $W\in LH(\widehat{\mathcal{B}}c_K)$ there are $U,V\in \Indban$ such that there is a short exact sequence:
\begin{equation*}
    0\rightarrow I(U)\rightarrow I(V)\rightarrow W\rightarrow 0.
\end{equation*}
Furthermore, every element in the essential image of $I:\Indban\rightarrow LH(\widehat{\mathcal{B}}c_K)$ is flat with respect to $\widetilde{\otimes}_K$. Thus, we also have left derived functor:
\begin{equation*}
   -\widetilde{\otimes}_K^{\mathbb{L}}-:\operatorname{D}(LH(\widehat{\mathcal{B}}c_K))\times \operatorname{D}(LH(\widehat{\mathcal{B}}c_K))\rightarrow \operatorname{D}(LH(\widehat{\mathcal{B}}c_K)). 
\end{equation*}
The equivalence of categories $I:\operatorname{D}(\widehat{\mathcal{B}}c_K)\rightarrow \operatorname{D}(LH(\widehat{\mathcal{B}}c_K))$ allows us to get an identification: 
\begin{equation*}
    I(-\overrightarrow{\otimes}_K^{\mathbb{L}}-)=I(-)\widetilde{\otimes}_K^{\mathbb{L}}I(-).
\end{equation*}
Thus, we will identify $\operatorname{D}(\widehat{\mathcal{B}}c_K)$ and $\operatorname{D}(LH(\widehat{\mathcal{B}}c_K))$, and simply write:
\begin{equation*}
    -\overrightarrow{\otimes}_K^{\mathbb{L}}-=-\widetilde{\otimes}_K^{\mathbb{L}}-.
\end{equation*}
Recall that $LH(\widehat{\mathcal{B}}c_K)$ is a Grothendieck abelian category. Thus, the version in $LH(\widehat{\mathcal{B}}c_K)$ of the adjunction $(\ref{equation tensor-hom adjunction complete bornological modules})$, together with \cite[Theorem 14.4.8]{Kashiwara2006} imply that we have a right derived functor:
\begin{equation*}
    R\underline{\Hom}_{LH(\widehat{\mathcal{B}}c_K)}(-,-):\operatorname{D}(LH(\widehat{\mathcal{B}}c_K))^{\op}\times \operatorname{D}(LH(\widehat{\mathcal{B}}c_K))\rightarrow \operatorname{D}(LH(\widehat{\mathcal{B}}c_K)).
\end{equation*}
such that for $U^{\bullet},V^{\bullet},W^{\bullet}\in \operatorname{D}(LH(\widehat{\mathcal{B}}c_K))$ we have a derived adjunction:
\begin{multline}\label{equation derived tensor-hom adjunction bornological modules}
    \Hom_{\operatorname{D}(LH(\widehat{\mathcal{B}}c_K))}(U^{\bullet}\widetilde{\otimes}^{\mathbb{L}}_KV^{\bullet},W^{\bullet})\\
    =\Hom_{\operatorname{D}(LH(\widehat{\mathcal{B}}c_K))}(U^{\bullet},R\underline{\Hom}_{LH(\widehat{\mathcal{B}}c_K)}(V^{\bullet},W^{\bullet})).
\end{multline}
\begin{obs}\label{obs abuse of notation internal hom}
As $\Indban$ does not have enough injectives, the hom functor $\Hom_{\Indban}(-,-)$ may not be explicitly right derivable. However,if it is, then by $\textnormal{\cite[Proposition 1.3.16]{schneiders1999quasi}}$ we have:
\begin{equation*}
    R\underline{\Hom}_{\Indban}(-,-)=R\underline{\Hom}_{LH(\widehat{\mathcal{B}}c_K)}(I(-),I(-)).
\end{equation*}
Thus, we will follow the conventions in $\textnormal{\cite{bode2021operations}}$ and write:
\begin{equation*}
    R\underline{\Hom}_{\widehat{\mathcal{B}}c_K}(-,-)=R\underline{\Hom}_{LH(\widehat{\mathcal{B}}c_K)}(-,-),
\end{equation*}
in an abuse of notation. We will follow similar conventions for the different versions of this functor. 
\end{obs}
This functor can be computed explicitly. Namely, let $V^{\bullet}\in \operatorname{D}(\widehat{\mathcal{B}}c_K)$. There is a complex $\widetilde{V}^{\bullet}\in\operatorname{K}(\Indban)$ with a quasi-isomorphism $I(\widetilde{V})^{\bullet}\rightarrow V^{\bullet}$. Let $W^{\bullet}\in \operatorname{D}(\widehat{\mathcal{B}}c_K)$. As $LH(\widehat{\mathcal{B}}c_K)$ is a Grothendieck abelian category, there is a homotopically injective complex $\widetilde{W}^{\bullet}$ such that there is a quasi-isomorphism $I(W)^{\bullet}\rightarrow \widetilde{W}^{\bullet}$. Then by \cite[Theorem 14.4.8]{Kashiwara2006}, we have a chain of identities:
\begin{multline}\label{equation explicit computation of RHom}
    R\underline{\Hom}_{\Indban}(V^{\bullet},W^{\bullet})=\underline{\Hom}_{LH(\widehat{\mathcal{B}}c_K)}(I(\widetilde{V}^{\bullet}),\widetilde{W}^{\bullet})\\
    =\operatorname{Tot}_{\pi}\left(\underline{\Hom}^{\bullet,\bullet}_{LH(\widehat{\mathcal{B}}c_K)}(I(\widetilde{V}^{-\bullet}),\widetilde{W}^{\bullet})\right),
\end{multline}
which we will encounter in many computations.
\section{Background: Sheaves of Ind-Banach spaces}\label{Section sheaves of Ind-Banach spaces}
In this chapter we will continue introducing the background material needed for the rest of the paper. In particular, this chapter is devoted to giving a quick introduction to the theory of sheaves with values on quasi-abelian categories. Namely, we will start by endowing the category of sheaves of Ind-Banach spaces with a quasi-abelian structure, and determining its left heart. We continue by constructing the sheafification functor, and use it to extend the closed symmetric monoidal structure on $\Indban$ to the sheaf setting. Later, we will discuss the theory of derived functors of certain important functors (\emph{i.e.} pushforward, inner homomorphisms, tensor products). We finish this chapter by analyzing some situations which will be of use in our study of Hochschild (co)-homology, and in particular in the version of Kashiwara's equivalence given in Section \ref{Section hochschild cohomology}. We point out that most of the contents in this chapter are not original, and that we have only included proofs of the facts that we were not able to find in the literature. In particular, we only claim originality of the contents of Sections \ref{section base-change bimodule} and \ref{section closed immersions}.  As before, we have included references of all the statements not proved in the text.
\subsection{Sheaves of Ind-Banach spaces}
Let $X$ be a rigid space. We will now show how the previous constructions extend to the category of sheaves on $X$ with values on Ind-Banach spaces. This may be performed on general $G$-topological spaces, but we do not need such generality here.

\begin{defi}
Let $X$ be a rigid space, and denote its underlying $G$-topological space by $X_{\operatorname{An}}$. For any quasi-abelian category $\mathcal{C}$ we define the following categories:
\begin{enumerate}[label=(\roman*)]
    \item The category of presheaves of $X$ with values in $\mathcal{C}$ is:
    \begin{equation*}
        \operatorname{PreShv}(X,\mathcal{C}):= \operatorname{Functors}(X_{\operatorname{An}}^{\op},\mathcal{C}).
    \end{equation*}
    \item The category of sheaves on $X$ with values in $\mathcal{C}$ is the full subcategory  of $\operatorname{PreShv}(X,\mathcal{C})$ given by presheaves $\mathcal{F}$ such that for every admissible open $U\subset X$ and every admissible cover $(U_i)_{i\in I}$ of $U$ the diagram:
    \begin{equation*}
        0\rightarrow \mathcal{F}(U) \rightarrow \prod_{i\in I}\mathcal{F}(U_i)\rightrightarrows \prod_{i,j\in I}\mathcal{F}(U_i\cap U_j),
    \end{equation*}
    is strictly exact in $\mathcal{C}$. We denote this category by $\operatorname{Shv}(X,\mathcal{C})$.
\end{enumerate}
\end{defi}

The categories $\operatorname{PreShv}(X,\mathcal{C})$, and $\operatorname{Shv}(X,\mathcal{C})$ are quasi-abelian categories in a natural way (see \cite[Section 2.2]{schneiders1999quasi}). Furthermore, if $\mathcal{C}$ is elementary, there is a sheafification functor:
\begin{equation*}
    (-)^a:\operatorname{PreShv}(X,\mathcal{C})\rightarrow \operatorname{Shv}(X,\mathcal{C}),
\end{equation*}
which is strongly exact, and left adjoint to the inclusion $\operatorname{Shv}(X,\mathcal{C})\rightarrow\operatorname{PreShv}(X,\mathcal{C})$. In particular, $\operatorname{Shv}(X,\mathcal{C})$ is a reflexive subcategory of $\operatorname{PreShv}(X,\mathcal{C})$. Thus, every limit in $\operatorname{Shv}(X,\mathcal{C})$ agrees with its limit in $\operatorname{PreShv}(X,\mathcal{C})$, and every colimit in $\operatorname{Shv}(X,\mathcal{C})$ is the sheafification of the corresponding colimit in $\operatorname{PreShv}(X,\mathcal{C})$. Furthermore, we have canonical equivalences of abelian categories:
\begin{align*}
    LH(\operatorname{PreShv}(X,\mathcal{C}))=\operatorname{PreShv}(X,LH(\mathcal{C})),\\
    LH(\operatorname{Shv}(X,\mathcal{C}))=\operatorname{Shv}(X,LH(\mathcal{C})).
\end{align*}
The functor $I:\operatorname{PreShv}(X,\mathcal{C})\rightarrow \operatorname{PreShv}(X,LH(\mathcal{C}))$, is defined by the rule:
\begin{equation*}
    I(\mathcal{F})(U)=I(\mathcal{F}(U)),
\end{equation*}
for every $\mathcal{F}\in \operatorname{PreShv}(X,\mathcal{C})$, and every admissible open subspace $U\subset X$. Sheafification and inclusion into the left heart fit into the following commutative diagram:
\begin{equation*}
\begin{tikzcd}
{\operatorname{PreShv}(X,\mathcal{C})} \arrow[r, "I"] \arrow[d, "(-)^{a}"] & {\operatorname{PreShv}(X,LH(\mathcal{C}))} \arrow[d, "(-)^a"] \\
{\operatorname{Shv}(X,\mathcal{C})} \arrow[r, "I"]                         & {\operatorname{Shv}(X,LH(\mathcal{C}))}                      
\end{tikzcd}   
\end{equation*}
If $\mathcal{C}$ is not elementary, then filtered colimits are not strongly exact, and do not commute with $I$. Hence, there are problems when defining a sheafification functor, and the diagram above will not be commutative. For this reasons, we will restrict to the case where $\mathcal{C}$ is elementary. We remark that, in general, $\operatorname{Shv}(X,\mathcal{C})$
does not have enough projective nor injective objects. Hence, it is not a good category from the homological point of view. However if $\mathcal{C}$
has good enough properties (\emph{i.e.} it is elementary, complete and co-complete), then $LH(\mathcal{C})$ is a Grothendieck abelian category. Thus, $\operatorname{Shv}(X,LH(\widehat{\mathcal{B}}c_K))$ is also a Grothendieck abelian category. Thus, it has enough injective objects.
This is the case, for example, if $\mathcal{C}=\Indban$.\bigskip

The closed symmetric monoidal category structure of $\Indban$ extends naturally to $\operatorname{Shv}(X,\Indban)$. Namely, for $\mathcal{F},\mathcal{G}\in \operatorname{Shv}(X,\Indban)$, we define $\mathcal{F}\overrightarrow{\otimes}_K\mathcal{G}$ as the sheafification of the tensor product presheaf:
\begin{equation*}
    U\mapsto \mathcal{F}(U)\overrightarrow{\otimes}_K\mathcal{G}(U).
\end{equation*}
As shown above, $\overrightarrow{\otimes}_K$ is exact in $\Indban$. Hence, the tensor product of presheaves is exact. As sheafification is strongly exact, it follows that:
\begin{equation*}
    -\overrightarrow{\otimes}_K-:\operatorname{Shv}(X,\Indban)\times \operatorname{Shv}(X,\Indban)\rightarrow \operatorname{Shv}(X,\Indban),
\end{equation*}
is an exact functor. We may repeat this construction in $\operatorname{Shv}(X,LH(\widehat{\mathcal{B}}c_K))$ to get another version of tensor product: 
\begin{equation*}
 -\widetilde{\otimes}_K-:\operatorname{Shv}(X,LH(\widehat{\mathcal{B}}c_K))\times \operatorname{Shv}(X,LH(\widehat{\mathcal{B}}c_K))\rightarrow \operatorname{Shv}(X,LH(\widehat{\mathcal{B}}c_K)).   
\end{equation*}
As $I$ commutes with sheafification, the identity in $(\ref{equation extension of closed symmetric monoidal category structure})$ shows that:
\begin{equation*}
    I(-\overrightarrow{\otimes}_K-)=I(-)\widetilde{\otimes}_KI(-).
\end{equation*}
We will now construct the inner homomorphism functor for sheaves of Ind-Banach spaces. For $\mathcal{F},\mathcal{G}\in \operatorname{Shv}(X,\Indban)$, we define $\underline{\mathcal{H}om}_{\Indban}(\mathcal{F},\mathcal{G})$  by setting for every admissible open $U\subset X$:
\begin{multline}\label{equation definition inner homomorphism}
    \underline{\mathcal{H}om}_{\Indban}(\mathcal{F},\mathcal{G})(U)\\
    =\operatorname{Eq}\left(\prod_{V\subset U}\underline{\Hom}_{\Indban}(\mathcal{F}(V),\mathcal{G}(V))\rightrightarrows \prod_{W\subset V}\underline{\Hom}_{\Indban}(\mathcal{F}(V),\mathcal{G}(W)) \right).
\end{multline}
 The maps in the equalizer are the following: For admissible open 
subspaces:
\begin{equation*}
    W\subset V\subset U,
\end{equation*}
a homomorphism $f\in \underline{\Hom}_{\Indban}(\mathcal{F}(V),\mathcal{G}(V))$ is mapped to the composition:
\begin{equation*}
    \mathcal{F}(V)\xrightarrow[]{f}\mathcal{G}(V)\rightarrow\mathcal{G}(W).
\end{equation*}
Similarly, every $g\in \underline{\Hom}_{\Indban}(\mathcal{F}(W),\mathcal{G}(W))$ is mapped to:
\begin{equation*}
    \mathcal{F}(V)\rightarrow\mathcal{F}(W)\xrightarrow[]{g}\mathcal{G}(W).
\end{equation*}
 We may define $\underline{\mathcal{H}om}_{LH(\widehat{\mathcal{B}}c_K)}(-,-)$ analogously, and the fact that $I$ commutes with equalizers, together with $(\ref{equation extension of closed symmetric monoidal category structure 2})$ show that the following equation holds:
 \begin{equation*}
    I(\underline{\mathcal{H}om}_{\Indban}(-,-))=\underline{\mathcal{H}om}_{LH(\widehat{\mathcal{B}}c_K)}(I(-),I(-)).
 \end{equation*}
Furthermore, we have the following adjunction:
\begin{prop}\label{prop closed symmetric monoidal structure on sheaves of Indban}
 These functors endow $\operatorname{Shv}(X,\Indban)$ with the structure of  a closed symmetric monoidal categories. In particular, we have an adjunction:
\begin{multline}\label{equation tensor-hom adjunctions for sheaves}
   \Hom_{\operatorname{Shv}(X,\Indban)}(U\overrightarrow{\otimes}_KV,W)\\
   =\Hom_{\operatorname{Shv}(X,\Indban)}(U,\underline{\mathcal{H}om}_{\Indban}(V,W)),
\end{multline}
where $U,V,W\in \operatorname{Shv}(X,\Indban)$. Furthermore, the analogous statement holds for  $\operatorname{Shv}(X,LH(\widehat{\mathcal{B}}c_K))$. In particular, for $\mathcal{U},\mathcal{V},\mathcal{W}\in \operatorname{Shv}(X,LH(\widehat{\mathcal{B}}c_K))$, we have :
\begin{multline}\label{equation tensor-hom adjunction left heart}
\Hom_{\operatorname{Shv}(X,LH(\widehat{\mathcal{B}}c_K))}(\mathcal{U}\widetilde{\otimes}_K\mathcal{V},\mathcal{W})\\
=\Hom_{\operatorname{Shv}(X,LH(\widehat{\mathcal{B}}c_K))}(\mathcal{U},\underline{\mathcal{H}om}_{LH(\widehat{\mathcal{B}}c_K)}(\mathcal{V},\mathcal{W})).
\end{multline}   
\end{prop}
Let $\mathscr{A}$ be a sheaf of complete Ind-Banach algebras on $X$ ( i.e. a monoid in 
$\operatorname{Shv}(X,\Indban)$). We denote the category of Ind-Banach $\mathscr{A}$-modules by $\Mod_{\Indban}(\mathscr{A})$. As before, this is a quasi-abelian category, and we have:
\begin{equation*}
    LH(\Mod_{\Indban}(\mathscr{A}))=\Mod_{LH(\widehat{\mathcal{B}}c_K)}(I(\mathscr{A})).
\end{equation*}
The constructions in Section \ref{Section tensor-hom in bornological spaces} make sense in any closed symmetric monoidal category. In particular, we have a tensor product $-\overrightarrow{\otimes}_{\mathscr{A}}-$, an inner homomorphisms functor $\underline{\mathcal{H}om}_{\mathscr{A}}(-,-)$, and their counterparts in $\Mod_{LH(\widehat{\mathcal{B}}c_K)}(I(\mathscr{A}))$ are constructed analogously. Furthermore, the same arguments as above show that we have an identification:
\begin{equation*}
    I(\underline{\mathcal{H}om}_{\mathscr{A}}(-,-))=\underline{\mathcal{H}om}_{I(\mathscr{A})}(I(-),I(-)).
\end{equation*}
Similarly, the adjunctions of equations $(\ref{equation relative tensor hom adjunction})$, $(\ref{equation tensor relative hom adjunction})$ hold in $\operatorname{Shv}(X,\Indban)$, and $\operatorname{Shv}(X,LH(\widehat{\mathcal{B}}c_K))$ respectively.\bigskip

As before, for $\mathcal{M}\in \Mod_{\Indban}(\mathscr{A}^{\op})$ and $\mathcal{N}\in \Mod_{\Indban}(\mathscr{A})$ we have:
\begin{equation*}  I(\mathcal{M}\overrightarrow{\otimes}_{\mathscr{A}}\mathcal{N})=I(\mathcal{M})\widetilde{\otimes}_{I(\mathscr{A})}I(\mathcal{N}),
\end{equation*}
if and only if the canonical map:
\begin{equation*}
    \mathcal{M}\overrightarrow{\otimes}_K\mathscr{A}\overrightarrow{\otimes}_K\mathcal{N}\rightarrow \mathcal{M}\overrightarrow{\otimes}_K\mathcal{N},
\end{equation*}
is a strict epimorphism. In particular, we have an extension-restriction of scalars adjunction:
\begin{equation*}
  I(\mathscr{A})\widetilde{\otimes}_K-:\operatorname{Shv}(X,LH(\widehat{\mathcal{B}}c_K))\leftrightarrows \Mod_{LH(\widehat{\mathcal{B}}c_K)}(I(\mathscr{A})): \operatorname{Forget}(-).  
\end{equation*}
This adjunction satisfies that $\operatorname{Forget}$ commutes with (co)-limits, and reflects monomorphisms and epimorphisms. Furthermore, for every $\mathcal{M}\in \Mod_{LH(\widehat{\mathcal{B}}c_K)}(I(\mathscr{A}))$, the canonical map:
\begin{equation*}
    I(\mathscr{A})\widetilde{\otimes}_K\operatorname{Forget}(\mathcal{M})\rightarrow \mathcal{M},
\end{equation*}
is an epimorphism. As $\operatorname{Shv}(X,LH(\widehat{\mathcal{B}}c_K))$ is a Grothendieck abelian category, this implies that $\Mod_{LH(\widehat{\mathcal{B}}c_K)}(I(\mathscr{A}))$ is also Grothendieck. In particular, it has enough injective objects. Following the conventions in \cite{bode2021operations}, we introduce the following:
\begin{defi}
Let $\mathscr{A}$ be a sheave of Ind-Banach algebras on $X$. We say $\mathscr{A}$ is a sheaf of complete bornological algebras if for every admissible open $U\subset X$ the Ind-Banach space $\mathscr{A}(U)$ is a complete bornological space. We define a sheaf of complete bornological $\mathscr{A}$-modules analogously.
\end{defi}
The necessity for this definition stems from the fact that $\widehat{\mathcal{B}}c_K$ is not an elementary quasi-abelian category. Hence, there is no sheafification functor, and therefore no tensor product in $\operatorname{Shv}(X,\widehat{\mathcal{B}}c_K)$. Thus, it is not clear how to define sheaves of complete bornological algebras directly, and the same problem arises when defining sheaves of complete bornological modules. In any case, by Proposition \ref{prop dissection functor and algebras}, our current definition is enough for our purposes.
\subsection{Derived pushforward and pullback}
Let $f:Y\rightarrow X$ be a morphism of rigid spaces. As usual, there is a pushforward-pullback adjunction, which fits into the following diagram:
\begin{equation}\label{equation diagram pushforward and pullback}
\begin{tikzcd}
{f^{-1}_{LH}:\operatorname{Shv}(X,LH(\widehat{\mathcal{B}}c_K))} \arrow[r]       & {\operatorname{Shv}(Y,LH(\widehat{\mathcal{B}}c_K)):f_{*,LH}} \arrow[l, shift left=2]       \\
{f^{-1}:\operatorname{Shv}(X,\Indban)} \arrow[u, "I"] \arrow[r] & {\operatorname{Shv}(Y,\Indban):f_*} \arrow[u, "I"] \arrow[l, shift left=2]
\end{tikzcd}
\end{equation}
As usual, the pushforward: $f_*:\operatorname{Shv}(Y,\Indban)\rightarrow \operatorname{Shv}(X,\Indban)$, is defined by sending $\mathcal{F}\in \operatorname{Shv}(Y,\Indban)$ to the sheaf defined on admissible open subspaces $U\subset X$ by the rule:
\begin{equation*}
    f_*\mathcal{F}(U)=\mathcal{F}(f^{-1}(U)).
\end{equation*}
The pushforward between the left hearts $f_{*,LH}$ is defined analogously.\\
The pullback functor $f^{-1}:\operatorname{Shv}(X,\Indban)\rightarrow \operatorname{Shv}(Y,\Indban)$ is defined by sending a sheaf $\mathcal{F}\in \operatorname{Shv}(X,\Indban)$ to the sheafification of the presheaf: 
\begin{equation*}
   f^{-1}_{\operatorname{Pre}}\mathcal{F}(U)=\varinjlim_{f(U)\subset V}\mathcal{F}(V), 
\end{equation*}
 where $U\subset Y$, and the $V\subset f(U)$ are admissible open subspaces. The left heart version $f^{-1}_{LH}$ is constructed via the analogous procedure.\bigskip

As discussed above, the fact that $\Indban$ is an elementary quasi-abelian category implies that $\operatorname{Shv}(X,\Indban)$ is also elementary. In particular, filtered colimits are strongly exact in $\operatorname{Shv}(X,\Indban)$. As sheafification is also strongly exact, it follows that $f^{-1}$
is strongly exact. Therefore, it extends a triangulated functor between the derived categories:
\begin{equation*}
    f^{-1}:\operatorname{D}(\operatorname{Shv}(X,\Indban))\rightarrow \operatorname{D}(\operatorname{Shv}(Y,\Indban)).
\end{equation*}
By the same arguments, $f^{-1}_{LH}$ is also exact, so we obtain a derived functor:
\begin{equation*}
    f^{-1}_{LH}:\operatorname{D}(\operatorname{Shv}(X,LH(\widehat{\mathcal{B}}c_K)))\rightarrow \operatorname{D}(\operatorname{Shv}(Y,LH(\widehat{\mathcal{B}}c_K))).
\end{equation*}
Commutativity of diagram $(\ref{equation diagram pushforward and pullback})$ implies that both extensions get identified by $I$.\bigskip

As $\operatorname{Shv}(Y,\Indban)$ does not have enough injective objects, it is not clear whether if $f_*$ is explicitly right derivable or not. However, as $\operatorname{Shv}(Y,LH(\widehat{\mathcal{B}}c_K))$ is a Grothendieck abelian category, we have a right derived functor:
\begin{equation*}
    Rf_{*,LH}:\operatorname{D}(\operatorname{Shv}(Y,LH(\widehat{\mathcal{B}}c_K)))\rightarrow \operatorname{D}(\operatorname{Shv}(X,LH(\widehat{\mathcal{B}}c_K))).
\end{equation*}
Again, by commutativity of $(\ref{equation diagram pushforward and pullback})$, we may abuse notation and write $Rf_{*}=Rf_{*,LH}$.

\subsection{Derived Tensor-Hom adjunction II}\label{Section derived tensor-hom 2}
Let $\mathscr{A}$ be a sheaf of Ind-Banach algebras on $X$. In this section, we will now discuss the existence of derived functors for $-\overrightarrow{\otimes}_{\mathscr{A}}-$, $\underline{\mathcal{H}om}_{\mathscr{A}}(-,-)$, and their counterparts in the left heart. Setting $X=\Sp(K)$, we obtain the corresponding results for Ind-Banach $K$-algebras. First, notice that for any $V\in \operatorname{Shv}(X,\Indban)$, the functor:
\begin{equation*}
    V\overrightarrow{\otimes}_K-:\operatorname{Shv}(X,\Indban)\rightarrow \operatorname{Shv}(X,\Indban),
\end{equation*}
is exact. Indeed, let $0\rightarrow W_1\rightarrow W_2\rightarrow W_3\rightarrow 0$ be a strict exact sequence in $\operatorname{Shv}(X,\Indban)$. We need to show that the sequence:
\begin{equation*}
    0\rightarrow V\overrightarrow{\otimes}_KW_1\rightarrow V\overrightarrow{\otimes}_KW_2\rightarrow V\overrightarrow{\otimes}_KW_3\rightarrow 0,
\end{equation*}
is still a strict short exact sequence in $\operatorname{Shv}(X,\Indban)$.
By the adjunction $(\ref{equation tensor-hom adjunctions for sheaves})$, it is enough to show that $V\overrightarrow{\otimes}_KW_1\rightarrow V\overrightarrow{\otimes}_KW_2$ is a strict injection. Notice that $W_1\rightarrow W_2$ is a strict injection in $\operatorname{Shv}(X,\Indban)$ if and only if it is a strict injection in $\operatorname{PreShv}(X,\Indban)$. In particular, for every admissible open $U\subset X$, the map $W_1(U)\rightarrow W_2(U)$ is a strict injection. Hence, by exactness of the tensor product in $\Indban$, it follows that 
$V(U)\overrightarrow{\otimes}_KW_1(U)\rightarrow V(U)\overrightarrow{\otimes}_KW_2(U)$ is a strict injection. As sheafification is strongly exact, it follows that:
\begin{equation*}
V\overrightarrow{\otimes}_KW_1\rightarrow V\overrightarrow{\otimes}_KW_2,
\end{equation*}
is a strict injection, as we wanted to show. By exactness of $-\overrightarrow{\otimes}_K-$, we obtain a left derived functor:
\begin{multline}\label{equation derived functor tensor for sheaves}
    -\overrightarrow{\otimes}_K^{\mathbb{L}}-:\operatorname{D}(\operatorname{Shv}(X,\Indban))\times \operatorname{D}(\operatorname{Shv}(X,\Indban))\\
    \rightarrow \operatorname{D}(\operatorname{Shv}(X,\Indban)).
\end{multline}
As above, we may use the fact that every $W\in \operatorname{Shv}(X,LH(\widehat{\mathcal{B}}c_K))$ fits in a short exact sequence:
\begin{equation*}
    0\rightarrow I(U)\rightarrow I(V)\rightarrow W\rightarrow 0,
\end{equation*}
with  $U,V\in \operatorname{Shv}(X,\Indban))$, together with $(\ref{equation extension of closed symmetric monoidal category structure})$  to conclude that the left heart  $\operatorname{Shv}(X,LH(\widehat{\mathcal{B}}c_K))$ has enough flat objects. This, together with the fact that $\operatorname{Shv}(X,LH(\widehat{\mathcal{B}}c_K))$ is a Grothendieck abelian category, and the adjunction $(\ref{equation tensor-hom adjunction left heart})$, allows us to apply \cite[Theorem 14.4.8]{Kashiwara2006}  to obtain a pair of derived functors:
\begin{multline*}
-\widetilde{\otimes}_K^{\mathbb{L}}-:\operatorname{D}(\operatorname{Shv}(X,LH(\widehat{\mathcal{B}}c_K)))\times \operatorname{D}(\operatorname{Shv}(X,LH(\widehat{\mathcal{B}}c_K)))\rightarrow \operatorname{D}(\operatorname{Shv}(X,LH(\widehat{\mathcal{B}}c_K))),\\
 R\underline{\mathcal{H}om}_{LH(\widehat{\mathcal{B}}c_K)}(-,-):\operatorname{D}(\operatorname{Shv}(X,LH(\widehat{\mathcal{B}}c_K)))^{\op}\times \operatorname{D}(\operatorname{Shv}(X,LH(\widehat{\mathcal{B}}c_K)))\\
 \rightarrow \operatorname{D}(\operatorname{Shv}(X,LH(\widehat{\mathcal{B}}c_K))).
\end{multline*}
As before, the equivalence of triangulated categories $I:\operatorname{D}(\widehat{\mathcal{B}}c_K)\rightarrow \operatorname{D}(LH(\widehat{\mathcal{B}}c_K))$, induces an identification: $I(-\widehat{\otimes}_K^{\mathbb{L}}-)=I(-)\widetilde{\otimes}_K^{\mathbb{L}}I(-)$, and we may calculate the derived inner homomorphisms functor $R\underline{\mathcal{H}om}_{LH(\widehat{\mathcal{B}}c_K)}(-,-)$ explicitly using the procedure showcased in $(\ref{equation explicit computation of RHom})$. Repeating the abuse of notation in \ref{obs abuse of notation internal hom}, we will write $R\underline{\mathcal{H}om}_{\Indban}(-,-)=R\underline{\mathcal{H}om}_{LH(\widehat{\mathcal{B}}c_K)}(-,-)$. As a consequence of \cite[Theorem 14.4.8]{Kashiwara2006}, these functors fit into the following derived adjunction: 
\begin{prop}
For $U^{\bullet},V^{\bullet},W^{\bullet}\in \operatorname{D}(\operatorname{Shv}(X,\Indban))$ we have:
\begin{multline}\label{equation derived tensor-hom adjunction sheaves of bornological modules}
    \Hom_{\operatorname{D}(\operatorname{Shv}(X,\Indban))}(U^{\bullet}\overrightarrow{\otimes}^{\mathbb{L}}_KV^{\bullet},W^{\bullet})\\
    =\Hom_{\operatorname{D}(\operatorname{Shv}(X,(\Indban))}(U^{\bullet},R\underline{\mathcal{H}om}_{\Indban}(V^{\bullet},W^{\bullet})).
\end{multline}    
\end{prop}
The next goal is developing analogous adjunctions for the relative tensor product and inner homomorphism functors.\bigskip

Let $V\in\operatorname{Shv}(X,\Indban)$, $\mathcal{M}\in \Mod_{\Indban}(\mathscr{A}^{\op})$, $\mathcal{N}\in \Mod_{\Indban}(\mathscr{A})$. There are canonical isomorphisms of sheaves of Ind-Banach spaces:
\begin{equation*}
    \left(V\overrightarrow{\otimes}_K\mathscr{A}\right)\overrightarrow{\otimes}_{\mathscr{A}}\mathcal{N}=V\overrightarrow{\otimes}_K\mathcal{N}, \quad \mathcal{M}\overrightarrow{\otimes}_{\mathscr{A}}\left(\mathscr{A}\overrightarrow{\otimes}_KV\right)=\mathcal{M}\overrightarrow{\otimes}_KV.
\end{equation*}
In particular, $\mathscr{A}\overrightarrow{\otimes}_KV$ is a flat $\mathscr{A}$-module, and $V\overrightarrow{\otimes}_K\mathscr{A}$ is a flat $\mathscr{A}^{\op}$-module. Regard $\mathcal{N}$ as an object in  $\operatorname{Shv}(X,\Indban)$ via the forgetful functor. Then we have a strict epimorphism:
\begin{equation*}
    \mathscr{A}\overrightarrow{\otimes}_K\mathcal{N}\rightarrow \mathcal{N},
\end{equation*}
given by the action of $\mathscr{A}$ on $\mathcal{N}$. Hence, every object in $\Mod_{\Indban}(\mathscr{A})$ admits a strict epimorphism from a flat $\mathscr{A}$-module, and the analogous statement holds in $\Mod_{\Indban}(\mathscr{A}^{\op})$.\bigskip

As $\Mod_{\Indban}(\mathscr{A})$ is not an abelian category, this is not enough to show that $\overrightarrow{\otimes}_{\mathscr{A}}$ has a derived functor. However, we may use these constructions to obtain a derived functor for $\widetilde{\otimes}_{I(\mathscr{A})}$. Namely, we have the following identity:
\begin{equation*}
\mathcal{M}\widetilde{\otimes}_{I(\mathscr{A})}\left(I(\mathscr{A})\widetilde{\otimes}_KI(V)\right)=\mathcal{M}\widetilde{\otimes}_KI(V).
\end{equation*}
As $I(V)$ is flat for $\widetilde{\otimes}_K$, it follows that $I(\mathscr{A})\widetilde{\otimes}_KI(V)$ is a flat $I(\mathscr{A})$-module. The equivalence of abelian categories:
 \begin{equation*}
     \Mod_{LH(\widehat{\mathcal{B}}c_K)}(I(\mathscr{A}))=LH(\Mod_{\Indban}(\mathscr{A})),
 \end{equation*}
 implies that for every $I(\mathscr{A})$-module $\mathcal{M}$ there is an $\mathscr{A}$-module $\mathcal{N}$ which admits an epimorphism $I(\mathcal{N})\rightarrow \mathcal{M}$. As $I:\Mod_{\Indban}(\mathscr{A})\rightarrow \Mod_{LH(\widehat{\mathcal{B}}c_K)}(I(\mathscr{A}))$ preserves strict epimorphisms, we get an epimorphism $I(\mathscr{A})\widetilde{\otimes}_KI(\mathcal{N})\rightarrow \mathcal{M}$. In particular, $\Mod_{LH(\widehat{\mathcal{B}}c_K)}(I(\mathscr{A}))$ has enough flat objects. A similar argument shows that $\Mod_{LH(\widehat{\mathcal{B}}c_K)}(I(\mathscr{A}))^{\op})$ also has enough flat objects. As both are Grothendieck abelian categories, we obtain the following proposition:
\begin{prop}[{\cite[Proposition 3.34]{bode2021operations}}]\label{prop existence derived fucntor extension of scalars}
 The bifunctor $\widetilde{\otimes}_{I(\mathscr{A})}$ admits a derived bifunctor:
\begin{equation}\label{equation derived functor for relative tensor in left heart}
-\widetilde{\otimes}_{I(\mathscr{A})}^{\mathbb{L}}-:\operatorname{D}(\mathscr{A})\times \operatorname{D}(\mathscr{A}^{\op}) \rightarrow \operatorname{D}(\operatorname{Shv}(X,LH(\widehat{\mathcal{B}}c_K))).   
\end{equation}
Furthermore, the adjunction in $(\ref{equation relative tensor hom adjunction})$ lifts to a derived adjunction:
\begin{multline}\label{equation adjunction relative tensor-hom}
    \Hom_{\operatorname{D}(\operatorname{Shv}(X,LH(\widehat{\mathcal{B}}c_K)))}(\mathcal{M}^{\bullet}\widetilde{\otimes}_{I(\mathscr{A})}^{\mathbb{L}}\mathcal{N}^{\bullet},V^{\bullet})\\
    =\Hom_{\operatorname{D}(\mathscr{A})}(\mathcal{M}^{\bullet},R\underline{\mathcal{H}om}_{LH(\widehat{\mathcal{B}}c_K)}(\mathcal{N}^{\bullet},V^{\bullet})), 
\end{multline}
where $\mathcal{M}^{\bullet}\in \operatorname{D}(\mathscr{A})$, $\mathcal{N}^{\bullet}\in \operatorname{D}(\mathscr{A}^{\op})$, and $V^{\bullet}\in\operatorname{D}(\operatorname{Shv}(X,LH(\widehat{\mathcal{B}}c_K)))$.    
\end{prop}
Even though the functor $\overrightarrow{\otimes}_{\mathscr{A}}$ cannot be derived explicitly in the unbounded derived category, it is still possible to construct a derived functor:
\begin{equation}\label{equation derived functor relative tensor product in boinded above derived categories}
  -\overrightarrow{\otimes}_{\mathscr{A}}^{\mathbb{L}}-:\operatorname{D}^-(\mathscr{A})\times \operatorname{D}^-(\mathscr{A}^{\op}) \rightarrow \operatorname{D}^-(\operatorname{Shv}(X,LH(\widehat{\mathcal{B}}c_K))),  
\end{equation}
satisfying that $\overrightarrow{\otimes}_{\mathscr{A}}^{\mathbb{L}}=\widetilde{\otimes}_{I(\mathscr{A})}^{\mathbb{L}}$. In particular, for $\mathcal{M}\in \Mod_{\Indban}(\mathscr{A})$, $\mathcal{N}\in \Mod_{\Indban}(\mathscr{A}^{\op})$, we have the following identification of $I(\mathscr{A})$-modules:
\begin{equation*}
\operatorname{H}^0(\mathcal{M}\overrightarrow{\otimes}_{\mathscr{A}}^{\mathbb{L}}\mathcal{N})=I(\mathcal{M})\widetilde{\otimes}_{I(\mathscr{A})}I(\mathcal{N}).
\end{equation*}
 Notice that this is in general not isomorphic to $I(\mathcal{M}\overrightarrow{\otimes}_{\mathscr{A}}\mathcal{N})$. Details of this construction can be found in \cite[Proposition 3.36]{bode2021operations}.\bigskip
 
There is a version of the previous adjunction for $\underline{\mathcal{H}om}_{\mathscr{A}}$. In particular, a second application of \cite[Theorem 14.4.8]{Kashiwara2006} to this setting yields the following proposition:
\begin{prop}[{\cite[Proposition 3.34]{bode2021operations}}]
The bifunctor $\underline{\mathcal{H}om}_{I(\mathscr{A})}(-,-)$ admits a derived bifunctor:
\begin{equation*}
 R\underline{\mathcal{H}om}_{\mathscr{A}}(-,-):=R\underline{\mathcal{H}om}_{I(\mathscr{A})}(-,-):\operatorname{D}(\mathscr{A})^{\op}\times \operatorname{D}(\mathscr{A})\rightarrow \operatorname{D}(\operatorname{Shv}(X,\Indban)).   
\end{equation*}
Furthermore,  the adjunction in $(\ref{equation tensor relative hom adjunction})$ lifts to a derived adjunction:
\begin{multline}\label{equation derived tensor relative hom adjunction}
    \Hom_{\operatorname{D}(\mathscr{A})}(\mathcal{M}^{\bullet}\overrightarrow{\otimes}^{\mathbb{L}}_K\mathcal{V}^{\bullet},\mathcal{N}^{\bullet})\\
    =\Hom_{\operatorname{D}(\operatorname{Shv}(X,\Indban))}(\mathcal{V}^{\bullet},R\underline{\mathcal{H}om}_{\mathscr{A}}(\mathcal{M}^{\bullet},\mathcal{N}^{\bullet})),
\end{multline}
where $\mathcal{M}^{\bullet},\mathcal{N}^{\bullet}\in \operatorname{D}(\mathscr{A})$, and $\mathcal{V}^{\bullet}\in \operatorname{D}(\operatorname{Shv}(X,\Indban)$.   
\end{prop}
As before, $R\underline{\mathcal{H}om}_{\mathscr{A}}(-,-)$ may be calculated following the procedure in  $(\ref{equation explicit computation of RHom})$.
\begin{defi}
 Let $\mathcal{M}^{\bullet},\mathcal{N}^{\bullet}\in \operatorname{D}(\mathscr{A})$, and $n\in \mathbb{Z}$. We define the  $\underline{\mathcal{E}xt}^n_{\mathscr{A}}(\mathcal{M}^{\bullet},\mathcal{N}^{\bullet})$ sheaves as:
 \begin{equation*}
     \underline{\mathcal{E}xt}^n_{\mathscr{A}}(\mathcal{M}^{\bullet},\mathcal{N}^{\bullet})=\operatorname{H}^n(R\underline{\mathcal{H}om}_{\mathscr{A}}(\mathcal{M}^{\bullet},\mathcal{N}^{\bullet})).
  \end{equation*}   
\end{defi}
This is an object of $\operatorname{Shv}(X,LH(\widehat{\mathcal{B}}c_K))$, and for $\mathcal{M},\mathcal{N}\in \Mod_{\Indban}(\mathscr{A})$ we have:
\begin{equation*}
\underline{\mathcal{E}xt}^0_{\mathscr{A}}(\mathcal{M},\mathcal{N})=\underline{\mathcal{H}om}_{I(\mathscr{A})}(I(\mathcal{M}),I(\mathcal{N}))=I(\underline{\mathcal{H}om}_{\mathscr{A}}(\mathcal{M},\mathcal{N})),
\end{equation*}
where the first identity follows by the definition of derived functor, and the second identity is $(\ref{equation extension of inner hom to left heart})$.\\

Our next goal is giving an explicit description of the inner ext sheaves. For this reason, we recall the following well-known construction:
\begin{defi}
 Let $U\subset X$ be an admissible open subset. The extension by zero functor:
 \begin{equation*}
     i_{U!}:\Mod_{LH(\widehat{\mathcal{B}}c_K)}(I(\mathscr{A})_{\vert U})\rightarrow \Mod_{LH(\widehat{\mathcal{B}}c_K)}(I(\mathscr{A})),
 \end{equation*}
is defined for any $\mathcal{M}\in \Mod_{LH(\widehat{\mathcal{B}}c_K)}(I(\mathscr{A})_{\vert U})$ by  the formula:
\begin{equation*}
    i_{U!}\mathcal{M}=\left(i^P_{U!}\mathcal{M}\right)^{a},
\end{equation*}
where $i^P_{U!}\mathcal{M}$  is the presheaf defined on each admissible open subspace $V\subset X$ by:
\begin{equation*}
i^p_{U!}\mathcal{M}(V)=\mathcal{M}(V) \textnormal{ if } V\subset U, \textnormal{ and }     i^P_{U!}\mathcal{M}(V)=0 \textnormal{ if } V\not\subset U.
\end{equation*}
\end{defi}
We now show that extension by zero satisfies the usual properties:
\begin{prop}\label{prop extension by zero}    
Extension by zero satisfies the following properties:
    \begin{enumerate}[label=(\roman*)]
        \item $i_{U!}:\Mod_{LH(\widehat{\mathcal{B}}c_K)}(I(\mathscr{A})_{\vert U})\rightarrow \Mod_{LH(\widehat{\mathcal{B}}c_K)}(I(\mathscr{A}))$ is exact.
        \item $i_{U!}$ is left adjoint to the restriction functor:
        \begin{equation*}
            i^{-1}:\Mod_{LH(\widehat{\mathcal{B}}c_K)}(I(\mathscr{A}))\rightarrow \Mod_{LH(\widehat{\mathcal{B}}c_K)}(I(\mathscr{A})_{\vert U}).
        \end{equation*}
        \item $i^{-1}:\operatorname{D}(\mathscr{A})\rightarrow \operatorname{D}(\mathscr{A})_{\vert U}$ preserves homotopically injective objects.
    \end{enumerate}
    In particular, we have a derived adjunction $i_{U!}:\operatorname{D}(\mathscr{A}_{\vert U})\rightarrow\operatorname{D}(\mathscr{A}) :i_U^{-1}$.
\end{prop}
\begin{proof}
The first statement is \cite[Tag 03DI]{stacks-project}, the second one is  \cite[Tag 03DJ]{stacks-project}. Statement $(iii)$ follows by the fact that for a homotopically injective complex $\mathcal{I}^{\bullet}\in \operatorname{D}(\mathscr{A})$ and an acyclic complex $C^{\bullet}\in \operatorname{D}(\mathscr{A}_{\vert U})$ we have:
\begin{equation*}
    \Hom_{\operatorname{K}(I(\mathscr{A})_{\vert U})}(C^{\bullet},\mathcal{I}^{\bullet}_{\vert U})=\Hom_{\operatorname{K}(I(\mathscr{A}))}(i_{U!}C^{\bullet},\mathcal{I}^{\bullet})=0,
\end{equation*}
where the middle map follows by the fact that $i_{U!}$ is exact, so it preserves acyclicity.\\
In order to show the derived adjunction, choose complexes $\mathcal{M}^{\bullet}\in \operatorname{D}(\mathscr{A}_{\vert U})$ and $\mathcal{N}^{\bullet}\in \operatorname{D}(\mathscr{A})$. As $\Mod_{LH(\widehat{\mathcal{B}}c_K)}(I(\mathscr{A}))$ is a Grothendieck abelian category, we have a homotopically injective object $\mathcal{I}^{\bullet}$ with a quasi-isomorphism $\mathcal{N}^{\bullet}\rightarrow \mathcal{I}^{\bullet}$. But then we have:
\begin{multline*}
  \Hom_{\operatorname{D}(\mathscr{A})}(i_{U!}\mathcal{M}^{\bullet},\mathcal{N}^{\bullet})=\Hom_{\operatorname{K}(I(\mathscr{A}))}(i_{U!}\mathcal{M}^{\bullet},\mathcal{I}^{\bullet})
  =\Hom_{\operatorname{K}(I(\mathscr{A}_{\vert U}))}(\mathcal{M}^{\bullet},\mathcal{I}_{\vert U}^{\bullet})\\
  =\Hom_{\operatorname{D}(\mathscr{A}_{\vert U})}(\mathcal{M}^{\bullet},\mathcal{I}_{\vert U}^{\bullet})=\Hom_{\operatorname{D}(\mathscr{A}_{\vert U})}(\mathcal{M}^{\bullet},\mathcal{N}_{\vert U}^{\bullet}),
\end{multline*}
where the last two identities follow because $i^{-1}_{U}$ is exact and preserves homotopically injective objects.
\end{proof}
The next goal is relating extension by zero with inner homomorphisms and tensor products:
\begin{Lemma}\label{lemma extension}
Let $\mathcal{M}\in \Mod_{\Indban}(\mathscr{A}_{\vert U})$. Then $i_{U!}I(\mathcal{M})$ is in the image of the functor:
\begin{equation*}
    I:\Mod_{\Indban}(\mathscr{A})\rightarrow\Mod_{LH(\widehat{\mathcal{B}}c_K)}(I(\mathscr{A})).
\end{equation*}   
\end{Lemma}
\begin{proof}
 Follows at once by the fact that the functor:
 \begin{equation*}
     I:\operatorname{PreShv}(X,\Indban)\rightarrow \operatorname{PreShv}(X,LH(\widehat{\mathcal{B}}c_K)),
 \end{equation*}
commutes with sheafification.  
\end{proof}
\begin{Lemma}\label{Lemma iso tensor extension restriction}
  Let $\mathcal{M}\in \operatorname{D}(\mathscr{A})$, and $V^{\bullet}\in \operatorname{D}(\operatorname{Shv}(U,LH(\widehat{\mathcal{B}}c_K)))$. There is a canonical isomorphism in  $\operatorname{D}(\mathscr{A})$:
  \begin{equation*}
      i_{U!}\left( \mathcal{M}_{\vert U}^{\bullet}\widetilde{\otimes}_K^{\mathbb{L}}V^{\bullet}\right)=\mathcal{M}^{\bullet}\widetilde{\otimes}_K^{\mathbb{L}}i_{U!}V^{\bullet}.
  \end{equation*}
\end{Lemma}
\begin{proof}
Choose $\widetilde{\mathcal{M}}^{\bullet}\in \operatorname{K}(\mathscr{A})$, and $\widetilde{V}^{\bullet}\in \operatorname{K}(\operatorname{Shv}(U,\Indban))$ such that we have quasi-isomorphisms:
\begin{equation*}
    I(\widetilde{\mathcal{M}}^{\bullet})\rightarrow \mathcal{M}^{\bullet}, \quad I(\widetilde{V}^{\bullet})\rightarrow V^{\bullet}. 
\end{equation*}
In this case, by Lemma \ref{lemma extension} we have the following identities:
\begin{equation*}
  \mathcal{M}^{\bullet}\widetilde{\otimes}_K^{\mathbb{L}}i_{U!}V^{\bullet}=I(\widetilde{\mathcal{M}}^{\bullet})\widetilde{\otimes}_Ki_{U!}I(\widetilde{V}^{\bullet}), \quad   \mathcal{M}_{\vert U}^{\bullet}\widetilde{\otimes}_K^{\mathbb{L}}V^{\bullet}=I(\widetilde{\mathcal{M}}^{\bullet}_{\vert U})\widetilde{\otimes}_KI(\widetilde{V}^{\bullet}).
\end{equation*}
By construction of the derived tensor product and the extension by zero, for each $n\in\mathbb{Z}$, we have that  $\left(\mathcal{M}^{\bullet}\widetilde{\otimes}_K^{\mathbb{L}}i_{U!}V^{\bullet}\right)^n$ is the sheafification of the following presheaf :
\begin{equation*}
    W\mapsto \bigoplus_{r+s=n}\Gamma(W,I(\widetilde{\mathcal{M}})^{r})\widetilde{\otimes}_K\Gamma(W,I(\widetilde{V})^{s}) \textnormal{ if } W\subset U, \quad 0 \textnormal{ if } W\not\subset U.
\end{equation*}
As this presheaf determines $i_{U!}\left(\mathcal{M}_{\vert U}^{\bullet}\widetilde{\otimes}_K^{\mathbb{L}}V^{\bullet}\right)^n=i_{U!}\left(I(\widetilde{\mathcal{M}})_{\vert U}^{\bullet}\widetilde{\otimes}_KI(\widetilde{V})^{\bullet}\right)^n$, we are done.
\end{proof}
\begin{prop}\label{prop restriction of inner hom}
For each admissible open subspace $U\subset X$ there is a natural isomorphism of derived functors:
\begin{equation*}
    R\underline{\mathcal{H}om}_{\mathscr{A}}(-,-)_{\vert U}=R\underline{\mathcal{H}om}_{\mathscr{A}_{\vert U}}((-)_{\vert U},(-)_{\vert U}).
\end{equation*}
\end{prop}
\begin{proof}
Let $U\subset X$ be an admissible open subspace, and choose complexes $\mathcal{M}^{\bullet},\mathcal{N}^{\bullet}\in \operatorname{D}(\mathscr{A})$, and  $V^{\bullet}\in \operatorname{D}(\operatorname{Shv}(U,\Indban))$. Then we have the following identities:
\begin{align*}
    \Hom_{\operatorname{D}(\operatorname{Shv}(U,\Indban))}(V^{\bullet},R\underline{\mathcal{H}om}_{\mathscr{A}_{\vert U}}(\mathcal{M}_{\vert U}^{\bullet},\mathcal{N}^{\bullet}_{\vert U}))
    =\Hom_{\operatorname{D}(\mathscr{A}_{\vert U})}(\mathcal{M}_{\vert U}^{\bullet}\overrightarrow{\otimes}_K^{\mathbb{L}}V^{\bullet},\mathcal{N}^{\bullet}_{\vert U})&\\
=\Hom_{\operatorname{D}(\mathscr{A})}(i_{U!}\left(\mathcal{M}_{\vert U}^{\bullet}\overrightarrow{\otimes}_K^{\mathbb{L}}V^{\bullet}\right),\mathcal{N}^{\bullet})
    =\Hom_{\operatorname{D}(\mathscr{A})}(\mathcal{M}^{\bullet}\overrightarrow{\otimes}_K^{\mathbb{L}}i_{U!}V^{\bullet},\mathcal{N}^{\bullet})&\\
    =\Hom_{\operatorname{D}(\operatorname{Shv}(X,\Indban))}(i_{U!}V^{\bullet},R\underline{\mathcal{H}om}_{\mathscr{A}}(\mathcal{M}^{\bullet},\mathcal{N}^{\bullet}))&\\
    =\Hom_{\operatorname{D}(\operatorname{Shv}(U,\Indban))}(V^{\bullet},R\underline{\mathcal{H}om}_{\mathscr{A}}(\mathcal{M}^{\bullet},\mathcal{N}^{\bullet})_{\vert U})&,
\end{align*}
so the isomorphism follows by Yoneda.
\end{proof}
As a consequence of this, it follows that if $\mathcal{M}^{\bullet},\mathcal{N}^{\bullet}\in \operatorname{D}(\mathscr{A})$ have cohomology supported on a closed subvariety $Y\subset X$, then $R\underline{\mathcal{H}om}_{\mathscr{A}}(\mathcal{M}^{\bullet},\mathcal{N}^{\bullet})$ also has cohomology supported on $Y$.
\begin{defi}
Let $U\subset X$ be an admissible open subspace. We define the functor:
\begin{multline*}
    R\underline{\Hom}_{\mathscr{A}_{\vert U}}((-)_{\vert U},(-)_{\vert U}):=R\Gamma(U,R\underline{\mathcal{H}om}_{\mathscr{A}}(-,-)):\operatorname{D}(\mathscr{A})^{\op}\times \operatorname{D}(\mathscr{A})\\
    \rightarrow \operatorname{D}(LH(\widehat{\mathcal{B}}c_K)).
\end{multline*}
For each $n\in\mathbb{Z}$, we define:
\begin{equation*}
    \underline{\operatorname{Ext}}^n_{\mathscr{A}_{\vert U}}((-)_{\vert U},(-)_{\vert U})=\operatorname{H}^n\left(R\underline{\Hom}_{\mathscr{A}_{\vert U}}((-)_{\vert U},(-)_{\vert U})\right).
\end{equation*}
\end{defi}
We now recall the definition of the constant sheaf functor:
\begin{defi}
Let $\mathcal{C}$ be an elementary quasi-abelian category. The constant sheaf functor is the functor:
\begin{equation*}
    c:\mathcal{C}\rightarrow \operatorname{Shv}(X,\mathcal{C}),
\end{equation*}
defined by sending an object $x\in\mathscr{C}$ to the sheafification of the constant presheaf of value $x$. We call $c(x)$ the constant sheaf of stalk $x$. 
\end{defi}
Notice that $c$ is strongly exact, and fits into a commutative diagram:
\begin{equation}\label{equation derived adjunction constant sheaf global functions}
\begin{tikzcd}
c:\Indban \arrow[d, "I"] \arrow[r] & {\operatorname{Shv}(X,\Indban):\Gamma(X,-)} \arrow[d, "I"] \arrow[l, shift right=2] \\
c:LH(\widehat{\mathcal{B}}c_K) \arrow[r]            & {\operatorname{Shv}(X,LH(\widehat{\mathcal{B}}c_K)):\Gamma(X,-)} \arrow[l, shift right=2]           
\end{tikzcd}   
\end{equation}
where the pairs of horizontal maps are adjunctions. As the functor:
\begin{equation*}
   c:LH(\widehat{\mathcal{B}}c_K)\rightarrow \operatorname{Shv}(X,LH(\widehat{\mathcal{B}}c_K)), 
\end{equation*}
is exact, it induces a derived adjunction:
\begin{equation*}
    c:\operatorname{D}(LH(\widehat{\mathcal{B}}c_K))\leftrightarrows \operatorname{D}(\operatorname{Shv}(X,LH(\widehat{\mathcal{B}}c_K))):R\Gamma(X,-).
\end{equation*}
We may use this adjunction to show the following:
\begin{prop}\label{prop composition of global sections and inner hom}
The functor $ \widetilde{\Hom}_{I(\mathscr{A})}(-,-):=\Gamma(X,\underline{\mathcal{H}om}_{I(\mathscr{A})}(-,-))$
admits a right derived functor:
\begin{equation*}
    R\widetilde{\Hom}_{\mathscr{A}}(-,-):\operatorname{D}(\mathscr{A}^{\op})\times\operatorname{D}(\mathscr{A})\rightarrow \operatorname{D}(\widehat{\mathcal{B}}c_K).
\end{equation*}
Furthermore, there is a derived adjunction:
\begin{equation*}
\Hom_{\operatorname{D}(\mathscr{A})}(\mathcal{M}^{\bullet}\overrightarrow{\otimes}^{\mathbb{L}}_Kc(\mathcal{V}^{\bullet}),\mathcal{N}^{\bullet})=\Hom_{\operatorname{D}(LH(\widehat{\mathcal{B}}c_K))}(\mathcal{V}^{\bullet},R\widetilde{\Hom}_{\mathscr{A}}(\mathcal{M}^{\bullet},\mathcal{N}^{\bullet})),        
\end{equation*}
where $\mathcal{M}^{\bullet},\mathcal{N}^{\bullet}\in \operatorname{D}(\mathscr{A})$, and $\mathcal{V}^{\bullet}\in\operatorname{D}(LH(\widehat{\mathcal{B}}c_K))$.
Furthermore, we have:
\begin{equation*}
R\widetilde{\Hom}_{\mathscr{A}}(-,-)=R\Gamma(X,-)\circ R\underline{\mathcal{H}om}_{\mathscr{A}}(-,-).   
\end{equation*}
\end{prop}
\begin{proof}
The adjunction $c:LH(\widehat{\mathcal{B}}c_K)\leftrightarrows \operatorname{Shv}(X,LH(\widehat{\mathcal{B}}c_K)):\Gamma(X,-)$, together with the adequate version of $(\ref{equation tensor relative hom adjunction})$, induces an adjunction:
\begin{equation*}
\Hom_{I(\mathscr{A})}(\mathcal{M}\widetilde{\otimes}_Kc(\mathcal{V}),\mathcal{N})=\Hom_{LH(\widehat{\mathcal{B}}c_K)}(\mathcal{V},\widetilde{\Hom}_{I(\mathscr{A})}(\mathcal{M},\mathcal{N})),    
\end{equation*}
where $\mathcal{M},\mathcal{N}\in \Mod_{LH(\widehat{\mathcal{B}}c_K)}(I(\mathscr{A}))$, and $\mathcal{V}\in LH(\widehat{\mathcal{B}}c_K)$. As stated above, the objects in the essential image of $I:\operatorname{Shv}(X,\Indban)\rightarrow \operatorname{Shv}(X,LH(\widehat{\mathcal{B}}c_K))$ are flat with respect to $-\widetilde{\otimes}_K-$. In particular, commutativity of the diagram above shows that for every $\mathcal{V}\in LH(\widehat{\mathcal{B}}c_K)$ we may find some $W\in \Indban$ such that there is an epimorphism $I(c(W))\rightarrow c(\mathcal{V})$, and $I(c(W))$ is flat with respect to $-\widetilde{\otimes}_K-$. Therefore, the functor:
 \begin{equation*}
     -\widetilde{\otimes}_Kc(-):\Mod_{LH(\widehat{\mathcal{B}}c_K)}(I(\mathscr{A}))\times LH(\widehat{\mathcal{B}}c_K)\rightarrow \Mod_{LH(\widehat{\mathcal{B}}c_K)}(I(\mathscr{A})),
 \end{equation*}
satisfies the conditions of \cite[Theorem 14.4.8]{Kashiwara2006}, and it follows that $R\widetilde{\Hom}_{\mathscr{A}_{\vert U}}(-,-)$ exists, and it is right adjoint to $-\overrightarrow{\otimes}_K^{\mathbb{L}}c(-)$ in the derived category.\\
For the second part of the statement, choose $\mathcal{M}^{\bullet},\mathcal{N}^{\bullet}\in \operatorname{D}(\mathscr{A})$, and $\mathcal{V}^{\bullet}\in\operatorname{D}(LH(\widehat{\mathcal{B}}c_K))$. Then by the identity in $(\ref{equation derived adjunction constant sheaf global functions})$ we have:
 \begin{multline*}
 \Hom_{\operatorname{D}(LH(\widehat{\mathcal{B}}c_K))}(\mathcal{V}^{\bullet},R\Gamma(X,R\underline{\mathcal{H}om}_{\mathscr{A}}(\mathcal{M}^{\bullet},\mathcal{N}^{\bullet})))\\= 
 \Hom_{\operatorname{D}(\operatorname{Shv}(X,LH(\widehat{\mathcal{B}}c_K)))}(c(\mathcal{V}^{\bullet}),R\underline{\mathcal{H}om}_{\mathscr{A}}(\mathcal{M}^{\bullet},\mathcal{N}^{\bullet}))\\
 =\Hom_{\operatorname{D}(\mathscr{A})}(\mathcal{M}^{\bullet}\overrightarrow{\otimes}^{\mathbb{L}}_Kc(\mathcal{V}^{\bullet}),\mathcal{N}^{\bullet}). 
 \end{multline*}
 Hence, Yoneda's Lemma implies that $R\widetilde{\Hom}_{\mathscr{A}}(-,-)=R\Gamma(X,-)\circ R\underline{\mathcal{H}om}_{\mathscr{A}}(-,-)$.
\end{proof}
We can give a simple characterization of the $\underline{\mathcal{E}xt}_{\mathscr{A}}$ sheaves in some situations:
\begin{Lemma}\label{Lemma internal hom of flat and injective is injective}
Let $\mathcal{M}\in \operatorname{D}^-(\mathscr{A})$, and $\mathcal{I}\in \operatorname{D}(\mathscr{A})$ be homotopically injective. Then the complex:
\begin{equation*}
    \underline{\mathcal{H}om}_{I(\mathscr{A})}(I(\mathcal{M}^{\bullet}),\mathcal{I}^{\bullet})=\operatorname{Tot}_{\pi}\left(\underline{\mathcal{H}om}^{\bullet,\bullet}_{I(\mathscr{A})}(I(\mathcal{M}^{-\bullet}),\mathcal{I}^{\bullet})\right),
\end{equation*}
is a homotopically injective complex in $\operatorname{D}(\widehat{\mathcal{B}}c_K)$. 
\end{Lemma}
\begin{proof}
Let $\mathcal{N}^{\bullet}\in \operatorname{K}(I(\mathscr{A}))$ be an acyclic complex. We need to show that:
\begin{equation*}
    \Hom_{\operatorname{K}(I(\mathscr{A}))}(\mathcal{N}^{\bullet},\underline{\mathcal{H}om}_{I(\mathscr{A})}(I(\mathcal{M}^{\bullet}),\mathcal{I}^{\bullet}))=\Hom_{\operatorname{K}(I(\mathscr{A}))}(\mathcal{N}^{\bullet}\widetilde{\otimes}_KI(\mathcal{M}^{\bullet}),\mathcal{I}^{\bullet})=0.
\end{equation*}
As $\mathcal{I}^{\bullet}$ is homotopically injective, it suffices to show that $\mathcal{N}^{\bullet}\widetilde{\otimes}_KI(\mathcal{M}^{\bullet})$ is an acyclic complex. For each $n\in \mathbb{Z}$, we have the truncations: $\sigma_{\geq n}I(\mathcal{M}^{\bullet})$, given by:
\begin{equation*}
    \sigma_{\geq n}I(\mathcal{M}^{m})=I(\mathcal{M}^m) \textnormal{ for } m\geq  n, \textnormal{ and } \sigma_{\geq n}I(\mathcal{M}^{m})=0 \textnormal{ for } m< n.
\end{equation*}
For each $n\in\mathbb{Z}$, the truncation $\sigma_{\geq n}I(\mathcal{M}^{\bullet})$  is in the essential image of the functor:
\begin{equation*}
    I:\operatorname{K}^b(\mathscr{A})\rightarrow \operatorname{K}^b(I(\mathscr{A})),
\end{equation*}
and we have a filtered colimit $I(\mathcal{M}^{\bullet})=\varinjlim \sigma_{\geq n}I(\mathcal{M}^{\bullet})$. \newline
As tensor products commute with colimits, and $LH(\widehat{\mathcal{B}}c_K)$ is a Grothendieck abelian category, the following identity holds for each $m\in\mathbb{Z}$: 
\begin{equation*} \operatorname{H}^{m}\left(\mathcal{N}^{\bullet}\widetilde{\otimes}_KI(\mathcal{M}^{\bullet})\right)=\varinjlim_n\operatorname{H}^m\left(\mathcal{N}^{\bullet}\widetilde{\otimes}_K\sigma_{\geq n}I(\mathcal{M}^{\bullet})\right).   
\end{equation*}
 For each $n\in\mathbb{Z}$, the complex $\sigma_{\geq n}I(\mathcal{M}^{\bullet})$ is a bounded complex of flat objects. Hence, as $\mathcal{N}$ is acyclic, a standard spectral sequence argument shows that $\mathcal{N}^{\bullet}\widetilde{\otimes}_K\sigma_{\geq n}I(\mathcal{M}^{\bullet})$ is acyclic for each $n\in\mathbb{Z}$. In particular, $\mathcal{N}^{\bullet}\widetilde{\otimes}_KI(\mathcal{M}^{\bullet})$ is also acyclic.
\end{proof}
\begin{prop}\label{prop expression of higher ext functors}
Let $\mathcal{M}^{\bullet}\in \operatorname{D}^-(\mathscr{A})$, $\mathcal{N}^{\bullet}\in \operatorname{D}(\mathscr{A})$. For any $n\in \mathbb{Z}$ the inner $n$-th Ext sheaf  
$\underline{\mathcal{E}xt}^n_{\mathscr{A}}(\mathcal{M}^{\bullet},\mathcal{N}^{\bullet})$ is the sheafification of the presheaf:
\begin{equation*}
   U\mapsto \underline{\operatorname{Ext}}^n_{\mathscr{A}_{\vert U}}(\mathcal{M}^{\bullet}_{\vert U},\mathcal{N}^{\bullet}_{\vert U}).
\end{equation*}
\end{prop}
\begin{proof}
As $\Mod_{LH(\widehat{\mathcal{B}}c_K)}(I(\mathscr{A}))$ is a Grothendieck abelian category, there is a homotopically injective object $\mathcal{I}^{\bullet}$ with a quasi-isomorphism $I(\mathcal{N}^{\bullet})\rightarrow \mathcal{I}^{\bullet}$. By $(\ref{equation explicit computation of RHom})$
it follows that $R\underline{\mathcal{H}om}_{\mathscr{A}}(\mathcal{M}^{\bullet},\mathcal{N}^{\bullet})$
may be represented in $\operatorname{D}(\operatorname{Shv}(X,LH(\widehat{\mathcal{B}}c_K)))$ by the following complex:
\begin{equation*}
    R\underline{\mathcal{H}om}_{\mathscr{A}}(\mathcal{M}^{\bullet},\mathcal{N}^{\bullet})=\underline{\mathcal{H}om}_{I(\mathscr{A})}(I(\mathcal{M}^{\bullet}),\mathcal{I}^{\bullet})=\operatorname{Tot}_{\pi}\left(\Hom^{\bullet,\bullet}_{I(\mathscr{A})}(I(\mathcal{M}^{-\bullet}),\mathcal{I}^{\bullet})\right).
\end{equation*}
By Lemma \ref{Lemma internal hom of flat and injective is injective}, $\underline{\mathcal{H}om}_{I(\mathscr{A})}(I(\mathcal{M}^{\bullet}),\mathcal{I}^{\bullet})$ is a homotopically injective complex. 
Thus, for each admissible open subspace $U\subset X$ we have:
\begin{equation*}
    R\Gamma(U,\underline{\mathcal{H}om}_{\mathscr{A}}(\mathcal{M}^{\bullet},\mathcal{N}^{\bullet}))=\Gamma(U,\underline{\mathcal{H}om}_{I(\mathscr{A})}(I(\mathcal{M}^{\bullet}),\mathcal{I}^{\bullet})).
\end{equation*}
Fix some integer $n\in\mathbb{Z}$. Then $\underline{\mathcal{E}xt}^n_{\mathscr{A}}(\mathcal{M}^{\bullet},\mathcal{N}^{\bullet})$ is the sheafification of the $n$-th cohomology presheaf of $R\underline{\mathcal{H}om}_{\mathscr{A}}(\mathcal{M}^{\bullet},\mathcal{N}^{\bullet})$. In particular, it is the sheafification of the presheaf defined for each admissible open $U\subset X$ by:
\begin{multline*}
\operatorname{H}^n\left(\Gamma(U,\underline{\mathcal{H}om}_{I(\mathscr{A})}(I(\mathcal{M}^{\bullet}),\mathcal{I}^{\bullet}) \right)=\operatorname{H}^n\left(R\Gamma(U,\underline{\mathcal{H}om}_{I(\mathscr{A})}(I(\mathcal{M}^{\bullet}),\mathcal{I}^{\bullet}) \right)\\
=\operatorname{H}^n\left(R\Gamma(U,R\underline{\mathcal{H}om}_{\mathscr{A}}(\mathcal{M}^{\bullet},\mathcal{N}^{\bullet}) \right)=\underline{\operatorname{Ext}}_{\mathscr{A}_{\vert U}}(\mathcal{M}_{\vert U},\mathcal{N}_{\vert U}),
\end{multline*}
and this is the identity we wanted to show.
\end{proof}

\subsection{Base-change along a bimodule}\label{section base-change bimodule}
In the next two sections, we will develop some technical tools needed for the proof of the version of Kashiwara's equivalence presented in Proposition \ref{prop derived internal hom and kashiwara}.\bigskip

Let $\mathscr{B}$ be another sheaf of Ind-Banach algebras on $X$, and let $\mathcal{S}$ be a complete $(\mathscr{B},\mathscr{A})$-bimodule. That is, a sheaf $\mathcal{S}\in \operatorname{Shv}(X,\Indban)$ equipped with actions: 
\begin{equation*}
    \mathscr{B}\overrightarrow{\otimes}_K\mathcal{S}\rightarrow \mathcal{S}, \quad \mathcal{S}\overrightarrow{\otimes}_K\mathscr{A}\rightarrow \mathcal{S},
\end{equation*}
satisfying that the following diagram is commutative:
\begin{equation*}
\begin{tikzcd}
\mathscr{B}\overrightarrow{\otimes}_K\mathcal{S}\overrightarrow{\otimes}_K\mathscr{A} \arrow[r] \arrow[d] & \mathscr{B}\overrightarrow{\otimes}_K\mathcal{S} \arrow[d] \\
\mathcal{S}\overrightarrow{\otimes}_K\mathscr{A} \arrow[r]                                         & \mathcal{S}                                        
\end{tikzcd}  
\end{equation*}
Our interest in bimodules stems from the fact that every $(\mathscr{B},\mathscr{A})$-bimodule induces an adjunction:
\begin{equation*}
    \Mod_{\Indban}(\mathscr{A})\leftrightarrows \Mod_{\Indban}(\mathscr{B}).
\end{equation*}
To wit, we have an extension of scalars functor:
\begin{equation*}
    \mathcal{S}\overrightarrow{\otimes}_{\mathscr{A}}-:\Mod_{\Indban}(\mathscr{A})\rightarrow \Mod_{\Indban}(\mathscr{B}),
\end{equation*}
given by sending a left $\mathscr{A}$-module $\mathcal{M}$ to the sheaf  of Ind-Banach spaces $\mathcal{S}\overrightarrow{\otimes}_{\mathscr{A}}\mathcal{M}$ with the structure of a $\mathscr{B}$-module given by the formula:
\begin{equation*}
\mathscr{B}\overrightarrow{\otimes}_K\left(\mathcal{S}\overrightarrow{\otimes}_{\mathscr{A}}\mathcal{M} \right)=\left(\mathscr{B}\overrightarrow{\otimes}_K\mathcal{S}\right)\overrightarrow{\otimes}_{\mathscr{A}}\mathcal{M}\rightarrow  \mathcal{S}\overrightarrow{\otimes}_{\mathscr{A}}\mathcal{M}.
\end{equation*}
This functor is often times called extension of scalars with respect to $\mathcal{S}$. In the other direction, we have the restriction of scalars functor:
\begin{equation*}
    \underline{\mathcal{H}om}_{\mathscr{B}}(\mathcal{S},-):\Mod_{\Indban}(\mathscr{B})\rightarrow \Mod_{\Indban}(\mathscr{A}),
\end{equation*}
is defined by sending a left $\mathscr{B}$-module $\mathcal{N}$ to the sheaf of Ind-Banach spaces $\underline{\mathcal{H}om}_{\mathscr{B}}(\mathcal{S},\mathcal{N})$ with the $\mathscr{A}$-module structure given by:
\begin{equation*}
    \mathscr{A}\overrightarrow{\otimes}_K\underline{\mathcal{H}om}_{\mathscr{B}}(\mathcal{S},\mathcal{N})\rightarrow \underline{\mathcal{H}om}_{\mathscr{B}}(\mathcal{S}\overrightarrow{\otimes}_K\mathscr{A},\mathcal{N})\rightarrow \underline{\mathcal{H}om}_{\mathscr{B}}(\mathcal{S},\mathcal{N}),
\end{equation*}
as in $(\ref{equation map in tensor hom adjunction})$. This functor is denoted restriction of scalars with respect to $\mathcal{S}$.\\
If $\mathscr{A}=K$, we recover 
the extension and restriction of scalars adjunction $(\ref{equation tensor relative hom adjunction})$. Furthermore, this construction only uses the properties of symmetric monoidal categories. In particular, $I(\mathcal{S})$ is a $(I(\mathscr{B}),I(\mathscr{A}))$-bimodule, and we may repeat this constructions in the left hearts to obtain a pair of functors $I(\mathcal{S})\widetilde{\otimes}_{I(\mathscr{A})}-$, $\underline{\mathcal{H}om}_{I(\mathscr{B})}(I(\mathcal{S}),-)$. Let us now investigate the properties of these functors:
\begin{Lemma}\label{lemma adjunction extension and restriction along a bimodule}
Extension and restriction of scalars along $\mathcal{S}$ fits into an adjunction:
\begin{equation*}   \mathcal{S}\overrightarrow{\otimes}_{\mathscr{A}}-:\Mod_{\Indban}(\mathscr{A})\leftrightarrows \Mod_{\Indban}(\mathscr{B}):\underline{\mathcal{H}om}_{\mathscr{B}}(\mathcal{S},-).
\end{equation*}
In particular, for $\mathcal{M}\in \Mod_{\Indban}(\mathscr{A})$, and $\mathcal{N}\in \Mod_{\Indban}(\mathscr{B})$, we have:
\begin{equation*}
  \Hom_{\mathscr{B}}(\mathcal{S}\overrightarrow{\otimes}_{\mathscr{A}}\mathcal{M},\mathcal{N})=\Hom_{\mathscr{A}}(\mathcal{M},\underline{\mathcal{H}om}_{\mathscr{B}}(\mathcal{S},\mathcal{N})).
\end{equation*}
The equivalent adjunction holds for $I(\mathcal{S})\widetilde{\otimes}_{I(\mathscr{A})}-$ and $\underline{\mathcal{H}om}_{I(\mathscr{B})}(I(\mathcal{S}),-)$.
\end{Lemma}
\begin{proof}
Choose $f\in \Hom_{\mathscr{B}}(\mathcal{S}\overrightarrow{\otimes}_{\mathscr{A}}\mathcal{M},\mathcal{N})$. This is equivalent to a map  $f:\mathcal{S}\overrightarrow{\otimes}_{\mathscr{A}}\mathcal{M}\rightarrow\mathcal{N}$ such that the following diagram of sheaves of Ind-Banach spaces commutes:
\begin{equation}\label{equation first diagram adjunction relative extension of scalars}
\begin{tikzcd}
\mathscr{B}\overrightarrow{\otimes}_K\left(\mathcal{S}\overrightarrow{\otimes}_{\mathscr{A}}\mathcal{M}\right) \arrow[d] \arrow[r, "\operatorname{Id}\overrightarrow{\otimes}_Kf"] & \mathscr{B}\overrightarrow{\otimes}_K\mathcal{N} \arrow[d] \\
\mathcal{S}\overrightarrow{\otimes}_{\mathscr{A}}\mathcal{M} \arrow[r, "f"]                                                                                          & \mathcal{N}                                        
\end{tikzcd}
\end{equation}
By the sheaf version of $(\ref{equation relative tensor hom adjunction})$, the map  $f:\mathcal{S}\overrightarrow{\otimes}_{\mathscr{A}}\mathcal{M}\rightarrow\mathcal{N}$ is equivalent to an $\mathscr{A}$-linear map $g:\mathcal{M}\rightarrow \underline{\mathcal{H}om}_{\Indban}(\mathcal{S},\mathcal{N})$. On the other hand, by construction, we have:
\begin{equation*}
   \underline{\mathcal{H}om}_{\mathscr{B}}(\mathcal{S},\mathcal{N})=\operatorname{Eq}\left(\underline{\mathcal{H}om}_{\Indban}(\mathcal{S},\mathcal{N})\rightrightarrows \underline{\mathcal{H}om}_{\Indban}(\mathscr{B}\overrightarrow{\otimes}_K\mathcal{S},\mathcal{N}) \right).
\end{equation*}
Notice that the fact that $\mathcal{S}$ is a $(\mathscr{B},\mathscr{A})$-bimodule implies that the maps in the equalizer are $\mathscr{A}$-linear. Hence, $f:\mathcal{S}\overrightarrow{\otimes}_{\mathscr{A}}\mathcal{M}\rightarrow\mathcal{N}$ induces an $\mathscr{A}$-linear map $\mathcal{M}\rightarrow \underline{\mathcal{H}om}_{\mathscr{B}}(\mathcal{S},\mathcal{N})$ if and only if $g$ factors through the equalizer. Thus, we need to  describe the maps in the equalizer.\\ 
The first map sends a morphism $\varphi\in \underline{\mathcal{H}om}_{\Indban}(\mathcal{S},\mathcal{N})$ to the composition:
\begin{equation*}
   \mathscr{B}\overrightarrow{\otimes}_K\mathcal{S}\rightarrow\mathcal{S}\xrightarrow[]{\varphi} \mathcal{N}, 
\end{equation*}
where the first map is the action of $\mathscr{B}$ on $\mathcal{S}$. By functoriality of $(\ref{equation relative tensor hom adjunction})$, it follows that the composition:
\begin{equation*}
    \mathcal{M}\rightarrow \underline{\mathcal{H}om}_{\Indban}(\mathcal{S},\mathcal{N})\rightarrow \underline{\mathcal{H}om}_{\Indban}(\mathscr{B}\overrightarrow{\otimes}_K\mathcal{S},\mathcal{N}),
\end{equation*}
is obtained via the adjunction $(\ref{equation relative tensor hom adjunction})$ from the following morphism:
\begin{equation*}
\left(\mathscr{B}\overrightarrow{\otimes}_K\mathcal{S}\right)\overrightarrow{\otimes}_{\mathscr{A}}\mathcal{M}\rightarrow \mathcal{S}\overrightarrow{\otimes}_{\mathscr{A}}\mathcal{M} \rightarrow \mathcal{N}, 
\end{equation*}
where the first map is given by the action of $\mathscr{B}$ on $\mathcal{S}$. Using the canonical isomorphism of sheaves:
\begin{equation*}
\left(\mathscr{B}\overrightarrow{\otimes}_K\mathcal{S}\right)\overrightarrow{\otimes}_{\mathscr{A}}\mathcal{M}\cong \mathscr{B}\overrightarrow{\otimes}_K\left(\mathcal{S}\overrightarrow{\otimes}_{\mathscr{A}}\mathcal{M}\right), 
\end{equation*}
it follows that this map is the lower part of the diagram $(\ref{equation first diagram adjunction relative extension of scalars})$.\\
The second map of the equalizer sends a morphism $\varphi\in \underline{\mathcal{H}om}_{\widehat{\mathcal{B}}c_K}(\mathcal{S},\mathcal{N})$ to the composition:
\begin{equation*}
   \mathscr{B}\overrightarrow{\otimes}_K\mathcal{S}\xrightarrow[]{\operatorname{Id}\overrightarrow{\otimes}_K\varphi}\mathscr{B}\overrightarrow{\otimes}_K\mathcal{N}\rightarrow \mathcal{N}. 
\end{equation*}
Thus, the composition $\mathcal{M}\rightarrow \underline{\mathcal{H}om}_{\Indban}(\mathcal{S},\mathcal{N})\rightarrow \underline{\mathcal{H}om}_{\Indban}(\mathscr{B}\overrightarrow{\otimes}_K\mathcal{S},\mathcal{N})$ is readily seen to be induced by the upper part of $(\ref{equation first diagram adjunction relative extension of scalars})$.
Hence, we get a map $\mathcal{M}\rightarrow \underline{\mathcal{H}om}_{\mathscr{B}}(\mathcal{S},\mathcal{N})$ as wanted. The construction of the map in the other direction is analogous.
\end{proof}
\begin{Lemma}\label{Lemma derived adjunction relative extension and restriction}
 Let $\mathcal{M}\in LH(\Mod_{\Indban}(\mathscr{A}))$, and $\mathcal{N}\in LH(\Mod_{\Indban}(\mathscr{B}))$. We have a functorial canonical isomorphism in $\operatorname{Shv}(X,LH(\widehat{\mathcal{B}}c_K))$:
 \begin{equation*}
     \underline{\mathcal{H}om}_{I(\mathscr{B})}(I(\mathcal{S})\widetilde{\otimes}_{I(\mathscr{A})}\mathcal{M},\mathcal{N})=\underline{\mathcal{H}om}_{I(\mathscr{A})}(\mathcal{M},\underline{\mathcal{H}om}_{I(\mathscr{B})}(I(\mathcal{S}),\mathcal{N})).
 \end{equation*}
\end{Lemma}
\begin{proof}
We prove this by Yoneda. For simplicity, let $\mathcal{C}:=\operatorname{Shv}(X,LH(\widehat{\mathcal{B}}c_K))$, and let  $\mathcal{T}\in \mathcal{C}$. We have the following chain of isomorphisms of $K$-vector spaces:
\begin{align*}
\Hom_{\mathcal{C}}(\mathcal{T},\underline{\mathcal{H}om}_{I(\mathscr{B})}(I(\mathcal{S})\widetilde{\otimes}_{I(\mathscr{A})}\mathcal{M},\mathcal{N}))=\Hom_{I(\mathscr{B})}(\left(I(\mathcal{S})\widetilde{\otimes}_{I(\mathscr{A})}\mathcal{M}\right)\widetilde{\otimes}_K\mathcal{T},\mathcal{N})\\
=\Hom_{I(\mathscr{B})}(I(\mathcal{S})\widetilde{\otimes}_{I(\mathscr{A})}\left(\mathcal{M}\widetilde{\otimes}_K\mathcal{T}\right),\mathcal{N})&\\=\Hom_{I(\mathscr{A})}(\mathcal{M}\widetilde{\otimes}_K\mathcal{T},\underline{\mathcal{H}om}_{I(\mathscr{B})}(I(\mathcal{S}),\mathcal{N}))&\\
=\Hom_{\mathcal{C}}(\mathcal{T},\underline{\mathcal{H}om}_{I(\mathscr{A})}(\mathcal{M},\underline{\mathcal{H}om}_{I(\mathscr{B})}(I(\mathcal{S},\mathcal{N})))&.
\end{align*}
Thus the isomorphism follows by Yoneda' Lemma.
\end{proof}
We will now deal with the derived functors associated to this adjunction. Instead of dealing with the general case, we will assume that $\mathcal{S}$ satisfies that for any $\mathcal{M}\in \Mod_{\Indban}(\mathscr{A})$ we have:
\begin{equation*}
\mathcal{S}\overrightarrow{\otimes}^{\mathbb{L}}_{\mathscr{A}}\mathcal{M}=I(\mathcal{S})\widetilde{\otimes}_{I(\mathscr{A})}I(\mathcal{M})= I(\mathcal{S}\overrightarrow{\otimes}_{\mathscr{A}}\mathcal{M}).
\end{equation*} 
In this case, the construction of the functor $\mathcal{S}\overrightarrow{\otimes}_{\mathscr{A}}^{\mathbb{L}}-$ shown in $(\ref{equation derived functor relative tensor product in boinded above derived categories})$ implies that we have a derived functor $\mathcal{S}\overrightarrow{\otimes}_{\mathscr{A}}^{\mathbb{L}}-:\operatorname{D}(\mathscr{A})\rightarrow \operatorname{D}(\mathscr{B})$ of the extension of scalars. Similarly, as $LH(\Mod_{\Indban}(\mathscr{B}))$ is a Grothendieck abelian category, we get a derived functor
$R\underline{\mathcal{H}om}_{\mathscr{B}}(\mathcal{S},-):\operatorname{D}(\mathscr{B})\rightarrow \operatorname{D}(\mathscr{A})$.\\
The idea is that our assumptions on $\mathcal{S}$ imply that the adjunction from Lemma \ref{lemma adjunction extension and restriction along a bimodule} lifts to a derived adjunction. In particular, we have the following proposition:
\begin{prop}\label{prop derived tensor-hom adjunction wrt a flat bimodule}
Let $\mathscr{A},\mathscr{B}$ be sheaves of Ind-Banach $K$-algebras, and let $\mathcal{S}$ be a $(\mathscr{B},\mathscr{A})$-bimodule. Assume $\mathcal{S}$ satisfies  that for any $\mathcal{M}\in \Mod_{\Indban}(\mathscr{A})$ we have:
\begin{equation*}
\mathcal{S}\overrightarrow{\otimes}^{\mathbb{L}}_{\mathscr{A}}\mathcal{M}=I(\mathcal{S})\widetilde{\otimes}_{I(\mathscr{A})}I(\mathcal{M})= I(\mathcal{S}\overrightarrow{\otimes}_{\mathscr{A}}\mathcal{M}).
\end{equation*}
Then for $\mathcal{M}^{\bullet}\in \operatorname{D}^-(\mathscr{A})$ and $\mathcal{N}^{\bullet}\in \operatorname{D}(\mathscr{B})$ we have a functorial identification in $\operatorname{D}(\operatorname{Shv}(X,LH(\widehat{\mathcal{B}}c_K)))$:
\begin{equation*}
    R\underline{\mathcal{H}om}_{\mathscr{B}}(\mathcal{S}\overrightarrow{\otimes}^{\mathbb{L}}_{\mathscr{A}}\mathcal{M}^{\bullet},\mathcal{N}^{\bullet})
=R\underline{\mathcal{H}om}_{\mathscr{A}}(\mathcal{M}^{\bullet},R\underline{\mathcal{H}om}_{\mathscr{B}}(\mathcal{S},\mathcal{N}^{\bullet})).
\end{equation*}
\end{prop}
\begin{proof}
Let $\mathcal{I}^{\bullet}$ be a homotopically injective object equipped with a quasi-isomorphism $I(\mathcal{N})^{\bullet}\rightarrow \mathcal{I}^{\bullet}$. Our assumptions on $\mathcal{S}$ imply that we have the following identity in $\operatorname{D}(\mathscr{B})$:
\begin{equation*}
\mathcal{S}\overrightarrow{\otimes}^{\mathbb{L}}_{\mathscr{A}}\mathcal{M}^{\bullet}=I(\mathcal{S})\widetilde{\otimes}_{I(\mathscr{A})}I(\mathcal{M}^{\bullet}).
\end{equation*}
 Thus, $R\underline{\mathcal{H}om}_{\mathscr{B}}(\mathcal{S}\overrightarrow{\otimes}^{\mathbb{L}}_{\mathscr{A}}\mathcal{M}^{\bullet},\mathcal{N}^{\bullet})$ is quasi-isomorphic to the complex:
\begin{equation*}
\underline{\mathcal{H}om}_{I(\mathscr{B})}(I(\mathcal{S})\widetilde{\otimes}_{I(\mathscr{A}})I(\mathcal{M}^{\bullet}),\mathcal{I}^{\bullet})=\operatorname{Tot}_{\pi}\left(\underline{\mathcal{H}om}_{I(\mathscr{B})}^{\bullet,\bullet}(I(\mathcal{S})\widetilde{\otimes}_{I(\mathscr{A})}I(\mathcal{M}^{-\bullet}),\mathcal{I}^{\bullet})\right).    
\end{equation*}
Let $n,m\in\mathbb{Z}$. By Lemma \ref{Lemma derived adjunction relative extension and restriction} we have an identification in  $\operatorname{Shv}(X,LH(\widehat{\mathcal{B}}c_K))$:
\begin{equation*}
\underline{\mathcal{H}om}_{I(\mathscr{B})}(I(\mathcal{S})\widetilde{\otimes}_{I(\mathscr{A})}I(\mathcal{M}^{-n}),\mathcal{I}^{m})= \underline{\mathcal{H}om}_{I(\mathscr{A})}(I(\mathcal{M}^{-n}),\underline{\mathcal{H}om}_{I(\mathscr{B})}(I(\mathcal{S}),\mathcal{I}^{m})).    
\end{equation*}
In particular, we have an identification:
\begin{equation*}
    R\underline{\mathcal{H}om}_{\mathscr{B}}(\mathcal{S}\overrightarrow{\otimes}^{\mathbb{L}}_{\mathscr{A}}\mathcal{M}^{\bullet},\mathcal{N}^{\bullet})\cong \underline{\mathcal{H}om}_{I(\mathscr{A})}(I(\mathcal{M}^{-\bullet}),\underline{\mathcal{H}om}_{I(\mathscr{B})}(I(\mathcal{S}),\mathcal{I}^{\bullet})).
\end{equation*}
By construction, we have $\underline{\mathcal{H}om}_{I(\mathscr{B})}(I(\mathcal{S}),I(\mathcal{N}^{\bullet}))\cong \underline{\mathcal{H}om}_{I(\mathscr{B})}(I(\mathcal{S}),\mathcal{I}^{\bullet})$ in $\operatorname{D}(\mathscr{A})$.\newline Furthermore, as $I(\mathcal{S})$ is bounded above, it follows by Lemma \ref{Lemma internal hom of flat and injective is injective} that the complex  $\underline{\mathcal{H}om}_{I(\mathscr{B})}(I(\mathcal{S}),\mathcal{I}^{\bullet})$ is homotopically injective. Therefore, we have the following identities in $\operatorname{D}(\operatorname{Shv}(X,LH(\widehat{\mathcal{B}}c_K)))$:
\begin{equation*}
    R\underline{\mathcal{H}om}_{\mathscr{A}}(\mathcal{M}^{\bullet},R\underline{\mathcal{H}om}_{\mathscr{B}}(\mathcal{S},\mathcal{N}^{\bullet}))\cong \underline{\mathcal{H}om}_{I(\mathscr{A})}(I(\mathcal{M}^{-\bullet}),\underline{\mathcal{H}om}_{I(\mathscr{B})}(I(\mathcal{S}),\mathcal{I}^{\bullet}))
\end{equation*}
Thus, we arrive at a the following canonical isomorphisms in $\operatorname{D}(\operatorname{Shv}(X,LH(\widehat{\mathcal{B}}c_K)))$:
\begin{equation*}
 R\underline{\mathcal{H}om}_{\mathscr{B}}(\mathcal{S}\overrightarrow{\otimes}^{\mathbb{L}}_{\mathscr{A}}\mathcal{M}^{\bullet},\mathcal{N}^{\bullet})
=R\underline{\mathcal{H}om}_{\mathscr{A}}(\mathcal{M}^{\bullet},R\underline{\mathcal{H}om}_{\mathscr{B}}(\mathcal{S},\mathcal{N}^{\bullet})),   
\end{equation*}
which is precisely what we wanted to show.
\end{proof}
We point out that a more general statement can be shown by using an argument similar to \cite[Proposition 3.36]{bode2021operations}. However, we will not develop this here.
\subsection{Closed immersions}\label{section closed immersions}
Let now $Y$ be another rigid space, and assume we have a closed immersion $i:Y\rightarrow X$. Let $\mathscr{B}=i^{-1}\mathscr{A}$. As $i^{-1}$ commutes with tensor products, this is a sheaf of Ind-Banach $K$-algebras on $Y$. Furthermore, we have:
\begin{defi}
Let $X$ be a rigid space and $Y\subset X$ be an analytically closed subset. We say a sheaf $\mathcal{F}\in \operatorname{Shv}(X,\Indban)$ (resp. in $\operatorname{Shv}(X,LH(\widehat{\mathcal{B}}c_K))$) is supported on $Y$ if $\mathcal{F}(U)=0$ for every admissible open $U\subset X$ such that $U\cap Y=\varnothing$.   
\end{defi}
\begin{prop}\label{prop sheaves of modules supported on the diagonal general case}
Let $\Mod_{\Indban}(\mathscr{A})_{Y}$ denote the full subcategory of\newline $\Mod_{\Indban}(\mathscr{A})$ given by the $\mathscr{A}$-modules supported on $Y$. There is an equivalence of quasi-abelian categories:
\begin{equation*}
    i_*: \Mod_{\Indban}(\mathscr{B})\leftrightarrows\Mod_{\Indban}(\mathscr{A})_{Y}:i^{-1}.
\end{equation*}
\end{prop}
\begin{proof}
By \cite[Appendix A]{ardakov2015d}, $i_*$ and $i^{-1}$ yield equivalences of categories between $\operatorname{Shv}(Y,\Indban)$ and   $\operatorname{Shv}(X,\Indban)_{Y}$. Thus, the result follows by the facts that $i^{-1}\mathscr{A}=\mathscr{B}$, and that $i^{-1}$ commutes with tensor products.  
\end{proof}
Hence, we have a commutative diagram:
\begin{equation}\label{equation commutative diagram pushforward-pullback and I}
\begin{tikzcd}
{i_{*,LH}:LH(\Mod_{\Indban}(\mathscr{B}))} \arrow[r, shift left=2]           & LH(\Mod_{\Indban}(\mathscr{A}))_{Y}:i^{-1}_{LH} \arrow[l]             \\
i_*:\Mod_{\Indban}(\mathscr{B}) \arrow[u, "I"] \arrow[r, shift left=2] & \Mod_{\Indban}(\mathscr{A})_{Y}:i^{-1} \arrow[u, "I"] \arrow[l]
\end{tikzcd} 
\end{equation}
where the horizontal maps are mutually inverse equivalences of quasi-abelian (resp. abelian) categories, and the vertical maps are the inclusions into the left hearts. The next goal is studying derived versions of these adjunctions. Let us start with the following lemma:
\begin{Lemma}\label{Lemma properties of inclusion of derived category of modules supported on a closed subspace}
Let $\operatorname{D}(\operatorname{Shv}(X,\Indban)_{Y})$ be the derived category of sheaves supported on $Y$. The following hold:
\begin{enumerate}[label=(\roman*)]
    \item The mutually inverse equivalences $i_{*},i^{-1}$ extend to a derived equivalence:
    \begin{equation*}
        i_{*}:\operatorname{D}(\operatorname{Shv}(Y,\Indban))\leftrightarrows\operatorname{D}(\operatorname{Shv}(X,\Indban)_{Y}):i^{-1}.
    \end{equation*}
    \item The canonical map:
    \begin{equation*}
        \operatorname{D}(\operatorname{Shv}(X,\Indban)_{Y})\rightarrow \operatorname{D}(\operatorname{Shv}(X,\Indban)),
    \end{equation*}
     is fully faithful, its essential image 
     are the complexes whose cohomology is supported on $Y$.
\end{enumerate}
The analogous properties hold for the functors:
\begin{equation*}
    i_{*}:\Mod_{\Indban}(\mathscr{B})\leftrightarrows\Mod_{\Indban}(\mathscr{A})_{Y}:i^{-1}.
\end{equation*}
\end{Lemma}
\begin{proof}
 Statement $(i)$ is clear by the fact that:
 \begin{equation*}
     i_*:\operatorname{Shv}(Y,\Indban)\leftrightarrows\operatorname{Shv}(X,\Indban)_{Y}:i^{-1},
 \end{equation*}
are mutually inverse equivalences of quasi-abelian categories.\\

For statement $(ii)$, notice that the pushforward $i_{*,LH}:\operatorname{Shv}(Y,LH(\widehat{\mathcal{B}}c_K))\rightarrow\operatorname{Shv}(X,LH(\widehat{\mathcal{B}}c_K))$ has an exact left adjoint. Hence, its extension to the derived category preserves homotopically injective objects. Therefore, for objects $V^{\bullet},W^{\bullet}\in\operatorname{D}(\operatorname{Shv}(X,LH(\widehat{\mathcal{B}}c_K))_{Y})$ there is a pair of homotopically injective complexes $\mathcal{I}_1^{\bullet},\mathcal{I}_2^{\bullet}\in \operatorname{D}(\operatorname{Shv}(Y,LH(\widehat{\mathcal{B}}c_K)))$ with quasi-isomorphisms:
\begin{equation*}
  V^{\bullet}\rightarrow i_{*,LH}\mathcal{I}_1^{\bullet}, \quad   W^{\bullet}\rightarrow i_{*,LH}\mathcal{I}_2^{\bullet}.
\end{equation*}
Hence, we have the following chain of identities::
 \begin{align*}
   \Hom_{\operatorname{D}(\operatorname{Shv}(X,LH(\widehat{\mathcal{B}}c_K)))}(V^{\bullet},W^{\bullet})&=\Hom_{\operatorname{D}(\operatorname{Shv}(X,LH(\widehat{\mathcal{B}}c_K)))}(i_{*,LH}\mathcal{I}_1^{\bullet},i_{*,LH}\mathcal{I}_2^{\bullet})\\
         &=\Hom_{\operatorname{K}(\operatorname{Shv}(X,LH(\widehat{\mathcal{B}}c_K)))}(i_{*,LH}\mathcal{I}_1^{\bullet},i_{*,LH}\mathcal{I}_2^{\bullet})\\&=\Hom_{\operatorname{K}(\operatorname{Shv}(Y,LH(\widehat{\mathcal{B}}c_K)))}(\mathcal{I}_1^{\bullet},\mathcal{I}_2^{\bullet})&\\
  &{}     =\Hom_{\operatorname{D}(\operatorname{Shv}(Y,LH(\widehat{\mathcal{B}}c_K)))}(\mathcal{I}_1^{\bullet},\mathcal{I}_2^{\bullet})\\&=\Hom_{\operatorname{D}(\operatorname{Shv}(X,LH(\widehat{\mathcal{B}}c_K))_Y)}(i_{*,LH}\mathcal{I}_1^{\bullet},i_{*,LH}\mathcal{I}_2^{\bullet})&\\
&{}    =\Hom_{\operatorname{D}(\operatorname{Shv}(X,LH(\widehat{\mathcal{B}}c_K))_Y)}(V^{\bullet},W^{\bullet})&.
 \end{align*}
 Thus, the triangulated functor:
 \begin{equation*}
     \operatorname{D}(\operatorname{Shv}(X,\Indban)_{Y})\rightarrow \operatorname{D}(\operatorname{Shv}(X,\Indban)),
 \end{equation*}
is fully faithful. Hence, it is  an equivalence onto its essential image.\\
Let $C^{\bullet}\in \operatorname{D}(\operatorname{Shv}(X,\Indban))$ be a complex with cohomology supported on $Y$. We have a canonical morphism of complexes $C^{\bullet}\rightarrow i_*i^{-1}C^{\bullet}$. It suffices to show that this is a quasi-isomorphism. As both $i_{*}$, and $i^{-1}$  are exact, for each $n\in\mathbb{Z}$ we have:
 \begin{equation*} \operatorname{H}^{n}\left(i_*i^{-1}C^{\bullet}\right)=i_{*,LH}i^{-1}_{LH}\operatorname{H}^{n}\left(C^{\bullet}\right) =\operatorname{H}^{n}\left(C^{\bullet}\right), 
 \end{equation*}
 where the last identity follows by the fact that $\operatorname{H}^{n}\left(C^{\bullet}\right)\in \operatorname{Shv}(X,LH(\widehat{\mathcal{B}}c_K))$ is supported on $Y$. Thus, the map $C^{\bullet}\rightarrow i_{*}i^{-1}C^{\bullet}$ is a quasi-isomorphism, as required.
\end{proof}
We may use this lemma to study the derived inner hom of complexes with support contained in $Y$. This will be instrumental for the version of Kashiwara equivalence given in Proposition \ref{prop derived internal hom and kashiwara}. In order to keep the notation manageable, we will write:
\begin{equation*}
    \operatorname{D}(\mathscr{A})=\operatorname{D}(\Mod_{\Indban}(\mathscr{A})),
\end{equation*}
and follow a similar convention for other sheaves of Ind-Banach algebras.
\begin{Lemma}\label{Lemma support of inner homomorphism of sheaves with support in a closed subspace}
Let $\mathcal{M}^{\bullet},\mathcal{N}^{\bullet}\in \operatorname{D}(\mathscr{A})_Y$. Then $R\underline{\mathcal{H}om}_{\mathscr{A}}(\mathcal{M}^{\bullet},\mathcal{N}^{\bullet})\in \operatorname{D}(\mathscr{A})_Y$.   
\end{Lemma}
\begin{proof}
We need to show that $R\underline{\mathcal{H}om}_{\mathscr{A}}(\mathcal{M}^{\bullet},\mathcal{N}^{\bullet})\in \operatorname{D}(\operatorname{Shv}(X,LH(\widehat{\mathcal{B}}c_K)))$ has cohomology supported on $Y$.
This follows by Proposition \ref{prop restriction of inner hom}, as for every $U\subset X$ such that $U\cap Y$ is empty, we have the following identities in $\operatorname{D}(\operatorname{Shv}(U,LH(\widehat{\mathcal{B}}c_K)))$:
    \begin{equation*}
   R\underline{\mathcal{H}om}_{\mathscr{A}}(\mathcal{M}^{\bullet},\mathcal{N}^{\bullet})_{\vert U}= R\underline{\mathcal{H}om}_{\mathscr{A}_{\vert U}}(\mathcal{M}_{\vert U}^{\bullet},\mathcal{N}^{\bullet}_{\vert U})= 0,
    \end{equation*}
    where the last identity follows by the fact that $\mathcal{M}_{\vert U}^{\bullet}=\mathcal{N}^{\bullet}_{\vert U}=0$ in $\operatorname{D}(\mathscr{A}_{\vert U})_{Y\cap U}$, because their cohomology is supported on $Y$.
\end{proof}
\begin{prop}\label{Proposition cohomology and closed immersions}
    There is a natural isomorphism of derived functors:
    \begin{equation*}
        R\underline{\mathcal{H}om}_{\mathscr{B}}(-,-)=i^{-1}R\underline{\mathcal{H}om}_{\mathscr{A}}(i_{*}(-),i_{*}(-)).
    \end{equation*}
\end{prop}
\begin{proof}
Let $\mathcal{M}^{\bullet}\in \operatorname{D}(\operatorname{Shv}(Y,LH(\widehat{\mathcal{B}}c_K)))$, and $\mathcal{N}^{\bullet},\mathcal{T}^{\bullet}\in \operatorname{D}(\mathscr{B})$. Then we have:
\begin{align*}
    \Hom_{\operatorname{D}(\operatorname{Shv}(Y,LH(\widehat{\mathcal{B}}c_K)))}(&\mathcal{M}^{\bullet},i^{-1}R\underline{\mathcal{H}om}_{\mathscr{A}}(i_{*}\mathcal{N}^{\bullet},i_{*}\mathcal{T}^{\bullet}))\\
    &=\Hom_{\operatorname{D}(\operatorname{Shv}(X,LH(\widehat{\mathcal{B}}c_K)))}(i_{*}\mathcal{M}^{\bullet},R\underline{\mathcal{H}om}_{\mathscr{A}}(i_{*}\mathcal{N}^{\bullet},i_{*}\mathcal{T}^{\bullet}))\\
    &=\Hom_{\operatorname{D}(\mathscr{A})}(i_{*}\mathcal{N}^{\bullet}\overrightarrow{\otimes}_K^{\mathbb{L}}i_{*}\mathcal{M}^{\bullet},i_{*}\mathcal{T}^{\bullet})&\\ &=\Hom_{\operatorname{D}(\mathscr{B})}(\mathcal{N}^{\bullet}\overrightarrow{\otimes}_K^{\mathbb{L}}\mathcal{M}^{\bullet},\mathcal{T}^{\bullet})&\\
   &=\Hom_{\operatorname{D}(\operatorname{Shv}(Y,LH(\widehat{\mathcal{B}}c_K)))}(\mathcal{M}^{\bullet},R\underline{\mathcal{H}om}_{\mathscr{B}}(\mathcal{N}^{\bullet},\mathcal{T}^{\bullet})).&
\end{align*}
where the first identity follows by Lemmas \ref{Lemma properties of inclusion of derived category of modules supported on a closed subspace}, and \ref{Lemma support of inner homomorphism of sheaves with support in a closed subspace}, and the third one by the fact that tensor products commute with pullbacks. Thus, the isomorphism holds by Yoneda.
\end{proof}
Many of these properties are also shared by the derived tensor product:
\begin{Lemma}\label{Lemma suport of derived tensor product}
Let $\mathcal{M}^{\bullet}\in \operatorname{D}(\mathscr{A}^{\op})_Y$ and $\mathcal{N}^{\bullet}\in \operatorname{D}(\mathscr{A})_Y$. Then $\mathcal{M}^{\bullet}\overrightarrow{\otimes}^{\mathbb{L}}_{\mathscr{A}}\mathcal{N}^{\bullet}$ is supported on $Y$.
\end{Lemma}
\begin{proof}
By \cite[Theorem 3.24]{bode2021operations}, there are full additive subcategories:
\begin{equation*}
    \widetilde{P}_r\subset LH(\Mod_{\Indban}(\mathscr{A}^{\op})), \quad \widetilde{P}_l\subset LH(\Mod_{\Indban}(\mathscr{A})),
\end{equation*}
such that the pair $(\widetilde{P}_r,\widetilde{P}_l)$ is projective  with respect to $-\widetilde{\otimes}_{I(\mathscr{A})}-$ (cf. \cite[Definition 10.3.9]{Kashiwara2006}).
Let 
$\mathcal{P}_r$ be the full subcategory of
$LH(\Mod_{\Indban}(\mathscr{A}^{\op}))_Y$ given by the objects contained in $\widetilde{P}_r$, and define 
$\mathcal{P}_l$ analogously. We claim that $\mathcal{P}_r$ and $\mathcal{P}_l$ are full additive subcategories which generate $LH(\Mod_{\Indban}(\mathscr{A}^{\op}))_Y$  and $LH(\Mod_{\Indban}(\mathscr{A}))_Y$ respectively. Indeed, start by choosing an object  $\mathcal{X}\in LH(\Mod_{\Indban}(\mathscr{A}^{\op}))_Y$. Then $\mathcal{X}\widetilde{\otimes}_KI(\mathscr{A})$ is in $\mathcal{P}_r$ and we have a surjection:
\begin{equation*}
    \mathcal{X}\widetilde{\otimes}_KI(\mathscr{A})\rightarrow \mathcal{X}.
\end{equation*}
Hence, our claim holds for $\mathcal{P}_r$, and an analogous argument works for $\mathcal{P}_l$.\\ 

As $LH(\Mod_{\Indban}(\mathscr{A}^{\op}))_Y$  and $LH(\Mod_{\Indban}(\mathscr{A}))_Y$ are Grothendieck abelian categories, it follows by \cite[Lemma 14.4.1]{Kashiwara2006} that there are quasi-isomorphisms 
$P^{\bullet}\rightarrow \mathcal{M}^{\bullet}$, and $Q^{\bullet}\rightarrow \mathcal{M}^{\bullet}$ such that $P^{\bullet}$ and $Q^{\bullet}$ are complexes whose objects are direct sums of elements in $\mathcal{P}_r$ and $\mathcal{P}_l$ respectively. In particular, they have support contained on $Y$. By the proof of \cite[Theorem 14.4.8]{Kashiwara2006}, we have:
\begin{equation*}
\mathcal{M}^{\bullet}\overrightarrow{\otimes}^{\mathbb{L}}_{\mathscr{A}}\mathcal{N}^{\bullet}=\operatorname{Tot}_{\oplus}\left(P^{\bullet}\widetilde{\otimes}^{\bullet,\bullet}_{I(\mathscr{A})} Q^{\bullet} \right).  
\end{equation*}
Thus, $\mathcal{M}^{\bullet}\overrightarrow{\otimes}^{\mathbb{L}}_{\mathscr{A}}\mathcal{N}^{\bullet}$ has cohomology supported on $Y$, as we wanted to show.
\end{proof}
\begin{prop}\label{prop derived tensor products and closed immersions}
There is a natural isomorphism of functors:
\begin{equation*}
    i_*\left(-\overrightarrow{\otimes}^{\mathbb{L}}_{\mathscr{B}}-\right)=i_*(-)\overrightarrow{\otimes}^{\mathbb{L}}_{\mathscr{A}}i_*(-).
\end{equation*}
\end{prop}
\begin{proof}
Let $\mathcal{M}^{\bullet}\in \operatorname{D}(\mathscr{B}^{\op})$, $\mathcal{N}^{\bullet}\in \operatorname{D}(\mathscr{B})$, and $C^{\bullet}\in \operatorname{D}(\operatorname{Shv}(X,LH(\widehat{\mathcal{B}}c_K)))$. Then we have:
\begin{align*}
\Hom_{\operatorname{D}(\operatorname{Shv}(Y,LH(\widehat{\mathcal{B}}c_K)))}(\mathcal{M}^{\bullet}&\overrightarrow{\otimes}^{\mathbb{L}}_{\mathscr{B}}\mathcal{N}^{\bullet},C^{\bullet})= \Hom_{\operatorname{D}(\mathscr{B})}(\mathcal{N}^{\bullet}, R\underline{\mathcal{H}om}_{\Indban}(\mathcal{M}^{\bullet},C^{\bullet}))\\
&=\Hom_{\operatorname{D}(\mathscr{B})}(\mathcal{N}^{\bullet}, i^{-1}R\underline{\mathcal{H}om}_{\Indban}(i_*\mathcal{M}^{\bullet},i_*C^{\bullet}))\\&=\Hom_{\operatorname{D}(\mathscr{A})}(i_*\mathcal{N}^{\bullet}, R\underline{\mathcal{H}om}_{\Indban}(i_*\mathcal{M}^{\bullet},i_*C^{\bullet}))\\
&=\Hom_{\operatorname{D}(\operatorname{Shv}(X,LH(\widehat{\mathcal{B}}c_K)))}(i_*\mathcal{M}^{\bullet}\overrightarrow{\otimes}^{\mathbb{L}}_{\mathscr{A}}i_*\mathcal{N}^{\bullet},i_*C^{\bullet})\\&=\Hom_{\operatorname{D}(\operatorname{Shv}(X,LH(\widehat{\mathcal{B}}c_K)))}(i^{-1}\left(i_*\mathcal{M}^{\bullet}\overrightarrow{\otimes}^{\mathbb{L}}_{\mathscr{A}}i_*\mathcal{N}^{\bullet}\right),C^{\bullet}),
\end{align*}
where we use the derived adjunction $(\ref{equation adjunction relative tensor-hom})$, together with Lemma \ref{Lemma support of inner homomorphism of sheaves with support in a closed subspace}, and Proposition \ref{Proposition cohomology and closed immersions} applied to the case $\mathscr{A}=K$. Thus, we have the following identity in $\operatorname{D}(\operatorname{Shv}(Y,LH(\widehat{\mathcal{B}}c_K)))$:
\begin{equation*}
   \mathcal{M}^{\bullet}\overrightarrow{\otimes}^{\mathbb{L}}_{\mathscr{B}}\mathcal{N}^{\bullet}=i^{-1}\left(i_*\mathcal{M}^{\bullet}\overrightarrow{\otimes}^{\mathbb{L}}_{\mathscr{A}}i_*\mathcal{N}^{\bullet}\right). 
\end{equation*}
By lemmas \ref{Lemma properties of inclusion of derived category of modules supported on a closed subspace}, and \ref{Lemma suport of derived tensor product}, this identity implies the proposition.
\end{proof}
\section{Background: Lie algebroids and co-admissible modules}\label{Section background Lie algebroids}
Now that we have established a technical framework in which to develop our formalism of Hochschild (co)-homology, we  will introduce the particular objects to which we will apply it. In particular, we will devote this chapter to introducing the sheaves of (Fréchet-Stein) completions of universal enveloping algebras of Lie algebroids, and the category of sheaves of co-admissible modules attached to them. In latter stages, we will define $\mathcal{C}$-complexes, which are a homological generalization of sheaves of co-admissible modules, and discuss some of the operations that can be performed on them. We would like to point out that we do not claim any originality here. In particular, the theory of co-admissible modules over sheaves of Fréchet-Stein completions of universal enveloping algebras of Lie algebroids was first developed by Ardakov-Wadsley in \cite{ardakov2019}, and the six functor formalism for $\mathcal{C}$-complexes was developed by Bode in \cite{bode2021operations}.
\subsection{Lie algebroids}
In this section, we will give a brief review of the theory of Fréchet-Stein completions of universal enveloping algebras of Lie algebroids on a smooth rigid space $X$, and the theory of co-admissible modules attached to them, as developed in \cite{ardakov2019}, \cite{ardakov2015d}, \cite{Ardakov_Bode_Wadsley_2021}. 
\begin{obs}\label{remark all derivations are bounded}
Recall the following result from \cite[Section 2.4]{ardakov2019}: Let $A$ be an affinoid algebra and $M$ be a finite $A$-module. Regard $M$ as a Banach $K$-vector space with respect to its topology as a finite $A$-module. Then any $K$-linear derivation $f:A\rightarrow M$ is bounded. In particular, we have:
\begin{equation*}
    \Der_K^b(A):= \{\textnormal{bounded derivations }f:A\rightarrow A\}=\Der_K(A).
\end{equation*}
\end{obs}
For a rigid space $X$, we let $\mathcal{T}_{X/K}$ be the tangent sheaf, and $\Omega^1_{X/K}$ be the sheaf of Kähler differentials. By definition, for an affinoid open subspace $U\subset X$ we have:
\begin{equation*}
    \mathcal{T}_{X/K}(U)=\Der_K(\OX_X(U)).
\end{equation*}
Given $f\in \OX_X(U)$ and $v\in \mathcal{T}_{X/K}(U)$, we denote the evaluation of $v$ at $f$ by $v(f)$.
\begin{defi}
We define the following objects:
\begin{enumerate}[label=(\roman*)]
    \item $\textnormal{(\cite[Section 2]{rinehart1963differential})}$ Let $R$ be a commutative ring and $A$ be a $R$-algebra. A smooth $(R,A)$-Lie algebra is a finite projective $A$-module $L$ with a structure of  $R$-Lie algebra and an $A$-linear map of $R$-Lie algebras:
    \begin{equation*}
        \rho:L\rightarrow \Der_{R}(A),
    \end{equation*}
called the anchor map. It satisfies that for  $a\in A$, $x,y\in L$ we have: 
    \begin{equation*}
        [ax,y]=a[x,y]+\rho(x)(a)y.
    \end{equation*}
    \item $\textnormal{(\cite[Definition 6.1]{ardakov2019})}$ Let $A$ be an affinoid $K$-algebra, $\mathcal{A}\subset A$  be an affine formal model of $A$, and $L$ be a $(K,A)$-Lie algebra. A smooth $(\mathcal{R},\mathcal{A})$-Lie lattice of $L$
    is a smooth $(\mathcal{R},\mathcal{A})$-Lie algebra $\mathcal{L}\subset L$ such that: $L=\mathcal{L}\otimes_{\mathcal{R}}K$ as $(K,A)$-Lie algebras.
    \item $\textnormal{(\cite[Definition 9.1]{ardakov2019})}$ A Lie algebroid on a rigid $K$-variety $X$ is a sheaf of $(K,\OX_{X})$-Lie algebras $\mathscr{L}$ such that $\mathscr{L}$ is a locally finite-free $\OX_X$-module.
\end{enumerate}
\end{defi}
Notice that $\mathcal{T}_{X/K}$ is canonically a sheaf of $(K,\OX_{X})$-Lie algebras, and that  it is a Lie algebroid if and only if $X$ is a smooth rigid space.\bigskip

Let $A$ be an affinoid algebra. As shown in \cite{rinehart1963differential}, to every smooth $(K,A)$-Lie algebra $L$ we can associate a universal enveloping algebra $U(L)$. This is a noetherian $K$-algebra with a filtration $F_{\bullet}U(L)$ which satisfies the PBW theorem. That is, there is a canonical isomorphism of graded $A$-algebras:
\begin{equation*}
    \gr_FU(L)=\operatorname{Sym}_A(L).
\end{equation*}
We call $F_{\bullet}U(L)$ the PBW filtration. Furthermore, if $\mathcal{A}\subset A$ is an affine formal model, and $\mathcal{L}\subset L$ is a $(\mathcal{R},\mathcal{A})$-Lie lattice, then we can also form a universal algebra $U(\mathcal{L})$. By 
\cite[Proposition 2.3]{ardakov2019}, the algebra $U(\mathcal{L})$ is $\mathcal{R}$-flat, it is noetherian if $K$ is discretely valued, and satisfies:
\begin{equation*}
    U(L)=K\otimes_{\mathcal{R}}U(\mathcal{L}).
\end{equation*}
Let $X=\Sp(A)$ be the affinoid space associated to $A$. In virtue of the contents of \cite[Section 9]{ardakov2019}, we may uniquely extend $L$ to a sheaf of $(K,\OX_X)$-Lie algebras on $X$. Namely, let $\mathscr{L}=\operatorname{Loc}(L)$ be the unique sheaf on $X$ defined on affinoid subdomains $U\subset X$ by the rule:
\begin{equation*}
    \mathscr{L}(U)=\OX_X(U)\otimes_AL.
\end{equation*}
Then the $(K,A)$-Lie algebra structure on $L$ extends uniquely to a Lie algebroid structure on $\mathscr{L}$. Furthermore, every Lie algebroid on $X$ arises in this way.\\

Let $\mathscr{L}$ be a Lie algebroid on $X$. For every affinoid admissible open subspace $V=\Sp(B)\subset X$ we have the following identity:
\begin{equation*}
    U(\mathscr{L}(V))=U(B\otimes_AL)=B\otimes_AU(L).
\end{equation*}
In particular, we have a sheaf of filtered $K$-algebras $U(\mathscr{L})$ on $X$ such that:
\begin{equation*}
    U(\mathscr{L})(V)=U(\mathscr{L}(V)),
\end{equation*}
We call $U(\mathscr{L})$ the sheaf of enveloping algebras of the Lie algebroid  $\mathscr{L}$.
\begin{defi}[{\cite[Section 3]{schneider2002algebras}}]
A $K$-algebra $\mathscr{A}$ is called a left (right) Fréchet-Stein algebra if there is an inverse system of Banach $K$-algebras $(\mathscr{A}_{n})_{n\geq 0}$  satisfying the following conditions:
\begin{enumerate}[label=(\roman*)]
    \item The transition maps $\mathscr{A}_{n+1}\rightarrow \mathscr{A}_n$ have dense image and are left (right) flat for each $n\geq 0$.
    \item Each $\mathscr{A}_n$ is a two-sided noetherian $K$-algebra.
    \item There is an isomorphism of $K$-algebras $\mathscr{A}\rightarrow \varprojlim_n\mathscr{A}_n$.
\end{enumerate}
If the transition maps are left and right flat, we will say that $\mathscr{A}$ is a (two-sided) Fréchet-Stein algebra.
\end{defi}
Back to our previous setting, let $A$ be an affinoid $K$-algebra with an affine formal model $\mathcal{A}\subset A$, and let $L$ be a smooth $(K,A)$-Lie algebra with a smooth $(\mathcal{R},\mathcal{A})$-Lie lattice $\mathcal{L}$. For every $n\geq 0$, the $\mathcal{A}$-module $\pi^n\mathcal{L}$ is still a smooth $(\mathcal{R},\mathcal{A})$-Lie lattice. In particular, we can form the enveloping algebra $U(\pi^n\mathcal{L})$. Let $\widehat{U}(\pi^n\mathcal{L})$ be the $\pi$-adic completion of $U(\pi^n\mathcal{L})$. Then $\widehat{U}(\pi^n\mathcal{L})$ is a $\mathcal{R}$-flat, $\pi$-adically complete, and $\widehat{U}(\pi^{n+1}\mathcal{L})_K$ is two sided noetherian $K$-algebra. Assume in addition that $\mathcal{L}$ satisfies $[\mathcal{L},\mathcal{L}]\subset \pi^2 \mathcal{L}$ and $\rho(\mathcal{L})(\mathcal{A})\subset \pi \mathcal{A}$. Then the canonical maps:
\begin{equation*}
    \widehat{U}(\pi^{n+1}\mathcal{L})_K\rightarrow \widehat{U}(\pi^n\mathcal{L})_K,
\end{equation*}
are
two sided flat and have dense image (cf. \cite[Theorem 4.1.11]{ardakov2021equivariant}). Notice that for an arbitrary smooth $(\mathcal{R},\mathcal{A})$-Lie lattice $\mathcal{L}$, the lattice $\pi\mathcal{L}$ satisfies both the assumptions above. In this situation, we can form the following Fréchet-Stein algebra:
\begin{equation*}
    \wideparen{U}(L)=\varprojlim \widehat{U}(\pi^n\mathcal{L})_K,
\end{equation*}
which we call the Fréchet-Stein enveloping algebra of $L$. The topological $K$-algebra $\wideparen{U}(L)$ does not depend on the choice of affine formal model or smooth Lie-lattice, only on the isomorphism class of $L$ as a $(K,A)$-Lie algebra. We can use the particular description of Fréchet-Stein algebras to introduce a relevant category of modules:
\begin{defi}[{\cite[Section 3]{schneider2002algebras}}]
 Let $\mathscr{A}=\varprojlim_n \mathscr{A}_n$ be a Fréchet-Stein algebra, and $M$ be a left (right) $\mathscr{A}$-module. We say $M$ is a co-admissible module if for each $n\geq 0$ the $\mathscr{A}_n$-modules $M_n=\mathscr{A}_n\otimes_{\mathscr{A}}M$ are finite, and the canonical map of $\mathscr{A}$-modules
$M\rightarrow \varprojlim_n M_n$ is an isomorphism. We denote the category of co-admissible modules over $\mathscr{A}$ with $\mathscr{A}$-linear maps by $\mathcal{C}(\mathscr{A})$.   
\end{defi}
In particular, the previous discussion shows the existence of $\mathcal{C}(\wideparen{U}(L))$, the category of co-admissible modules over the Fréchet-Stein enveloping algebra $\wideparen{U}(L)$. Most of the contents of this paper are centered around this category and its generalizations. As the name suggests, Fréchet-Stein algebras are in particular Fréchet algebras. Furthermore, every co-admissible module admits a canonical Fréchet topology. Thus, by \cite[Theorem 6.4]{bode2021operations} we may regard $\wideparen{U}(L)$ as an Ind-Banach algebra, 
and $\mathcal{C}(\wideparen{U}(L))$ as an abelian subcategory of $\Mod_{\Indban}(\wideparen{U}(L))$.\bigskip

The canonical map $U(L)\rightarrow \wideparen{U}(L)$ is faithfully flat and has dense image. Furthermore, for any affinoid subdomain $V=\Sp(B)\subset X=\Sp(A)$, the restriction map $U(L)\rightarrow B\otimes_AU(L)$ extends uniquely to a
morphism:
\begin{equation*}
    \wideparen{U}(L)\rightarrow \wideparen{U}(B\otimes_A L)=B\overrightarrow{\otimes}_K\wideparen{U}(L).
\end{equation*}
Let $\mathscr{L}=\operatorname{Loc}(L)$. These maps allows us to define a sheaf of Ind-Banach algebras $\wideparen{U}(\mathscr{L})$ on $X$, which is defined by the rule:
\begin{equation*}
    V\mapsto \wideparen{U}(\mathscr{L})(V)=\wideparen{U}(\mathscr{L}(V)).
\end{equation*}
It can be shown that $\wideparen{U}(\mathscr{L})$ has trivial higher \v{C}ech cohomology groups. 
\begin{defi}[{\cite[Section 8 ]{ardakov2019}}]
Let $\mathcal{M}$ be a sheaf of Ind-Banach $\wideparen{U}(\mathscr{L})$-modules. We say $\mathcal{M}$ is a co-admissible module if $\mathcal{M}(X)\in\mathcal{C}(\wideparen{U}(\mathscr{L}(X)))$, and for every affinoid subdomain $V=\Sp(B)\subset X$ we have:
\begin{equation*}
   \mathcal{M}(V)=\wideparen{U}(\mathscr{L}(V))\overrightarrow{\otimes}_{\wideparen{U}(\mathscr{L}(X))}\mathcal{M}(X). 
\end{equation*}
We let $\mathcal{C}(\wideparen{U}(\mathscr{L}))$ be the abelian category of co-admissible $\wideparen{U}(\mathscr{L})$-modules with $\wideparen{U}(\mathscr{L})$-linear maps.     
\end{defi}
Given a co-admissible $\wideparen{U}(\mathscr{L}(X))$-module $\mathcal{M}$, we obtain a sheaf of co-admissible $\wideparen{U}(\mathscr{L})$-modules on $X$ by setting:
\begin{equation*}
    \operatorname{Loc}(\mathcal{M})(V)=\wideparen{U}(\mathscr{L}(V))\overrightarrow{\otimes}_{\wideparen{U}(\mathscr{L}(X))}\mathcal{M},
\end{equation*}
where $V=\Sp(B)\subset X$ is an affinoid subdomain. We may condense the main results on co-admissible modules into the following proposition:
\begin{prop}
Let $X=\Sp(A)$ be a smooth affinoid space equipped with a Lie algebroid $\mathscr{L}$. Then the following hold:
\begin{enumerate}[label=(\roman*)]
    \item $\textnormal{(\cite[Theorems 8.2, 8.4]{ardakov2019})}$ There are mutually inverse equivalences of abelian categories:
    \begin{equation*}
        \Gamma(X,-):\mathcal{C}(\wideparen{U}(\mathscr{L}))\leftrightarrows \mathcal{C}(\wideparen{U}(\mathscr{L}(X))):\operatorname{Loc}(-).
    \end{equation*}
    Furthermore, objects in  $\mathcal{C}(\wideparen{U}(\mathscr{L}))$ have trivial higher \v{C}ech cohomology.
    \item $\textnormal{(\cite[Theorem 8.4]{ardakov2019})}$ A sheaf of $\wideparen{U}(\mathscr{L})$-modules $\mathcal{M}$ is co-admissible if and only if for every admissible affinoid cover $\{V_i\}_{i\in I}$, we have $\mathcal{M}_{\vert V_i}\in \mathcal{C}(\wideparen{U}(\mathscr{L}_{\vert V_i}))$.
    \item  $\textnormal{(\cite[Proposition 5.5]{bode2021operations})}$ Let $\mathcal{M}\in \mathcal{C}(\wideparen{U}(\mathscr{L}(X)))$. Then $\mathcal{M}$ is an $A$-nuclear Fréchet space (cf. \cite[Definition 5.3]{bode2021operations}).
\end{enumerate}
\end{prop}
If $X$ is an arbitrary rigid analytic space with a Lie algebroid $\mathscr{L}$, we define its category of co-admissible $\wideparen{U}(\mathscr{L})$-modules as the full subcategory of $\Mod_{\Indban}(\wideparen{U}(\mathscr{L}))$ given by the objects which are co-admissible when restricted to an admissible affinoid cover. The previous proposition shows that this is well defined. Indeed, if $X$ is affinoid, we recover our previous definition.
\subsection{\texorpdfstring{Operations on C-complexes}{}}
Now that we have defined the category of co-admissible modules, the next step is studying their homological algebra. In particular, we will start by recalling the definition of $\mathcal{C}$-complexes from \cite[Section 6]{bode2021operations}. These complexes are to be understood as a $p$-adic analog of the bounded complexes with coherent cohomology in the algebraic theory of $\D$-modules. After that, we will introduce $p$-adic versions of some of the classical operations in the theory of $\D$-modules (direct image, inverse image, \emph{etc}). \bigskip

Let $X=\Sp(A)$ be a smooth affinoid space, $\mathcal{A}\subset A$  an affine formal model, and 
$\mathcal{T}_{\mathcal{A}}\subset \mathcal{T}_{X/K}(X)$ an $(\mathcal{R},\mathcal{A})$-Lie lattice. As shown above, the choice of a $(\mathcal{R},\mathcal{A})$-Lie lattice induces a  Fréchet-Stein presentation:
\begin{equation*}
    \wideparen{\D}_X(X)=\varprojlim_n\widehat{U}(\pi^n\mathcal{T}_{\mathcal{A}})_K.
\end{equation*}
We can use this presentation to make the following definition:
\begin{defi}[{\cite[Section 8]{bode2021operations}}]\label{defi C complexes}
Let $C^{\bullet}\in \operatorname{D}(\wideparen{\D}_X)$. We say $C^{\bullet}$ is a $\mathcal{C}$-complex if the following conditions hold:
\begin{enumerate}[label=(\roman*)]
    \item $\operatorname{H}^i(C^{\bullet})\in \mathcal{C}(\wideparen{\D}_X)$ for all $i\in \mathbb{Z}$.
    \item For $n\geq 0$ we have $\widehat{U}(\pi^n\mathcal{T}_{\mathcal{A}})_K\overrightarrow{\otimes}_{\wideparen{\D}_X(X)}\Gamma(X,\operatorname{H}^i(C^{\bullet}))=0$ for all but finitely many $i\in \mathbb{Z}$.
\end{enumerate}
Let $\operatorname{D}_{\mathcal{C}}(\wideparen{\D}_X)$ be the full subcategory of $\operatorname{D}(\wideparen{\D}_X)$ given by the $\mathcal{C}$-complexes. 
\end{defi}
The definition of $\mathcal{C}$-complex does not depend on the choice of affine formal model or Lie lattice. Furthermore, $\operatorname{D}_{\mathcal{C}}(\wideparen{\D}_X)$ is a full triangulated subcategory of $\operatorname{D}(\wideparen{\D}_X)$. For a general smooth rigid space $X$, we say a complex $C^{\bullet}\in \operatorname{D}(\wideparen{\D}_X)$ is a $\mathcal{C}$-complex if its restriction to an (any) admissible affinoid cover is a $\mathcal{C}$-complex. We will now describe the direct image and inverse image functors, as defined in \cite{bode2021operations}.
The homological machinery of Ind-Banach spaces allows to give a formulation which is pretty similar to that of the classical theory of algebraic $\D$-modules. Let $f:X\rightarrow Y$ be a morphism of smooth rigid analytic spaces. Let us start by introducing the following objects:
\begin{defi}[{\cite[Lemmas 5.3,5.5]{bode2021operations}}]
The transfer bimodules associated to $f:X\rightarrow Y$ are:
 \begin{align*}
     \wideparen{\D}_{X\rightarrow Y}=\OX_X\overrightarrow{\otimes}_{f^{-1}\OX_Y}f^{-1}\wideparen{\D}_Y,\\
     \wideparen{\D}_{Y \leftarrow X}=\Omega_X\overrightarrow{\otimes}_{\OX_X} \wideparen{\D}_{X\rightarrow Y}\overrightarrow{\otimes}_{f^{-1}\OX_Y}f^{-1}\Omega_Y^{-1}.
 \end{align*} 
\end{defi}
 The transfer bimodule $\wideparen{\D}_{X\rightarrow Y}$ is a bornological $(\wideparen{\D}_X,f^{-1}\wideparen{\D}_Y)$-bimodule,  and $\wideparen{\D}_{Y \leftarrow X}$ is a bornological $(f^{-1}\wideparen{\D}_Y,\wideparen{\D}_X)$-bimodule. We use the transfer bimodules to define the following operations:
 \begin{defi}[{\cite[Section 7.3,7.4]{bode2021operations}}]
 We define the following functors:
     \begin{enumerate}[label=(\roman*)]
         \item The inverse image functor $f^!:\operatorname{D}(\wideparen{\D}_{Y})\rightarrow \wideparen{\D}_{X}$ is defined via the following formula:
         \begin{multline*}
             f^!\mathcal{M}=\wideparen{\D}_{X\rightarrow Y}\overrightarrow{\otimes}^{\mathbb{L}}_{f^{-1}\wideparen{\D}_Y}f^{-1}\mathcal{M}[\operatorname{dim}(X)-\operatorname{dim}(Y)]\\
             =\OX_X\overrightarrow{\otimes}^{\mathbb{L}}_{f^{-1}\OX_Y}f^{-1}\mathcal{M}[\operatorname{dim}(X)-\operatorname{dim}(Y)].
         \end{multline*}
         \item The direct image functor $f_+:\operatorname{D}(\wideparen{\D}_{X})\rightarrow \wideparen{\D}_{Y}$ is defined by the following formula:
         \begin{equation*}
             f_+\mathcal{N}=Rf_*\left(\wideparen{\D}_{Y \leftarrow X}\overrightarrow{\otimes}^{\mathbb{L}}_{\wideparen{\D}_X}\mathcal{N}\right).
         \end{equation*}
     \end{enumerate}
 \end{defi}
We point out that there are versions of these operators for right modules, defined analogously. Some of the good properties of the classic versions of these operations survive to the rigid analytic setting. To us, the most important will be rigid analytic version of Kashiwara's equivalence: Let $i:X\rightarrow Y$ be a closed immersion of smooth rigid analytic
$K$-varieties. Let $\operatorname{D}_{\mathcal{C}^X}(\wideparen{\D}_Y)$ denote the category of $\mathcal{C}$-complexes on $Y$ with cohomology supported on $X$. We have the following:
\begin{teo}[{\cite[Proposition 9.5]{bode2021operations}}]
For $\mathcal{M}\in \operatorname{D}_{\mathcal{C}}(\wideparen{\D}_X)$, and  $\mathcal{N}\in \operatorname{D}_{\mathcal{C}^X}(\wideparen{\D}_Y)$  there is a functorial isomorphism:
\begin{equation*}
    R\Hom_{\operatorname{D}(\wideparen{\D}_Y)}(i_+\mathcal{M},\mathcal{N})=R\Hom_{\operatorname{D}(\wideparen{\D}_X)}(\mathcal{M},i^!\mathcal{N}).
\end{equation*}
In particular, $i_+:\operatorname{D}_{\mathcal{C}}(\wideparen{\D}_X)\leftrightarrows \operatorname{D}_{\mathcal{C}^X}(\wideparen{\D}_Y):i^!$ are mutually inverse equivalences of triangulated categories. Furthermore, $i_+$ and $i^!$ send co-admissible modules into co-admissible modules.
\end{teo}

\section{Products of Lie algebroids and the bi-enveloping algebra}\label{Chapter products of Lie algebroids}
Let $X$ be a smooth and separated rigid analytic space equipped with a Lie algebroid $\mathscr{L}$.  The final goal of this paper is developing a formalism of Hochschild (co)-homology for sheaves of Ind-Banach $\wideparen{U}(\mathscr{L})$-bimodules. Define the enveloping algebra of $\wideparen{U}(\mathscr{L})$ as the following sheaf of Ind-Banach $K$-algebras on $X$:
\begin{equation*}
    \wideparen{U}(\mathscr{L})^e:=\wideparen{U}(\mathscr{L})\overrightarrow{\otimes}_K\wideparen{U}(\mathscr{L})^{\op}.
\end{equation*}
Ideally, we would wish to study the following derived functors:
\begin{align*}
   \mathcal{HH}^{\bullet}(-):= R\underline{\mathcal{H}om}_{\wideparen{U}(\mathscr{L})^e}(\wideparen{U}(\mathscr{L}),-),&\\ 
  &\mathcal{HH}_{\bullet}(-):=\wideparen{U}(\mathscr{L})\overrightarrow{\otimes}^L_{\wideparen{U}(\mathscr{L})^e}-.
\end{align*}
The first problem we encounter with this definition is that the sheaf $\wideparen{U}(\mathscr{L})^e$ is pretty badly behaved. In particular, its algebraic properties are not clear, and it is rather hard to determine its global sections. In particular, it is not clear if $\wideparen{U}(\mathscr{L})^e$ is a sheaf of Fréchet-Stein algebras (in fact, this is probably false), and this prevents us from using the material developed in previous sections.\bigskip

In order to circumvent these issues, we will slightly modify the definition of Hochschild (co)-homology that we wish to study. In particular, we will introduce two auxiliary sheaves of Ind-Banach $K$-algebras on the product $X^2$, which we will denote by $\wideparen{E}(\mathscr{L})$, and $\wideparen{U}(\mathscr{L}^2)$.  As we will see, these sheaves are much better behaved than $\wideparen{U}(\mathscr{L})^e$, and there are exact functors from the category of Ind-Banach $\wideparen{U}(\mathscr{L})^e$-modules to the categories of Ind-Banach modules over 
$\wideparen{E}(\mathscr{L})$  and  $\wideparen{U}(\mathscr{L}^2)$ respectively. We will now give a brief description of these sheaves, and explain some of the rationale behind their definitions.\bigskip

Let us deal with  $\wideparen{E}(\mathscr{L})$ first. The sheaf $\wideparen{E}(\mathscr{L})$ is called the sheaf of complete bi-enveloping algebras of $\mathscr{L}$, and its defined as the following sheaf of Ind-Banach $K$-algebras on the product $X^2:=X\times X$:
\begin{equation*}
    \wideparen{E}(\mathscr{L}):=\wideparen{U}(\mathscr{L}) \overrightarrow{\boxtimes}\wideparen{U}(\mathscr{L})^{\op}= \OX_{X^2}\overrightarrow{\otimes}^{\mathbb{L}}_{p_1^{-1}\OX_X\overrightarrow{\otimes}_Kp_2^{-1}\OX_X} \left(p_1^{-1}\wideparen{U}(\mathscr{L})\overrightarrow{\otimes}_Kp_2^{-1}\wideparen{U}(\mathscr{L})^{\op}\right).
\end{equation*}
Notice that it is a priory not clear that $\wideparen{E}_X$ is contained in degree zero, nor that it is indeed an algebra. We will see, however, that both properties do hold. Let $\Delta:X\rightarrow X^2$ be the diagonal embedding. Assuming that $\wideparen{E}_X$ is an Ind-Banach algebra, it follows from the definition that there is a functor of quasi-abelian categories:
\begin{align*}
    \Delta_*:\Mod_{\Indban}(\wideparen{U}(\mathscr{L})^e)\rightarrow \Mod_{\Indban}(\wideparen{E}_X),\\
    \mathcal{M}\mapsto \OX_{X^2}\overrightarrow{\otimes}^{\mathbb{L}}_{p_1^{-1}\OX_X\overrightarrow{\otimes}_Kp_2^{-1}\OX_X}\Delta_*\mathcal{M}
\end{align*}
Using this functor, we will be able to define Hochschild (co)-homology as follows:
\begin{multline*}
   \mathcal{HH}^{\bullet}(-):= \Delta^{-1}R\underline{\mathcal{H}om}_{\wideparen{E}(\mathscr{L})}(\Delta_*\wideparen{U}(\mathscr{L}),\Delta_*(-)),\\ 
  \mathcal{HH}_{\bullet}(-):=\Delta^{-1}\left(\Delta_*\wideparen{U}(\mathscr{L})\overrightarrow{\otimes}^L_{\wideparen{E}(\mathscr{L})}\Delta_*(-)\right).
\end{multline*}
Hence, we may translate our homological problems from  $\wideparen{U}(\mathscr{L})^e$ to $\wideparen{E}(\mathscr{L})$. At first, it may seem like we are not gaining much with this change. However, it turns out that $\wideparen{E}(\mathscr{L})$ presents a much nicer behavior than  $\wideparen{U}(\mathscr{L})^e$, and the study of the category sheaves of Ind-Banach $\wideparen{E}(\mathscr{L})$-modules  takes most of the following section.\bigskip

Part of our interest in $\wideparen{E}(\mathscr{L})$ stems from our second sheaf of algebras, $\wideparen{U}(\mathscr{L}^2)$. This sheaf is defined as the following sheaf of Ind-Banach $K$-algebras on $X^2$:
\begin{equation*}
    \wideparen{U}(\mathscr{L}^2):=\wideparen{U}(\mathscr{L}) \overrightarrow{\boxtimes}\wideparen{U}(\mathscr{L})= \OX_{X^2}\overrightarrow{\otimes}^{\mathbb{L}}_{p_1^{-1}\OX_X\overrightarrow{\otimes}_Kp_2^{-1}\OX_X} \left(p_1^{-1}\wideparen{U}(\mathscr{L})\overrightarrow{\otimes}_Kp_2^{-1}\wideparen{U}(\mathscr{L})\right).
\end{equation*}
Most of the contents of the section are destined towards studying the properties of $\wideparen{U}(\mathscr{L}^2)$. As the notation suggests, we will define a Lie algebroid on $X^2$, which we denote by $\mathscr{L}^2$, satisfying that  $\wideparen{U}(\mathscr{L}^2)$ is the sheaf of complete enveloping algebras of $\mathscr{L}^2$. In particular, we can apply the contents of Section \ref{Section background Lie algebroids}, to the category of Ind-Banach $\wideparen{U}(\mathscr{L}^2)$-modules. The relation between $\wideparen{E}(\mathscr{L})$ and $\wideparen{U}(\mathscr{L}^2)$ will be further studied in later sections. Let us just mention that the side-switching
operations for $\wideparen{U}(\mathscr{L})$-modules can be extended to this setting, and that  this indices a natural equivalence of derived categories between  $\operatorname{D}(\wideparen{E}(\mathscr{L}))$ and $\operatorname{D}(\wideparen{U}(\mathscr{L}^2))$. This equivalence will later be used to give a more convenient description of  $\mathcal{HH}^{\bullet}(-)$ and $\mathcal{HH}_{\bullet}(-)$, and to obtain explicit calculations in certain situations.
\subsection{Products of Lie-Rinehart algebras}\label{Section products of Lie-Rinehart algebras}
Let $X$ and $Y$ be smooth and separated rigid analytic $K$-varieties, equipped with Lie algebroids $\mathscr{L}_X$ and $\mathscr{L}_Y$. In this section we will construct a Lie algebroid  $\mathscr{L}_X\times \mathscr{L}_Y$ on the product $X\times Y$. This Lie algebroid is called the product of $\mathscr{L}_X$ and $\mathscr{L}_Y$, and it can be characterized by the following universal property: Let $\mathscr{L}_Z$ be a Lie algebroid on $X\times Y$ satisfying that we have maps:
\begin{equation*}
    \mathscr{L}_X\rightarrow p_{X*}\mathscr{L}_Z, \quad \mathscr{L}_Y\rightarrow p_{Y*}\mathscr{L}_Z,
\end{equation*}
which are $\OX_X$-linear (resp. $\OX_Y$-linear) maps of sheaves of $K$-Lie algebras. Then there is a unique map of Lie algebroids 
$\mathscr{L}_X\times \mathscr{L}_Y\rightarrow \mathscr{L}_Z$ extending these maps. Thus, letting $f^*$ denote the pullback of coherent modules along a morphism of rigid analytic spaces $f:Z\rightarrow T$, it follows that  $\mathscr{L}_X\times \mathscr{L}_Y$ is in some sense the smallest Lie algebroid on $X\times Y$ containing the Lie algebroids  $p_X^{*}\mathscr{L}_X$ and $p_Y^{*}\mathscr{L}_Y$.\bigskip

Let $X=\Sp(A)$, $Y=\Sp(B)$ be smooth affinoid $K$-varieties, and choose affine formal models $\mathfrak{X}=\Spf(\mathcal{A})$, and $\mathfrak{Y}=\Spf(\mathcal{B})$. We let $C=A\widehat{\otimes}_KB$, so we have $X\times Y=\Sp(C)$. We will write the projections onto the first and second factors by $p_X:X\times Y\rightarrow X$, and $p_Y:X\times Y\rightarrow Y$ respectively. 
\begin{prop}\label{prop affine formal model on product}
$\mathcal{A}\widehat{\otimes}_{\mathcal{R}}\mathcal{B}$ is an admissible $\mathcal{R}$-algebra such that the following identity holds:
\begin{equation*}
    \left(\mathcal{A}\widehat{\otimes}_{\mathcal{R}}\mathcal{B}\right)\otimes_{\mathcal{R}}K=A\widehat{\otimes}_KB.
\end{equation*}
\end{prop}
\begin{proof}
Let  $\mathcal{C}=\mathcal{A}\widehat{\otimes}_{\mathcal{R}}\mathcal{B}$. We need to show that $\mathcal{C}$ is $\mathcal{R}$-flat and topologically of finite presentation. We start by choosing a pair of topological algebra presentations:
\begin{equation*}
    \mathcal{A}\cong \mathcal{R}\langle x_1,\cdots, x_n \rangle/\mathcal{I}_1, \quad \mathcal{B}\cong \mathcal{R}\langle y_1,\cdots, y_r \rangle/\mathcal{I}_2,
\end{equation*}
where $\mathcal{I}_1$, and $\mathcal{I}_2$ are finitely generated ideals of $\mathcal{R}\langle x_1,\cdots, x_n \rangle$
and $\mathcal{R}\langle y_1,\cdots, y_r \rangle$ respectively.  Consider the following  $\mathcal{R}\langle x_1,\cdots, x_n \rangle\otimes_{\mathcal{R}}\mathcal{R}\langle y_1,\cdots, y_r \rangle$-module:
\begin{equation*}
    \mathcal{J}:=\mathcal{I}_1\otimes_{\mathcal{R}}\mathcal{R}\langle y_1,\cdots, y_r \rangle +\mathcal{R}\langle x_1,\cdots, x_n \rangle \otimes_{\mathcal{R}}\mathcal{I}_2.
\end{equation*}
As all $\mathcal{R}$-modules mentioned above are flat, we have the following short exact sequence of $\mathcal{R}$-modules:
\begin{equation}\label{equation products of affine formal models}
    0\rightarrow \mathcal{J} \rightarrow \mathcal{R}\langle x_1,\cdots, x_n \rangle\otimes_{\mathcal{R}}\mathcal{R}\langle y_1,\cdots, y_r \rangle \rightarrow \mathcal{A}\otimes_{\mathcal{R}}\mathcal{B}\rightarrow 0.
\end{equation}
As $\mathcal{A}\otimes_{\mathcal{R}}\mathcal{B}$ is flat over $\mathcal{R}$, it does not have $\pi$-torsion. Thus, the $\pi$-adic completion of the previous sequence is still exact, and this shows that $\mathcal{C}:=\mathcal{A}\widehat{\otimes}_{\mathcal{R}}\mathcal{B}$ is topologically of finite presentation over $\mathcal{R}$. In order to show that $\mathcal{C}$ is flat, it suffices to see that it is torsion-free. We start by showing that it does not have $\pi$-torsion. As the algebra $\mathcal{R}\langle x_1,\cdots, x_n \rangle\widehat{\otimes}_{\mathcal{R}}\mathcal{R}\langle y_1,\cdots, y_r \rangle=\mathcal{R}\langle x_1,\cdots, x_n,y_1,\cdots, y_r \rangle$ is flat over $\mathcal{R}$, it suffices to show that the map $\widehat{\mathcal{J}}/\pi \rightarrow \mathcal{R}\langle x_1,\cdots, x_n,y_1,\cdots, y_r \rangle/\pi$ is injective. However, this follows by flatness of $\mathcal{A}\otimes_{\mathcal{R}}\mathcal{B}$ together with exactness of the sequence in equation (\ref{equation products of affine formal models}). Hence, $\mathcal{C}$ does not have $\pi$-torsion. As the completion of $\mathcal{A}\otimes_{\mathcal{R}}\mathcal{B}$ does not depend on the choice of the pseudo-uniformizer, the argument above shows that 
$\mathcal{C}$ is $\rho$-torsion free for every choice of pseudo-uniformizer  $\rho \in \mathcal{R}$. In other words, $\mathcal{C}$ is flat as a $\mathcal{R}$-module.
\end{proof}
Let $\mathfrak{X}\times\mathfrak{Y}=\Spf(\mathcal{C})$.
By Proposition \ref{prop affine formal model on product}, $\mathfrak{X}\times\mathfrak{Y}$  is an affine formal model of the rigid analytic space $X\times Y$.  
\begin{prop}[{\cite[Proposition 16.5]{eisenbud2013commutative}}]\label{prop product of differentials}
    Let $R$ be a commutative ring and $A,B$ two $R$-algebras. There is a canonical isomorphism of $A\otimes_RB$-modules:
    \begin{equation*}
        \Omega^1_{A\otimes_RB/R}=\Omega^1_{A/R}\otimes_RB\oplus A\otimes_R\Omega^1_{B/R}.
    \end{equation*}
\end{prop}
We can extend this proposition to the setting of formal schemes, as follows:
\begin{Lemma}\label{Lemma differentials on affine formal model of product}
We have the following identity of $\mathcal{C}$-modules:
 \begin{equation*}
\Omega^1_{\mathcal{C}/\mathcal{R}}=\Omega^1_{\mathcal{A}/\mathcal{R}}\widehat{\otimes}_{\mathcal{R}}\mathcal{B}\oplus \mathcal{A}\widehat{\otimes}_{\mathcal{R}}\Omega^1_{\mathcal{B}/\mathcal{R}}.   
 \end{equation*}     
In particular, we have the following identity:
\begin{equation*}
    \Der_{\mathcal
R}(\mathcal{C})=\Der_{\mathcal{R}}(\mathcal{A})\widehat{\otimes}_{\mathcal{R}}\mathcal{B}\oplus \mathcal{A}\widehat{\otimes}_{\mathcal{R}}\Der_{\mathcal{R}}(\mathcal{B})=\widehat{\Der_{\mathcal{R}}}(\mathcal{A}\otimes_{\mathcal{R}}\mathcal{B}).
\end{equation*}
\end{Lemma}
\begin{proof}
Let $D=\mathcal{A}\otimes_{\mathcal{R}}\mathcal{B}$. Notice that we have $\mathcal{C}=\widehat{D}$. Let $M$ be a $\mathcal{C}$-module. The $\pi$-adic topology on $M$ makes it a topological $\mathcal{C}$-module. Hence, we have the following chain of identities:
\begin{multline*}
   \Hom_{\mathcal{C}}(\Omega^1_{\mathcal{C}/\mathcal{R}},M) =\Der_{\mathcal{R}}(\mathcal{C},M)=\Der_{\mathcal{R}}(D,M)\\
   =\Hom_D(\Omega^1_{D/\mathcal{R}},M)=\Hom_{\widehat{D}}(\Omega^1_{D/\mathcal{R}}\otimes_D\mathcal{C},M),
\end{multline*}
where the second identity follows by the fact that $D$ is dense in $\mathcal{C}$, and $\mathcal{R}$-linear derivations are continuous with respect to the $\pi$-adic topology. Next, notice that by Lemma \ref{Lemma differentials on affine formal model of product} we have:
\begin{equation*}
    \Omega^1_{D/\mathcal{R}}=\Omega^1_{\mathcal{A}/\mathcal{R}}\otimes_{\mathcal{R}}\mathcal{B}\oplus \mathcal{A}\otimes_{\mathcal{R}}\Omega^1_{\mathcal{B}/\mathcal{R}}.
\end{equation*}
As $\mathcal{A}$ and $\mathcal{B}$ are affine formal models of $A$ and $B$ respectively, the modules $\Omega^1_{\mathcal{A}/\mathcal{R}}$  and $\Omega^1_{\mathcal{B}/\mathcal{R}}$ are finitely generated over $\mathcal{A}$ and $\mathcal{B}$. Furthermore, they are  flat as $\mathcal{R}$-modules. Thus, it follows by \cite[Theorem 7.3.4]{bosch2014lectures}
that $\Omega^1_{\mathcal{A}/\mathcal{R}}$  and $\Omega^1_{\mathcal{B}/\mathcal{R}}$ are finitely presented over $\mathcal{A}$ and $\mathcal{B}$ respectively. Hence, $\Omega^1_{D/\mathcal{R}}$ is a finitely presented $D$-module, and this implies that $\widehat{\Omega}^1_{D/\mathcal{R}}=\Omega^1_{D/\mathcal{R}}\otimes_D\mathcal{C}$. The last part of the lemma is a direct consequence of the isomorphism $\Der_{\mathcal{R}}(\mathcal{B})=\Hom_{\mathcal{B}}(\Omega_{\mathcal{B}/\mathcal{R}},\mathcal{B})$. 
\end{proof}
As $\mathfrak{X}\times\mathfrak{Y}$ is an affine formal model of $X\times Y$, we immediately get:

\begin{Lemma}\label{Lemma derivations on fiber products and completions}
   There is a canonical isomorphism of coherent $X\times Y$-modules:
    \begin{equation*}
        \Omega^1_{X\times Y/K}=p_X^*\Omega^1_{Y/K} \oplus p_Y^*\Omega^1_{X/K}.
    \end{equation*}
    In particular, we have the following isomorphism of $C$-modules:
    \begin{multline}\label{equation decomposition of vector fields on a product}
       \mathcal{T}_{X\times Y/K}(X\times Y)= \Hom_{\OX_{X\times Y}}(\Omega^1_{X\times Y/K},\OX_{X\times Y})\\
       =\Der_K(A)\widehat{\otimes}_KB\oplus A\widehat{\otimes}_K\Der_K(B).
    \end{multline}
\end{Lemma}
\begin{proof}
follows by Lemma \ref{prop product of differentials}, together with remark \ref{remark all derivations are bounded}.
\end{proof}
The commutator induces a $(K,A\widehat{\otimes}_KB)$-Lie algebra structure on the tangent sheaf $\mathcal{T}_{X\times Y/K}(X\times Y)$. We will now give an explicit description of this structure in terms of the decomposition obtained in equation $(\ref{equation decomposition of vector fields on a product})$.
\begin{Lemma}
    For $i=1,2$, let  $\omega_i=(v_i\otimes f_i )\oplus (g_i\otimes w_i)\in \mathcal{T}_{X\times Y/K}(X\times Y)$
    be vector fields on $X\times Y$. The following identity holds:
    \begin{multline}\label{description bracket product of vector fields}
[\omega_1 ,\omega_2]= ([v_1,v_2]\otimes f_1f_2) \oplus (g_2g_1\otimes [w_1,w_2])\\+(-g_2v_1\otimes w_2(f_1))\oplus (v_1(g_2)\otimes f_1w_2)
  + (g_1v_2\otimes w_1(f_2))\oplus (-v_2(g_1)\otimes f_2w_1).   
    \end{multline}     
\end{Lemma}
\begin{proof}
 Consider the following derivations:
\begin{equation*}
    \delta_1=[(v_1\otimes f_1)\oplus 0 ,(v_2\otimes f_2)\oplus 0] \textnormal{, } \delta_2=[0\oplus (g_1\otimes w_1) ,0\oplus (g_2\otimes w_2)].
\end{equation*}
For any $a\otimes b\in A\widehat{\otimes}_KB$ we have:
\begin{multline*}
  \delta_1(a\otimes b)=((v_1\otimes f_1) (v_2\otimes f_2) - (v_2\otimes f_2)(v_1\otimes f_1)) (a\otimes b) \\
 =(v_1(v_2(a))\otimes f_1f_2b)-(v_2(v_1(a))\otimes f_1f_2b)\\
 =([v_1,v_2](a)\otimes f_1f_2b)=([v_1,v_2]\otimes f_1f_2)(a\otimes b).   
\end{multline*}
Thus, we have $\delta_1=([v_1,v_2]\otimes f_1f_2)\oplus 0$. Symmetrically, we also get the identity $\delta_2=0\oplus(g_1g_2\otimes[w_1,w_2])$. Next, consider the derivations:
\begin{equation*}
    \delta_3=[(v_1\otimes f_1)\oplus 0 ,0 \oplus (g_2\otimes w_2)],\textnormal{  } \delta_4 =[0\oplus (g_1\otimes w_1) ,(v_2\otimes f_2)\oplus 0].
\end{equation*}
For any $a\otimes b\in A\widehat{\otimes}_KB$ we have:
\begin{align*}
     \delta_3(a\otimes b)&=v_1(g_2a)\otimes f_1w_2(b)-g_2v_1(a)\otimes w_2(f_1b)\\
   &=(g_2v_1(a)+av_1(g_2))\otimes f_1w_2(b)-g_2v_1(a)\otimes (f_1w_2(b)+bw_2(f_1))\\
   &=av_1(g_2)\otimes f_1w_2(b) - g_2v_1(a)\otimes bw_2(f_1).
\end{align*}
Thus, we have:
\begin{equation*}
 \delta_3 = (-g_2v_1\otimes w_2(f_1))\oplus (v_1(g_2)\otimes f_1w_2), \textnormal{ and } 
\delta_4=(g_1v_2\otimes w_1(f_2))\oplus (-v_2(g_1)\otimes f_2w_1).   
\end{equation*}
As Lie brackets are bilinear, we have $ [\omega_1,\omega_2]=\sum_{i=1}^4\delta_i$, as wanted.
\end{proof}

Let now  $(L_1,[-,-]_1,\tau_1)$ be a finite free $(K,A)$-Lie algebra,
and $(L_2,[-,-]_2,\tau_2)$ be a finite free $(K,B)$-Lie algebra. For simplicity in the notation, we will often just denote them by $L_1$ and $L_2$, and omit the sub-indices from the brackets and anchor maps when no confusion is possible. We make the following definition:
\begin{defi}
  Let $L_1\times L_2$ be the $A\widehat{\otimes}_KB$-module given by:
    \begin{equation*}
        L_1\times L_2:=p_X^*L_1\oplus p_Y^*L_2=L_1\widehat{\otimes}_KB\oplus A\widehat{\otimes}_KL_2,
    \end{equation*}
    where the second identity holds because $L_1$ and $L_2$ are finitely generated over $A$ and $B$ respectively.
\end{defi}
As $A\widehat{\otimes}_KB$ is an affinoid algebra, it induces a unique Banach topology on $L_1\times L_2$. Consider the $A\otimes_KB$-module:
\begin{equation*}
    \mathfrak{L}=L_1\otimes_KB\oplus A\otimes_KL_2.
\end{equation*}
It is clear that $\mathfrak{L}$ is a dense subspace of $L_1\times L_2$. In particular, considering $\mathfrak{L}$ as a locally convex $K$-vector space with respect to the subspace topology, it follows that it is a topological $A\otimes_KB$-module, such that its completion agrees with 
$L_1\times L_2$. Furthermore, by Lemma \ref{Lemma derivations on fiber products and completions} we have:
\begin{equation*}  \Der_K(A\widehat{\otimes}_KB)=\widehat{\Der}_K(A\otimes_K B).
\end{equation*}
Hence, in order to construct a $(K,A\widehat{\otimes}_KB)$-Lie algebra structure on $L_1\times L_2$, it suffices to construct a
$(K,A\otimes_KB)$-Lie algebra structure on $\mathfrak{L}$ such that the Lie bracket and the anchor map are continuous.

Let us start by recalling some constructions on filtered algebras:
\begin{defi}
\begin{enumerate}[label=(\roman*)]
    \item Let $R$ be a ring and let $A$ and $B$ be two filtered $R$-algebras. Assume the filtrations $F_{\bullet}A$ and $F_{\bullet}B$ satisfy  $R\subset F_0A,F_0B$. For each $i,j\in\mathbb{Z}$, let $\overline{F_iA\otimes_RF_jB}$ be the image of the map:
    \begin{equation*}
        F_iA\otimes_RF_jB\rightarrow A\otimes_R B. 
    \end{equation*}
We obtain a filtration on the tensor product $C=A\otimes_{R}B$ given by the following family of $\mathcal{R}$-submodules:
    \begin{equation*}
        F_nC=\sum_{i+j=n}\overline{F_iA\otimes_{\mathcal{R}}F_jB}.
    \end{equation*}
We denote this the tensor filtration. Unless mentioned, we will regard $C$ as a filtered algebra with respect to the tensor filtration.
\item For graded $R$-algebras $A$ and $B$, satisfying $R\subset A_0$ and $R\subset B_0$, we define the tensor grading on $C$ using the analogous construction.
\end{enumerate}     
\end{defi}
\begin{obs}\label{remark tensor grading and tensor filtration}
\begin{enumerate}[label=(\roman*)]
    \item If $F_{\bullet}A$ and $F_{\bullet}B$ are exhaustive, then  $F_{\bullet}C$ is also exhaustive.
    \item The tensor filtration (grading) on $R[x,y]$ induced by the isomorphism $R[x,y]\cong R[x]\otimes_{R}R[y]$ is the usual filtration (grading) by degree of monomials. 
\end{enumerate}
\end{obs}
We may now apply these definitions to out case of interest:
 \begin{Lemma}\label{Lemma pbw on products}
     The following hold:
     \begin{enumerate}[label=(\roman*)]
         \item $F_0\left(U(L_1)\otimes_KU(L_2)\right)=A\otimes_KB$
         \item $[A\otimes_KB,F_n\left(U(L_1)\otimes_KU(L_2)\right)]\subset F_{n-1}\left(U(L_1)\otimes_KU(L_2)\right)$.
         \item $\gr_F\left(U(L_1)\otimes_KU(L_2)\right)$ is an $A\otimes_KB$-algebra.
     \end{enumerate}
 \end{Lemma}
 \begin{proof}
By definition we have the following identities:
\begin{equation*}
    F_0\left(U(L_1)\otimes_KU(L_2)\right)=F_0U(L_1)\otimes_K F_0U(L_2)=A\otimes_KB,
\end{equation*}
so claim $(i)$ holds. For statement $(ii)$, choose $a\otimes b\in A\otimes_K B$, $v\in F_iU(L_1)$, $w\in F_jU(L_2)$ satisfying $i+j=n$. Then we have the following chain of identities:
\begin{equation*}
[a\otimes b,v\otimes w]=av\otimes bw -va\otimes wb= av\otimes([b,w]+wb)- va \otimes wb =[a,v]\otimes wb+ av\otimes[b,w].
\end{equation*}
By the PBW Theorem \cite[Theorem 3.1]{rinehart1963differential}, we have $[a,v]\in F_{i-1}U(L_1)$, and $[b,w]\in F_{j-1}U(L_2)$. Hence, we have 
$[a\otimes b,v\otimes w]\in F_{n-1}\left(U(L_1)\otimes_KU(L_2)\right)$, so statement $(ii)$ holds. Passing to the associated graded we have:
\begin{equation*}
    A\otimes_K B\subset Z(\gr_F\left(U(L_1)\otimes_KU(L_2)\right)),
\end{equation*}
which shows statement $(iii)$.
 \end{proof}
\begin{prop}\label{PBW version for enveloping algebras}
 We have a canonical isomorphism of graded $A\otimes_KB$-algebras:
    \begin{equation*}      
   \operatorname{Sym}_{A}(L_1)\otimes_K\operatorname{Sym}_{B}(L_2)\rightarrow \gr_F\left(U(L_1)\otimes_KU(L_2)\right).
    \end{equation*}
\end{prop}
Thus, $U(L_1)\otimes_KU(L_2)$ is a filtered $K$-algebra such that its associated graded is commutative. In other words, it is an an almost commutative filtered $K$-algebra (cf. \cite[Definition 3.4]{ardakov2013irreducible}).
\begin{proof}
By construction of the tensor filtration, the inclusions:
\begin{equation*}
    U(L_1)\rightarrow U(L_1)\otimes_KU(L_2), \textnormal{ and } U(L_2)\rightarrow U(L_1)\otimes_KU(L_2),
\end{equation*}
are maps of filtered $K$-algebras.  Hence, we get a map of graded $K$-algebras:
\begin{equation*}
    \psi:\operatorname{Sym}_{A}(L_1)\otimes_K\operatorname{Sym}_{B}(L_2)\rightarrow \gr_F(U(L_1)\otimes_KU(L_2)),
\end{equation*}
where the grading on the left-hand side is the tensor grading.
By Lemma \ref{Lemma pbw on products}, the zero-graded $K$-algebra on both sides is $A\otimes_KB$. Hence, this is in fact a morphism of $A\otimes_KB$-algebras.\\ 
 Let $x_1,\cdots,x_s$ be an ordered basis of $L_1$ as an $A$-module, and  $y_1,\cdots, y_m$ be an ordered basis of $L_2$ as a $B$-module. By the PBW Theorem \cite[Theorem 3.1]{rinehart1963differential}, for each $i\geq 0$ the $A$-module $F_{i}U(L_1)$ is free, on a basis given by $A$-linear combinations of elements of the form:
 \begin{equation*}
     x_1^{\alpha_1}\cdots x_n^{\alpha_n}, \textnormal{ with } \sum_{s=1}^n\alpha_s\leq i.
 \end{equation*}
 By symmetry, we get an analogous description of the $B$-module structure of $F_{i}U(L_2)$.

Therefore, we have a basis of $F_{l}\left(U(L_1)\otimes_KU(L_2)\right)=\sum_{i+j=l}F_iU(L_1)\otimes_{K}F_jU(L_2)$
as an $A\otimes_KB$-module given by the following simple tensors:
\begin{equation*}
    x_1^{\alpha_1}\cdots x_n^{\alpha_n}\otimes y_1^{\beta_1}\cdots y_m^{\beta_m}, \textnormal{ such that  } \sum_{s=1}^n \alpha_s+ \sum_{s=1}^m \beta_s\leq l.
\end{equation*}
In particular, $F_{l}\left(U(L_1)\otimes_KU(L_2)\right)/F_{l-1}\left(U(L_1)\otimes_KU(L_2)\right)$ is a free $A\otimes_KB$-module, with a basis  given by the simple tensors:
\begin{equation*}
    x_1^{\alpha_1}\cdots x_n^{\alpha_n}\otimes y_1^{\beta_1}\cdots y_m^{\beta_m}, \textnormal{ such that }\sum_{s=1}^n \alpha_s+ \sum_{s=1}^m \beta_s=l.
\end{equation*}
The family of all such tensors, for varying $l$, yields a basis of $\gr_F \left(U(L_1)\otimes_KU(L_2)\right)$ as a module over $A\otimes_KB$. Notice that $\operatorname{Sym}_{A}(L_1)$ and $\operatorname{Sym}_{B}(L_2)$ are polynomial algebras over $A$ and $B$ respectively. In particular, we have an isomorphism of graded $A$-algebras (resp. $B$-algebras):
\begin{equation*}
    \operatorname{Sym}_{A}(L_1)\rightarrow A[x_1,\cdots x_n], \textnormal{ } \operatorname{Sym}_{B}(L_2)\rightarrow B[y_1,\cdots y_m].
\end{equation*}
Hence, we have an isomorphism of graded $A\otimes_KB$-algebras:
\begin{multline*}
    \operatorname{Sym}_{A}(L_1)\otimes_{K}\operatorname{Sym}_{B}(L_2)\rightarrow A[x_1,\cdots x_n]\otimes_{K}B[y_1,\cdots y_m]\\
    \cong A\otimes_{K}B[x_1,\cdots x_n,y_1,\cdots y_m].
\end{multline*}
Notice that by remark \ref{remark tensor grading and tensor filtration} the grading on $A\otimes_{K}B[x_1,\cdots x_n,y_1,\cdots y_m]$ induced by the tensor grading is the usual grading on polynomial rings. Thus, $\psi$ maps a basis of $A\otimes_{K}B[x_1,\cdots x_n,y_1,\cdots y_m]$ as an $A\otimes_KB$-module to a basis of the graded algebra $\gr_F(U(L_1)\otimes_KU(L_2))$. Therefore, $\psi$ is an isomorphism of graded $A\otimes_KB$-algebras, as we wanted to show.
\end{proof}
\begin{coro}\label{coro decomposition filtration of tensor}
For any $i\geq 0$, we have a decomposition of $A\otimes_K B$-modules:
\begin{equation*}
    F_i\left( U(L_1)\otimes_K U(L_2)\right)= \bigoplus_{j=0}^i F_j U(L_1)\otimes_K F_{i-j} U(L_2).
\end{equation*}
In particular, $F_i U(L_1)\otimes_K U(L_2)$ is a finite-free $A\otimes_KB$ module.
\end{coro}
In particular, we have an injective map
of $K$-algebras:
\begin{equation*}
    j:A\otimes_KB\rightarrow U(L_1)\otimes_K U(L_2).
\end{equation*}
This gives $U(L_1)\otimes_K U(L_2)$ a structure of a left $A\otimes_KB$-module. On the other hand, the PBW Theorem also yields injective maps of $A$-modules (resp. $B$-modules):
\begin{equation*}
    i_1:L_1\rightarrow U(L_1), \quad i_2:L_2\rightarrow U(L_2).
\end{equation*}
 Thus, we get the following morphism of $A\otimes_KB$-modules:
\begin{equation*}
  i:\mathfrak{L}\rightarrow U(L_1)\otimes_KU(L_2).
\end{equation*}
We will define the Lie bracket on  $\mathfrak{L}$ using the commutator on $U(L_1)\otimes_KU(L_2)$:
\begin{coro}\label{coro lie bracket}
The following hold:
\begin{enumerate}[label=(\roman*)]
    \item $i:\mathfrak{L}\rightarrow U(L_1)\otimes_K U(L_2)$ is injective. Thus, we may identify $\mathfrak{L}$ with its image.
    \item $\mathfrak{L}=F_1(U(L_1)\otimes_K U(L_2))\setminus A\otimes_K B$.
    \item The commutator on $U(L_1)\otimes_K U(L_2)$ induces a $K$-Lie algebra structure:
    \begin{equation*}
        [-,-]:\mathfrak{L} \times \mathfrak{L} \longrightarrow \mathfrak{L}.
    \end{equation*}
    \item For $i=1,2$, let $\omega_i=(v_i\otimes f_i)\oplus (g_i\otimes w_i) \in \mathfrak{L}$. The bracket on $\mathfrak{L}$ is:
\begin{multline}\label{equation Lie bracket on a product}
[\omega_1 ,\omega_2]= ([v_1,v_2]\otimes f_1f_2) \oplus (g_2g_1\otimes [w_1,w_2])\\+(-g_2v_1\otimes\tau_2(w_2)(f_1))\oplus (\tau_1(v_1)(g_2)\otimes f_1w_2)\\
  + (g_1v_2\otimes \tau_2(w_1)(f_2))\oplus (-\tau_1(v_2)(g_1)\otimes f_2w_1).          
\end{multline}
\end{enumerate}     
\end{coro}
\begin{proof}
Let $U=U(L_1)\otimes_K U(L_2)$. Assume that $L_1$ and $L_2$ have basis given by $x_1,\cdots,x_n$  and $y_1,\cdots,y_m$ as free modules over $A$ and $B$ respectively. Then, we have a basis of   $F_1(U)\setminus A\otimes_KB$ given by the simple tensors $x_i\otimes 1$, $1\otimes y_j$, where $1\leq i\leq n$ and $1\leq j\leq m$. As this is a basis of $\mathfrak{L}$ over $A\otimes_K B$, it follows that $(i)$ and $(ii)$ hold.\\ 
As the associated graded $\gr(U)$ is commutative, we have $[F_iU,F_jU]\subset F_{i+j-1}U$ for all $i,j\geq 0$. In particular, we have: $[\mathfrak{L},\mathfrak{L}]\subset F_1U=\mathfrak{L}\oplus (A\otimes_KB)$. Thus, statement $(iii)$ holds if and only if $[\mathfrak{L},\mathfrak{L}]\cap A\otimes_K B=0$. Hence, it is enough to show $(iv)$. However, as the commutator bracket is $K$-bilinear, and we have $[x,a]=\tau_1(x)(a)$, $[y,b]=\tau_1(y)(b)$ for all $x\in L_1$, $y\in L_2$, $a\in A$, and $b\in B$, we can repeat the calculations leading up to equation $(\ref{description bracket product of vector fields})$ to show that statement $(iv)$ holds. 
\end{proof}
As a consequence, we get the following:
\begin{Lemma}\label{Lemma universal enveloping algebra and uncompleted products}
Consider the map:
\begin{equation*}
    \tau_1\times\tau_2=\tau_1\otimes_K \Id_B\oplus \Id_A\otimes_K\tau_2:\mathfrak{L}\rightarrow \Der_K(A\otimes_KB).
\end{equation*}
The following hold:
\begin{enumerate}[label=(\roman*)]
    \item $\tau_1\times\tau_2$ is an $A\otimes_KB$-linear map of $K$-Lie algebras.
    \item $(\mathfrak{L},[-,-],\tau_1\times\tau_2)$ is a $(K,A\otimes_KB)$-Lie algebra.
    \item We have a canonical isomorphism of filtered $K$-algebras:
    \begin{equation*}
        U(\mathfrak{L})\rightarrow U(L_1)\otimes_KU(L_2).
    \end{equation*}
\end{enumerate}
\end{Lemma}
\begin{proof}
By definition, the anchor map $\tau_1$ ( resp. $\tau_2$) is a morphism of $A$-modules (resp. $B$-modules). Thus, it follows that $\tau_1\times \tau_2$ is $A\otimes_KB$-linear. Furthermore, both are maps of $K$-Lie algebras. One can check that applying $\tau_1\times\tau_2$ to equation $(\ref{equation Lie bracket on a product})$ yields equation $(\ref{description bracket product of vector fields})$, thus showing that $\tau_1\times\tau_2$ is a morphism of $K$-Lie algebras. Hence, statement $(i)$ holds.\\
 Statement $(ii)$ follows by inspection of $(\ref{equation Lie bracket on a product})$. In order to show $(iii)$, it is enough to show that $U(L_1)\otimes_KU(L_2)$ satisfies the universal property of enveloping algebras with respect to the maps:
\begin{equation*}
    i:\mathfrak{L}\rightarrow U(L_1)\otimes_KU(L_2), \quad j:A\otimes_KB\rightarrow U(L_1)\otimes_KU(L_2).
\end{equation*}
Let $D$ be an associative $K$-algebra, together with a map of $K$-Lie algebras $i_D:\mathfrak{L}\rightarrow D$, and a map of $K$-algebras $j_D:A\otimes_KB\rightarrow D$. We also assume that for each $f\in A\otimes_KB$, and each $v\in \mathfrak{L}$ we have $j_D(f)i_D(v)=i_D(fv)$. Consider the restrictions of $i_D$ and $j_D$ to $L_1\otimes_K 1$ and  $A\otimes_K 1$ respectively. This induces a unique map of $K$-algebras $U(L_1)\rightarrow D$. Similarly, we get a unique map of $K$-algebras $U(L_2)\rightarrow D$. Hence, we get a unique map of $K$-algebras $U(L_1)\otimes_KU(L_2)\rightarrow D$ extending $i_D$ and $j_D$, as wanted.
\end{proof}
The next step is showing that the $(K,A\otimes_KB)$-Lie algebra structure on $\mathfrak{L}$ extends to a $(K,A\widehat{\otimes}_KB)$-Lie algebra structure $L_1\times L_2$. In particular, we need to show that the bracket and anchor map are continuous. As $L_1$ and $L_2$ are free, we may choose an $(\mathcal{R},\mathcal{A})$-Lie lattice $\mathcal{L}_1$ of $L_1$, and a  $(\mathcal{R},\mathcal{B})$-Lie lattice $\mathcal{L}_2$ of  $L_2$. The topology of $L_1\times L_2$ as a finite $A\widehat{\otimes}_KB$-module is induced by the $\pi$-adic topology on the finite $\mathcal{A}\widehat{\otimes}_{\mathcal{R}}\mathcal{B}$-module:
\begin{equation*}
    \mathcal{L}_1\times \mathcal{L}_2=\mathcal{L}_1\widehat{\otimes}_{\mathcal{R}}\mathcal{B}\oplus \mathcal{A}\widehat{\otimes}_{\mathcal{R}}\mathcal{L}_2.
\end{equation*}
Notice that this lattice is the $\pi$-adic completion of the 
$\mathcal{A}\otimes_{\mathcal{R}}\mathcal{B}$-module:
\begin{equation*}
   \mathcal{M}= \mathcal{L}_1\otimes_{\mathcal{R}}\mathcal{B}\oplus \mathcal{A}\otimes_{\mathcal{R}}\mathcal{L}_2.
\end{equation*}
As $\mathcal{R}$-linear maps between spaces with the $\pi$-adic topology are always continuous, it suffices to show that  the restriction of the bracket and anchor map of $\mathfrak{L}$ to $\mathcal{M}$ makes $\mathcal{M}$ a $(\mathcal{R},\mathcal{A}\otimes_{\mathcal{R}}\mathcal{B})$-Lie algebra. 
\begin{Lemma}\label{Lemma continuity of bracket}
The following hold:
\begin{enumerate}[label=(\roman*)]
    \item $\tau_1\times\tau_2(\mathcal{M})\subset \Der_{\mathcal{R}}(\mathcal{A}\otimes_{\mathcal{R}}\mathcal{B})$.
    \item $[\mathcal{M},\mathcal{M}]\subset \mathcal{M}$.
\end{enumerate}
In particular, $(\mathcal{M},[-.-],\tau_1\times\tau_2)$ is a $(\mathcal{R},\mathcal{A}\otimes_{\mathcal{R}}\mathcal{B})$-Lie algebra.
\end{Lemma}
\begin{proof}

By definition, both $\mathcal{L}_i$ are Lie lattices. Thus, we have:
\begin{equation*}
    \tau_1(\mathcal{L}_1)\subset \Der_{\mathcal{R}}(\mathcal{A}), \quad \tau_2(\mathcal{L}_2)\subset \Der_{\mathcal{R}}(\mathcal{B}).
\end{equation*}
Therefore, we have:
   \begin{equation*}
       \tau_1\times\tau_2(\mathcal{M})\subset \Der_{\mathcal{R}}(\mathcal{A})\otimes_{\mathcal{R}}\mathcal{B}\oplus \mathcal{A}\otimes_{\mathcal{R}}\Der_{\mathcal{R}}(\mathcal{B})= \Der_{\mathcal{R}}(\mathcal{A}\otimes_{\mathcal{R}}\mathcal{B}).
   \end{equation*}
For the second part, notice that equation $(\ref{equation Lie bracket on a product})$ gives an explicit description of the Lie bracket on $\mathfrak{L}$. As $\mathcal{L}_1$ and $\mathcal{L}_2$ are Lie lattices, they are closed under their respective brackets, as well as closed by multiplication by elements in $\mathcal{A}$ and $\mathcal{B}$ respectively. Hence, we may use equation $(\ref{equation Lie bracket on a product})$ to show that $[\mathcal{M},\mathcal{M}]\subset \mathcal{M}$.
\end{proof}
We may summarize our findings so far into the following proposition:
\begin{prop}\label{prop product of Lie rinehart algebras}
The following hold:
\begin{enumerate}[label=(\roman*)]
    \item $(L_1\times L_2, [-,-],\tau_1\times \tau_2)$, is a finite-free $(K,A\widehat{\otimes}_KB)$-Lie algebra.
    \item The inclusion $j_1:L_1\rightarrow L_1\times L_2$ given by $x\mapsto x\otimes 1 $ is an $A$-linear map of $K$-Lie algebras satisfying $\tau_1=(\tau_1\times \tau_2)\circ j_1$.
    \item The inclusion $j_2:L_2\rightarrow L_1\times L_2$ given by $y\mapsto 1\otimes y$ is a $B$-linear map of $K$-Lie algebras satisfying $\tau_2=(\tau_1\times \tau_2)\circ j_2$.
\end{enumerate}
\end{prop}
Now that we have a local description of $L_1\times L_2$, we will commit our efforts to giving an explicit description of the universal property claimed at the beginning of the section.
\begin{Lemma}\label{Lemma existence of liftings to enveloping algebras}
Let $L$ be a $(K,A)$-Lie algebra and $M$ be a $(K,A\widehat{\otimes}_KB)$-Lie algebra. Denote their anchor maps by $\tau_L$ and $\tau_M$ respectively. Consider a $K$-linear map $\theta: L\rightarrow M$ satisfying the following:
\begin{enumerate}[label=(\roman*)]
    \item $\theta$ is an $A$-linear map of $K$-Lie algebras.
    \item $\tau_L=\tau_M\circ \theta$, where we identify $\Der_K(A)$ with its image inside $\Der_K(A\widehat{\otimes}_KB)$.
\end{enumerate}
Then there is a unique extension $U(\theta):U(L_1)\rightarrow U(M)$, which is a morphism of filtered $K$-algebras.
\begin{proof}
    Clear by the universal property of universal enveloping algebras.
\end{proof}
\end{Lemma}
\begin{prop}\label{prop universal property of product}
Let $M$ be a $(K,A\widehat{\otimes}_KB)$-Lie algebra, such that $M$ is coherent as a $A\widehat{\otimes}_KB$-module. Assume we have $K$-linear maps:
\begin{equation*}
    \theta_1:L_1\rightarrow M, \quad  \theta_2:L_2\rightarrow M,
\end{equation*}
satisfying the conditions in Lemma \ref{Lemma existence of liftings to enveloping algebras}. Then there is a unique map of $(K,A\widehat{\otimes}_KB)$-Lie algebras: $\theta:L_1\times L_2 \rightarrow M$ which induces the following factorization for $j=1,2$:
\begin{equation*}
\begin{tikzcd}
L_i \arrow[r, "j_i"] \arrow[rr, "\theta_i"', bend left] & L_1\times L_2 \arrow[r, "\theta"] & M
\end{tikzcd}
\end{equation*}
\end{prop}
\begin{proof}
    By Lemma \ref{Lemma existence of liftings to enveloping algebras}, we get a unique map of associative $K$-algebras:
    \begin{equation*}
        \psi:U(L_1)\otimes_KU(L_2)\rightarrow U(M).
    \end{equation*}
Endowing the left hand side with the tensor filtration, it follows that this is a map of filtered $K$-algebras. In particular, we get a map: 
    \begin{equation*}
        \theta:\mathfrak{L}:= L_1\otimes_K B\oplus A\otimes_K L_2 \rightarrow M.
    \end{equation*}
By Lemma \ref{Lemma universal enveloping algebra and uncompleted products}, we have a canonical isomorphism $U(\mathfrak{L})\rightarrow U(L_1)\otimes_KU(L_2)$. Thus, $\theta$ is a $A\otimes_KB$-linear map of $K$-Lie algebras which satisfies that $\tau_1\times \tau_2=\tau_M\circ \theta$.
As $M$ is coherent, it has a unique topology as a finite $A\widehat{\otimes}_KB$-module. As this topology is complete, it follows that $\theta$ extends uniquely to a map
$\theta:L_1\times L_2\rightarrow M$, which is a morphism of $(K,A\widehat{\otimes}_KB)$-Lie algebras.
\end{proof}
Let $X,Y$ be smooth separated rigid analytic varieties, and consider Lie algebroids $\mathscr{L}_X$,$\mathscr{L}_Y$ on $X$ and $Y$ respectively. Consider the following coherent module:
\begin{equation*}
\mathscr{L}_X\times\mathscr{L}_Y=p_X^*\mathscr{L}_X\oplus p_Y^*\mathscr{L}_Y.
\end{equation*}
Using the above, we can  show that $ \mathscr{L}_X\times\mathscr{L}_Y$ is also a Lie algebroid on $X\times Y$.\\ 
Indeed, as $\mathscr{L}_X$ is a Lie algebroid, we can choose an admissible affinoid cover $\mathfrak{U}=\left( U_i\right)_{i\in I}$ such that for all $i\in I$ there is some affine formal model $\mathcal{A}_i$ of $\OX_X(U_i)$ satisfying that $\mathscr{L}_X(U_i)$ admits a finite free $\mathcal{A}_i$-Lie lattice.
As we assumed that $X$ is separated, we may also freely assume that $\mathfrak{U}$ is closed under finite intersections. We let  $\mathfrak{V}=\left( V_j\right)_{j\in J}$ be an admissible affinoid cover of $Y$ satisfying the analogous conditions for $\mathscr{L}_Y$. Thus, we get an admissible affinoid cover $\mathfrak{U\times V}=\left(U_i\times V_j \right)_{i\in I,j\in J}$ of $X\times Y$ which is closed under finite intersections.\bigskip

Let $\mathscr{L}_X\times\mathscr{L}_{Y\vert U_i\times V_j}$ be the restriction of  $\mathscr{L}_X\times\mathscr{L}_Y$ to  $U_i\times V_j$.
By construction, we have the following identity:
\begin{equation*}
    \mathscr{L}_X\times\mathscr{L}_Y(U_i\times V_j)= \mathscr{L}_X(U_i)\times\mathscr{L}_Y(V_j),
\end{equation*}
and  by Theorem \ref{prop product of Lie rinehart algebras} this has a canonical $(K,\OX_X(U_i\times V_j))$-Lie algebra structure.\\ 
As $U_i\times V_j$ is an affinoid space, \cite[Lemma 9.2]{ardakov2019} shows that  the $(K,\OX_X(U_i\times V_j))$-Lie algebra structure on $\mathscr{L}_X\times\mathscr{L}_Y(U_i\times V_j)$ extends uniquely to a Lie algebroid structure on $\mathscr{L}_X\times\mathscr{L}_{Y\vert U_i\times V_j}$. In order to show that all these Lie algebroids are compatible, we just need to show that they agree on the intersections of elements in  $\mathfrak{U\times V}$. For any, $i_1,i_2\in I$ and $j_1,j_2\in J$,
let $U_{i_1i_2}=U_{i_1}\cap U_{i_2}$, and $V_{j_1j_2}=V_{j_1}\cap V_{j_2}$.
We then have:  
\begin{equation*}
    U_{i_1}\times V_{j_1}\cap U_{i_2}\times V_{j_2}=U_{i_1i_2}\times V_{j_1j_2}.
\end{equation*}
For simplicity, let $W=U_{i_1i_2}\times V_{j_1j_2}$.
We have the following identities of $\OX_X(W)$-modules:
\begin{equation*}
\mathscr{L}_X\times\mathscr{L}_{Y\vert U_{i_{1}}
\times V_{j_{1}}}(W)=\mathscr{L}_X\times\mathscr{L}_{Y\vert W}(W)=\mathscr{L}_X\times\mathscr{L}_{Y\vert U_{i_{2}}
\times V_{j_{2}}}(W).    
\end{equation*}
Again by \cite[Lemma 9.2]{ardakov2019}, it suffices to show that these are, in fact, identifications of $(K,\OX_X(W))$-Lie algebras. However, by equation $(\ref{equation Lie bracket on a product})$, the brackets and anchor maps of these algebras are uniquely determined by the Lie-Rinehart algebra structures of $\mathscr{L}_X(U_{i_1i_2})$ and $\mathscr{L}_Y(V_{j_1j_2})$, so they must agree.
\begin{teo}\label{teo product of Lie algebroids}
    Let $X$, $Y$ be smooth and separated rigid analytic varieties, with Lie algebroids $\mathscr{L}_X$, $\mathscr{L}_Y$. There is a unique lie algebroid $\mathscr{L}_X\times\mathscr{L}_Y$ on $X\times Y$ satisfying the following properties: 
    \begin{enumerate}[label=(\roman*)]
        \item For all affinoid subspaces $U\subset X$, 
    and $V\subset Y$ such that $\mathscr{L}_X(U)$, $\mathscr{L}_Y(V)$ are free modules  over $\OX_X(U)$ and $\OX_X(V)$ respectively, we have an identification of $(K,\OX_X(U\times V))$-Lie algebras:
    \begin{equation*}
        \mathscr{L}_X\times\mathscr{L}_Y(U\times V)=\mathscr{L}_X(U)\times\mathscr{L}_Y(V).
    \end{equation*}
    \item The canonical morphisms:
    \begin{equation*}
        i_X:\mathscr{L}_X\rightarrow p_{X*} \left(\mathscr{L}_X\times\mathscr{L}_Y\right), \quad i_Y:\mathscr{L}_Y\rightarrow p_{Y*}\left( \mathscr{L}_X\times\mathscr{L}_Y\right),
    \end{equation*}
   are injective $\OX_X$-linear (resp. $\OX_Y$-linear) maps  of sheaves of  $K$-Lie algebras which satisfy: 
   \begin{equation*}
       \tau_X=p_{X*}\left(\tau_X\times\tau_Y\right)\circ i_X, \quad \tau_Y=p_{Y*}\left(\tau_X\times\tau_Y\right)\circ i_Y.
   \end{equation*}
   \item Let $\mathscr{M}$ be a Lie algebroid on $X\times Y$, and assume that there are maps:
    \begin{equation*}
        \theta_X:\mathscr{L}_X\rightarrow p_{X*}\mathscr{M}, \quad\theta_Y:\mathscr{L}_Y\rightarrow p_{Y*}\mathscr{M},
    \end{equation*}
    which satisfy  conditions analogous to those in statement $(ii)$. Then there is a unique morphism of Lie algebroids $\theta:\mathscr{L}_X\times\mathscr{L}_Y\rightarrow \mathscr{M}$,
    which satisfies:
    \begin{equation*}
        \theta_X=p_{X*}\theta\circ i_X, \quad \theta_Y=p_{Y*}\theta\circ i_Y.
    \end{equation*}
    \end{enumerate}
We call $\mathscr{L}_X\times\mathscr{L}_Y$ the product of $\mathscr{L}_X$ and $\mathscr{L}_Y$, and we will write $\mathscr{L}_X^2:=\mathscr{L}_X\times \mathscr{L}_X$.
\end{teo}
\begin{obs}
    Notice that if $\mathscr{L}=\mathcal{T}_{X/K}$, then $\mathcal{T}_{X/K}^2=\mathcal{T}_{X^2/K}$.
\end{obs}
Perhaps the most interesting case is $Y=X$ and $\mathscr{L}_X=\mathscr{L}_Y=\mathcal{T}_{X/K}$. As $X$ is separated, the diagonal map $\Delta:X\rightarrow X^2$ being a closed immersion, and we may apply the version of Kashiwara's equivalence developed in \cite{ardakov2015d} to study the relation between the categories of co-admissible modules over $\wideparen{\D}_X$ and the category of co-admissible modules of $\wideparen{\D}_{X^2}$ supported on the diagonal. 
\subsection{Universal enveloping algebras of products of Lie-Rinehart algebras}\label{Section universal enveloping algebras and products}
In this section, we will study the relation between the tensor product of the universal enveloping algebras of two finite-free Lie-Rinehart algebras $L_1,L_2$ defined over smooth affinoid algebras $A$ and $B$, and the universal enveloping algebra of the product $L_1\times L_2$. in particular, we will show that:
\begin{equation*}
    U(L_1\times L_2)=\varinjlim \widehat{F}_s\left(U(L_1)\otimes_K U(L_2) \right),
\end{equation*}
where $\widehat{F}_s\left(U(L_1)\otimes_K U(L_2) \right)$ denotes the completion of $F_s\left(U(L_1)\otimes_K U(L_2) \right)$ with respect to its topology as a finite $A\widehat{\otimes}_KB$-module. As in the previous section, this reinforces the intuition that $U(L_1\times L_2)$ is the smallest filtered $K$-algebra which contains $A\widehat{\otimes}_KB$ and $U(L_1)\otimes_K U(L_2)$.\bigskip

As in the beginning of the previous section, we let $X=\Sp(A)$, $Y=\Sp(B)$ be smooth affinoid $K$-varieties, and choose affine formal models $\mathfrak{X}=\Spf(\mathcal{A})$, and $\mathfrak{Y}=\Spf(\mathcal{B})$ of $X$ and $Y$ respectively. We let $C=A\widehat{\otimes}_KB$, and $\mathcal{C}=\mathcal{A}\widehat{\otimes}_{\mathcal{R}}\mathcal{B}$,  and we choose  $\mathfrak{X}\times\mathfrak{Y}=\Spf(\mathcal{C})$ as affine formal model of $X\times Y=\Sp(C)$. As before, we regard $A,B$, and $C$ as Banach $K$-algebras with respect to the norm induced by the associated formal models. Fix a pair of finite-free Lie-Rinehart algebras $L_1$, $L_2$ over $A$ and $B$ respectively, and let $\mathcal{L}_1$ and $\mathcal{L}_2$ be a finite free $(\mathcal{R},\mathcal{A})$-Lie lattice and a  finite-free $(\mathcal{R},\mathcal{B})$-Lie lattice respectively.
\begin{Lemma}
    We have a canonical map of filtered $K$-algebras:
    \begin{equation*}
        j:U(L_1)\otimes_K U(L_2)\rightarrow U(L_1\times L_2).
    \end{equation*}
\end{Lemma}
\begin{proof} 
This was part of the proof of Proposition \ref{prop universal property of product}.
\end{proof}

Recall that for any $n\geq 0$, the $\mathcal{A}$-module $\pi^n\mathcal{L}_1$ is a smooth $\mathcal{A}$-Lie lattice, and the analogous holds for $\pi^n\mathcal{L}_2$. This leads us to the following definition:
\begin{defi}
   Let $n,m\geq 0$ be non-negative integers. We define the $\mathcal{R}$-algebra:
   \begin{equation*}   V(\mathcal{L}_1,\mathcal{L}_2)_{n,m}:=U(\pi^n\mathcal{L}_1)\otimes_{\mathcal{R}}U(\pi^m\mathcal{L}_2),
   \end{equation*}
   and regard it as a filtered $\mathcal{R}$-algebra with respect to the tensor filtration.
\end{defi}
Notice that for each $n,m\geq 0$ we have a canonical isomorphism of filtered $K$-algebras:
\begin{equation*}
    V(\mathcal{L}_1,\mathcal{L}_2)_{n,m}\otimes_{\mathcal{R}} K\rightarrow U(L_1)\otimes_{K}U(L_2).
\end{equation*}

\begin{coro}\label{coro decomposition filtrations of tensors formal level}
Let $n,m\geq 0$ be non-negative integers. The following holds:
\begin{enumerate}[label=(\roman*)]
    \item There is an isomorphism of graded $\mathcal{A}\otimes_{\mathcal{R}}\mathcal{B}$-algebras:
    \begin{equation*}
      \operatorname{Sym}_{\mathcal{A}}(\pi^n\mathcal{L}_1)\otimes_{\mathcal{R}}\operatorname{Sym}_{\mathcal{B}}(\pi^m\mathcal{L}_2)\rightarrow \gr_F(V(\mathcal{L}_1,\mathcal{L}_2)_{n,m}).
    \end{equation*}
    \item For any $s\geq 0$, we have a decomposition of $\mathcal{A}\otimes_{\mathcal{R}} \mathcal{B}$-modules:
\begin{equation*}
    F_s V(\mathcal{L}_1,\mathcal{L}_2)_{n,m}= \bigoplus_{i=0}^s F_i U(\pi^n\mathcal{L}_1)\otimes_{\mathcal{R}} F_{s-i} U(\pi^m\mathcal{L}_2).
\end{equation*}
In particular, $F_s V(\mathcal{L}_1,\mathcal{L}_2)_{n,m}$ is a finite-free $\mathcal{A}\otimes_{\mathcal{R}}\mathcal{B}$ module.
\end{enumerate}
\end{coro}
\begin{proof}
The proof of Proposition \ref{PBW version for enveloping algebras} carries over verbatim.
\end{proof}
For every $i\geq 0$ we have that  $F_i U(L_1)$ is a finitely generated $A$-module. Hence, we get a unique topology on $F_i U(L_1)$ making it a topological $A$-module. Similarly, $F_j U(L_2)$ has a unique complete topology making it a topological $B$-module for every $j\geq 0$. From now on, we will regard the filtration terms $F_i U(L_1)$ and $F_j U(L_2)$ as topological modules with respect to this topology.\bigskip

Fix some non-negative integers $n,m\geq 0$. The $\mathcal{A}$-module $F_i U(\pi^n\mathcal{L}_1)$ is an open lattice of $F_i U(L_1)$. Let  $q_{i,n}$  be its associated gauge norm (cf. \cite[Chapter I.2]{schneider2013nonarchimedean}). We may then  regard $F_i U(L_1)$, together with the gauge $q_{i,n}$, as a complete, normed $A$-module. Analogously, we regard 
$F_j U(L_2)$ as a normed $B$-module with respect to the gauge $q_{j,m}$ induced by $F_j U(\pi^m\mathcal{L}_2)$. By Corollary \ref{coro decomposition filtration of tensor}, for each $s\geq 0$ we have the following identity of $A\otimes_KB$-modules:
\begin{equation*}
    F_s \left(U(L_1)\otimes_KU(L_2)\right)= \bigoplus_{i=0}^s F_i U(L_1)\otimes_K F_{s-i} U(L_2).
\end{equation*}
Consider the following $\mathbb{R}$-valued function:
\begin{equation*}
    q_{s,n,m}(\oplus_{i=0}^sa_i)=\operatorname{max}\{ q_{i,n}\otimes q_{s-i,m}(a_i) \textnormal{ : } 0\leq i\leq s\}.
\end{equation*}
A quick calculation shows that $q_{s,n,m}$ is a $K$-norm on $F_s \left(U(L_1)\otimes_KU(L_2)\right)$, making it a normed $A\otimes_KB$-module. Furthermore, the norm on $q_{s,n,m}$ is the gauge of the following lattice:
\begin{equation*}
     F_sV(\mathcal{L}_1,\mathcal{L}_2)_{n,m}= \bigoplus_{i=0}^s F_i U(\pi^n\mathcal{L}_1)\otimes_{\mathcal{R}} F_{s-i} U(\pi^m\mathcal{L}_2).
\end{equation*}
Let $\widehat{F}_sV(\mathcal{L}_1,\mathcal{L}_2)_{n,m}$ be the $\pi$-adic completion of $F_sV(\mathcal{L}_1,\mathcal{L}_2)_{n,m}$, and let :
\begin{equation*}
    \widehat{F}_s \left(U(L_1)\otimes_KU(L_2)\right),
\end{equation*}
be the completion of $U(L_1)\otimes_KU(L_2)$ with respect to $q_{s,n,m}$. We have the following identity of Banach $K$-vector spaces:
\begin{equation*}
    \widehat{F}_s \left(U(L_1)\otimes_KU(L_2)\right)=\widehat{F}_sV(\mathcal{L}_1,\mathcal{L}_2)_{n,m}\otimes_{\mathcal{R}}K.
\end{equation*}

For each $s\geq 0$, we have inclusions $F_sV(\mathcal{L}_1,\mathcal{L}_2)_{n,m}\subset F_{s+1}V(\mathcal{L}_1,\mathcal{L}_2)_{n,m}$. Hence, we get the following inductive limit of complete $\mathcal{R}$-modules:
\begin{equation*} \mathcal{A}\widehat{\otimes}_\mathcal{R}\mathcal{B}\rightarrow \widehat{F}_1V(\mathcal{L}_1,\mathcal{L}_2)_{n,m}\rightarrow \cdots \rightarrow \widehat{F}_sV(\mathcal{L}_1,\mathcal{L}_2)_{n,m} \rightarrow \cdots, 
\end{equation*}
We can now make the following definition:
\begin{defi}
Let $n,m\geq 0$. We define:
\begin{equation*}
    \Tilde{V}(\mathcal{L}_1,\mathcal{L}_2)_{n,m}=\varinjlim \widehat{F}_sV(\mathcal{L}_1,\mathcal{L}_2)_{n,m}.
\end{equation*}
We regard $\Tilde{V}(\mathcal{L}_1,\mathcal{L}_2)_{n,m}$ as a topological $\mathcal{A}\widehat{\otimes}_\mathcal{R}\mathcal{B}$-module with respect to the colimit topology.
\end{defi}
We will now show some basic properties of $\Tilde{V}(\mathcal{L}_1,\mathcal{L}_2)_{n,m}$:
\begin{prop}\label{prop denseness and basic properties}
Let $n,m\geq 0$. The following hold:
 \begin{enumerate}[label=(\roman*)]
 \item The map $\widehat{F}_sV(\mathcal{L}_1,\mathcal{L}_2)_{n,m} \rightarrow \widehat{F}_{s+1}V(\mathcal{L}_1,\mathcal{L}_2)_{n,m}$ is a homeomorphism onto its image.
 \item $\widehat{F}_1V(\mathcal{L}_1,\mathcal{L}_2)_{n,m} = \left(\mathcal{A}\widehat{\otimes}_{\mathcal{R}}\mathcal{B}\right)\oplus\left(\pi^n\mathcal{L}_1\widehat{\otimes}_\mathcal{R}\mathcal{B}\oplus \mathcal{A}\widehat{\otimes}_\mathcal{R}\pi^m\mathcal{L}_2\right)$.
\item The map $V(\mathcal{L}_1,\mathcal{L}_2)_{n,m}\rightarrow  \Tilde{V}(\mathcal{L}_1,\mathcal{L}_2)_{n,m}$ is injective and has dense image.
 \end{enumerate}
\end{prop}
\begin{proof}
Claim $(i)$ follows by the fact that  $F_sV(\mathcal{L}_1,\mathcal{L}_2)_{n,m}$ and $F_{s+1}V(\mathcal{L}_1,\mathcal{L}_2)_{n,m}$ are finite-free $\mathcal{A}\otimes_{\mathcal{R}}\mathcal{B}$-modules, and $\pi$-adic completion commutes with finite direct sums. For $(ii)$, notice that by the PBW Theorem we have: 
\begin{equation*} 
F_1V(\mathcal{L}_1,\mathcal{L}_2)_{n,m} = \mathcal{A}\otimes_{\mathcal{R}}\mathcal{B}\oplus \pi^n\mathcal{L}_1\otimes_\mathcal{R}\mathcal{B}\oplus \mathcal{A}\otimes_\mathcal{R}\pi^m\mathcal{L}_2.
\end{equation*}
Again, using the fact that $\pi$-adic completion commutes with finite direct sums yields the result.
For statement $(iii)$, it is enough to show that  the map $F_sV(\mathcal{L}_1,\mathcal{L}_2)_{n,m}\rightarrow  \widehat{F}_sV(\mathcal{L}_1,\mathcal{L}_2)_{n,m}$ is injective and has dense image for each $s\geq 0$. As this map is simply the map into the $\pi$-adic completion, it is enough to show that the $\pi$-adic topology on $F_sV(\mathcal{L}_1,\mathcal{L}_2)_{n,m}$ is separated. The $\pi$-adic topology on  $F_sV(\mathcal{L}_1,\mathcal{L}_2)_{n,m}$agrees with the subspace topology with respect to $F_s\left(U(L_1)\otimes_K U(L_2)\right)$. As this is the direct sum of finitely many tensor products of two Banach spaces, it is separated. Hence, the $\pi$-adic topology on  $F_sV(\mathcal{L}_1,\mathcal{L}_2)_{n,m}$ is separated.
\end{proof}
We may use this lemma to add some extra structure to $\Tilde{V}(\mathcal{L}_1,\mathcal{L}_2)_{n,m}$:
\begin{prop}\label{prop algebra structure on inj limit of completions of enveloping algebra}
The following hold:
 \begin{enumerate}[label=(\roman*)]
 \item There is a unique continuous $\mathcal{R}$-algebra structure on $\Tilde{V}(\mathcal{L}_1,\mathcal{L}_2)_{n,m}$  such that:
 \begin{equation*}
    V(\mathcal{L}_1,\mathcal{L}_2)_{n,m}\rightarrow \Tilde{V}(\mathcal{L}_1,\mathcal{L}_2)_{n,m},
 \end{equation*}
is a $\mathcal{R}$-algebra homomorphism. 
 \item $\Tilde{V}(\mathcal{L}_1,\mathcal{L}_2)_{n,m}$ is generated as an algebra by $\mathcal{A}\widehat{\otimes}_{\mathcal{R}}\mathcal{B}$, $\pi^n\mathcal{L}_1\otimes 1$, and $1\otimes\pi^m\mathcal{L}_2$.
 \item  The filtration $F_\bullet \Tilde{V}(\mathcal{L}_1,\mathcal{L}_2)_{n,m}$ given by:
 \begin{equation*}
     F_s\Tilde{V}(\mathcal{L}_1,\mathcal{L}_2)_{n,m}:=\widehat{F}_{s}V(\mathcal{L}_1,\mathcal{L}_2)_{n,m}, \textnormal{ for } s\geq 0,
 \end{equation*}
 is a positive and exhaustive $\mathcal{R}$-algebra filtration on $\Tilde{V}(\mathcal{L}_1,\mathcal{L}_2)_{n,m}$.
 \end{enumerate}
\end{prop}
\begin{proof}
Choose non-negative integers $r,s\geq 0$. As $ V(\mathcal{L}_1,\mathcal{L}_2)_{n,m}$ is a filtered algebra,  multiplication fits into the following diagram:
\begin{equation*} 
F_s V(\mathcal{L}_1,\mathcal{L}_2)_{n,m}\times F_r V(\mathcal{L}_1,\mathcal{L}_2)_{n,m}\rightarrow F_{s+r} V(\mathcal{L}_1,\mathcal{L}_2)_{n,m}.
\end{equation*}
As multiplication is $\mathcal{R}$-linear, we may apply $\pi$-adic completion to this map. Taking colimits, we get a map $\Tilde{V}(\mathcal{L}_1,\mathcal{L}_2)_{n,m}\times \Tilde{V}(\mathcal{L}_1,\mathcal{L}_2)_{n,m}\rightarrow\Tilde{V}(\mathcal{L}_1,\mathcal{L}_2)_{n,m}$. This is a $\mathcal{R}$-linear multiplication map extending the product on $V(\mathcal{L}_1,\mathcal{L}_2)_{n,m}$. As the product is continuous and $V(\mathcal{L}_1,\mathcal{L}_2)_{n,m}$ is dense in $\Tilde{V}(\mathcal{L}_1,\mathcal{L}_2)_{n,m}$, it follows that $\Tilde{V}(\mathcal{L}_1,\mathcal{L}_2)_{n,m}$ is a $\mathcal{R}$-algebra. Furthermore, by construction of the product, $F_{\bullet}\Tilde{V}(\mathcal{L}_1,\mathcal{L}_2)_{n,m}$ is a positive and exhaustive $\mathcal{R}$-algebra filtration. This shows $(i)$ and $(iii)$.\\
We will now show claim $(ii)$. Choose any $x\in  \Tilde{V}(\mathcal{L}_1,\mathcal{L}_2)_{n,m}$. There is a minimal $s\geq 0$ such that $x\in F_s\Tilde{V}(\mathcal{L}_1,\mathcal{L}_2)_{n,m}$. By definition of $\pi$-adic completion, there is a sequence $(z_r)_{r\geq 0}\subset F_{s}V(\mathcal{L}_1,\mathcal{L}_2)_{n,m}$ such that we have $x=\sum_{r\geq 0}\pi^rz_r$. Let $x_1,\cdots x_t$ be a basis of $\mathcal{L}_1$ as an $\mathcal{A}$-module, and let $y_1,\cdots,y_l$ be a basis of $\mathcal{L}_2$ as a $\mathcal{B}$-module. For any multi-index $\lambda =(\lambda_1,\cdots,\lambda_{t+l})\in \mathbb{N}^{t+l}$, we define $\vert \lambda \vert = \sum_{i=1}^{l+t}\lambda_i$. We will write:
\begin{equation*}
    (\pi^nx\otimes \pi^my)^{\lambda}=(\pi^nx_1)^{\lambda_1}\cdots (\pi^nx_t)^{\lambda_t}\otimes (\pi^my_1)^{\lambda_{t+1}}\cdots (\pi^my_l)^{\lambda_{t+l}}.
\end{equation*}
For simplicity, we will assume $n=m=1$, the general case is analogous. By the PBW Theorem, we may express each of the $z_r$ in the following way:
\begin{equation*}
    z_r=\sum_{\vert \lambda\vert \leq s} a^r_{\lambda}  (x\otimes y)^{\lambda},
\end{equation*}
where all  $a^r_{\lambda}\in\mathcal{A}\otimes_{\mathcal{R}}\mathcal{B}$. Furthermore, this expression is unique. We have the following identities:
\begin{multline*}
    x=\sum_{r\geq 0}\pi^rz_r=\sum_{r\geq 0}\pi^r\left(\sum_{\vert \lambda\vert \leq s} a^r_{\lambda} (x\otimes y)^{\lambda}\right)\\
    =
    \sum_{\vert \lambda\vert \leq s}\left(\sum_{r\geq 0}\pi^r a^r_{\lambda} \right)(x\otimes y)^{\lambda}=\sum_{\vert \lambda\vert \leq s}c_{\lambda} (x\otimes y)^{\lambda},
\end{multline*}
where the second identity follows by the fact that there are only finitely many different multi-indices $\lambda$ such that $\vert\lambda\vert \leq s$. As $\mathcal{A}\widehat{\otimes}_{\mathcal{R}}\mathcal{B}$ is $\pi$-adically complete, it follows that it contains each of the $c_{\lambda}$. The expression on the right hand side is an element in the subalgebra of $\Tilde{V}(\mathcal{L}_1,\mathcal{L}_2)_{n,m}$ generated by $\mathcal{A}\widehat{\otimes}_{\mathcal{R}}\mathcal{B}$, $\pi^n\mathcal{L}_1$, and $\pi^m\mathcal{L}_2$, as we wanted.
\end{proof}
\begin{coro}\label{coro universal enveloping algebras of products and colimits}
Consider the topological $K$-algebra:
\begin{equation*}
    \Tilde{V}(L_1,L_2):=\Tilde{V}(\mathcal{L}_1,\mathcal{L}_2)_{n,m}\otimes_{\mathcal{R}}K.
\end{equation*}
The following hold:
\begin{enumerate}[label=(\roman*)]
    \item The canonical map $U(L_1)\otimes_KU(L_2)\rightarrow\Tilde{V}(L_1,L_2)$ is a map of $K$-algebras with dense image. 
    \item The map $U(L_1)\otimes_KU(L_2)\rightarrow U(L_1\times L_2)$ extends to an isomorphism of filtered $K$-algebras:
    \begin{equation*}
        \Tilde{V}(L_1,L_2)\rightarrow U(L_1\times L_2).
    \end{equation*}
    In particular, we have $U(L_1\times L_2)=\varinjlim \widehat{F}_s\left(U(L_1)\otimes_KU(L_2) \right)$.
\end{enumerate}
\end{coro}
\begin{proof}
    Statement $(i)$ is clear. For statement $(ii)$, recall that for each $n,m\geq 0$, the product $L_1\times L_2$ admits the following smooth $(\mathcal{R},\mathcal{C})$-Lie lattice:
    \begin{equation*}
        \pi^n\mathcal{L}_1\times \pi^m\mathcal{L}_2=\pi^n\mathcal{L}_1\widehat{\otimes}_{\mathcal{R}}\mathcal{B}\oplus \mathcal{A}\widehat{\otimes}_{\mathcal{R}}\pi^m\mathcal{L}_2.
    \end{equation*}
The map of filtered $K$-algebras $U(L_1)\otimes_KU(L_2)\rightarrow U(L_1\times L_2)$ 
maps $\pi^n\mathcal{L}_1\otimes_{\mathcal{R}}\mathcal{B}\oplus \mathcal{A}\otimes_{\mathcal{R}}\pi^m\mathcal{L}_2$ to $\pi^n\mathcal{L}_1\times \pi^m\mathcal{L}_2$. Hence, it induces a map of filtered $\mathcal{R}$-algebras:
\begin{equation*}
    \psi:V(\mathcal{L}_1,\mathcal{L}_2)_{n,m}\rightarrow U(\pi^n\mathcal{L}_1\times \pi^m\mathcal{L}_2).
\end{equation*}
For each $s\geq 0$, the $\mathcal{C}$-module $F_sU(\pi^n\mathcal{L}_1\times \pi^m\mathcal{L}_2)$ is finite. In particular, it is $\pi$-adically complete. Thus, taking $\pi$-adic completions, we get an extension $\widehat{F}_sV(\mathcal{L}_1,\mathcal{L}_2)_{n,m}\rightarrow F_sU(\pi^n\mathcal{L}_1\times \pi^m\mathcal{L}_2)$. Notice that this is in fact an isomorphism. Indeed, the PBW theorems \ref{coro decomposition filtrations of tensors formal level}
and \cite[Theorem 3.1]{rinehart1963differential}, give explicit basis as $\mathcal{C}$-modules of both modules, and these get identified by the map. Taking colimits, we get an isomorphism of filtered $\mathcal{C}$-modules:
\begin{equation*}
    \widehat{\psi}:\Tilde{V}(\mathcal{L}_1,\mathcal{L}_2)_{n,m}\rightarrow U(\pi^n\mathcal{L}_1\times \pi^m\mathcal{L}_2),
\end{equation*}
which extends the map $V(\mathcal{L}_1,\mathcal{L}_2)_{n,m}\rightarrow U(\pi^n\mathcal{L}_1\times \pi^m\mathcal{L}_2)$. Thus, we only need to show that 
$\widehat{\psi}$ is a morphism of algebras. However, as this is a morphism of filtered $\mathcal{C}$-modules, it suffices to show that for each $s\geq 0$ we have a commutative diagram:
\begin{equation*}
\begin{tikzcd}
{F_s\Tilde{V}(\mathcal{L}_1,\mathcal{L}_2)_{n,m}\times F_r\Tilde{V}(\mathcal{L}_1,\mathcal{L}_2)_{n,m}} \arrow[d, "\widehat{\psi}\times \widehat{\psi}"] \arrow[r, "m"] & {F_{r+s}\Tilde{V}(\mathcal{L}_1,\mathcal{L}_2)_{n,m}} \arrow[d,"\widehat{\psi}" ] \\
F_sU(\pi^n\mathcal{L}_1\times \pi^m\mathcal{L}_2)\times F_rU(\pi^n\mathcal{L}_1\times \pi^m\mathcal{L}_2) \arrow[r, "m"]                                   & F_{r+s}U(\pi^n\mathcal{L}_1\times \pi^m\mathcal{L}_2)               
\end{tikzcd}
\end{equation*}
This follows by the fact that the maps $F_sV(\mathcal{L}_1,\mathcal{L}_2)_{n,m}\rightarrow F_s\Tilde{V}(\mathcal{L}_1,\mathcal{L}_2)_{n,m}$ have dense image.
\end{proof}
\begin{coro}\label{coro colimit topology integral level}
 For each $n,m\geq 0$ there is a canonical isomorphism of filtered $\mathcal{R}$-algebras:
 \begin{equation*}
     U(\pi^n\mathcal{L}_1\times \pi^m\mathcal{L}_2)=\Tilde{V}(\mathcal{L}_1,\mathcal{L}_2)_{n,m}=\varinjlim_{s\geq 0} \widehat{F}_s\left(U(\pi^n\mathcal{L}_1)\otimes_{\mathcal{R}}U(\mathcal{L}_2)\right).
 \end{equation*}
\end{coro}
\begin{proof}
    This follows by the proof of Corollary \ref{coro universal enveloping algebras of products and colimits}.
\end{proof}
    Notice that the previous Corollary endows $U(L_1\times L_2)$ with a locally convex topology. This topology is in fact the finest locally convex topology of $U(L_1\times L_2)$ as a $A\widehat{\otimes}_KB$-module. In particular, as $U(L_1\times L_2)$ is an inductive limit of finite $A\widehat{\otimes}_KB$-modules. Hence,
    given a topological $A\widehat{\otimes}_KB$-module $M$, it follows that every $A\widehat{\otimes}_KB$-linear map  $U(L_1\times L_2)\rightarrow M$ is continuous. Thus, the topology on $U(L_1\times L_2)$ is independent of the choice of Lie lattices.
\subsection{\texorpdfstring{Fréchet Completions}{}}
Let $X$ and $Y$ be smooth and separated rigid varieties with Lie algebroids $\mathscr{L}_X$ and $\mathscr{L}_Y$. In this section we will show that there is a canonical isomorphism of sheaves of complete bornological spaces:
\begin{multline*}
    \wideparen{U}(\mathscr{L}_X\times \mathscr{L}_Y) \simeq\wideparen{U}(\mathscr{L}_X) \overrightarrow{\boxtimes}\wideparen{U}(\mathscr{L}_Y)\\
    = \OX_{X\times Y}\overrightarrow{\otimes}^{\mathbb{L}}_{p_X^{-1}\OX_X\overrightarrow{\otimes}_Kp_Y^{-1}\OX_Y} \left(p_X^{-1}\wideparen{U}(\mathscr{L}_X)\overrightarrow{\otimes}_Kp_Y^{-1}\wideparen{U}(\mathscr{L}_Y)\right).
\end{multline*}
We start by recalling the following well-known result:
\begin{teo}[{\cite[Corollary 9.4.2]{bosch2014lectures}}]\label{teo flat maps are open}
    Let $f:X\rightarrow Y$ be a flat morphism of quasi-compact quasi-separated rigid spaces. Then $f$ sends admissible open subspaces to admissible open subspaces, and admissible covers to admissible covers.
\end{teo}
We will most of the times apply this theorem to the projections $p_X:X\times Y\rightarrow X$, and 
$p_Y:X\times Y\rightarrow Y$, which are flat morphisms of quasi-separated rigid spaces. A first application of this theorem is the following lemma:
\begin{Lemma}\label{lemma structure of pullback}
The sheaves $p_{X}^{-1}\wideparen{U}(\mathscr{L}_X)$ and $p_{Y}^{-1}\wideparen{U}(\mathscr{L}_Y)$  are sheaves of complete bornological $K$-algebras on $X\times Y$.   
\end{Lemma}
\begin{proof}
 As the pullback of sheaves is strict symmetric monoidal with respect to $\overrightarrow{\otimes}_K$, it is clear that  $p_{X}^{-1}\wideparen{U}(\mathscr{L}_X)$ and $p_{Y}^{-1}\wideparen{U}(\mathscr{L}_Y)$  are sheaves of Ind-Banach $K$-algebras on $X\times Y$. Thus, we need to see that they take values on complete bornological spaces. We only do this for $p_{X}^{-1}\wideparen{U}(\mathscr{L}_X)$, the other case being analogous. By assumption, $X$ and $Y$ are separated. In particular, the intersection of two affinoid spaces in $X\times Y$ is still affinoid. This, fact together with the fact that the dissection functor:
 \begin{equation*}
\operatorname{diss}:\widehat{\mathcal{B}}c_K\rightarrow \Indban, 
 \end{equation*}
 commutes with inverse limits, implies that it suffices to show that the sections of $p_{X}^{-1}\wideparen{U}(\mathscr{L}_X)$ on small enough affinoid spaces are complete bornological spaces.\\
Thus, we may assume that $X$ and $Y$ are affinoid spaces. As $X\times Y$ is quasi-separated, we can define the following site: we let $(X\times Y)_w$ be the site with objects the quasi-compact subspaces of $X\times Y$ and covers being the finite covers. As every admissible cover of $X\times Y$ can be refined by such a cover, there is an equivalence of quasi-abelian categories:
\begin{equation}
    \operatorname{Shv}((X\times Y)_w,\Indban)\rightarrow \operatorname{Shv}(X\times Y,\Indban).
\end{equation}
In particular, it is enough to study the restriction of $p_{X}^{-1}\wideparen{U}(\mathscr{L}_X)$ to $(X\times Y)_w$.\bigskip

Notice that the projection $p_X:X\times Y\rightarrow Y$ is a flat morphism of quasi-separated rigid spaces, and that every object in $(X\times Y)_w$ is quasi-compact. In particular, it follows by Theorem \ref{teo flat maps are open} that the pullback of presheaves:
\begin{equation*}
   p^{-1}_{X,\operatorname{pre}}\wideparen{U}(\mathscr{L}_X)(V)\mapsto \varinjlim_{p_X(V)\subset U}\wideparen{U}(\mathscr{L}_X)(U)=\wideparen{U}(\mathscr{L}_X)(p_X(V)), \textnormal{ where } V\in (X\times Y)_w.
\end{equation*}
is already a sheaf of Ind-Banach spaces on $(X\times Y)_w$. By construction, $\wideparen{U}(\mathscr{L}_X)(p_X(V))$ is a complete bornological space (in fact, it is a Fréchet space). Hence, $p_x^{-1}\wideparen{U}(\mathscr{L}_X)(V)$ is a complete bornological space, as we wanted to show.
\end{proof}

Recall from Theorem \ref{teo product of Lie algebroids} that there is a canonical morphism of $\OX_X$-modules:
\begin{equation*}
    j_X:\mathscr{L}_X\rightarrow p_{X,*}\left(\mathscr{L}_X\times \mathscr{L}_Y \right).
\end{equation*}
As $\mathscr{L}_X\times \mathscr{L}_Y$ is a coherent $\OX_{X\times Y}$-module, we may regard it as a sheaf of complete bornological $\OX_{X\times Y}$-modules. This implies that $p_{X,*}\left(\mathscr{L}_X\times \mathscr{L}_Y \right)$ is a sheaf of complete bornological $\OX_X$-modules. Furthermore, it is clear that the map $j_X$ is bounded. In particular, it is an $\OX_X$-linear morphism in $\operatorname{Shv}(X,\Indban)$. Hence, applying the pushforward-pullback adjunction for sheaves of Ind-Banach spaces, we have an $p_X^{-1}\OX_X$-linear map of sheaves of Ind-Banach $K$-Lie algebras:
\begin{equation*}
    i_X:p_{X}^{-1}\mathscr{L}_X\rightarrow \mathscr{L}_X\times \mathscr{L}_Y.
\end{equation*}
This map satisfies the following property:
\begin{Lemma}\label{lemma map to enveloping of product}
    The map $i_X:p_{X}^{-1}\mathscr{L}_X\rightarrow \mathscr{L}_X\times \mathscr{L}_Y$ extends uniquely to a morphism of sheaves of complete bornological $K$-algebras:
    \begin{equation*}
        p_{X}^{-1}\wideparen{U}(\mathscr{L}_X)\rightarrow \wideparen{U}(\mathscr{L}_X\times \mathscr{L}_Y).
    \end{equation*}
\end{Lemma}
\begin{proof}
We can freely assume that $X=\Sp(A)$, and $Y=\Sp(B)$ are affinoid spaces, and that $\mathscr{L}_X$ and $\mathscr{L}_Y$ are free. Choose affine formal models $\mathcal{A}$, $\mathcal{B}$, and free Lie lattices $\mathcal{L}_X$, $\mathcal{L}_Y$ of $\mathscr{L}_X(X)$ and $\mathscr{L}_Y(Y)$.\\
By the pushforward-pullback adjunction, it is enough to show that:
\begin{equation*}
   i_X:\mathscr{L}_X\rightarrow p_{X,*}\left(\mathscr{L}_X\times \mathscr{L}_Y\right), 
\end{equation*}
extends uniquely to a morphism of sheaves of complete Ind-Banach $K$-algebras:
\begin{equation*}
    \wideparen{U}(\mathscr{L}_X)\rightarrow p_{X*}\wideparen{U}(\mathscr{L}_X\times \mathscr{L}_Y).
\end{equation*}
Let $U\subset X$ be an affinoid space. Then $\wideparen{U}(\mathscr{L}_X(U))$ and $\wideparen{U}(\mathscr{L}_X\times \mathscr{L}_Y(U\times Y))$ are complete bornological $K$-algebras. Furthermore, by \cite[Corollary 5.19]{bode2021operations},  it follows that $\wideparen{U}(\mathscr{L}_X(U))$ is the completion of $U(\mathscr{L}_X(U))$, and $\wideparen{U}(\mathscr{L}_X\times \mathscr{L}_Y(U\times Y))$ is the completion of $U(\mathscr{L}_X\times \mathscr{L}_Y(U\times Y))$ (with respect to the subspace bornologies). In particular, if the extension exists, it is unique. By the same token, it suffices to show that the canonical map:
\begin{equation*}
    U(i_X):U(\mathscr{L}_X(U))\rightarrow U(\mathscr{L}_X\times \mathscr{L}_Y(U\times Y)),
\end{equation*}
is bounded. We show this for $U=X$, the general case is analogous.\\
A subset $B\subset U(\mathscr{L}_X(X))$ is bounded if and only if for every $n\geq 0$ there is some $m\in\mathbb{Z}$ such that:
\begin{equation*}
    B\subset \pi^mU(\pi^n\mathcal{L}_X(X)).
\end{equation*}
The bounded subsets of $U(\mathscr{L}_X\times \mathscr{L}_Y(X\times Y))$ have an analogous description in terms of the lattices $\pi^n\mathcal{L}_X\times \pi^n\mathcal{L}_Y$. By construction $U(i_X)(\pi^n\mathcal{L}_X)\subset \pi^n\mathcal{L}_X\times \pi^n\mathcal{L}_Y$. Thus, $U(i_X)$ maps bounded subsets to bounded subsets, as wanted. For general $X,Y$, we use the fact that we have an admissible cover of $X\times Y$ by products of affinoid spaces satisfying the above.
\end{proof}
Let us introduce the following useful space:
\begin{defi}[{\cite[Section 5.5]{bode2021operations}}]
Let $s\geq 0$ be an integer and $K\langle x_1,\cdots, x_s \rangle$ be the (bornologification of) the Tate algebra on $s$ variables. We define the following complete bornological space:
\begin{equation*}
    K\{x_1,\cdots, x_s \}:= \varprojlim_n K\langle \pi^nx_1,\cdots, \pi^nx_s\rangle.
\end{equation*}
\end{defi}
Notice that $K\{x_1,\cdots, x_s \}$ is the bornologification of a nuclear Fréchet space. Furthermore, the space $K\{x_1,\cdots, x_s \}$ has a natural structure as complete bornological $K$-algebra. Namely, it is the space of holomorphic functions on $\mathbb{A}^s_K$. However, this algebra structure will rarely be of use, so we just regard  $K\{x_1,\cdots, x_s \}$ as a complete bornological $K$-vector space. Let us start with the following lemmas:
\begin{Lemma}\label{lemma local freeness of algebras over structure sheaf}
Let $X$ be a smooth rigid space with a Lie algebroid $\mathscr{L}$. Assume there are sections $s_1,\dots,s_r\in \Gamma(X,\mathscr{L})$ inducing an isomorphism of $\OX_X$-modules:
\begin{equation*}
    \bigoplus_{i=1}^r\OX_X\rightarrow \mathscr{L}.
\end{equation*}
Then there is an isomorphism of sheaves of complete bornological $\OX_X$-modules:
\begin{equation*}
    \OX_X\overrightarrow{\otimes}_KK\{x_1,\cdots, x_s \}\rightarrow \wideparen{U}(\mathscr{L}),
\end{equation*}
which is the identity on $\OX_X$ and sends $x_i$ to $s_i$ for $1\leq i\leq r$.
\end{Lemma}
\begin{proof}
 This is shown in \cite[
Lemma 6.5]{bode2021operations}.   
\end{proof}
\begin{Lemma}\label{lemma flatness of algebras over structure sheaf}
 Let $X$ be a smooth rigid space with a Lie algebroid $\mathscr{L}$. Then  $\wideparen{U}(\mathscr{L})$ is a strongly flat Ind-Banach $\OX_X$-module.  
\end{Lemma}
\begin{proof}
As flatness is a local property and $\mathscr{L}$ is a Lie algebroid, we can assume that the conditions of Lemma \ref{lemma local freeness of algebras over structure sheaf} hold. In particular, we have:
\begin{equation*}
    \wideparen{U}(\mathscr{L})\simeq \OX_X\overrightarrow{\otimes}_KK\{x_1,\cdots, x_s \},
\end{equation*}
as sheaves of complete bornological $\OX_X$-modules. Thus, it suffices to show that $K\{x_1,\cdots, x_s \}$ is strongly flat in $\Indban$. This is shown in \cite[Section 5.5]{bode2021operations}.
\end{proof}
\begin{prop}\label{prop enveloping of product as product of envelopings}
The morphism of sheaves of  complete bornological $K$-algebras:
\begin{equation*}
   p_X^{-1}\wideparen{U}(\mathscr{L}_X)\overrightarrow{\otimes}_Kp_Y^{-1}\wideparen{U}(\mathscr{L}_Y)\rightarrow \wideparen{U}(\mathscr{L}_X\times \mathscr{L}_Y),
\end{equation*}
extends to an isomorphism of sheaves of complete bornological $\OX_{X\times Y}$-modules:
\begin{equation}\label{equation map isomorphism enveloping of product}
    \wideparen{U}(\mathscr{L}_X) \overrightarrow{\boxtimes}\wideparen{U}(\mathscr{L}_Y)\rightarrow \wideparen{U}(\mathscr{L}_X\times \mathscr{L}_Y),
\end{equation}
which is the identity on $\OX_{X\times Y}$. In particular, for each pair of affinoid open subspaces $U\subset X$, $V\subset Y$  such that $\mathscr{L}_{X\vert U}$ and $\mathscr{L}_{Y\vert V}$ are free, we have:
\begin{equation*}
 \wideparen{U}(\mathscr{L}_X\times \mathscr{L}_Y)(U\times V)=\wideparen{U}(\mathscr{L}_X(U))\overrightarrow{\otimes}_K\wideparen{U}(\mathscr{L}_Y(V)).   
\end{equation*}
\end{prop}
\begin{proof}
First, notice that by Lemma \ref{lemma local freeness of algebras over structure sheaf} the  sheaf $p_X^{-1}\wideparen{U}(\mathscr{L}_X)\overrightarrow{\otimes}_Kp_Y^{-1}\wideparen{U}(\mathscr{L}_Y)$ is strongly flat over $p_X^{-1}\OX_X\overrightarrow{\otimes}_Kp_Y^{-1}\OX_Y$. Thus, the map in equation (\ref{equation map isomorphism enveloping of product}) is a well-defined morphism of Ind-Banach $\OX_{X\times Y}$-modules, the existence of which is a consequence of Lemma \ref{lemma map to enveloping of product}. We need to see it is an isomorphism. As this can be done locally, we may assume that $X=\Sp(A)$, $Y=\Sp(B)$, and that $\mathscr{L}_X$, and $\mathscr{L}_Y$ are free, of ranks $r$ and $s$ respectively. As the conditions of Lemma \ref{lemma local freeness of algebras over structure sheaf} hold, we have isomorphisms of  complete bornological $\OX_X$-modules (resp. $\OX_Y$-modules):
\begin{equation}\label{equation proof Proposition completion of enveloping algebra of product}
 \wideparen{U}(\mathscr{L}_X)\simeq \OX_{X}\overrightarrow{\otimes}_KK\{x_1,\cdots,x_{r}\}, \textnormal{ and } \wideparen{U}(\mathscr{L}_Y)\simeq \OX_{Y}\overrightarrow{\otimes}K\{x_{r+1},\cdots,x_{r+s}\},
\end{equation}
induced by choosing basis $x_1,\cdots,x_r$, and $x_{r+1},\cdots,x_{r+s}$ of $\mathscr{L}_X(X)$ and $\mathscr{L}_Y(Y)$ respectively. Using the fact that the pullback of functors is strong symmetric monoidal, we have the following identities:
\begin{equation*}
    p_X^{-1}\wideparen{U}(\mathscr{L}_X)\overrightarrow{\otimes}_Kp_Y^{-1}\wideparen{U}(\mathscr{L}_Y)\simeq \left(p_X^{-1}\OX_{X}\overrightarrow{\otimes}_Kp_{Y}^{-1}\OX_{Y}\right)\overrightarrow{\otimes}_KK\{x_1,\cdots,x_{r+s}\}.
\end{equation*}
Similarly, the construction of $\mathscr{L}_X\times\mathscr{L}_Y$ shows that $x_1,\cdots,x_{r+s}$ are a basis of $\mathscr{L}_X\times\mathscr{L}_Y(X\times Y)$, so that we also have:
\begin{equation*}
    \wideparen{U}(\mathscr{L}_X\times \mathscr{L}_Y)\simeq \OX_{X\times Y}\overrightarrow{\otimes}_KK\{x_{1},\cdots,x_{r+s}\},
\end{equation*}
and it follows that (\ref{equation map isomorphism enveloping of product}) is an isomorphism. \\
For the second statement, notice that the tensor product of two sheaves of Ind-Banach spaces $\mathcal{F}$, $\mathcal{G}$ is the sheafification of the presheaf: 
\begin{equation*}
    W\mapsto \mathcal{F}(W)\overrightarrow{\otimes}_K\mathcal{G}(W).
\end{equation*}
As shown in \cite[Section 5.5]{bode2021operations}, the  functor $-\overrightarrow{\otimes}_K K\{x_1,\cdots, x_{r+s}\}$ is strongly exact. Thus, the restriction of the presheaf:
\begin{equation*}
    W\mapsto \OX_{X\times Y}(W)\overrightarrow{\otimes}_K K\{x_1,\cdots, x_{r+s}\},
\end{equation*}
to the site $(X\times Y)_w$ (cf. the proof of Lemma \ref{lemma structure of pullback}) is already a sheaf. As we have $\OX_{X\times Y}(X\times Y)=A\widehat{\otimes}_KB$, the result follows by the isomorphisms above.
\end{proof}
\begin{obs}
Notice that the isomorphism of sheaves of Ind-Banach spaces:
\begin{equation*}
    \wideparen{U}(\mathscr{L}_X)\simeq \OX_{X}\overrightarrow{\otimes}_KK\{x_1,\cdots,x_{r}\},
\end{equation*}
is an isomorphism of sheaves of  $K$-algebras if and only if $\mathscr{L}_X$ has trivial bracket.
\end{obs}
\begin{coro}\label{coro expression of enveloping algebra of product}
Let $X=\Sp(A), Y=\Sp(B)$ be smooth affinoid spaces with finite-free Lie algebroids $L_X$, $L_Y$. Then the Ind-Banach $K$-algebra:
\begin{equation*}
   \wideparen{U}(L_X)\overrightarrow{\otimes}_K\wideparen{U}(L_Y), 
\end{equation*}
is a two-sided Fréchet-Stein algebra.
\end{coro} 
\begin{proof}
By Proposition \ref{prop enveloping of product as product of envelopings}, we have an identification of Ind-Banach $K$-algebras:
\begin{equation*}
   \wideparen{U}(L_X)\overrightarrow{\otimes}_K\wideparen{U}(L_Y)     =\wideparen{U}(L_X\times L_Y). 
\end{equation*}
Hence, it suffices to show the result for $\wideparen{U}(L_X\times L_Y)$. This is \cite[Theorem 6.7]{ardakov2019}.
\end{proof}
\subsection{Sheaves of complete bi-enveloping algebras}\label{Section The sheaf of complete bi-enveloping algebras}
For the rest of this section, we fix a smooth and separated rigid analytic space $X$ with a Lie algebroid $\mathscr{L}$.
In order to study the Hochschild (co)-homology of $\mathscr{L}$, we need to find a convenient way to study the category of modules over the enveloping algebra of $\wideparen{U}(\mathscr{L}):$
\begin{equation*}
    \wideparen{U}(\mathscr{L})^e:=\wideparen{U}(\mathscr{L})\overrightarrow{\otimes}_K\wideparen{U}(\mathscr{L})^{\op}.
\end{equation*} 
To this effect, we introduce the sheaf of complete bornological $K$-algebras $\wideparen{E}(\mathscr{L})$:

\begin{defi}\label{defi sheaf of complete bornological K algebras}
Consider the following sheaf of complete bornological spaces:
\begin{equation*}
\wideparen{E}(\mathscr{L}):=\wideparen{U}(\mathscr{L}) \overrightarrow{\boxtimes}\wideparen{U}(\mathscr{L})^{\op}= \OX_{X^2}\overrightarrow{\otimes}^{\mathbb{L}}_{p_1^{-1}\OX_X\overrightarrow{\otimes}_Kp_2^{-1}\OX_X} \left(p_1^{-1}\wideparen{U}(\mathscr{L})\overrightarrow{\otimes}_Kp_2^{-1}\wideparen{U}(\mathscr{L})^{\op}\right).
\end{equation*}
We call $\wideparen{E}(\mathscr{L})$ the sheaf of complete bi-enveloping algebras of $\mathscr{L}$. If $\mathscr{L}=\mathcal{T}_{X/K}$, we will write $\wideparen{E}_X:=\wideparen{E}(\mathcal{T}_{X/K})$, and call $\wideparen{E}_X$ the sheaf of (complete) bi-enveloping algebras of $X$.
\end{defi}
Notice that it is not a priory clear that 
$\wideparen{E}(\mathscr{L})$ is a sheaf of Ind-Banach algebras. Indeed, as $p_1^{-1}\OX_X\overrightarrow{\otimes}_Kp_2^{-1}\OX_X$ is not central in 
$p_1^{-1}\wideparen{U}(\mathscr{L})\overrightarrow{\otimes}_Kp_2^{-1}\wideparen{U}(\mathscr{L})^{\op}$, it is not immediate that the algebra structure extends uniquely to the completed tensor product defining $\wideparen{E}(\mathscr{L})$. In order to overcome this difficulty, we will use some special features of the enveloping algebras of Lie-Rinehart algebras to define the product on $\wideparen{E}(\mathscr{L})$ locally, and then we will show that these uniquely defined products glue to a global Ind-Banach algebra structure.\bigskip

Let us paint the local picture first: Let $X=\Sp(A)$ be a smooth affinoid space with a free Lie algebroid $\mathscr{L}$. Let $\mathfrak{X}=\Spf(\mathcal{A})$ be an affine formal model, and let $\mathcal{L}$ be a free $\mathcal{A}$-Lie lattice of $L=\mathscr{L}(X)$. Choose a basis of $\mathcal{L}$ as an $\mathcal{A}$-module $x_1,\cdots,x_d$. We then have the following proposition:
\begin{prop}\label{prop anti-homomorphism}
By the PBW theorem, every element in $U(L_2)$  admits a unique expression $\beta=\sum f_{(r_1,\cdots,r_d)} x_1^{r_1}\cdots x_d^{r_d}$. Consider the following expression:
\begin{equation}\label{equation anti-automorphism}
    \beta^t=\sum(-1)^{\sum r_i}x_1^{r_1}\cdots x_d^{r_d}f_r.
\end{equation}
Then the following statements hold:
\begin{enumerate}[label=(\roman*)]
    \item $(\beta^t)^t=1$.
    \item The map $\beta\mapsto \beta^t$ induces isomorphisms of filtered $K$-algebras:
    \begin{equation*}
         T:U(L_2)\rightarrow U(L_2)^{\op}, \quad T^{-1}:U(L_2)^{\op}\rightarrow U(L_2).
    \end{equation*}
\end{enumerate}
\end{prop}
\begin{proof}
This is done in \cite[Section 1.2]{hotta2007d}.
\end{proof}
By construction of $T$, it follows that for each $n\geq 0$, $T$ induces an isomorphism of filtered $\mathcal{R}$-algebras:
\begin{equation*}
    T_n:U(\pi^n\mathcal{L})\rightarrow U(\pi^n\mathcal{L})^{\op}.
\end{equation*}
Therefore, the following proposition is straightforward:
\begin{prop}\label{prop anti-isomorphism fréchet-stein level}
The map $T:U(L)\rightarrow U(L)^{\op}$ satisfies the following:
\begin{enumerate}[label=(\roman*)]
    \item $T$ extends uniquely to an isomorphism of Banach $K$-algebras:
    \begin{equation*}
        \widehat{T}_n:\widehat{U}(\pi^n\mathcal{L})_K\rightarrow \widehat{U}(\pi^n\mathcal{L})_K^{\op}, \textnormal{ for each } n\geq 0.
    \end{equation*}
    \item $T$ extends uniquely to an isomorphism of Fréchet-Stein algebras:
    \begin{equation*}
        \wideparen{T}: \wideparen{U}(L)\rightarrow \wideparen{U}(L)^{\op}.
    \end{equation*}
\end{enumerate}
The map $T^{-1}:U(L)^{\op}\rightarrow U(L)$ satisfies analogous properties.
\end{prop}
The fact that $\mathscr{L}$ is globally free has the following consequence:
\begin{coro}
There are mutually inverse isomorphisms of sheaves of complete bornological $K$-algebras on $X$:
\begin{equation*}
    \wideparen{T}:\wideparen{U}(\mathscr{L})\leftrightarrows \wideparen{U}(\mathscr{L})^{\op}:\wideparen{T}^{-1}.
\end{equation*}
Which are defined on affinoid spaces $U\subset X$ by $\wideparen{T}:\wideparen{U}(\mathscr{L}(U))\rightarrow \wideparen{U}(\mathscr{L}(U))^{\op}$.
\end{coro}
\begin{proof}
The choice of a basis of $\mathcal{L}$ induces an isomorphism of sheaves of complete bornological $K$-vector spaces $\wideparen{U}(\mathscr{L})\rightarrow \OX_X \overrightarrow{\otimes}_KK\{x_1,\cdots,x_d\}$. Using this identification, it follows that for any pair of affinoid subdomains $W\subset V\subset X$, the restriction maps $\wideparen{U}(\mathscr{L})(V)\rightarrow \wideparen{U}(\mathscr{L})(W)$ are identified with the image under the dissection functor of the maps:
\begin{equation*}
    \OX_X(V)\overrightarrow{\otimes}_KK\{x_1,\cdots,x_d\}\rightarrow \OX_X(W)\overrightarrow{\otimes}_KK\{x_1,\cdots,x_d\}.
\end{equation*},
In particular, the restriction maps are the identity on  the $K\{x_1,\cdots,x_d\}$ factor, and then it is clear that $\wideparen{T}$ commutes with the restriction maps. Hence, it defines a map of sheaves, which is an isomorphism by Proposition \ref{prop anti-isomorphism fréchet-stein level}.
\end{proof}
\begin{defi}
We endow $\wideparen{E}(\mathscr{L})$ with the unique product satisfying that the morphism of sheaves on Ind-Banach spaces:
\begin{equation*}
\wideparen{\mathbb{T}}:=\operatorname{Id}\otimes\left(\operatorname{Id}\overrightarrow{\otimes}_Kp_2^{-1}\wideparen{T}\right):\wideparen{U}(\mathscr{L}^2)\rightarrow \wideparen{E}(\mathscr{L}),
\end{equation*}
is an isomorphism of sheaves of Ind-Banach algebras.
\end{defi}
Notice that, under the current definition, it may seem that the Ind-Banach algebra structure on $\wideparen{E}(\mathscr{L})$ depends on the choice of a basis of global sections of $\mathscr{L}$. Indeed, by construction, $\wideparen{\mathbb{T}}$ is obtained via (the pullback of) the following composition:
\begin{equation*}
  \wideparen{U}(\mathscr{L})\simeq \OX_X\overrightarrow{\otimes}_KK\{x_1,\cdots,x_d\}\rightarrow  \OX_X\overrightarrow{\otimes}_KK\{x_1,\cdots,x_d\}\simeq \wideparen{U}(\mathscr{L})^{\op},
\end{equation*}
where the first and third maps are induced by the choice of a global basis for $\mathscr{L}$, and the middle map is obtained from Propositions \ref{prop anti-homomorphism} and \ref{prop anti-isomorphism fréchet-stein level}. 
\begin{Lemma}\label{Lemma product in E}
Let $x_1,\cdots,x_d$ and $y_1,\cdots,y_d$ be two basis of global sections of $\mathscr{L}$. The base-change induces a unique commutative diagram:
\begin{equation*}
\begin{tikzcd}
	{\OX_X\overrightarrow{\otimes}_KK\{x_1,\cdots,x_d\}} & {(\OX_X\overrightarrow{\otimes}_KK\{x_1,\cdots,x_d\})^{\op}} \\
	{\OX_X\overrightarrow{\otimes}_KK\{y_1,\cdots,y_d\}} & {(\OX_X\overrightarrow{\otimes}_KK\{y_1,\cdots,y_d\})^{\op}}
	\arrow["{\wideparen{T}}", from=1-1, to=1-2]
	\arrow[from=1-1, to=2-1]
	\arrow[from=1-2, to=2-2]
	\arrow["{\wideparen{T}}", from=2-1, to=2-2]
\end{tikzcd}
\end{equation*}
where all maps are isomorphisms of sheaves of Ind-Banach algebras. In particular, the isomorphism class of $\wideparen{E}(\mathscr{L})$ is independent of the choice of a basis for $\mathscr{L}$.
\end{Lemma}
\begin{proof}
The leftmost vertical maps are defined via the compositions:
\begin{equation*}
    \OX_X\overrightarrow{\otimes}_KK\{x_1,\cdots,x_d\}\rightarrow \wideparen{U}(\mathscr{L})\rightarrow \OX_X\overrightarrow{\otimes}_KK\{y_1,\cdots,y_d\},
\end{equation*}
where the first map is the one obtained in Lemma \ref{lemma local freeness of algebras over structure sheaf}, and the second one is the inverse. The vertical map on the right hand side is obtained by applying $(-)^{\op}$ to this map. In order to see that the diagram commutes, it is enough to test if on the structure sheaf $\OX_X$ and the elements of the basis $x_1,\cdots,x_d$. This is a straightforward calculation. For the second part, we apply the functor $\operatorname{Id}\otimes\left(\operatorname{Id}\overrightarrow{\otimes}_Kp_2^{-1}(-)\right)$ to obtain a commutative diagram of sheaves of Ind-Banach spaces:
\begin{equation}\label{equation independence of algebra structure from choice of basis}
    \begin{tikzcd}
	{\wideparen{U}(\mathscr{L}^2)} & {\wideparen{E}(\mathscr{L})} \\
	{\wideparen{U}(\mathscr{L}^2)} & {\wideparen{E}(\mathscr{L})}
	\arrow[from=1-1, to=1-2]
	\arrow[from=1-1, to=2-1]
	\arrow[from=1-2, to=2-2]
	\arrow[from=2-1, to=2-2]
\end{tikzcd}
\end{equation}
where every morphism is an isomorphism of sheaves of Ind-Banach spaces. Notice that the leftmost vertical map is an isomorphism of sheaves of Ind-Banach algebras. Thus, the rightmost vertical map is an isomorphism of sheaves of Ind-Banach algebras if the two horizontal maps are.
\end{proof}
Back to the general setting, we let $X$ be a smooth and separated rigid analytic space with a Lie algebroid $\mathscr{L}$. The previous construction shows that, at least locally, $\wideparen{E}(\mathscr{L})$ can be given the structure of a sheaf of Ind-Banach algebras extending the products on $p_1^{-1}\wideparen{U}(\mathscr{L})\overrightarrow{\otimes}_Kp_2^{-1}\wideparen{U}(\mathscr{L})^{\op}$ and $\OX_{X^2}$.  The idea now is showing that these structures can be combined into a global structure:
\begin{teo}\label{teo product on E}
Let $X$ be a smooth and separated rigid analytic space with a lie algebroid $\mathscr{L}$, and let $\wideparen{E}(\mathscr{L})$ be the sheaf of complete bi-enveloping algebras of $\mathscr{L}$. Then $\wideparen{E}(\mathscr{L})$ admits a unique structure as a sheaf of complete bornological algebras satisfying the following properties: 
\begin{enumerate}[label=(\roman*)]
    \item The canonical maps: 
    \begin{equation*}
        p_1^{-1}\wideparen{U}(\mathscr{L})\overrightarrow{\otimes}_Kp_2^{-1}\wideparen{U}(\mathscr{L})^{\op}\rightarrow \wideparen{E}(\mathscr{L}), \quad \OX_{X^2}\rightarrow \wideparen{E}(\mathscr{L}),
    \end{equation*}
    are morphisms of sheaves of Ind-Banach algebras.
    \item Let $U\subset X$ be an affinoid subdomain satisfying that $\mathscr{L}_{\vert U}$ is globally free. A choice of a basis of global sections $s_1,\cdots,s_d\in \Gamma(U,\mathscr{L})$ induces a unique isomorphism of sheaves of Ind-Banach spaces on $X^2$:
\begin{equation*}
    \wideparen{U}(\mathscr{L}^2)_{\vert U\times U}=\wideparen{E}(\mathscr{L})_{\vert U\times U}\simeq \OX_{U\times U}\overrightarrow{\otimes}_KK\{x_1,\cdots,x_{2d}\}.
\end{equation*}
Then the canonical morphism of sheaves of Ind-Banach spaces:
\begin{multline*}
    \wideparen{\mathbb{T}}:\wideparen{U}(\mathscr{L}^2)_{\vert U\times U}\simeq \OX_{U\times U}\overrightarrow{\otimes}_KK\{x_1,\cdots,x_{2d}\}\\
    \rightarrow \OX_{U\times U}\overrightarrow{\otimes}_KK\{x_1,\cdots,x_{2d}\}\simeq \wideparen{E}(\mathscr{L})_{\vert U\times U},
\end{multline*}
where the middle map is defined as in Definition \ref{Lemma product in E} is an isomorphism of sheaves of Ind-Banach algebras. 
\end{enumerate}

\end{teo}
\begin{proof}
Uniqueness is clear. Thus, we only need to show that such a structure exists. As $\mathscr{L}$ is a Lie algebroid, we can choose an admissible affinoid cover  $\mathfrak{U}$ of $X$ satisfying that for each $U\in \mathfrak{U}$ the restriction $\mathscr{L}_{\vert U}$ admits a basis of global sections. Choose affinoid subdomains $U,V,W\in \mathfrak{U}$. A seen before, any choice of a basis of global sections of $\mathscr{L}_{\vert U}$ induces an isomorphism of sheaves of Ind-Banach spaces:
\begin{equation*}
    \wideparen{\mathbb{T}}_U:\wideparen{U}(\mathscr{L}^2)_{\vert U\times U}\rightarrow \wideparen{E}(\mathscr{L})_{\vert U\times U},
\end{equation*}
endowing $\wideparen{E}(\mathscr{L})_{\vert U\times U}$ with a unique product. Choose a basis $(s_1^U,\cdots, s_d^U)$ of $\mathscr{L}_{\vert U}$, and let $\wideparen{E}(\mathscr{L})_{\vert U\times U}^U$ be the corresponding sheaf of Ind-Banach algebras. Additionally, we let  $\wideparen{E}(\mathscr{L})_{\vert V\times V}^V$, and $\wideparen{E}(\mathscr{L})_{\vert W\times W}^W$  be the analogous algebras obtained from choices of basis of $\mathscr{L}$ over $V$ and $W$, which we denote by $(s_1^V,\cdots, s_d^V)$ and $(s_1^W,\cdots, s_d^W)$ respectively.  It suffices to see that these products glue to a global Ind-Banach algebra structure on $\wideparen{E}(\mathscr{L})$. That is, we need to find isomorphisms of sheaves of Ind-Banach algebras:
\begin{multline*}
  \rho_{U,V}:  \wideparen{E}(\mathscr{L})^U_{\vert (U\cap V)\times (U\cap V)}\rightarrow  \wideparen{E}(\mathscr{L})^V_{\vert (U\cap V)\times (U\cap V)},\\
  \rho_{V,W}:  \wideparen{E}(\mathscr{L})^V_{\vert (V\cap W)\times (V\cap W)}\rightarrow  \wideparen{E}(\mathscr{L})^W_{\vert (V\cap W)\times (V\cap W)},
\end{multline*}
and show that they satisfy the cocycle condition. In order to construct this map, we look at the diagram (\ref{equation independence of algebra structure from choice of basis})  and notice that, by construction of the product of Lie algebroids, the choice of the basis $(s_1^U,\cdots, s_d^U)$ induces a unique basis of $\mathscr{L}^2_{\vert U\times U}$. Thus, we can define the following map:
\begin{multline*}
    \rho_{U,V}:\wideparen{E}(\mathscr{L})^U_{\vert (U\cap V)\times (U\cap V)}\xrightarrow[]{\wideparen{\mathbb{T}}_U^{-1}}\wideparen{U}(\mathscr{L}^2)_{\vert (U\cap V)\times (U\cap V)}\\
    \xrightarrow[]{\tau_{U,V}}  \wideparen{U}(\mathscr{L}^2)_{\vert (U\cap V)\times (U\cap V)} \xrightarrow[]{\wideparen{\mathbb{T}}_V} \wideparen{E}(\mathscr{L})^V_{\vert (U\cap V)\times (U\cap V)},
\end{multline*}
where $\tau_{U,V}$ is the unique automorphism of $\wideparen{U}(\mathscr{L}^2)_{\vert (U\cap V)\times (U\cap V)}$ induced by the base change in $\mathscr{L}^2_{\vert (U\cap V)\times (U\cap V)}$ from  the basis induced by $(s_1^U,\cdots, s_d^U)$ to the basis induced by $(s_1^V,\cdots, s_d^V)$. It follows that the $\rho_{U,V}$ satisfy the cocycle condition if and only if the $\tau_{U,V}$ do, and this is clear since $\mathscr{L}^2$ is a coherent module in $X^2$.
\end{proof}
From now on, we will always regard $\wideparen{E}(\mathscr{L})$ as a sheaf of Ind-Banach algebras with respect to this product. As can be seen from the definition, the algebras $\wideparen{E}(\mathscr{L})$ and $\wideparen{U}(\mathscr{L}^2)$ are closely related. This relation will be further explored in future sections, where we will study the category of Ind-Banach 
$\wideparen{E}(\mathscr{L})$-modules via comparing it with the category of Ind-Banach  $\wideparen{U}(\mathscr{L}^2)$-modules. However, before doing that, let us first motivate the above definition by showcasing 
an important feature of $\wideparen{E}(\mathscr{L})$:
\begin{Lemma}\label{lemma sheaves of modules supported on the diagonal}
Let $\Mod_{\Indban}(p_1^{-1}\wideparen{U}(\mathscr{L})\overrightarrow{\otimes}_Kp_2^{-1}\wideparen{U}(\mathscr{L})^{\op})_{\Delta}$ denote the full subcategory of $\Mod_{\Indban}(p_1^{-1}\wideparen{U}(\mathscr{L})\overrightarrow{\otimes}_Kp_2^{-1}\wideparen{U}(\mathscr{L})^{\op})$ given by the modules supported on the diagonal. There is an equivalence of quasi-abelian categories:
\begin{equation*}
    \Delta_*: \Mod_{\Indban}(\wideparen{U}(\mathscr{L})^e)\leftrightarrows\Mod_{\Indban}(p_1^{-1}\wideparen{U}(\mathscr{L})\overrightarrow{\otimes}_Kp_2^{-1}\wideparen{U}(\mathscr{L})^{\op})_{\Delta}:\Delta^{-1}.
\end{equation*}
\end{Lemma}
\begin{proof}
There is an isomorphism of sheaf of Ind-Banach algebras:
\begin{equation*}
\wideparen{U}(\mathscr{L})^e=\Delta^{-1}\left(    p_1^{-1}\wideparen{U}(\mathscr{L})\overrightarrow{\otimes}_Kp_2^{-1}\wideparen{U}(\mathscr{L})^{\op}\right).
\end{equation*}
Thus, this is a special case of \ref{prop sheaves of modules supported on the diagonal general case}. 
\end{proof}
\begin{prop}\label{prop extension functor}
Let $X$ be a smooth and separated rigid analytic space with a lie algebroid $\mathscr{L}$. There is a strongly right exact functor:
 \begin{align*}
     \Delta^E_*:\Mod_{\Indban}(\wideparen{U}(\mathscr{L})^e)\rightarrow \Mod_{\Indban}(\wideparen{E}(\mathscr{L})),\\
     \mathcal{M}\mapsto  \wideparen{E}(\mathscr{L})\overrightarrow{\otimes}_{    p_1^{-1}\wideparen{U}(\mathscr{L})\overrightarrow{\otimes}_Kp_2^{-1}\wideparen{U}(\mathscr{L})^{\op}} \Delta_*(\mathcal{M})
 \end{align*}
Analogously, we also have a right exact functor between the left hearts:
\begin{equation*}
    \Delta^{I(E)}_*:\Mod_{\Indban}(I(\wideparen{U}(\mathscr{L}))^e)\rightarrow \Mod_{\Indban}(I(\wideparen{E}(\mathscr{L}))),
\end{equation*}
and this functor admits a left derived functor:
\begin{equation*}
    \mathbb{L}\Delta_*^E:=\mathbb{L}(\Delta^{I(E)}_*)\operatorname{D}(\wideparen{U}(\mathscr{L}))^e)\rightarrow \operatorname{D}(\wideparen{E}_X).
\end{equation*}
\end{prop}
\begin{proof}
The existence of the functors, together with their strong right exactness is a consequence of Lemma \ref{lemma sheaves of modules supported on the diagonal}, together with the fact that tensor products have a right adjoint, and therefore preserve arbitrary cokernels. In order to see that $\Delta^{I(E)}_*$  admits a left derived functor, we again invoke Lemma \ref{lemma sheaves of modules supported on the diagonal} to reduce the claim to showing that the extension of scalars:
\begin{equation*}
    \Mod_{LH(\widehat{\mathcal{B}}c_K)}(p_1^{-1}I(\wideparen{U}(\mathscr{L}))\Tilde{\otimes}_{I(K)}p_2^{-1}I(\wideparen{U}(\mathscr{L}))^{\op})\rightarrow \Mod_{LH(\widehat{\mathcal{B}}c_K)}(I(\wideparen{E}(\mathscr{L}))),
\end{equation*}
admits a left derived functor. However, this follows by the existence of flat resolutions for modules, as seen in Proposition \ref{prop existence derived fucntor extension of scalars}.
\end{proof}
As in previous instances, tensor products do not (generally) behave nicely with respect to the passage the left heart. Namely, given $M\in \Mod_{\Indban}(\wideparen{U}(\mathscr{L})^e)$ we generally do \emph{not} have:
\begin{equation*}
  I(\Delta_*^E(M))=\Delta_*^{I(E)}(I(\mathcal{M})).
\end{equation*}
However, we will later identify a full subcategory of $\Mod_{\Indban}(\wideparen{U}(\mathscr{L})^e)$ in which this identity always holds.
\subsection{Definition of Hochschild (co)-homology}
Again, we fix a smooth and separated rigid analytic space $X$ with a Lie algebroid $\mathscr{L}$. Building on the material developed in previous sections, we are finally ready to give our first formal definition of Hochschild (co)-homology for Ind-Banach $\wideparen{U}(\mathscr{L})$-bimodules:
\begin{defi}\label{defi hichschild cohomology}
For any  $\mathcal{M}^{\bullet}\in\operatorname{D}(\wideparen{U}(\mathscr{L})^e)$ we define the following objects: 
\begin{enumerate}[label=(\roman*)]
    \item The inner Hochschild cohomology complex of $\wideparen{U}(\mathscr{L})$ with coefficients in $\mathcal{M}^{\bullet}$ is the following complex in $\operatorname{D}(\operatorname{Shv}(X,\Indban))$:
\begin{equation*}
  \mathcal{HH}^{\bullet}(\wideparen{U}(\mathscr{L}),\mathcal{M}^{\bullet}):=\Delta^{-1}R\underline{\mathcal{H}om}_{\wideparen{E}_X}(\EDelta\wideparen{U}(\mathscr{L}),\EDelta\mathcal{M}^{\bullet}),
\end{equation*}
and we call $\mathcal{HH}^n(\wideparen{U}(\mathscr{L}),\mathcal{M}^{\bullet}):=\operatorname{H}^n\left(\mathcal{HH}^{\bullet}(\wideparen{U}(\mathscr{L}),\mathcal{M}^{\bullet})\right)$ the $n$-th Hochschild cohomology sheaf of $\wideparen{U}(\mathscr{L})$ with coefficients in $\mathcal{M}^{\bullet}$. 
\item The Hochschild cohomology complex of $\wideparen{U}(\mathscr{L})$ with coefficients in $\mathcal{M}^{\bullet}$ is the following complex in $\operatorname{D}(\Indban)$:
\begin{equation*}
    \operatorname{HH}^{\bullet}(\wideparen{U}(\mathscr{L}),\mathcal{M}^{\bullet})=R\Gamma(X,\mathcal{HH}^{\bullet}(\wideparen{U}(\mathscr{L}),\mathcal{M}^{\bullet})).
\end{equation*}
As before, we call $\operatorname{HH}^{n}(\wideparen{U}(\mathscr{L}),\mathcal{M}^{\bullet})=\operatorname{H}^n\left(\operatorname{HH}^{\bullet}(\wideparen{U}(\mathscr{L}),\mathcal{M}^{\bullet})\right)$ the $n$-th Hochschild cohomology space (group) of $\wideparen{U}(\mathscr{L})$ with coefficients in $\mathcal{M}^{\bullet}$. 
\end{enumerate}    
Additionally, if $\mathscr{A}$ is an Ind-Banach algebra and $\mathcal{M}\in \operatorname{D}(\mathscr{A}^e)$, we define:
\begin{equation*}
    \operatorname{HH}^{\bullet}(\mathscr{A},\mathcal{M}^{\bullet})=R\underline{\operatorname{Hom}}_{\mathscr{A}^e}(\mathscr{A},\mathcal{M}^{\bullet}),
\end{equation*}
and  call this the Hochschild cohomology complex of $\mathscr{A}$ with coefficients in $\mathcal{M}^{\bullet}$.
\end{defi}
Notice that for any $\mathcal{M}\in \Mod_{\Indban}(\wideparen{U}(\mathscr{L})^e)$, the $n$-th Hochschild cohomology sheaf $\mathcal{HH}^n(\wideparen{U}(\mathscr{L}),\mathcal{M})$ is (generally) not an object in $\operatorname{Shv}(X,\Indban)$, but rather an object in the left heart  $\operatorname{Shv}(X,LH(\widehat{\mathcal{B}}c_K))$. The analogous phenomena also  holds for the Hochschild cohomology groups.\\
Ideally, we would like to show that inner Hochschild cohomology $\mathcal{HH}^{\bullet}(\wideparen{U}(\mathscr{L}),-)$ can be expressed as a derived inner homomorphism functor in the derived category of $\wideparen{U}(\mathscr{L})^e$-modules. Namely, we would like to show that for any $\mathcal{M}^{\bullet}\in\operatorname{D}(\wideparen{U}(\mathscr{L})^e))$ there is an identification:
\begin{equation*}
    \mathcal{HH}^{\bullet}(\wideparen{U}(\mathscr{L}),\mathcal{M}^{\bullet})=R\underline{\mathcal{H}om}_{\wideparen{U}(\mathscr{L})^e}(\wideparen{U}(\mathscr{L}),\mathcal{M}^{\bullet}).
\end{equation*}
Unfortunately, this does not seem to be correct. Furthermore, there is no clear way of comparing these two functors. A more sensible approach would be replacing $\mathcal{M}^{\bullet}$ by $\Delta^{-1}\mathbb{L}\Delta^E_*\mathcal{M}^{\bullet}$ in the right hand side. In this case, it can be shown that there is a natural transformation:
\begin{equation*}
   \mathcal{HH}^{\bullet}(\wideparen{U}(\mathscr{L}),-)\rightarrow R\underline{\mathcal{H}om}_{\wideparen{U}(\mathscr{L})^e}(\wideparen{U}(\mathscr{L}),\Delta^{-1}\mathbb{L}\Delta^E_*(-)), 
\end{equation*}
but we are unable to determine if this is always a natural transformation.\bigskip

Let us now paint the homological side of the picture. We start with the following observation: Let $\mathscr{A}$ be a sheaf of Ind-Banach algebras on $X$. We remark that $\mathscr{A}$ also has a canonical structure as a right $\mathscr{A}^e$-module. Indeed, a right $\mathscr{A}^e$-module is a left module over $(\mathscr{A}^e)^{\op}=\mathscr{A}^{\op}\widehat{\otimes}_K\mathscr{A}$. This is precisely the enveloping algebra of $\mathscr{A}^{\op}$. Hence, $\mathscr{A}^{\op}$ has a canonical structure as a right $\mathscr{A}^e$-module. We define the structure on $\mathscr{A}$ via the identity $\mathscr{A}=\mathscr{A}^{\op}$ as objects in $\operatorname{Shv}(X,\Indban)$. We may generalize this as follows:
\begin{Lemma}\label{Lemma right module structure on pushforward}
$\EDelta\wideparen{U}(\mathscr{L})$ is canonically a complex of right $\wideparen{E}(\mathscr{L})$-modules.    
\end{Lemma}
\begin{proof}
Let $A=p_1^{-1}\wideparen{U}(\mathscr{L})\overrightarrow{\otimes}_Kp_2^{-1}\wideparen{U}(\mathscr{L})^{\op}$ for simplicity. By the discussion above, we know that $\Delta_*\wideparen{U}(\mathscr{L})$ is a left module over $A^{\op}$, so we have a complex of right $\wideparen{E}(\mathscr{L})$-modules:
\begin{multline*}
    \wideparen{E}(\mathscr{L})^{\op}\overrightarrow{\otimes}_{A^{\op}}^{\mathbb{L}}\Delta_*\wideparen{U}(\mathscr{L})=(\OX_{X^2}\overrightarrow{\otimes}^{\mathbb{L}}_{p_1^{-1}\OX_X\overrightarrow{\otimes}_Kp_1^{-1}\OX_X}A)^{\op}\overrightarrow{\otimes}_{A^{\op}}^{\mathbb{L}}\Delta_*\wideparen{U}(\mathscr{L})\\
    =\OX_{X^2}\overrightarrow{\otimes}^{\mathbb{L}}_{p_1^{-1}\OX_X\overrightarrow{\otimes}_Kp_1^{-1}\OX_X}\Delta_*\wideparen{U}(\mathscr{L})=\EDelta\wideparen{U}(\mathscr{L}),
\end{multline*}
and this is precisely what we wanted to show.
\end{proof}
Thus, we are free to regard $\EDelta\wideparen{U}(\mathscr{L})$ as a right Ind-Banach $\wideparen{E}(\mathscr{L})$-module, and we will do so without further notice. We now have the following definition:
\begin{defi}
For any  $\mathcal{M}^{\bullet}\in\operatorname{D}(\wideparen{U}(\mathscr{L})^e)$ we define the following objects: 
\begin{enumerate}[label=(\roman*)]
    \item The inner Hochschild homology complex of $\wideparen{U}(\mathscr{L})$ with coefficients in $\mathcal{M}^{\bullet}$ is the following complex in $\operatorname{D}(\operatorname{Shv}(X,\Indban))$:
\begin{equation*}
  \mathcal{HH}_{\bullet}(\wideparen{U}(\mathscr{L}),\mathcal{M}^{\bullet}):= \Delta^{-1}\left(\EDelta\wideparen{U}(\mathscr{L})\overrightarrow{\otimes}^{\mathbb{L}}_{\wideparen{E}(\mathscr{L})}\EDelta\mathcal{M}^{\bullet}\right),
\end{equation*}
and we call $\mathcal{HH}_n(\wideparen{U}(\mathscr{L}),\mathcal{M}^{\bullet}):=\operatorname{H}^{-n}\left(\mathcal{HH}_{\bullet}(\wideparen{U}(\mathscr{L}),\mathcal{M}^{\bullet})\right)$ the $n$-th Hochschild homology sheaf of $\wideparen{U}(\mathscr{L})$ with coefficients in $\mathcal{M}^{\bullet}$. 
\item The Hochschild homology complex of $\wideparen{U}(\mathscr{L})$ with coefficients in $\mathcal{M}^{\bullet}$ is the following complex in $\operatorname{D}(\Indban)$:
\begin{equation*}
    \operatorname{HH}_{\bullet}(\wideparen{U}(\mathscr{L}),\mathcal{M}^{\bullet})=R\Gamma(X,\mathcal{HH}_{\bullet}(\wideparen{U}(\mathscr{L}),\mathcal{M}^{\bullet})).
\end{equation*}
As before, we call $\operatorname{HH}_{n}(\wideparen{U}(\mathscr{L}),\mathcal{M}^{\bullet})=\operatorname{H}^{-n}\left(\operatorname{HH}_{\bullet}(\wideparen{U}(\mathscr{L}),\mathcal{M}^{\bullet})\right)$ the $n$-th Hochschild homology space (group) of $\wideparen{U}(\mathscr{L})$ with coefficients in $\mathcal{M}^{\bullet}$. 
\end{enumerate}    
Additionally, if $\mathscr{A}$ is an Ind-Banach algebra and $\mathcal{M}\in \operatorname{D}(\mathscr{A}^e)$, we define:
\begin{equation*}
    \operatorname{HH}_{\bullet}(\mathscr{A},\mathcal{M}^{\bullet})=\mathscr{A}\overrightarrow{\otimes}^{\mathbb{L}}_{\mathscr{A}^e}\mathcal{M}^{\bullet}
\end{equation*}
and  call this the Hochschild homology complex of $\mathscr{A}$ with coefficients in $\mathcal{M}^{\bullet}$.
\end{defi}
Unlike in the cohomological case, we can show the following proposition:
\begin{prop}\label{prop bi-enveloping algebra and HHo}
There are canonical natural equivalences of functors:
\begin{multline*}
    \mathcal{HH}_{\bullet}(\wideparen{U}(\mathscr{L}),-)\xrightarrow[]{\simeq}\wideparen{U}(\mathscr{L})\overrightarrow{\otimes}^{\mathbb{L}}_{\wideparen{U}(\mathscr{L})^e}\left(\Delta^{-1}\EDelta(-)\right)\\
    \xrightarrow[]{\simeq}\left(\Delta^{-1}\EDelta\wideparen{U}(\mathscr{L})\right)\overrightarrow{\otimes}^{\mathbb{L}}_{\wideparen{U}(\mathscr{L})^e}-.
\end{multline*}    
Moreover, we also have an equivalence $\Delta^{-1}\EDelta(-)=\Delta^{-1}\OX_{X^2}\overrightarrow{\otimes}_{\OX_X\overrightarrow{\otimes}_K\OX_X}^{\mathbb{L}}(-)$.
\end{prop}
\begin{proof}
In order to simplify notation, let $A=p_1^{-1}\wideparen{U}(\mathscr{L})\overrightarrow{\otimes}_Kp_2^{-1}\wideparen{U}(\mathscr{L})^{\op}$. Choose any $\mathcal{M}^{\bullet}\in \operatorname{D}(\wideparen{U}(\mathscr{L})^e)$. We have the following chain of identities:
\begin{align*}
    \mathcal{HH}_{\bullet}(\wideparen{U}(\mathscr{L}),\mathcal{M}^{\bullet}):=& \Delta^{-1}\left(\EDelta\wideparen{U}(\mathscr{L})\overrightarrow{\otimes}^{\mathbb{L}}_{\wideparen{E}(\mathscr{L})}\EDelta\mathcal{M}^{\bullet}\right)\\ =&\Delta^{-1}\left(\left( \Delta_*\wideparen{U}(\mathscr{L})\overrightarrow{\otimes}_A^{\mathbb{L}}\wideparen{E}(\mathscr{L})\right)\overrightarrow{\otimes}^{\mathbb{L}}_{\wideparen{E}(\mathscr{L})}\left(\wideparen{E}(\mathscr{L})\overrightarrow{\otimes}_{A}^{\mathbb{L}}\Delta_*\mathcal{M}^{\bullet}\right)\right)\\
    =&\Delta^{-1}\left( \Delta_*\wideparen{U}(\mathscr{L})\overrightarrow{\otimes}_A^{\mathbb{L}}\wideparen{E}(\mathscr{L})\overrightarrow{\otimes}_{A}^{\mathbb{L}}\Delta_*\mathcal{M}^{\bullet}\right)\\
    =&\wideparen{U}(\mathscr{L})\overrightarrow{\otimes}_{\wideparen{U}(\mathscr{L})^e}^{\mathbb{L}}\Delta^{-1}\wideparen{E}(\mathscr{L})\overrightarrow{\otimes}_{\wideparen{U}(\mathscr{L})^e}^{\mathbb{L}}\mathcal{M}^{\bullet}.
\end{align*}
By construction, we have $\Delta^{-1}\wideparen{E}(\mathscr{L})=\Delta^{-1}\OX_{X^2}\overrightarrow{\otimes}_{\OX_X\overrightarrow{\otimes}_K\OX_X}^{\mathbb{L}}\wideparen{U}(\mathscr{L})^e$. Hence, we may rewrite the last complex of sheaves of Ind-Banach spaces as:
\begin{equation*}
    \wideparen{U}(\mathscr{L})\overrightarrow{\otimes}_{\wideparen{U}(\mathscr{L})^e}^{\mathbb{L}}\Delta^{-1}\OX_{X^2}\overrightarrow{\otimes}_{\OX_X\overrightarrow{\otimes}_K\OX_X}^{\mathbb{L}}\mathcal{M}^{\bullet}=\wideparen{U}(\mathscr{L})\overrightarrow{\otimes}^{\mathbb{L}}_{\wideparen{U}(\mathscr{L})^e}\Delta^{-1}\EDelta\mathcal{M}^{\bullet},
\end{equation*}
which is the first identity we wanted to show. The remaining ones are analogous.
\end{proof}

\section{Co-admissible modules over the bi-enveloping algebra}\label{Chapter co-admissible modules over the bi-enveloping algebra}
Let $X$ be a smooth and separated rigid analytic space with a Lie algebroid $\mathscr{L}$. In the previous chapter, we introduced two sheaves of complete bornological algebras, $\wideparen{U}(\mathscr{L}^2)$ and $\wideparen{E}(\mathscr{L})$, and showcased the relevance of the sheaf of bi-enveloping algebras $\wideparen{E}(\mathscr{L})$ in the study of Hochschild cohomology and homology. This chapter is devoted to bringing $\wideparen{U}(\mathscr{L}^2)$ into the limelight. By construction, $\wideparen{U}(\mathscr{L}^2)$  is the completion of the sheaf of enveloping algebras of the Lie algebroid $\mathscr{L}^2$ on $X^2$. In particular, its (derived) category of Ind-Banach modules presents great homological behavior, and has been extensively studied in \cite{bode2021operations}. The goal of this section is showing that many interesting constructions in  $ \Mod_{\Indban}(\wideparen{U}(\mathscr{L}^2))$, such as sheaves of co-admissible modules, can be translated to the setting of $\wideparen{E}(\mathscr{L})$-modules.\\
In order to do this, we will develop a formalism of side-switching operations for bimodules. In particular, we will obtain a pair of mutually inverse equivalences of quasi-abelian categories:
\begin{equation*}
    \operatorname{S}:\Mod_{\Indban}(\wideparen{U}(\mathscr{L}^2))\leftrightarrows \Mod_{\Indban}(\wideparen{E}(\mathscr{L})):\operatorname{S}^{-1},
\end{equation*}
which we call the side-switching equivalence. As shown above, if $\mathscr{L}$ is globally free on $X$, then there is a (non-canonical) isomorphism of sheaves of complete bornological algebras:
\begin{equation*}
   \wideparen{\mathbb{T}}: \wideparen{U}(\mathscr{L}^2)\rightarrow \wideparen{E}(\mathscr{L}).
\end{equation*}
In this situation, the previous equivalence can be explicitly obtained via extension of scalars along $\wideparen{\mathbb{T}}$. In the general case, the description of $\operatorname{S}$ is more involved, and will be discussed below.\bigskip 

The fact that $\wideparen{U}(\mathscr{L}^2)$ and $\wideparen{E}(\mathscr{L})$ are locally isomorphic has very strong repercussions in the structure of $\wideparen{E}(\mathscr{L})$. In particular, it shows that $\wideparen{E}(\mathscr{L})$ is locally a Fréchet-Stein algebra with $c$-flat restriction maps. Hence, we can define the (abelian) category of co-admissible $\wideparen{E}(\mathscr{L})$-modules, which we denote by $\mathcal{C}(\wideparen{E}(\mathscr{L}))$. As one may suspect, the side-switching equivalence restricts to an equivalence of abelian categories:
\begin{equation*}
    \operatorname{S}:\mathcal{C}(\wideparen{U}(\mathscr{L}^2))\leftrightarrows \mathcal{C}(\wideparen{E}(\mathscr{L})):\operatorname{S}^{-1}.
\end{equation*}
Even better, we can work in the derived setting to obtain a category of $\mathcal{C}$-complexes of $\wideparen{E}(\mathscr{L})$-modules, and the derived version of the side-switching equivalence induces an equivalence of triangulated categories between the corresponding categories of $\mathcal{C}$-complexes over $\wideparen{E}(\mathscr{L})$ and $\wideparen{U}(\mathscr{L}^2)$.\bigskip

Another upshot of the side-switching equivalence is that we obtain new formulas to calculate the inner Hochschild (co)-homology complex. Namely, we define the immersion functor as the following composition:
\begin{equation*}
    \mathbb{L}\Delta_*^{\operatorname{S}}:\operatorname{D}(\wideparen{U}(\mathscr{L})^e)\xrightarrow[]{\mathbb{L}\Delta_*^E}\operatorname{D}(\wideparen{E}(\mathscr{L}))\xrightarrow[]{\operatorname{S}^{-1}}\operatorname{D}(\wideparen{U}(\mathscr{L}^2)).
\end{equation*}
As a consequence of the derived side-switching equivalence, we obtain the following formula for the inner Hochschild cohomology complex:
\begin{equation*}
    \mathcal{HH}^{\bullet}(\mathcal{M}^{\bullet})=\Delta^{-1}R\underline{\mathcal{H}om}_{\wideparen{U}(\mathscr{L}^2)}(\Ifunct\wideparen{U}(\mathscr{L}),\Ifunct\mathcal{M}^{\bullet}),
\end{equation*}
where $\mathcal{M}^{\bullet}\in \operatorname{D}(\wideparen{U}(\mathscr{L})^e)$. We point out that there are right module versions for all of the constructions above, which lead to analogous results for Hochschild homology.
\subsection{Side-switching operations for bimodules}\label{Section side-switching operations}
Let $X$ be a smooth and separated rigid analytic variety. In this section, we adapt the formalism of side-switching operations for Ind-Banach $\wideparen{\D}$-modules developed in \cite[Section 3]{ardakov2015d} and \cite[Section 6.3]{bode2021operations} to the context of bimodules. As in previous sections, we will work in greater generality, and study the case of a general Lie algebroid on $X$.\bigskip

Let us start by recalling the algebraic situation: Let $A$, $B$ be two $K$-algebras, with smooth Lie algebroids $L_1$ and $L_2$. We start with the following:
\begin{Lemma}
Assume $L_2$ is locally free of rank $n$ as a $B$-module, and let:
\begin{equation*}
    \Omega_{L_2}=\Hom_B(\bigwedge_B^nL_2,B).
\end{equation*}
Then $\Omega_{L_2}$ is a right $U(L_2)$-module, with action given by:
\begin{equation*}
    x\alpha=-\lambda_{\alpha}(x), \textnormal{ for } x\in \Omega_{L_2}, \textnormal{ and } \alpha\in L_2,
\end{equation*}
where $\lambda_{\alpha}(x)$ denotes the Lie derivative.
\end{Lemma}
\begin{proof}
This is shown in \cite[Proposition 2.8]{huebschmann2004duality}. For the construction of the Lie derivative, check \cite[pp. 10]{huebschmann2004duality}  and \cite[Equation 4.3]{rinehart1963differential}.
\end{proof}

The module $\Omega_{L_2}$ can be used to construct a Morita equivalence between $U(L_2)$ and $U(L_2)^{\op}$. For convenience of the reader, we now recall the construction from \cite[3.1]{ardakov2015d}: Let $M$ be a left $U(L_2)$-module, and consider the $B$-module $\Omega_{L_2}\otimes_BM$. This module has a canonical $U(L_2)^{\op}$-module structure given by the following expression:
\begin{equation}\label{equation side-changing}
    (d\otimes m)\alpha = (d\alpha)\otimes m- d\otimes(\alpha m),\textnormal{ for } d \in \Omega_{L_2}, m\in M,\textnormal{ and } \alpha\in L_2.
\end{equation}
Assume in addition that $M$ is a $U(L_1)\otimes_K U(L_2)$-module. Then $M$ has structures as a left $U(L_1)$-module, and left $U(L_2)$-module, which satisfy the following compatibility condition:
For $a\in U(L_1)$, $b\in U(L_2)$, let $a\cdot:M\rightarrow M$ be the $K$-linear endomorphism induced by the action of $a$, and $b\cdot:M\rightarrow M$ be the $K$-linear endomorphism induced by the action of $b$. Then  
 $a\cdot:M\rightarrow M$ is $U(L_2)$-linear, and 
$b\cdot:M\rightarrow M$ is $U(L_1)$-linear.\\
In this case, for any $x\in U(L_1)$, the map $x:M\rightarrow M$ is $B$-linear. Hence, we get an induced map:
\begin{equation*}
    \Id\otimes x:\Omega_{L_2}\otimes_BM\rightarrow \Omega_{L_2}\otimes_BM.
\end{equation*}
In particular, $\Omega_{L_2}\otimes_BM$ is a $U(L_1)$-module, and for any $\beta\in L_1,\alpha\in L_2$ we have:
\begin{equation}\label{equation 2 side-changing}
     (\beta(d\otimes m))\alpha= (d\alpha)\otimes \beta m- d\otimes(\alpha \beta m) =(d\alpha)\otimes \beta m- d\otimes(\beta\alpha  m)=\beta((d\otimes m)\alpha).
\end{equation}
Thus, $\Omega_{L_2}\otimes_BM$ is a $U(L_1)\otimes_K U(L_2)^{\op}$-module. Furthermore, a straightforward calculation shows that this construction is functorial.\bigskip

Let now $M$ be a right $U(L_2)$-module, and consider the $B$-module:
\begin{equation*}
   \Hom_B(\Omega_{L_2},M)=\Omega_{L_2}^{-1}\otimes_BM.
\end{equation*}
Again, this has a canonical structure as a left $U(L_2)$-module given by:
\begin{equation*}
    \alpha \phi(d)=\phi(d\alpha)-\phi(d)\alpha, \textnormal{ for } d \in \Omega_{L_2}, \phi\in \Hom_B(\Omega_{L_2},M),\textnormal{ and } \alpha\in L_2.
\end{equation*}
Assume that $M$ is a $U(L_1)\otimes_K U(L_2)^{\op}$-module. Arguing as above, it follows that the action of every $x\in U(L_1)$ is $B$-linear. Hence, we have a structure of a left $U(L_1)$-module on  $\Hom_B(\Omega_{L_2},M)$ given by composition. Furthermore, it is clear that this makes $\Hom_B(\Omega_{L_2},M)$ a $U(L_1)\otimes_K U(L_2)$-module, and that this operation is functorial. We may condense this discussion into the following proposition:
\begin{prop}\label{prop side-changing for abstract bimodules}
There is an equivalence of abelian categories:
\begin{equation*}
    \Omega_{L_2}^{-1}\otimes_B-:\Mod(U(L_1)\otimes_K U(L_2)^{\op})\leftrightarrows \Mod(U(L_1)\otimes_K U(L_2)):\Omega_{L_2}\otimes_B-\textnormal{.}
\end{equation*}
\end{prop}
\begin{proof}
  This follows by \cite[Proposition 3.1]{ardakov2015d}. One just needs to check that the canonical isomorphisms:
  \begin{equation*}
      M\rightarrow \Hom_B(\Omega_{L_2},\Omega_{L_2}\otimes_B M), \quad\Omega_{L_2}\otimes_B\Hom_B(\Omega_{L_2},N) \rightarrow N,
  \end{equation*}
  are $U(L_1)$-linear, and this is clear by the construction of the action.
\end{proof}
Following \cite{hotta2007d}, we can get a more explicit version of this equivalences whenever $L_2$ is free as a $B$-module. Indeed, assume that $L_2$ is a free $B$-module, with basis $\alpha_1,\cdots,\alpha_n$, and let $d\alpha_1,\cdots,d\alpha_n$ denote the dual basis of the dual vector bundle $\Hom_B(L_2,B)$. In particular,  we have:
\begin{equation*}
    \Omega_{L_2}=Bd\alpha_1\wedge\cdots\wedge d\alpha_n.
\end{equation*}
As in Proposition \ref{prop anti-homomorphism}, there is an isomorphism of algebras $T:U(L_2)\rightarrow U(L_2)^{\op}$ uniquely determined by the expression:
\begin{equation*}
    T(\sum f_r \alpha_1^{r_1}\cdots \alpha_n^{r_n})=\left( \sum(-1)^{\sum r_i}\alpha_1^{r_1}\cdots\alpha_d^{r_d}f_r\right)^t:=\sum(-1)^{\sum r_i}\alpha_1^{r_1}\cdots\alpha_d^{r_d}f_r.
\end{equation*}
Thus, we get an identification of the categories of left and right $U(L_2)$-modules given by extension of scalars along the map $T:U(L_2)\rightarrow U(L_2)^{\op}$. The idea is comparing this identification with the one obtained above.\\
It can be shown that the action of $\beta\in U(L_2)$ on an element $fd\alpha_1\wedge\cdots\wedge d\alpha_n\in \Omega_{L_2}$ is given by the following formula:
\begin{equation*}
    \beta(fd\alpha_1\wedge\cdots\wedge d\alpha_n)=(\beta^tf)d\alpha_1\wedge\cdots\wedge d\alpha_n.
\end{equation*}
Hence, we arrive at the following proposition:
\begin{prop}\label{prop equivalence of pullback and side-change for modules}
Let $M$ be an $U(L_2)$-module. Then we have a functorial isomorphism of $U(L_2)$-modules:
\begin{equation*}
    \Omega_{L_2}\otimes_BM\rightarrow T^*M=U(L_2)^{\op}\otimes_{U(L_2)}M.
\end{equation*}
In particular, there is a natural equivalence of functors $ \Omega_{L_2}\otimes_B- \simeq T^*$. Similarly, there is also a natural equivalence $\Omega_{L_2}^{-1}\otimes_B- \simeq T^{-1,*}$.
\end{prop}
\begin{proof}
This is also  done in \cite[Section 1.2]{hotta2007d}.
\end{proof}
Let $X$ be a separated smooth rigid analytic variety with a Lie algebroid $\mathscr{L}$. The goal of this section is extending the previous discussion to the categories of Ind-Banach modules over $\wideparen{U}(\mathscr{L}^2)$, and $\wideparen{E}(\mathscr{L})$. Let us start with the following lemmas:
\begin{Lemma}\label{Lemma preparation Lemma side-changing}
Let $X,Y$ be two quasi-compact, quasi-separated rigid spaces. Then every quasi-compact admissible open $U\subset X\times Y$ admits an admissible cover $\left( U_i\right)_{i=1}^n$ satisfying the following:
\begin{enumerate}[label=(\roman*)]
    \item Each of the $U_i$ is quasi-compact and quasi-separated.
    \item $p_Y(U_i)\subset Y$ is an affinoid admissible open subspace. 
\end{enumerate}
Furthermore, if $X$ and $Y$ are affinoid, then the $U_i$ can be taken to be affinoid.
\end{Lemma}
\begin{proof}
As $X,Y$ are quasi-compact and quasi-separated, it follows that $X\times Y$ is also quasi-compact and quasi-separated. Hence, by Theorem \ref{teo flat maps are open}, it follows that $p_Y(U)\subset Y$ is a quasi-compact admissible open. Hence, we get a factorization $U\rightarrow X\times p_Y(U)\rightarrow X\times Y$. Notice that all these maps are open immersions of quasi-compact admissible open subspaces. Therefore, as $X\times Y$ is quasi-separated, the map $U\rightarrow  X\times p_Y(U)$ is quasi-compact. Let $\left( V_i\right)_{i=1}^n$ be a finite affinoid cover of $p_Y(U)$. Then the family $U_i= U\cap \left(X\times V_i \right)$ satisfies our requirements.
\end{proof}
\begin{Lemma}\label{Lemma side-changing for bimodules}
Assume $X=\Sp(A)$ is affinoid and $\mathscr{L}$ is free. For every quasi-compact admissible open subspace $U\subset X$ there is a pair of mutually inverse equivalences of quasi-abelian categories:
\begin{multline*}
\Omega_{\mathscr{L}(X)}\overrightarrow{\otimes}_A-:\Mod_{\Indban}\left(\wideparen{U}(\mathscr{L})(U)\overrightarrow{\otimes}_K\wideparen{U}(\mathscr{L})(X)\right)\leftrightarrows\\ \Mod_{\Indban}\left(\wideparen{U}(\mathscr{L})(U)\overrightarrow{\otimes}_K\wideparen{U}(\mathscr{L})(X)^{\op}\right):\Omega_{\mathscr{L}(X)}^{-1}\overrightarrow{\otimes}_A-.
\end{multline*}
\end{Lemma}
\begin{proof}
As both $U$ and $X$ are quasi-compact, it follows that $\wideparen{U}(\mathscr{L})(U)$ and $\wideparen{U}(\mathscr{L})(X)$ are Fréchet spaces. In particular, they are metrizable. Hence, it follows by \cite[Proposition 4.25]{bode2021operations} that the Ind-Banach $K$-algebras:
\begin{equation*}
  \wideparen{U}(\mathscr{L})(U)\overrightarrow{\otimes}_K\wideparen{U}(\mathscr{L})(X), \quad \wideparen{U}(\mathscr{L})(U)\overrightarrow{\otimes}_K\wideparen{U}(\mathscr{L})(X)^{\op},  
\end{equation*}
are the images under the dissection functor $\operatorname{diss}:\widehat{\mathcal{B}}c_K\rightarrow \Indban$ of the following complete bornological $K$-algebras:
\begin{equation*}
  \wideparen{U}(\mathscr{L})(U)\widehat{\otimes}_K\wideparen{U}(\mathscr{L})(X), \quad   \wideparen{U}(\mathscr{L})(U)\widehat{\otimes}_K\wideparen{U}(\mathscr{L})(X)^{\op}.
\end{equation*}
The first step is showing that there are mutually inverse equivalences of quasi-abelian categories:
\begin{multline*}
\Omega_{\mathscr{L}(X)}\widehat{\otimes}_A-:\Mod_{\widehat{\mathcal{B}}c_K}\left(\wideparen{U}(\mathscr{L})(U)\widehat{\otimes}_K\wideparen{U}(\mathscr{L})(X)\right)\leftrightarrows\\
\Mod_{\widehat{\mathcal{B}}c_K}\left(\wideparen{U}(\mathscr{L})(U)\widehat{\otimes}_K\wideparen{U}(\mathscr{L})(X)^{\op}\right):
\Omega_{\mathscr{L}(X)}^{-1}\widehat{\otimes}_A-.
\end{multline*}
By \cite[Corollary 5.19]{bode2021operations}, we have that $\wideparen{U}(\mathscr{L}(X))$ is the completion of $U(\mathscr{L}(X))$ with respect to the induced bornology. As bornological completion commutes with tensor products, $\wideparen{U}(\mathscr{L})(U)\widehat{\otimes}_K\wideparen{U}(\mathscr{L})(X)$ must be the completion of $\wideparen{U}(\mathscr{L})(U)\otimes_KU(\mathscr{L})(X)$. Thus, we have a canonical equivalence of quasi-abelian categories:
    \begin{equation*}
     \Mod_{\widehat{\mathcal{B}}c_K}(\wideparen{U}(\mathscr{L})(U)\otimes_KU(\mathscr{L})(X))\rightarrow \Mod_{\widehat{\mathcal{B}}c_K}(\wideparen{U}(\mathscr{L})(U)\widehat{\otimes}_K\wideparen{U}(\mathscr{L})(X)).   
    \end{equation*}
Similarly, $\wideparen{U}(\mathscr{L})(U)\widehat{\otimes}_K\wideparen{U}(\mathscr{L})(X)^{\op}$ is the completion of $\wideparen{U}(\mathscr{L})(U)\otimes_KU(\mathscr{L})(X)^{\op}$. Thus, we have an equivalence: 
\begin{equation*}
 \Mod_{\widehat{\mathcal{B}}c_K}(\wideparen{U}(\mathscr{L})(U)\otimes_KU(\mathscr{L})(X)^{\op})\rightarrow \Mod_{\widehat{\mathcal{B}}c_K}(\wideparen{U}(\mathscr{L})(U)\widehat{\otimes}_K\wideparen{U}(\mathscr{L})(X)^{\op}). 
\end{equation*}
Let $\mathcal{M}\in \Mod_{\widehat{\mathcal{B}}c_K}(\wideparen{U}(\mathscr{L})(U)\otimes_KU(\mathscr{L})(X))$, and regard $\Omega_{\mathscr{L}(X)}\otimes_A\mathcal{M}$ as a bornological space. By Proposition \ref{prop side-changing for abstract bimodules}, this is an abstract $\wideparen{U}(\mathscr{L})(U)\otimes_KU(\mathscr{L})(X)^{\op}$-module. In order to see that it is a bornological module, it suffices to see that the actions of $\wideparen{U}(\mathscr{L})(U)$ and $U(\mathscr{L})(X)^{\op}$ are bounded. The action of $U(\mathscr{L})(X)^{\op}$ is bounded by \cite[Proposition 6.10]{bode2021operations}.\\
Let now $S\subset \wideparen{U}(\mathscr{L})(U)$, $V\subset \Omega_{\mathscr{L}(X)}$ and $W\subset \mathcal{M}$ be bounded $\mathcal{R}$-submodules. By definition, the bornology on $\Omega_{\mathscr{L}(X)}\otimes_A\mathcal{M}$ is the quotient bornology with respect to the map:
\begin{equation*}
    q:\Omega_{\mathscr{L}(X)}\otimes_K\mathcal{M}\rightarrow \Omega_{\mathscr{L}(X)}\otimes_A\mathcal{M}.
\end{equation*}
Hence, it is generated by bounded subsets of the form $q(V\otimes_{\mathcal{R}}W)$. Every element  $x\in \wideparen{U}(\mathscr{L})(U)$ acts on $\Omega_{\mathscr{L}(X)}\otimes_A\mathcal{M}$ by $\Id\otimes x$. Thus, we have:
\begin{equation*}
    \bigcup_{x\in S}x(q(V\otimes_{\mathcal{R}}W)) \subset q(V\otimes_{\mathcal{R}}\bigcup_{x\in S}x(W)). 
\end{equation*}
As the action of $\wideparen{U}(\mathscr{L})(U)$ on $\mathcal{M}$ is bounded, 
$\cup_{x\in S}x(W)$ is bounded. Thus, the action of $\wideparen{U}(\mathscr{L})(U)$ on $\Omega_{\mathscr{L}(X)}\otimes_A\mathcal{M}$ is bounded. By taking completions, we get our desired functor. The second functor is constructed analogously. In order to see that they are mutually inverse, we can use Proposition \ref{prop side-changing for abstract bimodules},
to obtain natural isomorphisms:
\begin{equation*}
          \mathcal{M}\rightarrow \Omega_{\mathscr{L}(X)}^{-1}\otimes_A\left(\Omega_{\mathscr{L}(X)}\otimes_A \mathcal{M}\right), \textnormal{ and } \Omega_{\mathscr{L}(X)}\otimes_A\left(\Omega_{\mathscr{L}(X)}^{-1}\otimes_A \mathcal{N}\right) \rightarrow \mathcal{N},
\end{equation*}
which are bounded by \cite[Theorem 6.11]{bode2021operations}. The completion of the maps yields the desired equivalence.\\
As a consequence, we get the following equivalences between the left hearts:
\begin{multline*}
\Omega_{\mathscr{L}(X)}\widetilde{\otimes}_A-:LH\left(\Mod_{\widehat{\mathcal{B}}c_K}\left(\wideparen{U}(\mathscr{L})(U)\widehat{\otimes}_K\wideparen{U}(\mathscr{L})(X)\right)\right)\leftrightarrows \\ LH \left(\Mod_{\widehat{\mathcal{B}}c_K}\left(\wideparen{U}(\mathscr{L})(U)\widehat{\otimes}_K\wideparen{U}(\mathscr{L})(X)^{\op}\right)\right):\Omega_{\mathscr{L}(X)}^{-1}\widetilde{\otimes}_A-.
\end{multline*}
In virtue of \cite[Proposition 4.31]{bode2021operations}, we have the following canonical identifications of abelian categories:
\begin{multline*}
LH\left(\Mod_{\widehat{\mathcal{B}}c_K}\left(\wideparen{U}(\mathscr{L})(U)\widehat{\otimes}_K\wideparen{U}(\mathscr{L})(X)\right)\right)\\
\cong    LH\left(\Mod_{\Indban}\left(\wideparen{U}(\mathscr{L})(U)\overrightarrow{\otimes}_K\wideparen{U}(\mathscr{L})(X)\right)\right),
\end{multline*}
\begin{multline*}
LH\left(\Mod_{\widehat{\mathcal{B}}c_K}\left(\wideparen{U}(\mathscr{L})(U)\widehat{\otimes}_K\wideparen{U}(\mathscr{L})(X)^{\op}\right)\right)\\
\cong
LH\left(\Mod_{\Indban}\left(\wideparen{U}(\mathscr{L})(U)\overrightarrow{\otimes}_K\wideparen{U}(\mathscr{L})(X)^{\op}\right)\right).
\end{multline*}
A fortiori, we obtain an equivalence of abelian categories:
\begin{multline*}
\Omega_{\mathscr{L}(X)}\widetilde{\otimes}_A-:LH\left(\Mod_{\Indban}\left(\wideparen{U}(\mathscr{L})(U)\overrightarrow{\otimes}_K\wideparen{U}(\mathscr{L})(X)\right)\right)\rightarrow\\
LH \left(\Mod_{\Indban}\left(\wideparen{U}(\mathscr{L})(U)\overrightarrow{\otimes}_K\wideparen{U}(\mathscr{L})(X)^{\op}\right)\right),
\end{multline*}
with inverse given by the following functor:
\begin{multline*}
\Omega_{\mathscr{L}(X)}^{-1}\widetilde{\otimes}_A-:LH\left(\Mod_{\Indban}\left(\wideparen{U}(\mathscr{L})(U)\overrightarrow{\otimes}_K\wideparen{U}(\mathscr{L})(X)^{\op}\right)\right)\rightarrow \\
LH\left(\Mod_{\Indban}\left(\wideparen{U}(\mathscr{L})(U)\overrightarrow{\otimes}_K\wideparen{U}(\mathscr{L})(X)\right)\right).    
\end{multline*}
We need to show that this induces an equivalence of quasi-abelian categories between the associated categories of Ind-Banach modules.\\
However, as $\Omega_{\mathscr{L}(X)}$ and $\Omega_{\mathscr{L}(X)}^{-1}$ are finite projective $A$-modules, they are a direct summand in a finite-free $A$-module. Therefore, the functor $\Omega_{\mathscr{L}(X)}\widetilde{\otimes}_A-$ maps Ind-Banach spaces to Ind-Banach spaces, and so maps the essential image of:
\begin{equation*}
    \Mod_{\Indban}\left(\wideparen{U}(\mathscr{L})(U)\overrightarrow{\otimes}_K\wideparen{U}(\mathscr{L})(X)\right)
\end{equation*}
inside its left heart to $\Mod_{\Indban}\left(\wideparen{U}(\mathscr{L})(U)\overrightarrow{\otimes}_K\wideparen{U}(\mathscr{L})(X)^{\op}\right)$. As the analogous  phenomena holds for $\Omega_{\mathscr{L}(X)}^{-1}\widetilde{\otimes}_A-$, the result holds.
\end{proof}
Notice that the procedure above is functorial. Recall the following notation:
\begin{defi*} 
Let $f:X\rightarrow Y$ be a morphism of rigid analytic spaces. Then we define the derived pullback along $f$ to be the following functor:
\begin{align*}
    f^*:\operatorname{D}(&\OX_Y)\rightarrow \operatorname{D}(\OX_X),\\    
    &\mathcal{M}\mapsto \OX_X\overrightarrow{\otimes}^{\mathbb{L}}_{f^{-1}\OX_Y}f^{-1}\mathcal{M}
\end{align*}
\end{defi*}

We may use this to show the following proposition:
\begin{prop}\label{prop side changing for bimodules}
Let $X$ be a separated smooth  rigid variety equipped with a Lie algebroid $\mathscr{L}$. There is a pair of mutually inverse equivalences of quasi-abelian categories:
\begin{multline*}
p_2^{*}\Omega_{\mathscr{L}}\overrightarrow{\otimes}_{\OX_{X^2}}-:\Mod_{\Indban}(\wideparen{U}(\mathscr{L}^2))\leftrightarrows\\ \Mod_{\Indban}(\wideparen{E}(\mathscr{L})):p_2^{*}\Omega_{\mathscr{L}}^{-1}\overrightarrow{\otimes}_{\OX_{X^2}}-.
\end{multline*} 
\end{prop}
\begin{proof}
First,  notice that by definition $\Omega_{\mathscr{L}}$ is a line bundle in $X$. In particular, both $p_2^{*}\Omega_{\mathscr{L}}$ and $p_2^{*}\Omega_{\mathscr{L}}^{-1}$ are concentrated in degree zero, so the above functors make sense. Next, we may assume that $X=\Sp(A)$ is affinoid and $\mathscr{L}$ is free. In this situation, both $\Omega_{\mathscr{L}}$ and $\Omega_{\mathscr{L}}^{-1}$ are globally free $\OX_X$-modules. In particular, for every $\mathcal{M} \in \Mod_{\Indban}(\wideparen{U}(\mathscr{L}^2))$, and every affinoid subdomain $V\subset X^2$ we have the following identity:
\begin{equation*}   \left(p_2^{-1}\Omega_{\mathscr{L}}\overrightarrow{\otimes}_{p_2^{-1}\OX_X}\mathcal{M}\right)(V)=\Omega_{\mathscr{L}}(p_2(V))\overrightarrow{\otimes}_{\OX_X(p_2(V))}\mathcal{M}(V).
\end{equation*}
For simplicity, let $\Tilde{M}=p_2^{-1}\Omega_{\mathscr{L}}\overrightarrow{\otimes}_{p_2^{-1}\OX_X}\mathcal{M}$. In order to define an action of $\wideparen{E}(\mathscr{L})$  on $\Tilde{M}$, it suffices to define it at the level of presheaves. In particular, for every affinoid subdomain $V\subset X^2$ we need to define an action:
\begin{multline*}
    \OX_{X^2}(V)\overrightarrow{\otimes}_{\OX_X(p_1(V))\overrightarrow{\otimes}_K\OX_X(p_2(V))}\left(\wideparen{U}(\mathscr{L})(p_1(V))\overrightarrow{\otimes}_K \wideparen{U}(\mathscr{L})(p_2(V))^{\op}\right)\times \Tilde{\mathcal{M}}(V)\\
    \rightarrow \Tilde{\mathcal{M}}(V),
\end{multline*}
functorial in  $\mathcal{M}(V)\in \Mod_{\Indban}(\wideparen{U}(\mathscr{L})(p_1(V))\overrightarrow{\otimes}_K \wideparen{U}(\mathscr{L})(p_2(V)))$, and $V\subset X^2$. As we mentioned above, $p_2^{*}\Omega_{\mathscr{L}}\simeq \OX_{X^2}$ as an $\OX_{X^2}$-module. This isomorphism induces an isomorphism of $\OX_{X^2}$-modules $\Tilde{\mathcal{M}} \cong \mathcal{M}$. Thus, we reduce the problem to obtaining a functorial action:
\begin{equation*}
    \left(\wideparen{U}(\mathscr{L})(p_1(V))\overrightarrow{\otimes}_K \wideparen{U}(\mathscr{L})(p_2(V))^{\op}\right)\times \Tilde{\mathcal{M}}(V)
    \rightarrow \Tilde{\mathcal{M}}(V),
\end{equation*}
such that the restriction to $\OX_X(p_1(V))\overrightarrow{\otimes}_K\OX_X(p_2(V))$ agrees with the restriction of the action of $\OX_{X^2}(V)$ on $\Tilde{\mathcal{M}}(V)$.\\ 
As $V$ is affinoid, $p_1(V)$ and $p_2(V)$ are quasi-compact open subspaces of $X$. Thus, $\wideparen{U}(\mathscr{L})(p_1(V))$, and $\wideparen{U}(\mathscr{L})(p_2(V))$ are Fréchet algebras. As before, we may apply \cite[Proposition 4.31]{bode2021operations} to conclude that $\wideparen{U}(\mathscr{L})(p_1(V))\overrightarrow{\otimes}_K \wideparen{U}(\mathscr{L})(p_2(V))$ is the image under $\operatorname{diss}:\widehat{\mathcal{B}}c_K\rightarrow \Indban$ of the complete bornological $K$-algebra:
\begin{equation*}
    \wideparen{U}(\mathscr{L})(p_1(V))\widehat{\otimes}_K \wideparen{U}(\mathscr{L})(p_2(V)).
\end{equation*}
Thus, every object in $\Mod_{\Indban}(\wideparen{U}(\mathscr{L})(p_1(V))\overrightarrow{\otimes}_K \wideparen{U}(\mathscr{L})(p_2(V)))$ is the cokernel of a strict morphism  of complete bornological $\wideparen{U}(\mathscr{L})(p_1(V))\widehat{\otimes}_K \wideparen{U}(\mathscr{L})(p_2(V))$-modules. In particular, there is a direct sum of Banach spaces $\bigoplus_{i\in I}W_i$ such that there is a strict epimorphism:
\begin{equation*}
   \left(\wideparen{U}(\mathscr{L})(p_1(V))\overrightarrow{\otimes}_K \wideparen{U}(\mathscr{L})(p_2(V))\right)\overrightarrow{\otimes}_K\bigoplus_iW_i\rightarrow \mathcal{M}(V). 
\end{equation*}
Hence, we may replace $\mathcal{M}$ with $\wideparen{U}(\mathscr{L}^2)\overrightarrow{\otimes}_K\bigoplus_iW_i$, and assume that $\mathcal{M}$ is a sheaf of complete bornological spaces (\emph{cf}. the proof of \cite[Theorem 6.11]{bode2021operations}). Again, freeness of $\Omega_{\mathscr{L}}$ implies that $\Tilde{\mathcal{M}}$ is a sheaf of complete bornological spaces.\\
Thus, we have reduced the problem to defining a bounded action:
\begin{equation*}
    \left(\wideparen{U}(\mathscr{L})(p_1(V))\otimes_K \wideparen{U}(\mathscr{L})(p_2(V))^{\op}\right)\times \Tilde{\mathcal{M}}(V)\rightarrow \Tilde{\mathcal{M}}(V).
\end{equation*}
Unlike for arbitrary Ind-Banach spaces, the fact that $\Tilde{\mathcal{M}}(V)$ is a bornological space implies that it has an underlying $K$-vector space. Hence, we may define an action by using the formulas in equations $(\ref{equation side-changing})$ and $(\ref{equation 2 side-changing})$. We just need to show that this action is bounded. By definition of the projective tensor product of bornological spaces, it suffices to show that the actions of  $\wideparen{U}(\mathscr{L})(p_1(V))$ and  $\wideparen{U}(\mathscr{L})(p_2(V))^{\op}$ on 
$\Tilde{\mathcal{M}}(V)$ are bounded.\\ 
The fact that the action of $\wideparen{U}(\mathscr{L})(p_1(V))$ is bounded follows as in Lemma \ref{Lemma side-changing for bimodules}. For the action of  $\wideparen{U}(\mathscr{L})(p_2(V))^{\op}$, choose a finite cover $\left( V_i\right)_{i=1}^n$ satisfying the conditions from Lemma \ref{Lemma preparation Lemma side-changing}. By Theorem \ref{teo flat maps are open}, together with functoriality of equation $(\ref{equation side-changing})$, we have the following commutative diagram:
\begin{equation*}
\begin{tikzcd}
\wideparen{U}(\mathscr{L})(p_1(V))\times \Tilde{\mathcal{M}}(V) \arrow[d] \arrow[r]        & \Tilde{\mathcal{M}}(V) \arrow[d]      \\
\prod_{i=1}^n\wideparen{U}(\mathscr{L})(p_1(V_i))\times \Tilde{\mathcal{M}}(V_i) \arrow[r] & \prod_{i=1}^n\Tilde{\mathcal{M}}(V_i)
\end{tikzcd}
\end{equation*}
where the vertical arrows are strict injections, and horizontal arrows are the action given by $(\ref{equation side-changing})$. By Lemma \ref{Lemma side-changing for bimodules} the lower horizontal map is bounded. Hence, as the vertical arrows are strict injections, it follows that the upper horizontal map is also bounded. After sheafification and completion, it follows that $\Tilde{M}$ is a module in $\Mod_{\Indban}(\wideparen{E}(\mathscr{L}))$. The construction of the other functor is completely analogous, and the fact that they are mutually inverse follows by Lemma \ref{Lemma side-changing for bimodules}.
\end{proof}
\begin{defi}
We will use the following notation to denote the functors defined in the previous proposition:
\begin{multline*}
\operatorname{S}:=p_2^{*}\Omega_{\mathscr{L}}\overrightarrow{\otimes}_{\OX_{X^2}}-:\Mod_{\Indban}(\wideparen{U}(\mathscr{L}^2))\leftrightarrows\\ \Mod_{\Indban}(\wideparen{E}(\mathscr{L})):p_2^{*}\Omega_{\mathscr{L}}^{-1}\overrightarrow{\otimes}_{\OX_{X^2}}-=:\operatorname{S}^{-1},
\end{multline*}
and we will call $\operatorname{S}$ and $\operatorname{S}^{-1}$ the side-switching operators for bimodules.
\end{defi}
As the side-switching operators are equivalences of quasi-abelian categories, they are strongly exact. Hence, they induce an equivalence of triangulated categories:
\begin{equation*}
    \operatorname{S}:\operatorname{D}(\wideparen{U}(\mathscr{L}^2))\leftrightarrows \operatorname{D}(\wideparen{E}(\mathscr{L})):\operatorname{S}^{-1}.
\end{equation*}
We can use this derived equivalence to extend the side switching operators to the respective left hearts. In particular, we define:
\begin{equation*}
    \operatorname{S}:LH(\Mod_{\Indban}(\wideparen{U}(\mathscr{L}^2)))\rightarrow LH(\Mod_{\Indban}(\wideparen{E}(\mathscr{L}))),
\end{equation*}
as the following composition:
\begin{multline*}
  LH(\Mod_{\Indban}(\wideparen{U}(\mathscr{L}^2)))\rightarrow \operatorname{D}(\wideparen{U}(\mathscr{L}^2))\xrightarrow[]{\operatorname{S}} \operatorname{D}(\wideparen{E}(\mathscr{L}))\\
  \xrightarrow[]{\operatorname{H}^0}LH(\Mod_{\Indban}(\wideparen{E}(\mathscr{L}))),
\end{multline*}
and  $\operatorname{S}^{-1}:LH(\Mod_{\Indban}(\wideparen{E}(\mathscr{L})))\rightarrow LH(\Mod_{\Indban}(\wideparen{U}(\mathscr{L}^2)))$ analogously. 
\begin{prop}
The functors:
\begin{equation*}
   \operatorname{S}:LH(\Mod_{\Indban}(\wideparen{U}(\mathscr{L}^2)))\leftrightarrows LH(\Mod_{\Indban}(\wideparen{E}(\mathscr{L}))):\operatorname{S}^{-1},
\end{equation*}
are mutually inverse equivalences of abelian categories.
\end{prop}
\begin{proof}
This is a consequence of \cite[Proposition 1.2.34]{schneiders1999quasi}.
\end{proof}
Notice that, by construction, we have a commutative diagram:
\begin{equation}\label{equation extension of S to the left heart}
\begin{tikzcd}
\operatorname{S}:LH(\Mod_{\Indban}(\wideparen{U}(\mathscr{L}^2))) \arrow[r]                               & LH(\Mod_{\Indban}(\wideparen{E}(\mathscr{L}))):\operatorname{S}^{-1} \arrow[l, shift left=-3]                                \\
\operatorname{S}:\Mod_{\Indban}(\wideparen{U}(\mathscr{L}^2)) \arrow[u, "I"] \arrow[r] & \Mod_{\Indban}(\wideparen{E}(\mathscr{L})):\operatorname{S}^{-1} \arrow[u, "I"] \arrow[l, shift right=3]
\end{tikzcd}
\end{equation}
We will now construct an analogous side-switching operation for right $\wideparen{E}(\mathscr{L})$-modules. First, we notice that we have the following identities:
\begin{multline*}
 \wideparen{U}(\mathscr{L}^2)^{\op}=\OX_{X^2}\overrightarrow{\otimes}_{p_1^{-1}\OX_X\overrightarrow{\otimes}_Kp_2^{-1}\OX_X} \left(p_1^{-1}\wideparen{U}(\mathscr{L})^{\op}\overrightarrow{\otimes}_Kp_2^{-1}\wideparen{U}(\mathscr{L})^{\op}\right),\\
 \quad    \wideparen{E}(\mathscr{L})^{\op}=\OX_{X^2}\overrightarrow{\otimes}_{p_1^{-1}\OX_X\overrightarrow{\otimes}_Kp_2^{-1}\OX_X} \left(p_1^{-1}\wideparen{U}(\mathscr{L})^{\op}\overrightarrow{\otimes}_Kp_2^{-1}\wideparen{U}(\mathscr{L})\right).
\end{multline*}
Where the Ind-Banach algebra structure on 
$\wideparen{U}(\mathscr{L}^2)^{\op}$ extends the ones on $\OX_{X^2}$ and $p_1^{-1}\wideparen{U}(\mathscr{L})^{\op}\overrightarrow{\otimes}_Kp_2^{-1}\wideparen{U}(\mathscr{L})^{\op}$, and the same holds for $\wideparen{E}(\mathscr{L})^{\op}$. Thus, mimicking the arguments for left modules, we may define side-switching operations for right modules in the following manner:
\begin{align*}
&\operatorname{S}_r:=p_2^{*}\Omega_{\mathscr{L}}^{-1}\overrightarrow{\otimes}_{\OX_{X^2}}-:\Mod_{\Indban}(\wideparen{U}(\mathscr{L}^2)^{\op})\rightarrow \Mod_{\Indban}(\wideparen{E}(\mathscr{L})^{\op}),\\
&\operatorname{S}_r^{-1}:=p_2^{*}\Omega_{\mathscr{L}}\overrightarrow{\otimes}_{\OX_{X^2}}-:\Mod_{\Indban}(\wideparen{E}(\mathscr{L})^{\op})\rightarrow \Mod_{\Indban}(\wideparen{U}(\mathscr{L}^2)^{\op}).
\end{align*}
The fact that these functors are well-defined mutually inverse equivalences of quasi-abelian categories is analogous to Proposition \ref{prop side changing for bimodules}.\bigskip

There is yet another pair of side switching operators which are relevant to our setting. Namely, by definition,
$\wideparen{U}(\mathscr{L}^2)$ is the sheaf of complete enveloping algebras of the Lie algebroid $\mathscr{L}^2$ on $X^2$. Thus, the appropriate version of \cite[Theorem 6.11]{bode2021operations}, yields an equivalence of quasi-abelian categories:
\begin{equation*} 
\Omega_{\mathscr{L}^2}\overrightarrow{\otimes}_{\OX_{X^2}}-:\Mod_{\Indban}(\wideparen{U}(\mathscr{L}^2))\leftrightarrows \Mod_{\Indban}(\wideparen{U}(\mathscr{L}^2)^{\op}):\Omega^{-1}_{\mathscr{L}^2}\overrightarrow{\otimes}_{\OX_{X^2}}-.
\end{equation*}
 This equivalence, together with the operations defined above, allows us to obtain a side-switching operation for Ind-Banach $\wideparen{E}(\mathscr{L})$-modules. Namely, we have the following equivalence of quasi-abelian categories:
 \begin{equation*}  \operatorname{S}_r\left(\Omega_{\mathscr{L}^2}\overrightarrow{\otimes}_{\OX_{X^2}}\operatorname{S}^{-1}(-) \right):\Mod_{\Indban}(\wideparen{E}(\mathscr{L}))\rightarrow \Mod_{\Indban}(\wideparen{E}(\mathscr{L})^{\op}).   
 \end{equation*}
We will now obtain an explicit description of this equivalence:
\begin{Lemma}
There is an isomorphism of complete bornological $\wideparen{U}(\mathscr{L}^2)^{\op}$-modules:
\begin{equation*}
\Omega_{\mathscr{L}^2}=\Omega_{\mathscr{L}}\overrightarrow{\boxtimes}\Omega_{\mathscr{L}}=p_1^{*}\Omega_{\mathscr{L}}\overrightarrow{\otimes}_{\OX_{X^2}}p_2^{*}\Omega_{\mathscr{L}}.
\end{equation*}
\end{Lemma}
\begin{proof}
 We can assume that $X$ is affine and $\mathscr{L}$ is free, in which case  the proof of \cite[Tag 0FMA]{stacks-project} carries over to this setting without major changes.
\end{proof}
\begin{prop}\label{prop equivalence left and right E(L)-modules}
 Consider the following functor:
\begin{equation*}
  \operatorname{T}:=\left(\Omega_{\mathscr{L}}\overrightarrow{\boxtimes}\Omega_{\mathscr{L}}^{-1}\right)\overrightarrow{\otimes}_{\OX_{X^2}}-:\Mod_{\Indban}(\wideparen{E}(\mathscr{L}))\rightarrow \Mod_{\Indban}(\wideparen{E}(\mathscr{L})^{\op}).
\end{equation*}
Then there is a canonical natural equivalence of functors:
\begin{equation*}
    \operatorname{T} \xrightarrow[]{\cong}\operatorname{S}_r\left(\Omega_{\mathscr{L}^2}\overrightarrow{\otimes}_{\OX_{X^2}}\operatorname{S}^{-1}(-) \right).
\end{equation*}
In particular, $\operatorname{T}$ is an equivalence of quasi-abelian categories.
\end{prop}
\begin{proof}
This is a routine calculation.
\end{proof}
As before, the fact that $\operatorname{T}$ is an equivalence of quasi abelian categories implies that it is strongly exact. Hence, it extends to the derived category. For simplicity, we will also denote this extension by $\operatorname{T}$. Recall from Lemma \ref{Lemma right module structure on pushforward} that the complex 
$\EDelta\wideparen{U}(\mathscr{L})$ is canonically an object in $\operatorname{D}(\wideparen{E}(\mathscr{L})^{\op})$. The explicit description of $\operatorname{T}$ given above allows us to show the following corollary:
\begin{coro}\label{coro image of enveloping algebra under T}
There is an isomorphism in $\operatorname{D}(\wideparen{E}(\mathscr{L})^{\op})$:
\begin{equation*}
    \operatorname{T}(\EDelta\wideparen{U}(\mathscr{L}))=\EDelta\wideparen{U}(\mathscr{L}).
\end{equation*}
\end{coro}
\begin{proof}
By construction of $\operatorname{T}$, we have the following identities:
\begin{align*}
\operatorname{T}(\EDelta\wideparen{U}(\mathscr{L})):=&p_1^{*}\Omega_{\mathscr{L}}\overrightarrow{\otimes}^{\mathbb{L}}_{\OX_{X^2}}\left(\EDelta\wideparen{U}(\mathscr{L})\right)\overrightarrow{\otimes}^{\mathbb{L}}_{\OX_{X^2}}p_2^{*}\Omega_{\mathscr{L}}^{-1}\\
=&p_1^{*}\Omega_{\mathscr{L}}\overrightarrow{\otimes}^{\mathbb{L}}_{\OX_{X^2}}\left(\OX_{X^2}\overrightarrow{\otimes}^{\mathbb{L}}_{p^{-1}_1\OX_X\overrightarrow{\otimes}_Kp^{-1}_2\OX_X}\Delta_*\wideparen{U}(\mathscr{L})\right)\overrightarrow{\otimes}^{\mathbb{L}}_{\OX_{X^2}}p_2^{*}\Omega_{\mathscr{L}}^{-1}\\
=&\OX_{X^2}\overrightarrow{\otimes}^{\mathbb{L}}_{p^{-1}_1\OX_X\overrightarrow{\otimes}_Kp^{-1}_2\OX_X}\Delta_*\left(\Omega_{\mathscr{L}}\overrightarrow{\otimes}_{\OX_X}\wideparen{U}(\mathscr{L})\overrightarrow{\otimes}_{\OX_X}\Omega_{\mathscr{L}}^{-1}\right)\\
=&\OX_{X^2}\overrightarrow{\otimes}^{\mathbb{L}}_{p^{-1}_1\OX_X\overrightarrow{\otimes}_Kp^{-1}_2\OX_X}\Delta_*(\wideparen{U}(\mathscr{L})^{\op})\\
=&\EDelta\wideparen{U}(\mathscr{L})^{\op}\\
=:&\EDelta\wideparen{U}(\mathscr{L}),
\end{align*}
where the third identity can be seen by applying $\Delta^{-1}$ to both sides, the third one is analogous to \cite[Lemma 1.2.7]{hotta2007d}, and the last one follows from the fact that the right $\wideparen{E}(\mathscr{L})$-module structure on $\EDelta\wideparen{U}(\mathscr{L})$ stems from the left $\wideparen{E}(\mathscr{L})^{\op}$-module structure on $\EDelta\wideparen{U}(\mathscr{L})^{\op}$, and both sheaves agree as sheaves of Ind-Banach spaces.
\end{proof}
This result will be instrumental to our study of Hochschild homology in the latter sections of the paper.\bigskip

We conclude this section by studying the behavior of 
the different derived functors discussed in Chapter \ref{Section sheaves of Ind-Banach spaces}
with respect to the side-switching operations:
\begin{Lemma}\label{Lemma free modules and side-switching}
Consider a sheaf of Ind-Banach spaces $\mathcal{V}\in \operatorname{Shv}(X,\Indban)$, together with modules $\mathcal{M}\in \Mod_{\Indban}(\wideparen{U}(\mathscr{L}^2))$, $\mathcal{N}\in \Mod_{\Indban}(\wideparen{E}(\mathscr{L}))$. There are canonical isomorphisms:
\begin{equation*}
   \operatorname{S}\left(\mathcal{M}\overrightarrow{\otimes}_K \mathcal{V}\right)= \operatorname{S}\left(\mathcal{M}\right)\overrightarrow{\otimes}_K \mathcal{V}, \quad   \operatorname{S}^{-1}\left(\mathcal{N}\overrightarrow{\otimes}_K \mathcal{V}\right)= \operatorname{S}^{-1}\left(\mathcal{N}\right)\overrightarrow{\otimes}_K \mathcal{V},
\end{equation*}  
in $\Mod_{\Indban}(\wideparen{E}(\mathscr{L}))$, and $\Mod_{\Indban}(\wideparen{U}(\mathscr{L}^2))$ respectively.
\end{Lemma}
\begin{proof}
By construction we have the following identity in of sheaves of Ind-Banach spaces:
\begin{multline*}
\operatorname{S}\left(\mathcal{M}\overrightarrow{\otimes}_K \mathcal{V}\right)=p_2^{-1}\Omega_{\mathscr{L}}\overrightarrow{\otimes}_{p_2^{-1}\OX_X}\left(\mathcal{M}\overrightarrow{\otimes}_K \mathcal{V} \right)\\
=\left(p_2^{-1}\Omega_{\mathscr{L}}\overrightarrow{\otimes}_{p_2^{-1}\OX_X}\mathcal{M}\right)\overrightarrow{\otimes}_K \mathcal{V}= \operatorname{S}\left(\mathcal{M}\right)\overrightarrow{\otimes}_K \mathcal{V}.    
\end{multline*}
We just need to show that these identifications are $\wideparen{E}(\mathscr{L})$-linear, which can be done locally. Thus, we may assume that $\mathcal{M}(X)$ and $\mathcal{V}(X)$ are complete bornological spaces, and then apply $(\ref{equation side-changing})$.
\end{proof}
\begin{Lemma}\label{Lemma side-changing and homs}
Choose $\mathcal{V}^{\bullet} \in \operatorname{D}(\operatorname{Shv}(X,LH(\widehat{\mathcal{B}}c_K)))$, along with $\mathcal{M}^{\bullet}\in \operatorname{D}(\wideparen{U}(\mathscr{L}^2))$, and  $\mathcal{N}^{\bullet}\in \operatorname{D}(\wideparen{E}(\mathscr{L}))$. There are canonical quasi-isomorphisms:
\begin{equation*}
   \operatorname{S}\left(\mathcal{M}^{\bullet}\overrightarrow{\otimes}^{\mathbb{L}}_K \mathcal{V}^{\bullet}\right)= \operatorname{S}\left(\mathcal{M}^{\bullet}\right)\overrightarrow{\otimes}^{\mathbb{L}}_K \mathcal{V}^{\bullet}, \quad   \operatorname{S}^{-1}\left(\mathcal{N}^{\bullet}\overrightarrow{\otimes}^{\mathbb{L}}_K \mathcal{V}^{\bullet}\right)= \operatorname{S}^{-1}\left(\mathcal{N}^{\bullet}\right)\overrightarrow{\otimes}^{\mathbb{L}}_K \mathcal{V}^{\bullet},
\end{equation*}
in $\operatorname{D}(\wideparen{E}(\mathscr{L}))$, and $\operatorname{D}(\wideparen{U}(\mathscr{L}^2))$ respectively.
\end{Lemma}
\begin{proof}
 We show the first isomorphism. By construction we have:
 \begin{align*}
 \operatorname{S}\left(\mathcal{M}^{\bullet}\overrightarrow{\otimes}^{\mathbb{L}}_K \mathcal{V}^{\bullet}\right)=&\operatorname{S}\left(\operatorname{Tot}_{\oplus}\left(I(\mathcal{M}^{\bullet})\widetilde{\otimes}^{\bullet,\bullet}_K I(\mathcal{V}^{\bullet})\right) \right)\\
 =& \operatorname{Tot}_{\oplus}\left(\operatorname{S}\left(I(\mathcal{M}^{\bullet})\widetilde{\otimes}^{\bullet,\bullet}_K I(\mathcal{V}^{\bullet})\right) \right)\\
=&\operatorname{Tot}_{\oplus}\left(\operatorname{S}\left(I(\mathcal{M}^{\bullet})\right)\widetilde{\otimes}^{\bullet,\bullet}_K I(\mathcal{V}^{\bullet}) \right)\\=&\operatorname{Tot}_{\oplus}\left(I\left(\operatorname{S}(\mathcal{M}^{\bullet})\right)\widetilde{\otimes}^{\bullet,\bullet}_K I(\mathcal{V}^{\bullet}) \right)\\
=&\operatorname{S}\left(\mathcal{M}^{\bullet}\right)\overrightarrow{\otimes}^{\mathbb{L}}_K \mathcal{V}^{\bullet},
 \end{align*}
 where the third identity follows by Lemma \ref{Lemma free modules and side-switching}, together with diagram $(\ref{equation extension of S to the left heart})$.
\end{proof}
\begin{prop}\label{Proposition internal hom and side-switching operators}
There are natural equivalences:
\begin{multline*}
    R\underline{\mathcal{H}om}_{\wideparen{U}(\mathscr{L}^2)}(-,-)\xrightarrow[]{\cong}R\underline{\mathcal{H}om}_{\wideparen{E}(\mathscr{L})}(\operatorname{S}(-),\operatorname{S}(-)),\\
    R\underline{\mathcal{H}om}_{\wideparen{E}(\mathscr{L})}(-,-)\xrightarrow[]{\cong} R\underline{\mathcal{H}om}_{\wideparen{U}(\mathscr{L}^2)}(\operatorname{S}^{-1}(-),\operatorname{S}^{-1}(-)).
\end{multline*}
\end{prop}
\begin{proof}
 Let $\mathcal{V}^{\bullet}\in \operatorname{D}(\operatorname{Shv}(X,LH(\widehat{\mathcal{B}}c_K)))$, $\mathcal{M}^{\bullet},\mathcal{N}^{\bullet}\in \operatorname{D}(\wideparen{U}(\mathscr{L}^2))$. Then we have:
 \begin{align*}
     \Hom_{\operatorname{D}(\operatorname{Shv}(X,LH(\widehat{\mathcal{B}}c_K)))}&(\mathcal{V}^{\bullet},R\underline{\mathcal{H}om}_{\wideparen{E}(\mathscr{L})}(\operatorname{S}(\mathcal{M}^{\bullet}),\operatorname{S}(\mathcal{N}^{\bullet})))\\
     =&\Hom_{\operatorname{D}(\wideparen{E}(\mathscr{L}))}(\operatorname{S}(\mathcal{M}^{\bullet})\overrightarrow{\otimes}_K^{\mathbb{L}}\mathcal{V}^{\bullet},\operatorname{S}(\mathcal{N}^{\bullet}))\\
     =&\Hom_{\operatorname{D}(\wideparen{E}(\mathscr{L}))}(\operatorname{S}\left(\mathcal{M}^{\bullet}\overrightarrow{\otimes}_K^{\mathbb{L}}\mathcal{V}^{\bullet}\right),\operatorname{S}(\mathcal{N}^{\bullet}))\\
     =&\Hom_{\operatorname{D}(\wideparen{U}(\mathscr{L}^2))}(\mathcal{M}^{\bullet}\overrightarrow{\otimes}_K^{\mathbb{L}}\mathcal{V}^{\bullet},\mathcal{N}^{\bullet})\\=&\Hom_{\operatorname{D}(\operatorname{Shv}(X,\widehat{\mathcal{B}}c_K))}(\mathcal{V}^{\bullet},R\underline{\mathcal{H}om}_{\wideparen{U}(\mathscr{L}^2)}(\mathcal{M}^{\bullet},\mathcal{N}^{\bullet})),
 \end{align*}
 and the result follows by Yoneda. The second isomorphism is analogous.
\end{proof}

We will now show the analogous results for derived tensor products:
\begin{Lemma}\label{Lemma side-changing and tensor products}
Choose $\mathcal{V}^{\bullet} \in \operatorname{D}(\operatorname{Shv}(X,LH(\widehat{\mathcal{B}}c_K)))$, and $\mathcal{M}^{\bullet}\in \operatorname{D}(\wideparen{U}(\mathscr{L}^2)^{\op})$, and $\mathcal{N}^{\bullet}\in \operatorname{D}(\wideparen{E}(\mathscr{L})^{\op})$. There are canonical quasi-isomorphisms:
\begin{multline*}
    R\underline{\mathcal{H}om}_{\Indban}(\mathcal{M}^{\bullet},\mathcal{V}^{\bullet})=\operatorname{S}^{-1}R\underline{\mathcal{H}om}_{\Indban}(\operatorname{S}_r(\mathcal{M}^{\bullet}),\mathcal{V}^{\bullet}),\\
  R\underline{\mathcal{H}om}_{\Indban}(\mathcal{N}^{\bullet},\mathcal{V}^{\bullet})=\operatorname{S}R\underline{\mathcal{H}om}_{\Indban}(\operatorname{S}_r^{-1}(\mathcal{N}^{\bullet}),\mathcal{V}^{\bullet}),
\end{multline*}
in $\operatorname{D}(\wideparen{U}(\mathscr{L}^2))$ and $\operatorname{D}(\wideparen{E}(\mathscr{L}))$ respectively.
\end{Lemma}
\begin{proof}
The proof is dual to the proofs of lemmas \ref{Lemma free modules and side-switching}, and \ref{Lemma side-changing and homs}.
\end{proof}
\begin{prop}\label{Proposition completed tensor and side-switching operators}
There are natural equivalences:
\begin{equation*}
    -\overrightarrow{\otimes}^{\mathbb{L}}_{\wideparen{U}(\mathscr{L}^2)}-= \operatorname{S}_r(-)\overrightarrow{\otimes}^{\mathbb{L}}_{\wideparen{E}(\mathscr{L})}\operatorname{S}(-), \textnormal{ }
    -\overrightarrow{\otimes}^{\mathbb{L}}_{\wideparen{E}(\mathscr{L})}-= \operatorname{S}_r^{-1}(-)\overrightarrow{\otimes}^{\mathbb{L}}_{\wideparen{U}(\mathscr{L}^2)}\operatorname{S}^{-1}(-).
\end{equation*}
\end{prop}
\begin{proof}
  Let $\mathcal{V}^{\bullet}\in \operatorname{D}(\operatorname{Shv}(X,LH(\widehat{\mathcal{B}}c_K)))$, $\mathcal{M}^{\bullet}\in \operatorname{D}(\wideparen{U}(\mathscr{L}^2)^{\op})$, and $\mathcal{N}^{\bullet}\in \operatorname{D}(\wideparen{U}(\mathscr{L}^2))$. Then we have: 
\begin{align*}
 \Hom_{\operatorname{D}(\operatorname{Shv}(X,LH(\widehat{\mathcal{B}}c_K)))}(\operatorname{S}_r(\mathcal{M})&\overrightarrow{\otimes}^{\mathbb{L}}_{\wideparen{E}(\mathscr{L})}\operatorname{S}(\mathcal{N}),\mathcal{V}^{\bullet})\\&=\Hom_{\operatorname{D}(\wideparen{E}(\mathscr{L}))}(\operatorname{S}(\mathcal{N}),R\underline{\mathcal{H}om}_{\Indban}(\operatorname{S}_r(\mathcal{M}),\mathcal{V}))\\&
 =\Hom_{\operatorname{D}(\wideparen{U}(\mathscr{L}^2))}(\mathcal{N},R\underline{\mathcal{H}om}_{\Indban}(\mathcal{M},\mathcal{V}))\\&=\Hom_{\operatorname{D}(\operatorname{Shv}(X,LH(\widehat{\mathcal{B}}c_K)))}(\mathcal{M}\overrightarrow{\otimes}^{\mathbb{L}}_{\wideparen{U}(\mathscr{L}^2)}\mathcal{N},\mathcal{V}^{\bullet}),
\end{align*}
where we are using Lemma \ref{Lemma side-changing and tensor products}. Thus, the result holds by Yoneda. 
\end{proof}
In order to keep the notation reasonable, we introduce the following definition:
\begin{defi}\label{defi immersion functor}
We define the immersion functor as the following composition:
\begin{equation*}
\Delta_{*}^{\operatorname{S}}:=\operatorname{S}^{-1}\circ\Delta^E_{*}:   \Mod_{\Indban}(\wideparen{U}(\mathscr{L})^e)\rightarrow \Mod_{\Indban}(\wideparen{U}(\mathscr{L}^2)).
\end{equation*} 
We also have a version for right modules:
\begin{equation*}
    \Delta_{*}^{\operatorname{S}_r}:=\operatorname{S}_r^{-1}\circ\Delta^E_{*}:   \Mod_{\Indban}(\wideparen{U}(\mathscr{L})^{e,\op})\rightarrow \Mod_{\Indban}(\wideparen{U}(\mathscr{L}^2)^{\op}).
\end{equation*}
Analogously, we can define versions for the left hearts, which we denote by $\Delta_*^{I(\operatorname{S})}$, and $\Delta_*^{I(\operatorname{S})_r}$, and we denote their left derived functors by $\Ifunct$ and $\RIfunct$ respectively.
\end{defi}
Notice that as the side-switching functor $\operatorname{S}$ is an equivalence of quasi-abelian categories, we get a factorization:
\begin{equation*}
    \Ifunct=\operatorname{S}^{-1}\circ \EDelta: \operatorname{D}(\wideparen{U}(\mathscr{L})^e)\rightarrow \operatorname{D}(\wideparen{U}(\mathscr{L}^2)),
\end{equation*}
and similarly for $\RIfunct$. As a consequence, we get the following corollary:
\begin{coro}\label{coro hochschild cohomology and the immersion functor}
Let $\mathcal{M}^{\bullet}\in \operatorname{D}(\wideparen{U}(\mathscr{L})^e)$. The following identities hold:
 \begin{enumerate}[label=(\roman*)]
     \item $\mathcal{HH}^{\bullet}(\wideparen{U}(\mathscr{L}),\mathcal{M}^{\bullet})=\Delta^{-1}R\underline{\mathcal{H}om}_{\wideparen{U}(\mathscr{L}^2)}(\Ifunct\wideparen{U}(\mathscr{L}),\Ifunct\mathcal{M}^{\bullet})$.
     \item $\mathcal{HH}_{\bullet}(\wideparen{U}(\mathscr{L}),\mathcal{M}^{\bullet})=\Delta^{-1}\left(\RIfunct\wideparen{U}(\mathscr{L})\overrightarrow{\otimes}^{\mathbb{L}}_{\wideparen{U}(\mathscr{L}^2)}\Ifunct\mathcal{M}^{\bullet}\right)$.
 \end{enumerate} 
\end{coro}
\begin{proof}
Identity $(i)$ follows by Proposition \ref{Proposition internal hom and side-switching operators}, $(ii)$ follows by Proposition \ref{Proposition completed tensor and side-switching operators}.
\end{proof}

\subsection{\texorpdfstring{Co-admissible modules over the bi-enveloping algebra I: The local case}{}}\label{Section Co-admissible modules over the bi-enveloping algebra}
As we will see, the sheaf of Ind-Banach algebras $\wideparen{E}(\mathscr{L})$ is locally Fréchet-Stein. Hence, it is a natural to wonder if there could be a well-behaved category of co-admissible $\wideparen{E}(\mathscr{L})$-modules. This section is devoted to defining such a category, and showing some of its most important features.\\
As discussed above, the category of Ind-Banach $\wideparen{E}(\mathscr{L})$-modules is equivalent to the category of Ind-Banach $\wideparen{U}(\mathscr{L}^2)$-modules. Hence, the idea is using this equivalence to derive the properties of co-admissible $\wideparen{E}(\mathscr{L})$-modules from those of co-admissible $\wideparen{U}(\mathscr{L}^2)$-modules.  Let $X$ be a smooth affinoid space with a Lie algebroid $\mathscr{L}$, and assume $\mathscr{L}$ admits a basis of global sections over $\OX_X$. As seen in Section \ref{Section The sheaf of complete bi-enveloping algebras}, such a choice induces an isomorphism of sheaves of Ind-Banach algebras:
\begin{equation*}
    \wideparen{\mathbb{T}}: \wideparen{U}(\mathscr{L}^2)\rightarrow \wideparen{E}(\mathscr{L}).
\end{equation*}
As a first consequence of this, we can show the following corollary:
\begin{coro}    
For any affinoid subdomain $V\subset X^2$, the Ind-Banach algebra $\wideparen{E}(\mathscr{L})(V)$ is a Fréchet-Stein algebra.
\end{coro}
\begin{proof}
As $\wideparen{U}(\mathscr{L}^2)$ is the sheaf of Fréchet-stein enveloping algebras of the Lie algebroid $\mathscr{L}^2$, it follows that $\wideparen{U}(\mathscr{L}^2)(V)$   is a Fréchet-Stein algebra. Thus, the isomorphism  $\wideparen{U}(\mathscr{L}^2)(V)\rightarrow \wideparen{E}(\mathscr{L})(V)$ makes the latter a Fréchet-Stein algebra. 
\end{proof}
Assume that $X=\Sp(A)$, and let $\mathfrak{X}=\Spf(\mathcal{A})$ be an affine formal model of $X$, and  $\mathcal{L}$ be a free $\mathcal{A}$-Lie lattice of $L=\mathscr{L}(X)$. Choose a basis of $\mathcal{L}$ as an $\mathcal{A}$-module $x_1,\cdots,x_d$. For later use, it  will be convenient to construct an explicit Fréchet-Stein presentation of the algebra $\wideparen{E}(\mathscr{L})(X)$:
\begin{prop}\label{prop explicit Fréchet-Stein presentation}
For $n,m\geq 0$ there is an isomorphism of Banach $K$-algebras:
\begin{equation*}
    \widehat{U}(\pi^n\mathcal{L})_K\overrightarrow{\otimes}_K\widehat{U}(\pi^m\mathcal{L})_K\rightarrow \widehat{U}(\pi^n\mathcal{L}\times \pi^m\mathcal{L})_K.
\end{equation*}
Thus, we have the following chain of isomorphisms of  Ind-Banach $K$-algebras:
    \begin{equation*}
        \wideparen{U}(\mathscr{L}^2)(X^2)=\wideparen{U}(\mathscr{L})(X)\overrightarrow{\otimes}_K\wideparen{U}(\mathscr{L})(X)=\varprojlim_s\left(\widehat{U}(\pi^{n+s}\mathcal{L})_K\overrightarrow{\otimes}_K\widehat{U}(\pi^{m+s}\mathcal{L})_K\right),
    \end{equation*}
\end{prop}
and the latter inverse limit is a Fréchet-Stein presentation of $\wideparen{U}(\mathscr{L}^2)(X^2)$.
\begin{proof}
By construction, for each $n,m\geq 0$ we have a map of filtered $\mathcal{R}$-algebras:
\begin{equation*}
 U(\pi^{n}\mathcal{L})\otimes_{\mathcal{R}}U(\pi^m\mathcal{L})\rightarrow U(\pi^n\mathcal{L}\times \pi^m\mathcal{L}).   
\end{equation*}
Taking the $\pi$-adic completion of both sides, we arrive at a map of $\mathcal{R}$-algebras:
\begin{equation*}
\varphi_{n,m}:\widehat{U}(\pi^n\mathcal{L})\widehat{\otimes}_{\mathcal{R}}\widehat{U}(\pi^m\mathcal{L})\rightarrow \widehat{U}(\pi^n\mathcal{L}\times \pi^m\mathcal{L}).
\end{equation*}
It is enough to show that this is an isomorphism of complete $\mathcal{R}$-algebras. By Corollary \ref{coro colimit topology integral level}, for any $r\geq 0$,  we have the following identities:
\begin{align*}
\widehat{U}(\pi^n\mathcal{L}\times \pi^m\mathcal{L})/\pi^r=& \varinjlim_s\left( \widehat{F}_s\left(U(\pi^n\mathcal{L})\otimes_{\mathcal{R}}U(\pi^m\mathcal{L})\right)/\pi^r\right)\\=&\varinjlim_s\left( F_s\left(U(\pi^n\mathcal{L})\otimes_{\mathcal{R}}U(\pi^m\mathcal{L})\right)/\pi^r\right)\\
=&\varinjlim_s\left( F_s\left(U(\pi^n\mathcal{L})\otimes_{\mathcal{R}}U(\pi^m\mathcal{L})\right)\right)/\pi^r\\=&U(\pi^n\mathcal{L})\otimes_{\mathcal{R}}U(\pi^m\mathcal{L})/\pi^r.
\end{align*}
Therefore, $\varphi_{n,m}$ is an isomorphism of complete $\mathcal{R}$-algebras. Additionally, there is a canonical isomorphism of Banach $K$-algebras:
\begin{equation*}
    \left(\widehat{U}(\pi^n\mathcal{L})\widehat{\otimes}_{\mathcal{R}}\widehat{U}(\pi^m\mathcal{L})\right)\otimes_{\mathcal{R}}K=\widehat{U}(\pi^n\mathcal{L})_K\overrightarrow{\otimes}_K\widehat{U}(\pi^m\mathcal{L})_K.
\end{equation*}
Combining these two facts, we get our desired isomorphism. The second part of the proposition follows by \ref{prop enveloping of product as product of envelopings}, together with the fact that:
\begin{equation*}
    \wideparen{U}(\mathscr{L}^2)(X)=\varprojlim_s \widehat{U}(\pi^{n+s}\mathcal{L}\times \pi^{m+s}\mathcal{L})_K,
\end{equation*}
is a Fréchet-Stein presentation.
\end{proof}
\begin{coro}\label{coro frechet-stein presentation of bi-enveloping algebra}
   There is a Fréchet-Stein presentation:
    \begin{equation*}
        \wideparen{E}(\mathscr{L})(X^2)=\wideparen{U}(\mathscr{L})(X)\overrightarrow{\otimes}_K\wideparen{U}(\mathscr{L})(X)^{\op}=\varprojlim_s\left(\widehat{U}(\pi^{n+s}\mathcal{L})_K\overrightarrow{\otimes}_K\widehat{U}(\pi^{m+s}\mathcal{L})_K^{\op}\right).
    \end{equation*}
\end{coro}
\begin{proof}
In virtue of Proposition \ref{prop anti-isomorphism fréchet-stein level}, for each $n,m\geq 0$ we have the following commutative diagram of morphisms of Banach $K$-algebras:
\begin{equation*}
\begin{tikzcd}[column sep=large]
	{\widehat{U}(\pi^{n+1}\mathcal{L})_K\overrightarrow{\otimes}_K\widehat{U}(\pi^{m+1}\mathcal{L})_K} & {\widehat{U}(\pi^{n+1}\mathcal{L})_K\overrightarrow{\otimes}_K\widehat{U}(\pi^{m+1}\mathcal{L})_K^{\op}} \\
	{\widehat{U}(\pi^{n}\mathcal{L})_K\overrightarrow{\otimes}_K\widehat{U}(\pi^{m}\mathcal{L})_K} & {\widehat{U}(\pi^n\mathcal{L})_K\overrightarrow{\otimes}_K\widehat{U}(\pi^m\mathcal{L})_K^{\op}}
	\arrow["{\operatorname{Id}\overrightarrow{\otimes}_K\widehat{t}_{m+1}}", from=1-1, to=1-2]
	\arrow[from=1-1, to=2-1]
	\arrow[from=1-2, to=2-2]
	\arrow["{\operatorname{Id}\overrightarrow{\otimes}_K\widehat{t}_{m}}", from=2-1, to=2-2]
\end{tikzcd}
\end{equation*}
By proposition \ref{prop anti-isomorphism fréchet-stein level}, the horizontal maps are isomorphisms of Banach $K$-algebras, and the leftmost vertical map is flat and has dense image by Proposition \ref{prop explicit Fréchet-Stein presentation}. Thus, the rightmost vertical map is also flat and has dense image. In particular, it follows that we have a Fréchet-Stein presentation:
\begin{equation*}
    \varprojlim_s\left(\widehat{U}(\pi^{n+s}\mathcal{L})_K\overrightarrow{\otimes}_K\widehat{U}(\pi^{m+s}\mathcal{L})_K^{\op}\right).
\end{equation*}
Furthermore, we have the following chain of isomorphisms of Ind-Banach algebras:
\begin{multline*}
 \wideparen{E}(\mathscr{L})(X^2)\xrightarrow[]{\wideparen{\mathbb{T}}} \wideparen{U}(\mathscr{L}^2)(X^2)=\varprojlim_s\left(\widehat{U}(\pi^{n+s}\mathcal{L})_K\overrightarrow{\otimes}_K\widehat{U}(\pi^{m+s}\mathcal{L})_K\right)\\
 \xrightarrow[]{} \varprojlim_s\left(\widehat{U}(\pi^{n+s}\mathcal{L})_K\overrightarrow{\otimes}_K\widehat{U}(\pi^{m+s}\mathcal{L})_K^{\op}\right),
\end{multline*}
so the proposition follows.
\end{proof}
As a consequence of the previous discussion, it follows that there is a well defined category of co-admissible bimodules over $\wideparen{E}(\mathscr{L})(X^2)$, and it makes sense to wonder if there could be a category of sheaves of co-admissible modules over $\wideparen{E}(\mathscr{L})$. In order to define and study such a category, we will analyze the behavior of the extension of scalars along the map $\wideparen{\mathbb{T}}:\wideparen{U}(\mathscr{L}^2)\rightarrow \wideparen{E}(\mathscr{L})$.
\begin{obs}\label{remark same results hold for right modules over E(L)}
We remark that all the following results have analogs for right modules, with the same proofs. Hence, to avoid repetition, we will only state and prove the results for left modules.
\end{obs}
\begin{defi}
Let $f:\mathscr{A}\rightarrow \mathscr{B}$ be a morphisms of sheaves of Ind-Banach algebras on a $G$-topological space $X$. We define the extension of scalars along $f$ as the following functor:
\begin{align*}
    f^*:\Mod_{\Indban}(\mathscr{A})\rightarrow \Mod_{\Indban}(\mathscr{B}),\\
    \mathcal{M}\mapsto f^*(\mathcal{M})=\mathscr{B}\overrightarrow{\otimes}_{\mathscr{A}}\mathcal{M},
\end{align*}
and the restriction of scalars along $f$ as the functor:
\begin{equation*}
    f_*:\Mod_{\Indban}(\mathscr{B})\rightarrow \Mod_{\Indban}(\mathscr{A}),
\end{equation*}
defined by regarding $\mathcal{M}\in \Mod_{\Indban}(\mathscr{B})$ as an $\mathscr{A}$-module via the canonical map:
\begin{equation*}
\mathscr{A}\overrightarrow{\otimes}_K\mathcal{M}\xrightarrow[]{f\overrightarrow{\otimes}_K\Id} \mathscr{B}\overrightarrow{\otimes}_K\mathcal{M}\rightarrow \mathcal{M}. 
\end{equation*}
\end{defi}
As the restriction functor $f_*$ does not change the underlying sheaf of Ind-Banach spaces, we will just write $\mathcal{N}$ instead of $f_*\mathcal{N}$ when no confusion is possible.\bigskip

Notice that both the extension and restriction of scalars are functorial, in the sense that for a pair of maps of sheaves of Ind-Banach algebras $f:\mathscr{A}\rightarrow\mathscr{B}$, and  $g:\mathscr{B}\rightarrow \mathscr{C}$ we have:
\begin{equation*}
    \left(gf\right)^*=g^*f^*, \quad \left(gf\right)_*=f_*g_*.
\end{equation*}
It is also clear that $\operatorname{Id}^*=\operatorname{Id}_*=\operatorname{Id}$. Hence, if $f:\mathscr{A}\rightarrow\mathscr{B}$ is an isomorphism of sheaves of bornological algebras, it follows that $f^*$ is an equivalence of quasi-abelian categories. Furthermore, we have:
\begin{prop}\label{prop extension-restriction of scalars complete bornological case}
 Extension and restriction of scalars are adjoint. In particular, it  $f:\mathscr{A}\rightarrow\mathscr{B}$ is a morphism of sheaves of Ind-Banach $K$-algebras, then:
\begin{equation*}   f^*:\Mod_{\Indban}(\mathscr{A})\leftrightarrows \Mod_{\Indban}(\mathscr{B}):f_*, 
\end{equation*}
is a pair of adjoint functors.
\end{prop}
\begin{proof}
Let $\mathcal{M}\in \Mod_{\Indban}(\mathscr{A})$ and $\mathcal{N}\in \Mod_{\Indban}(\mathscr{B})$. We define a map:
\begin{equation*}
 \Hom_{\mathscr{B}}(f^*\mathcal{M},\mathcal{N})\rightarrow \Hom_{\mathscr{A}}(\mathcal{M},f_*\mathcal{N}),   
\end{equation*}
by sending a map $\varphi\in \Hom_{\mathscr{B}}(f^*\mathcal{M},\mathcal{N})$ to the composition:
\begin{equation*}
    \mathcal{M}\cong \mathscr{A}\overrightarrow{\otimes}_{\mathscr{A}}\mathcal{M}\rightarrow \mathscr{B}\overrightarrow{\otimes}_{\mathscr{A}}\mathcal{M} \xrightarrow[]{\varphi} \mathcal{N}.
\end{equation*}
For the converse, we define a map $\Hom_{\mathscr{A}}(\mathcal{M},f_*\mathcal{N})\rightarrow Hom_{\mathscr{B}}(f^*\mathcal{M},\mathcal{N})$, by sending  $\psi \in \Hom_{\mathscr{A}}(\mathcal{M},f_*\mathcal{N})$ to the map:
\begin{equation*}
    f^*\mathcal{M}\xrightarrow[]{\operatorname{Id}\widehat{\otimes}_{\mathscr{A}} \psi}\mathscr{B}\widehat{\otimes}_{\mathscr{A}}\mathcal{N} \rightarrow \mathcal{N}, 
\end{equation*}
where the last map is the action of $\mathscr{B}$ on $\mathcal{N}$. A standard  calculation shows these maps are mutually inverse.
\end{proof}
We will now study the image of the category of co-admissible modules $\mathcal{C}(\wideparen{U}(\mathscr{L}^2))$ under the extension of scalars $\wideparen{\mathbb{T}}^*$. As above, there are Fréchet-Stein presentations:
\begin{multline*}
    \wideparen{E}(\mathscr{L})(X^2)=\varprojlim_s\left(\widehat{U}(\pi^{n+s}\mathcal{L})_K\overrightarrow{\otimes}_K\widehat{U}(\pi^{m+s}\mathcal{L})_K^{\op}\right),\\ \wideparen{U}(\mathscr{L}^2)(X^2)=\varprojlim_s\left(\widehat{U}(\pi^{n+s}\mathcal{L})_K\overrightarrow{\otimes}_K\widehat{U}(\pi^{m+s}\mathcal{L})_K\right).
\end{multline*}
To keep the notation simple, for each $n,m\geq 0$ we will define the following algebras:
\begin{equation*}
    \widehat{E}_{n,m}:=\widehat{U}(\pi^{n}\mathcal{L})_K\overrightarrow{\otimes}_K\widehat{U}(\pi^{m}\mathcal{L})_K^{\op}, \quad  \widehat{U}^2_{n,m}:=\widehat{U}(\pi^{n}\mathcal{L})_K\overrightarrow{\otimes}_K\widehat{U}(\pi^{m}\mathcal{L})_K.
\end{equation*}
Similarly, we write $\wideparen{E}=\wideparen{E}(\mathscr{L})(X^2)$, and $\wideparen{U}^2=\wideparen{U}(\mathscr{L}^2)(X^2)$.  For each co-admissible $\wideparen{U}^2$-module $\mathcal{M}\in\mathcal{C}(\wideparen{U})$,  we set $\mathcal{M}_{n,m}:=\widehat{U}^2_{n,m}\otimes_{\wideparen{U}^2}\mathcal{M}$. In particular, we have:
\begin{equation*}
    \mathcal{M}=\varprojlim_{n,m} \mathcal{M}_{n,m}.
\end{equation*}
Let us start with the following lemma:
\begin{Lemma}
For $n,m\geq 0$, there is an isomorphism of Banach $\wideparen{E}$-modules:
\begin{equation*}
    \widehat{E}_{n,m}=\wideparen{\mathbb{T}}^*\widehat{U}_{n,m}.
\end{equation*}
\end{Lemma}
\begin{proof}
 By Proposition \ref{prop anti-isomorphism fréchet-stein level}, we have a commutative diagram of Fréchet $K$-algebras:
 \begin{equation*}
\begin{tikzcd}
\wideparen{U}^2 \arrow[d, "{\pi_{n,m}}"] \arrow[r, "\wideparen{\mathbb{T}}"]                   & \wideparen{E} \arrow[d, "{\pi_{n,m}}"] \\
{\widehat{U}^2_{n,m}} \arrow[r, "\operatorname{Id}\overrightarrow{\otimes}_K\widehat{T}_n"] & {\widehat{E}_{n,m}}      
\end{tikzcd}
\end{equation*}
By the properties of  the functor $-\overrightarrow{\otimes}_{\wideparen{U}}-$, we obtain an $\wideparen{E}$-linear morphism of Ind-Banach spaces:
\begin{equation*}
    \wideparen{\mathbb{T}}^*\widehat{U}_{n,m}:= \wideparen{E}\overrightarrow{\otimes}_{\wideparen{U}}\widehat{U}^2_{n,m}\rightarrow \widehat{E}_{n,m}.
\end{equation*}
We need to show it is an isomorphism. It is enough to show that the sequence:
\begin{equation*}
 \wideparen{E}\overrightarrow{\otimes}_K\wideparen{U}\overrightarrow{\otimes}_K\widehat{U}_{n,m}^2\rightarrow  \wideparen{E}\overrightarrow{\otimes}_K\widehat{U}_{n,m}^2  \rightarrow \widehat{E}_{n,m} \rightarrow 0,
\end{equation*}
is strictly right exact. Thus, it suffices to show that $\widehat{E}_{n,m}$ satisfies the universal property of the cokernel of the map:
\begin{equation*}
    \wideparen{E}\overrightarrow{\otimes}_K\wideparen{U}\overrightarrow{\otimes}_K\widehat{U}_{n,m}^2\rightarrow  \wideparen{E}\overrightarrow{\otimes}_K\widehat{U}_{n,m}^2.
\end{equation*}
However, this follows by the fact that the upper and lower horizontal maps in the commutative diagram above are isomorphisms.
\end{proof}
\begin{prop}\label{prop extension of scalars co-admissible modules sections}
For any affinoid admissible open $V\subset X^2$, there is a pair of mutually inverse equivalences of abelian categories:
\begin{equation*}
    \wideparen{\mathbb{T}}^*:\mathcal{C}(\wideparen{U}(\mathscr{L}^2)(V))\leftrightarrows   \mathcal{C}(\wideparen{E}(\mathscr{L})(V)):\wideparen{\mathbb{T}}^{-1,*}.
\end{equation*}
\end{prop}
\begin{proof}
We will show this equivalence for $V=X^2$, the general case is analogous. As $\wideparen{\mathbb{T}}$ is an isomorphism, we have mutually inverse equivalences:
\begin{equation*}
    \wideparen{\mathbb{T}}^*:\Mod_{\Indban}(\wideparen{U}^2)\leftrightarrows \Mod_{\Indban}(\wideparen{E}):\wideparen{\mathbb{T}}^{-1,*}.
\end{equation*}
Thus, it  suffices to show that $\wideparen{\mathbb{T}}^*$ and $\wideparen{\mathbb{T}}^{-1,*}$ send co-admissible modules to co-admissible modules. Let then $M\in \mathcal{C}(\wideparen{U}^2)$ be a co-admissible module. For each $n,m\geq 0$,  define the following Banach $\widehat{E}_{n,m}$-module:
\begin{equation*}
    \wideparen{\mathbb{T}}^*M_{n,m}:=  \widehat{E}_{n,m}   \overrightarrow{\otimes}_{\wideparen{E}}\wideparen{\mathbb{T}}^*M.
\end{equation*}
We have the following chain of identities:
    \begin{align*}
      \wideparen{\mathbb{T}}^*M_{n,m}=\widehat{E}_{n,m}\overrightarrow{\otimes}_{\wideparen{E}}\wideparen{E}\overrightarrow{\otimes}_{\wideparen{U}^2}M=&\widehat{E}_{n,m}\overrightarrow{\otimes}_{\wideparen{U}^2}M\\=&\widehat{E}_{n,m}\overrightarrow{\otimes}_{\widehat{U}^2_{n,m}}\left(\widehat{U}^2_{n,m}\overrightarrow{\otimes}_{\wideparen{U}^2}M\right)\\=&\widehat{E}_{n,m}\overrightarrow{\otimes}_{\widehat{U}^2_{n,m}}M_{n,m},
    \end{align*}
where the last identity follows by \cite[Corollary 5.40]{bode2021operations}. As $M$ is a co-admissible $\wideparen{U}^2$-module, each of the $M_{n,m}$ is a finite module over $\wideparen{U}^2_{n,m}$. Hence, $\wideparen{\mathbb{T}}^*M_{n,m}$ is a finite module over $\widehat{E}_{n,m}$.
Furthermore, we have:
\begin{equation*}
    \wideparen{\mathbb{T}}^*M=\wideparen{E}\overrightarrow{\otimes}_{\wideparen{U}^2}(\varprojlim M_{n,m})=\varprojlim\wideparen{E}\overrightarrow{\otimes}_{\wideparen{U}^2}M_{n,m}=\varprojlim \wideparen{\mathbb{T}}^*M_{n,m},
\end{equation*}
where the second identity follows by the fact that $\wideparen{\mathbb{T}}^*$ is an equivalence, so it commutes with inverse limits, and the last one 
follows by the isomorphism:
\begin{equation*}
   \widehat{E}_{n,m}=\wideparen{E}\overrightarrow{\otimes}_{\wideparen{U}^2}\widehat{U}^2_{n,m}. 
\end{equation*}
The corresponding fact for $\wideparen{\mathbb{T}}^{-1,*}$ is analogous.
\end{proof}
Let $W\subset V\subset X^2$ be affinoid subdomains. We have a commutative diagram:
\begin{equation*}
\begin{tikzcd}
\wideparen{U}(\mathscr{L}^2)(V) \arrow[d] \arrow[r, "\wideparen{\mathbb{T}}"] & \wideparen{E}(\mathscr{L})(V) \arrow[d] \\
\wideparen{U}(\mathscr{L}^2)(W) \arrow[r, "\wideparen{\mathbb{T}}"]           & \wideparen{E}(\mathscr{L})(W)          
\end{tikzcd}
\end{equation*}
By \cite[Theorem 4.2.8]{ardakov2021equivariant}, the leftmost vertical map is $c$-flat. Furthermore, there are Fréchet-Stein presentations:
\begin{equation*}
    \wideparen{U}(\mathscr{L}^2)(V))=\varprojlim_n A_n, \textnormal{  } \wideparen{U}(\mathscr{L}^2)(W))=\varprojlim_n B_n,
\end{equation*}
satisfying that for each $n\geq 0$ the restriction map $\wideparen{U}(\mathscr{L}^2)(V))\rightarrow \wideparen{U}(\mathscr{L}^2)(W))$ induces a map $A_n\rightarrow B_n$ which is two-sided flat. As the horizontal maps are isomorphisms of Fréchet-Stein algebras, it follows that
the restriction map $\wideparen{E}(\mathscr{L})(V))\rightarrow \wideparen{E}(\mathscr{L})(W))$ is $c$-flat as well.\bigskip

Let $(W_i)_{i=1}^n$ be an affinoid cover of $V$. As $\wideparen{\mathbb{T}}:\wideparen{U}(\mathscr{L}^2)\rightarrow \wideparen{E}(\mathscr{L})$ is an isomorphism, and $V$ is affinoid, \cite[Theorem 8.2]{ardakov2019} shows that the augmented \v{C}ech complex:
\begin{equation}\label{equation augmented cech complex of bi-enveloping algebra}
    0\rightarrow \wideparen{E}(\mathscr{L})(V)\rightarrow \prod_{i=1}^n\wideparen{E}(\mathscr{L})(W_i)\rightarrow \cdots \rightarrow \wideparen{E}(\mathscr{L})(\cap_{i=1}^nW_i)\rightarrow 0,
\end{equation}
is a strict and exact complex. Let $M \in \mathcal{C}(\wideparen{E}(\mathscr{L})(V))$ be a co-admissible module and consider the following chain complex:
\begin{multline}\label{equation augmented cech complex of co-admissible module over bi-enveloping algebra}
    0\rightarrow M\rightarrow \prod_{i=1}^n\wideparen{E}(\mathscr{L})(W_i)\overrightarrow{\otimes}_{\wideparen{E}(\mathscr{L})(V)}M\rightarrow \\\cdots \rightarrow \wideparen{E}(\mathscr{L})(\cap_{i=1}^nW_i)\overrightarrow{\otimes}_{\wideparen{E}(\mathscr{L})(V)}M\rightarrow 0.
\end{multline}
As $(\ref{equation augmented cech complex of bi-enveloping algebra})$ is a bounded complex with $c$-flat transition maps, \cite[Proposition 5.33]{bode2021operations} shows that $(\ref{equation augmented cech complex of co-admissible module over bi-enveloping algebra})$ is strictly exact. Furthermore, for $1\leq i_1,\cdots,i_j\leq n$, the Fréchet module $\wideparen{E}(\mathscr{L})(\cap_{k=1}^jW_{i_k})\overrightarrow{\otimes}_{\wideparen{E}(\mathscr{L})(V)}M$ is a co-admissible $\wideparen{E}(\mathscr{L})(\cap_{k=1}^jW_{i_k})$-module.\\
As $V$ is an affinoid space, it is separated. Hence, in order to define a sheaf on the admissible $G$-topology, it is enough to define it on the site of affinoid subspaces with finite covers. Thus, we obtain a functor:
\begin{equation*}
    \operatorname{Loc}:\mathcal{C}(\wideparen{E}(\mathscr{L})(V))\rightarrow \Mod_{\Indban}(\wideparen{E}(\mathscr{L})_{\vert V}),
\end{equation*}
defined by assigning to each co-admissible module $M\in \mathcal{C}(\wideparen{E}(\mathscr{L})(V))$, and each affinoid subdomain $W\subset V$, the co-admissible $\wideparen{E}(\mathscr{L})(W)$-module:
\begin{equation*}
    \operatorname{Loc}(M)(W):=\wideparen{E}(\mathscr{L})(W)\overrightarrow{\otimes}_{\wideparen{E}(\mathscr{L})(V)}M.
\end{equation*}
We will refer to this functor as the localization functor.
\begin{prop}\label{prop properties of Loc}
 The following hold:
 \begin{enumerate}[label=(\roman*)]
     \item $\operatorname{Loc}(-)$ is a exact and fully faithful.
     \item If $W\subset V$ is affinoid, then $\operatorname{Loc}(\mathcal{M})(W)$ is a co-admissible $\wideparen{E}(\mathscr{L})(W)$-module.
     \item For each $M\in \mathcal{C}(\wideparen{E}(\mathscr{L})(V))$, we have $\operatorname{H}^j(V,\operatorname{Loc}(M))=0$ for all $j\geq 1$.
     \item For each $\mathcal{N}\in \Mod_{\Indban}(\wideparen{E}(\mathscr{L})_{\vert V})$ we have an isomorphism:
     \begin{equation*}
         \Hom_{\wideparen{E}(\mathscr{L})_{\vert V}}(\operatorname{Loc}(M),\mathcal{N})=\Hom_{\wideparen{E}(\mathscr{L})(V)}(M,\mathcal{N}(V)).
     \end{equation*}
 \end{enumerate}
\end{prop}
\begin{proof}
Consider a short exact sequence of co-admissible $\wideparen{E}(\mathscr{L})(V)$-modules:
\begin{equation*}
    0\rightarrow M_1\rightarrow M_2\rightarrow M_3\rightarrow 0.
\end{equation*}
As this is a short exact sequence of Fréchet spaces, it is strict by the open mapping theorem. Let $W\subset V$ be an affinoid subdomain. As shown above, the restriction map $\wideparen{E}(\mathscr{L})(V)\rightarrow \wideparen{E}(\mathscr{L})(W)$ is $c$-flat. Thus, by \cite[Proposition 5.33]{bode2021operations}, the sequence:
\begin{multline*}
    0\rightarrow \wideparen{E}(\mathscr{L})(W)\overrightarrow{\otimes}_{\wideparen{E}(\mathscr{L})(V)}M_1\rightarrow \wideparen{E}(\mathscr{L})(W)\overrightarrow{\otimes}_{\wideparen{E}(\mathscr{L})(V)}M_2\\\rightarrow \wideparen{E}(\mathscr{L})(W)\overrightarrow{\otimes}_{\wideparen{E}(\mathscr{L})(V)}M_3\rightarrow 0,
\end{multline*}
is a strict exact sequence. Thus, it follows that:
\begin{equation*}
    0\rightarrow \operatorname{Loc}(M_1)\rightarrow \operatorname{Loc}(M_2)\rightarrow \operatorname{Loc}(M_3)\rightarrow 0,
\end{equation*}
is a strict exact sequence in $\Mod_{\Indban}(\wideparen{E}(\mathscr{L})_{\vert V})$. In order to show that $\operatorname{Loc}(-)$ is fully faithful, it suffices to show $(iv)$. Let $f:\operatorname{Loc}(M)\rightarrow \mathcal{N}$ be a morphism of $\wideparen{E}(\mathscr{L})$-modules. For any affinoid subdomain $W\subset V$ there is a commutative diagram:
\begin{equation*}
\begin{tikzcd}
M \arrow[d] \arrow[r]                                                           & \mathcal{N}(V) \arrow[d] \\
\wideparen{E}(\mathscr{L})(W)\overrightarrow{\otimes}_{\wideparen{E}(\mathscr{L})(V)}M \arrow[r] & \mathcal{N}(W)          
\end{tikzcd}
\end{equation*}
 Notice that by Proposition \ref{prop extension-restriction of scalars complete bornological case}, we have:
 \begin{equation*}
     \Hom_{\wideparen{E}(\mathscr{L})(V)}(M,\mathcal{N}(W))=\Hom_{\wideparen{E}(\mathscr{L})(W)}(\wideparen{E}(\mathscr{L})(W)\overrightarrow{\otimes}_{\wideparen{E}(\mathscr{L})(V)}M,\mathcal{N}(W)).
 \end{equation*}
 Thus, $f:\operatorname{Loc}(M)\rightarrow \mathcal{N}$ is completely determined by the morphism $M\rightarrow \mathcal{N}(V)$ induced on global sections. Conversely, any map $M\rightarrow \mathcal{N}(V)$ induces a morphism $M\rightarrow \mathcal{N}(W)$ by composing with the restriction maps. Again by extension of scalars, we get commutative diagrams as above, and these determine a unique morphism of sheaves $\operatorname{Loc}(M)\rightarrow \mathcal{N}$. Statement $(ii)$ is \cite[Proposition 5.33]{bode2021operations}, and $(iii)$ follows by the fact that the higher \v{C}ech cohomology groups of $\operatorname{Loc}(M)$ vanish for all finite affinoid covers, as was shown above.
\end{proof}
\begin{defi}\label{defi localization functor for co-admissible E(l)-modules}
Let $V\subset X^2$ be an admissible affinoid open. We define the category of co-admissible $\wideparen{E}(\mathscr{L})_{\vert V}$-modules, $\mathcal{C}(\wideparen{E}(\mathscr{L})_{\vert V})$, as the essential image of the localization functor:
\begin{equation*}
    \operatorname{Loc}:\mathcal{C}(\wideparen{E}(\mathscr{L})(V))\rightarrow \Mod_{\Indban}(\wideparen{E}(\mathscr{L})_{\vert V}).
\end{equation*}
\end{defi}
We will now compare the co-admissible modules over $\wideparen{E}(\mathscr{L})_{\vert V}$ and $\wideparen{U}(\mathscr{L}^2)_{\vert V}$:
\begin{Lemma}\label{Lemma extension of scalars co-admissible modules}
Let $\mathcal{M}\in \mathcal{C}(\wideparen{U}(\mathscr{L}^2)_{\vert V})$ be a sheaf of co-admissible modules. For any pair of affinoid subdomains $W\subset V\subset X^2$ we have:
\begin{equation*}
    \wideparen{\mathbb{T}}^{*}\mathcal{M}(W)=\wideparen{E}(\mathscr{L})(W)\overrightarrow{\otimes}_{\wideparen{E}(\mathscr{L})(V)}\wideparen{\mathbb{T}}^{*}\mathcal{M}(V).
\end{equation*}
\end{Lemma}
\begin{proof}
By construction, we have:
\begin{equation*}
 \wideparen{\mathbb{T}}^{*}\mathcal{M}(W)=\wideparen{E}(\mathscr{L})(W)\overrightarrow{\otimes}_{\wideparen{U}(\mathscr{L}^2)(W)}\mathcal{M}(W).   
\end{equation*}
As $\mathcal{M}$ is co-admissible, and $W$ is affinoid, we have the following identity:
\begin{equation*}
  \mathcal{M}(W)=\wideparen{U}(\mathscr{L}^2)(W)\overrightarrow{\otimes}_{\wideparen{U}(\mathscr{L}^2)(V)}\mathcal{M}(V).  
\end{equation*}
Combining these two identities, we get:
\begin{equation*}
 \wideparen{\mathbb{T}}^{*}\mathcal{M}(W)=\wideparen{E}(\mathscr{L})(W)\overrightarrow{\otimes}_{\wideparen{U}(\mathscr{L}^2)(V)}\mathcal{M}(V)= \wideparen{E}(\mathscr{L})(W)\overrightarrow{\otimes}_{\wideparen{E}(\mathscr{L})(V)}\wideparen{\mathbb{T}}^*\mathcal{M}(V),  
\end{equation*}
and we get our desired isomorphism.
\end{proof}
\begin{prop}\label{prop equivalence of categories of co-admissible odules over E and U}
For each admissible affinoid open $V\subset X^2$, we have mutually inverse equivalences of abelian categories:
\begin{equation*}
  \wideparen{\mathbb{T}}^*:\mathcal{C}(\wideparen{U}(\mathscr{L}^2)_{\vert V})\leftrightarrows   \mathcal{C}(\wideparen{E}(\mathscr{L})_{\vert V}):\wideparen{\mathbb{T}}^{-1,*}.  
\end{equation*}
\end{prop}
\begin{proof}
 As $\wideparen{\mathbb{T}}^*$ and  $\wideparen{\mathbb{T}}^{-1,*}$ are mutually inverse equivalences of quasi-abelian categories between 
 $\Mod_{\Indban}(\wideparen{U}(\mathscr{L}^2)_{\vert V})$ and $\Mod_{\Indban}(\wideparen{E}(\mathscr{L})_{\vert V})$, it suffices to show that the essential image under $\wideparen{\mathbb{T}}^*$ of $\mathcal{C}(\wideparen{U}(\mathscr{L}^2)_{\vert V})$ is $\mathcal{C}(\wideparen{E}(\mathscr{L})_{\vert V})$. By Proposition \ref{prop extension of scalars co-admissible modules sections} and Lemma \ref{Lemma extension of scalars co-admissible modules}, we have a commutative diagram of functors:
 \begin{equation*}
\begin{tikzcd}
\mathcal{C}(\wideparen{U}(\mathscr{L}^2)(V)) \arrow[d, "\operatorname{Loc}"] \arrow[r, "\wideparen{\mathbb{T}}^*"] & \mathcal{C}(\wideparen{E}(\mathscr{L}^2)(V)) \arrow[d, "\operatorname{Loc}"] \\
\mathcal{C}(\wideparen{U}(\mathscr{L}^2)_{\vert V}) \arrow[r, "\wideparen{\mathbb{T}}^*"]                          & \mathcal{C}(\wideparen{E}(\mathscr{L})_{\vert V})                           
\end{tikzcd}   
 \end{equation*}
 the upper horizontal functor is an equivalence by Proposition \ref{prop extension of scalars co-admissible modules sections}, and the vertical maps are equivalences by \cite[Theorem 8.4]{ardakov2019}, and Proposition \ref{prop properties of Loc}. Hence, the lower horizontal map is also an equivalence.
\end{proof}
Moreover, we get a version of Kiehl's Theorem for co-admissible $\wideparen{E}(\mathscr{L})$-modules:
\begin{coro}\label{coro Kiehl's Theorem for E-modules}
Let $V\subset X^2$ be an admissible affinoid open and choose a module $\mathcal{M}\in \Mod_{\Indban}(\wideparen{E}(\mathscr{L})_{\vert V})$. The following are equivalent:
\begin{enumerate}[label=(\roman*)]
    \item $\mathcal{M}$ is a co-admissible $\wideparen{E}(\mathscr{L}_{\vert V})$-module.
    \item There is an admissible affinoid cover $(U_i)_{i=1}^n$ of $V$ such that: $\mathcal{M}_{\vert U_i}$ is a co-admissible $\wideparen{E}(\mathscr{L})_{\vert U_i}$-module for each $1\leq i\leq n$.
    \item For any admissible affinoid cover $(U_i)_{i\in I}$ of $V$, $\mathcal{M}_{\vert U_i}$ is a co-admissible $\wideparen{E}(\mathscr{L})_{\vert U_i}$-module for each $i\in I$.
\end{enumerate}
\end{coro}
\begin{proof}
In virtue of Proposition \ref{prop equivalence of categories of co-admissible odules over E and U}, it suffices to show the corresponding statement for co-admissible modules over $\wideparen{U}(\mathscr{L}^2)_{\vert V}$. This follows by Kiehl's Theorem for co-admissible modules over the Fréchet-Stein enveloping algebra of a Lie algebroid shown in  \cite[Proposition 8.4]{ardakov2019}.  
\end{proof}
Now that we have a well-defined theory of co-admissible $\wideparen{E}(\mathscr{L})$-modules, we may adapt the theory of $\mathcal{C}$-complexes from \cite{bode2021operations} to this setting. We remark that, even if the results from \cite[Section 6]{bode2021operations}
are only stated for $\wideparen{\D}$-modules, the results also hold for co-admissible modules over any Lie algebroid.
\begin{defi}\label{defi C-complex for E}
Let $V\subset X^2$ be an admissible affinoid open, and choose a Fréchet-Stein presentation $\wideparen{E}(\mathscr{L})(V)=\varprojlim_n A_n$. A complex $C^{\bullet}\in\operatorname{D}(\wideparen{E}(\mathscr{L})_{\vert V})$ is called a $\mathcal{C}$-complex if it satisfies the following properties:
    \begin{enumerate}[label=(\roman*)]
        \item $\operatorname{H}^i(C^{\bullet})\in \mathcal{C}(\wideparen{E}(\mathscr{L})_{\vert V})$ for all $i\in\mathbb{Z}$.
        \item For any $n\in\mathbb{Z}$, we have:
        \begin{equation*}
            A_n\overrightarrow{\otimes}_{\wideparen{E}(\mathscr{L})(V)}\left(\operatorname{H}^i(C^{\bullet})(V)\right)=0,
        \end{equation*}
        for all but finitely many $i\in \mathbb{Z}$.
    \end{enumerate}
    We let $\operatorname{D}_{\mathcal{C}}(\wideparen{E}(\mathscr{L})_{\vert V})$  be the full subcategory of $\operatorname{D}(\wideparen{E}(\mathscr{L})_{\vert V})$ given by all $\mathcal{C}$-complexes. Given a closed embedding $Y\rightarrow V$, we let $\operatorname{D}_{\mathcal{C}^Y}(\wideparen{E}(\mathscr{L})_{\vert V})$ be the full subcategory of $\operatorname{D}_{\mathcal{C}}(\wideparen{E}(\mathscr{L})_{\vert V})$ given by the $\mathcal{C}$-complexes with cohomology supported on $Y$.
\end{defi}
\begin{obs}
This definition does not depend on the presentation:
\begin{equation*}
  \wideparen{E}(\mathscr{L})(V)=\varprojlim_n A_n.  
\end{equation*}
Indeed, let $\wideparen{E}(\mathscr{L})(V)=\varprojlim_m B_m$ be another Fréchet-Stein presentation. For any $m\in \mathbb{Z}$ there is some $n\in\mathbb{Z}$ such that the projection $\wideparen{E}(\mathscr{L})(V)\rightarrow B_m$ factors as:
\begin{equation*}
    \wideparen{E}(\mathscr{L})(V)\rightarrow A_n\rightarrow B_m.
\end{equation*}
Hence, the fact that condition $(ii)$ in Definition \ref{defi C-complex for E} holds for a Fréchet-Stein presentation implies that it holds for every Fréchet-Stein presentation.
\end{obs}
We can relate this notion to the category of $\mathcal{C}$-complexes over $\wideparen{U}(\mathscr{L}^2)$ via the following proposition:
\begin{prop}\label{prop equivalence of side-changing at the level of C-complexes local version}
For each admissible affinoid open $V\subset X^2$, we have mutually inverse equivalences of categories:
\begin{equation*}
  \wideparen{\mathbb{T}}^*:\operatorname{D}_{\mathcal{C}}(\wideparen{U}(\mathscr{L}^2)_{\vert V})\leftrightarrows   \operatorname{D}_{\mathcal{C}}(\wideparen{E}(\mathscr{L})_{\vert V}):\wideparen{\mathbb{T}}^{-1,*}.  
\end{equation*} 
\end{prop}
\begin{proof}
We show the case $V=X^2$, the general case is analogous. Let $C^{\bullet}\in \operatorname{D}_{\mathcal{C}}(\wideparen{E}(\mathscr{L}))$. As $\wideparen{\mathbb{T}}^{-1,*}$ is strongly exact, we have $\operatorname{H}^i(\wideparen{\mathbb{T}}^{-1,*}C^{\bullet})=\wideparen{\mathbb{T}}^{-1,*}\operatorname{H}^i(C^{\bullet})$, which is a co-admissible $\wideparen{U}(\mathscr{L}^2)$-module for each $i\in \mathbb{Z}$ by \ref{prop equivalence of categories of co-admissible odules over E and U}. As in the proof of Proposition \ref{prop extension of scalars co-admissible modules sections}, for each $n,m\in \mathbb{Z}$ we have:
\begin{equation*}
    \widehat{U}_{n,m}\overrightarrow{\otimes}_{\wideparen{U}}\wideparen{\mathbb{T}}^{-1,*}\operatorname{H}^i(C^{\bullet})(X^2)=\wideparen{\mathbb{T}}^{-1,*}\left(\widehat{E}_{n,m}\overrightarrow{\otimes}_{\wideparen{E}}\operatorname{H}^i(C^{\bullet})(X^2)\right).
\end{equation*}
As $C^{\bullet}$ is a $\mathcal{C}$-complex, the right hand side is zero for all but finitely many $i\in\mathbb{Z}$. Thus, $\wideparen{\mathbb{T}}^{-1,*}C^{\bullet}$ is a $\mathcal{C}$-complex of $\wideparen{U}(\mathscr{L}^2)$-modules, as wanted. The fact that $\wideparen{\mathbb{T}}^{*}$ sends $\mathcal{C}$-complexes to $\mathcal{C}$-complexes is analogous.
\end{proof}
\begin{coro}\label{coro properties of side-changing of C-complexes}
The following hold:
\begin{enumerate}[label=(\roman*)]
    \item $\operatorname{D}_{\mathcal{C}}(\wideparen{E}(\mathscr{L})_{\vert V})$ is a triangulated subcategory of $\operatorname{D}(\wideparen{E}(\mathscr{L})_{\vert V})$.
    \item Let $Y\rightarrow V$ be a closed immersion. Then the equivalence of Proposition \ref{prop equivalence of side-changing at the level of C-complexes local version} restricts to an equivalence of triangulated categories:
    \begin{equation*}
        \wideparen{\mathbb{T}}^*:\operatorname{D}_{\mathcal{C}^Y}(\wideparen{U}(\mathscr{L}^2)_{\vert V})\leftrightarrows   \operatorname{D}_{\mathcal{C}^Y}(\wideparen{E}(\mathscr{L})_{\vert V}):\wideparen{\mathbb{T}}^{-1,*}.
    \end{equation*}
\end{enumerate}
\end{coro}
\begin{proof}
Statement $(i)$ follows by the corresponding fact for $\mathcal{C}$-complexes over $\wideparen{U}(\mathscr{L}^2)$, which is shown in \cite[Proposition 6.4]{bode2021operations}. Statement $(ii)$ holds because $\wideparen{\mathbb{T}}^*$  and  $\wideparen{\mathbb{T}}^{-1,*}$  do not change the supports.
\end{proof}
\subsection{Co-admissible modules over the bi-enveloping algebra II: Globalization}\label{section co-admis over bienvelop 2}
 In order to generalize the definition of co-admissible $\wideparen{E}(\mathscr{L})$-modules to the case where $X$ is not necessarily affinoid and $\mathscr{L}$ is not free, we will need to relate the functors $\wideparen{\mathbb{T}}^*$ and $\wideparen{\mathbb{T}}^{-1,*}$ to the globally defined side-switching operations defined in Section \ref{Section side-switching operations}. In particular, keeping the assumptions of the previous section, we have the following proposition:
 \begin{prop}\label{prop equivalence of pullback and side-changing sheaves}
There are canonical natural equivalences of functors:
\begin{equation*}
    \wideparen{\mathbb{T}}^* \simeq \operatorname{S}:=p_2^{*}\Omega_{\mathscr{L}}\overrightarrow{\otimes}_{\OX_{X^2}}-, \quad \wideparen{\mathbb{T}}^{-1,*} \simeq \operatorname{S}^{-1}:=p_2^{*}\Omega_{\mathscr{L}}^{-1}\overrightarrow{\otimes}_{\OX_{X^2}}-.
\end{equation*}
 \end{prop}
\begin{proof}
Let $V\subset X^2$ be an affinoid subdomain, and $M\in \Mod_{\Indban}(\wideparen{U}(\mathscr{L}^2)(V))$. We need to obtain a canonical identification:
\begin{equation*}
    \Omega_{\mathscr{L}(p_2(V))}\overrightarrow{\otimes}_{\OX_{X}(p_2(V))}M\rightarrow   \wideparen{E}(\mathscr{L})(V)\overrightarrow{\otimes}_{\wideparen{U}(\mathscr{L}^2)(V)}M,
\end{equation*}
which is functorial in $M$ and $V$. As in the proof of Proposition \ref{prop side changing for bimodules}, we can freely assume that $M$ is a complete bornological space. Consider the following bornological algebras:
 \begin{multline*}
     \mathscr{A}(V)=\wideparen{U}(\mathscr{L})(p_1(V))\otimes_K\wideparen{U}(\mathscr{L})(p_2(V)),\\ \mathscr{B}(V)=\wideparen{U}(\mathscr{L})(p_1(V))\otimes_K\wideparen{U}(\mathscr{L})^{\op}(p_2(V)).
 \end{multline*}
 By Proposition \ref{prop enveloping of product as product of envelopings} and definition \ref{defi sheaf of complete bornological K algebras}, we have the following identities of Ind-Banach algebras:
 \begin{equation*}
     \wideparen{U}(\mathscr{L}^2)(V)=\operatorname{diss}(\widehat{\mathscr{A}} \,\,), \quad \wideparen{E}(\mathscr{L})(V)=\operatorname{diss}(\widehat{\mathscr{B}}\,).
 \end{equation*}
As $M$ is bornological, it follows by \cite[Proposition 4.3]{bode2021operations} that we may regard $M$ as a complete bornological $\widehat{\mathscr{A}}$-module. A fortiori, this restricts to a structure as an $\mathscr{A}$-module. In particular, $M$ has an underlying $K$-vector space. Hence, we can use (a slight modification of) the discussion above Proposition \ref{prop equivalence of pullback and side-change for modules} to show that there is a canonical identification of bornological $\mathscr{B}$-modules:
\begin{equation*}
    \mathscr{B}\otimes_{\mathscr{A}}\mathcal{M}(V)=\Omega_{\mathscr{L}(p_2(V))}\otimes_{\OX_{X}(p_2(V))}\mathcal{M}(V).
\end{equation*}
Applying completion and the dissection functor, we obtain the identification we wanted. The fact that the isomorphism in Proposition \ref{prop equivalence of pullback and side-change for modules} is functorial, together with functoriality of completion and dissection shows that this process is functorial in both $V\subset X^2$, and $M\in \Mod_{\Indban}(\wideparen{U}(\mathscr{L}^2)(V))$, as we wanted to show.
\end{proof}

We will now extend the definition of co-admissible $\wideparen{E}(\mathscr{L})$-modules to the global case. For the rest of this section, we let $X$ be a smooth and separated rigid space equipped with a Lie algebroid $\mathscr{L}$.\bigskip

Let $(U_i)_{i\in I}$ be an admissible affinoid cover of $X$ such that $\mathscr{L}_{\vert U_i}$ is free for each $i\in I$, and let $U_i=\Sp(A_i)$. In this situation, the family $\mathfrak{U}=(U_i\times U_j)_{i,j\in I}$ is an admissible affinoid cover of $X^2$. For each $i,j\in I$, the disjoint union $V_{ij}=U_i\cup U_j$ is an affinoid space, and $\mathscr{L}$ induces a unique Lie algebroid on $V_{ij}$. Namely, we define $\mathscr{L}_{V_{ij}}$ as the unique Lie algebroid on $V_{ij}$ induced by the Lie algebroids $\mathscr{L}_{\vert U_i}$ and $\mathscr{L}_{\vert U_j}$ on $U_i$, and $U_j$ respectively. Notice that, in this situation, $V_{ij}$ satisfies the conditions at the beginning of Section \ref{Section Co-admissible modules over the bi-enveloping algebra}. Hence, we have a well defined category sheaves of co-admissible modules over $\wideparen{E}(\mathscr{L}_{V_{ij}})$. Furthermore, we have an open immersion:
\begin{equation*}
    U_i\times U_j\subset V_{ij}\times V_{ij}=V_{ij}^2.
\end{equation*}
Thus, we have a well-defined category $\mathcal{C}(\wideparen{E}(\mathscr{L}_{V_{ij}})_{\vert U_i\times U_j})$, as defined on \ref{defi localization functor for co-admissible E(l)-modules}.
\begin{Lemma}\label{Lemma 2 kiehl's Theorem for E-modules on general rigid spaces}
  There is an identification of sheaves of Ind-Banach algebras:
  \begin{equation*}
      \wideparen{E}(\mathscr{L})_{\vert U_i\times U_j}=\wideparen{E}(\mathscr{L}_{V_{ij}})_{\vert U_i\times U_j}.
  \end{equation*}
\end{Lemma}
\begin{proof}
    By construction, we have:
    \begin{align*}
   \wideparen{E}(\mathscr{L})_{\vert U_i\times U_j}=&\left(\OX_{X^2}\overrightarrow{\otimes}_{p_1^{-1}\OX_X\overrightarrow{\otimes}_Kp_2^{-1}\OX_X}\left(p_1^{-1}\wideparen{U}(\mathscr{L})\overrightarrow{\otimes}_Kp_2^{-1}\wideparen{U}(\mathscr{L})^{\op} \right)\right)_{\vert U_i\times U_j}\\
=&\OX_{U_i\times U_j}\overrightarrow{\otimes}_{p_1^{-1}\OX_{U_i}\overrightarrow{\otimes}_Kp_2^{-1}\OX_{U_2}}\left(p_1^{-1}\wideparen{U}(\mathscr{L}_{\vert U_i})\overrightarrow{\otimes}_Kp_2^{-1}\wideparen{U}(\mathscr{L}_{\vert U_j})^{\op}\right).     
    \end{align*}
    Similarly, we also have:
    \begin{align*}
 \wideparen{E}(\mathscr{L}_{V_{ij}})_{\vert U_i\times U_j}=&\left(\OX_{V_{ij}^2}\overrightarrow{\otimes}_{p_1^{-1}\OX_{V_{ij}}\overrightarrow{\otimes}_Kp_2^{-1}\OX_{V_{ij}}}p_1^{-1}\wideparen{U}(\mathscr{L}_{V_{ij}})\overrightarrow{\otimes}_Kp_2^{-1}\wideparen{U}(\mathscr{L}_{V_{ij}})^{\op} \right)_{\vert U_i\times U_j}\\
 =&\OX_{U_i\times U_j}\overrightarrow{\otimes}_{p_1^{-1}\OX_{U_i}\overrightarrow{\otimes}_Kp_2^{-1}\OX_{U_2}}\left(p_1^{-1}\wideparen{U}(\mathscr{L}_{_{V_{ij}}\vert U_i})\overrightarrow{\otimes}_Kp_2^{-1}\wideparen{U}(\mathscr{L}_{_{V_{ij}}\vert U_j})^{\op}\right).             
    \end{align*}
 Moreover, we have $\mathscr{L}_{_{V_{ij}}\vert U_i}=\mathscr{L}_{\vert U_i}$ and $\mathscr{L}_{_{V_{ij}}\vert U_j}=\mathscr{L}_{\vert U_j}$, so the statement holds.   
\end{proof}
Hence, for each pair of affinoid subdomains $U,V\subset X$ such that the restrictions of $\mathscr{L}$ to $U$ and $V$ are free, we have a well-defined category $\mathcal{C}(\wideparen{E}(\mathscr{L})_{\vert U\times V})$. We may use this to show the following:
\begin{Lemma}\label{Lemma kiehl's Theorem for E-modules on general rigid spaces}
Let $\mathcal{M}\in \Mod_{\Indban}(\wideparen{E}(\mathscr{L}))$. The following are equivalent:
\begin{enumerate}[label=(\roman*)]
    \item There is an admissible affinoid cover $(U_i)_{i\in I}$ of $X$ such that $\mathscr{L}_{\vert U_i}$ is free for each $i\in I$ and $\mathcal{M}_{\vert U_i\times U_j}$ is a co-admissible $\mathcal{C}(\wideparen{E}(\mathscr{L})_{\vert U_i\times U_j})$-module for each $i,j\in I$.
    \item For every admissible affinoid cover $(U_i)_{i\in I}$ of $X$ such that $\mathscr{L}_{\vert U_i}$ is free for each $i\in I$, we have that $\mathcal{M}_{\vert U_i\times U_j}$ is a co-admissible $\wideparen{E}(\mathscr{L})_{\vert U_i\times U_j}$-module for each $i,j\in I$.
\end{enumerate}
\end{Lemma}
\begin{proof}
Assume there is an admissible affinoid cover $(U_i)_{i\in I}$ satisfying the conditions in $(i)$, and let $(V_j)_{j\in J}$ be an admissible affinoid cover of $X$ such that $\mathscr{L}_{\vert V_j}$ is a free $\OX_{\vert V_j}$-module for each $j\in J$. As $X$ is separated, for every $j\in J$ we have an admissible affinoid cover $(U_i\cap V_j)_{i\in I}$ of $V_j$. Consider the following affinoid spaces:
\begin{equation*}
    W_{il}=(U_i\cap V_j)\times (U_l\cap V_k)=(U_i\times U_l)\cap (V_j\times V_k).
\end{equation*}
For each $j,k\in J$, we have an admissible affinoid cover of $V_j\times V_k$ given by $(W_{il})_{i,l\in I}$. By assumption, we have $\mathcal{M}_{\vert U_i\times U_l}\in \mathcal{C}(\wideparen{E}(\mathscr{L})_{\vert U_i\times U_l})$. Hence, by  Corollary \ref{coro Kiehl's Theorem for E-modules}, and Lemma \ref{Lemma 2 kiehl's Theorem for E-modules on general rigid spaces}, it follows that $\mathcal{M}_{\vert W_{il}}\in \mathcal{C}(\wideparen{E}(\mathscr{L})_{\vert W_{il}})$. Again, using these two results, it follows that $\mathcal{M}_{\vert V_j\times V_k}\in \mathcal{C}(\wideparen{E}(\mathscr{L})_{\vert V_j\times V_k})$ for each $j,k\in J$.
\end{proof}
\begin{defi}
Let $\mathcal{C}(\wideparen{E}(\mathscr{L}))$ be the full subcategory of $\Mod_{\Indban}(\wideparen{E}(\mathscr{L}))$ given by the modules satisfying the conditions of Lemma \ref{Lemma kiehl's Theorem for E-modules on general rigid spaces}. We call  $\mathcal{C}(\wideparen{E}(\mathscr{L}))$ the category of co-admissible modules over $\wideparen{E}(\mathscr{L})$.
\end{defi}
We will now study the properties of $\mathcal{C}(\wideparen{E}(\mathscr{L}))$:
\begin{Lemma}
    $\mathcal{C}(\wideparen{E}(\mathscr{L}))$ is an abelian category.
\end{Lemma}
\begin{proof}
As $\Mod_{\Indban}(\wideparen{E}(\mathscr{L}))$ is a quasi-abelian category, it suffices to show that every morphism $f:\mathcal{M}\rightarrow \mathcal{N}$ in $\mathcal{C}(\wideparen{E}(\mathscr{L}))$ is strict. As this is a local property, we may assume that $X$ is affinoid and $\mathscr{L}$ is free. In this case, by Proposition \ref{prop properties of Loc}
for any $V\in X^2$, the map $f:\mathcal{M}(V)\rightarrow\mathcal{N}(V)$ is a morphism of co-admissible modules over a Fréchet-Stein algebra. This is strict by \cite[Corollary 3.4]{schneider2002algebras}. Thus, the map $f:\mathcal{M}\rightarrow\mathcal{N}$ is strict, as we wanted.
\end{proof}
\begin{teo}\label{teo equivalence of co-admissible modules under side-changing}
 There are mutually inverse equivalences of abelian categories:
\begin{equation*}
\operatorname{S}:\mathcal{C}(\wideparen{U}(\mathscr{L}^2))\leftrightarrows\mathcal{C}(\wideparen{E}(\mathscr{L})):\operatorname{S}^{-1}.
\end{equation*}    
Furthermore, this equivalence preserves the supports of the sheaves.
\end{teo}
\begin{proof}
It suffices to show that  $\operatorname{S}$  and $\operatorname{S}^{-1}$ send co-admissible modules to co-admissible modules. This may be done locally, so we may assume $X$ is affinoid and $\mathscr{L}$ is a free $\OX_X$-module. But then this follows by Propositions \ref{prop equivalence of categories of co-admissible odules over E and U} and \ref{prop equivalence of pullback and side-changing sheaves}.
\end{proof}
We can finally give a global definition of $\mathcal{C}$-complexes:
\begin{defi}
A complex $C^{\bullet}\in \operatorname{D}(\wideparen{E}(\mathscr{L}))$ is called a $\mathcal{C}$-complex if there is an affinoid cover $(U_i)_{i\in I}$ of $X$ satisfying the following:
\begin{enumerate}[label=(\roman*)]
    \item $\mathscr{L}_{\vert U_i}$ is a globally free $\OX_{U_i}$-module for each $i\in I$.
    \item $C_{\vert U_i\times U_j}^{\bullet}\in \operatorname{D}_{\mathcal{C}}(\wideparen{E}(\mathscr{L})_{\vert U_i\times U_j})$ for each $i,j\in I$.
\end{enumerate}
We denote the full subcategory of all $\mathcal{C}$-complexes by $\operatorname{D}_{\mathcal{C}}(\wideparen{E}(\mathscr{L}))$.
\end{defi}
\begin{defi}
  Given a closed embedding $Y\rightarrow X^2$, we let $\operatorname{D}_{\mathcal{C}^Y}(\wideparen{E}(\mathscr{L}))$ be the full subcategory of $\mathcal{C}$-complexes with cohomology supported on $Y$.   
\end{defi}
After all this preparation, we are finally ready to obtain a side-switching equivalence for $\mathcal{C}$-complexes:
\begin{prop}\label{prop equivalence of side-changing at the level of C-complexes}
Let $X$ be a smooth and separated rigid variety with a Lie algebroid $\mathscr{L}$. We have mutually inverse equivalences of triangulated categories:
\begin{equation*}
  \operatorname{S}:\operatorname{D}_{\mathcal{C}}(\wideparen{U}(\mathscr{L}^2))\leftrightarrows   \operatorname{D}_{\mathcal{C}}(\wideparen{E}(\mathscr{L})):\operatorname{S}^{-1}.  
\end{equation*} 
Furthermore,  for any closed immersion $Y\rightarrow X$, these equivalences restrict to a pair of mutually inverse equivalences of triangulated categories:
\begin{equation*}
    \operatorname{S}:\operatorname{D}_{\mathcal{C}^Y}(\wideparen{U}(\mathscr{L}^2))\leftrightarrows   \operatorname{D}_{\mathcal{C}^Y}(\wideparen{E}(\mathscr{L})):\operatorname{S}^{-1}.
\end{equation*}
\end{prop}
\begin{proof}
This follows by Theorem \ref{teo equivalence of co-admissible modules under side-changing} and Proposition \ref{prop equivalence of side-changing at the level of C-complexes local version}.
\end{proof}
As stated in Remark \ref{remark same results hold for right modules over E(L)}, all the results thus far hold for the categories of right $\wideparen{E}(\mathscr{L})$-modules. In particular, we also have a category of right $\mathcal{C}$-complexes over $\wideparen{E}(\mathscr{L})$, and an equivalence of triangulated categories:
\begin{equation*}
    \operatorname{S}_r:\operatorname{D}_{\mathcal{C}}(\wideparen{U}(\mathscr{L}^2)^{\op})\leftrightarrows   \operatorname{D}_{\mathcal{C}}(\wideparen{E}(\mathscr{L})^{\op}):\operatorname{S}_r^{-1}.
\end{equation*}
Thus, we have the following corollary:
\begin{coro}\label{equivalence left and right categories of C-complexes}
There is an equivalence of triangulated categories:
\begin{equation*}
  \operatorname{T}:\operatorname{D}_{\mathcal{C}}(\wideparen{E}(\mathscr{L}))\rightarrow   \operatorname{D}_{\mathcal{C}}(\wideparen{E}(\mathscr{L})^{\op}),  
\end{equation*} 
which restricts to an equivalence of abelian categories:
\begin{equation*}
    \operatorname{T}:\mathcal{C}(\wideparen{E}(\mathscr{L}))\rightarrow \mathcal{C}(\wideparen{E}(\mathscr{L})^{\op}).
\end{equation*}
Furthermore, if $Y\rightarrow X$ is a closed immersion, then the restriction:
\begin{equation*}   \operatorname{T}:\operatorname{D}_{\mathcal{C}^Y}(\wideparen{E}(\mathscr{L}))\rightarrow \operatorname{D}_{\mathcal{C}^Y}(\wideparen{E}(\mathscr{L})^{\op}),
\end{equation*}
is also an equivalence of triangulated categories.  
\end{coro}
\begin{proof}
By Proposition \ref{prop equivalence left and right E(L)-modules} we have the identity:
\begin{equation*}
    \operatorname{T}=\operatorname{S}_r\left(\Omega_{\mathscr{L}^2}\overrightarrow{\otimes}_{\OX_{X^2}}\operatorname{S}^{-1}(-) \right).
\end{equation*}
Thus, it suffices to show that each of the functors on the right hand side satisfies the statements of the corollary. This is a consequence of Proposition \ref{prop equivalence of side-changing at the level of C-complexes}, as well as \cite[Theorem 4.11]{bode2021operations}.
\end{proof}
\begin{obs}
    In virtue of this corollary, we will from now on work with left $\wideparen{E}(\mathscr{L})$-modules unless otherwise mentioned. In particular, the results will be stated and proved for left modules. The equivalence  $\operatorname{T}$ ensures that the corresponding facts for right modules still hold.
\end{obs}
\section{Co-admissibility and the immersion functor}\label{Chapter co-admissibility and the immersion functor}
Let $X$ be a separated smooth rigid analytic space with a Lie algebroid $\mathscr{L}$. This section is devoted to understanding the behavior of co-admissibility with respect to the immersion functor: 
\begin{equation*}
\Delta_*^{\operatorname{S}}:\Mod_{\Indban}(\wideparen{U}(\mathscr{L})^e)\rightarrow \Mod_{\Indban}(\wideparen{U}(\mathscr{L}^2)).
\end{equation*}
In particular, our goal is studying certain subcategories of $\Mod_{\Indban}(\wideparen{U}(\mathscr{L})^e)$ 
such that their essential images under $\Delta^{\operatorname{S}}_*$ are Serre subcategories of  $\mathcal{C}(\wideparen{U}(\mathscr{L}^2))$. This will allow us to employ the homological machinery  of $\mathcal{C}$-complexes to our study of Hochschild (co)-homology. Due to some technical issues, we will focus on the case $\mathscr{L}=\mathcal{T}_{X/K}$. However, We would like to point out that we believe that most of the technology developed in this section should generalize to the case of an arbitrary Lie algebroid. For this reason, we will work in the general setting as much as we are currently able to.\bigskip

The first step is identifying  a notion of co-admissible $\wideparen{\D}_X$-bimodules that behaves well with respect to the immersion functor. Recall that  our interest of the bi-enveloping algebra $\wideparen{E}_X$ stems from the fact that $\wideparen{\D}_X^e$ is a very complicated sheaf. In particular, it is not clear if $\wideparen{\D}_X^e$ is locally Fréchet-Stein. In fact, it is not even clear if $\wideparen{\D}_X^e$ is a sheaf of complete bornological algebras. Thus, we cannot rely on the properties of $\wideparen{\D}_X^e$ to define a good notion of co-admissibility for bimodules.\\

As a first approximation, we will define
$\mathcal{C}^{\operatorname{Bi}}_B(\wideparen{\D}_X)$ as the full subcategory of  $\Mod_{\Indban}(\wideparen{\D}_X^e)$ given by the modules satisfying that they are co-admissible as left and right $\wideparen{\D}_X$-modules. This category satisfies some of the properties we are looking for. In particular, we will show below that if $X$ is an affinoid space admitting an étale map $X\rightarrow \mathbb{A}^n_K$, and $\mathcal{M}\in \mathcal{C}^{\operatorname{Bi}}_B(\wideparen{\D}_X)$, then $\mathcal{M}(X)$ is canonically a co-admissible $\wideparen{E}_X(X^2)$-module. However, it is not clear that $\Delta_*(\mathcal{M})$ is a co-admissible $\wideparen{E}_X$-module.  Indeed, under the assumptions above, $\Delta_*(\mathcal{M})$ is a co-admissible if and only if we have an isomorphism of $\wideparen{E}_X$-modules:
 \begin{equation*}
     \Delta^E_*(\mathcal{M})=\operatorname{Loc}(\mathcal{M}(X)),
 \end{equation*}
and it is not clear that the right hand side has support contained in $\Delta\subset X^2$.\bigskip

In order to remedy this, we need to impose conditions on $\mathcal{M}(X)$ which ensure that  $\operatorname{Loc}(\mathcal{M}(X))$ has support contained in the diagonal. This line of thought leads to the category of co-admissible diagonal bimodules, which we denote 
$\mathcal{C}_B^{\Bi}(\wideparen{\D}_X)_{\Delta}$. As we will see below, $\mathcal{C}_B^{\Bi}(\wideparen{\D}_X)_{\Delta}$ is a full abelian subcategory of $\mathcal{C}_B^{\Bi}(\wideparen{\D}_X)$
which satisfies that the immersion functor:
\begin{equation*}
    \Delta_*^{\operatorname{S}}:\mathcal{C}_B^{\Bi}(\wideparen{\D}_X)_{\Delta}\rightarrow \mathcal{C}(\wideparen{\D}_{X^2})_{\Delta},
\end{equation*}
is an exact immersion of abelian categories. Furthermore, the essential image of $\mathcal{C}_B^{\Bi}(\wideparen{\D}_X)_{\Delta}$ is a Serre subcategory of $\mathcal{C}(\wideparen{\D}_{X^2})_{\Delta}$.\bigskip

As above, these results are later generalized in the geometric and homological directions. In the geometric direction, we will show that the definitions can be generalized to an arbitrary separated and smooth rigid analytic space $X$. In the homological direction, we will define an adequate notion of diagonal $\mathcal{C}$-complexes, and show that the results above extend to the corresponding triangulated categories.
\subsection{Categories of co-admissible bimodules}\label{Section Categories of co-admissible bimodules}

Our study of Hochschild cohomology will be centered around the study of the following category:
\begin{defi}\label{defi sehaves of co-admissible bimodules}
Let $\mathcal{C}^{\operatorname{Bi}}_B(\wideparen{U}(\mathscr{L}))$
    be  the full subcategory of $\Mod_{\Indban}(U(\mathscr{L})^e)$ given by objects $\mathcal{M}$, such that $\mathcal{M}$ is co-admissible as a $\wideparen{U}(\mathscr{L})$-module and as a  $\wideparen{U}(\mathscr{L})^{\op}$-module. We call $\mathcal{C}^{\operatorname{Bi}}_B(\wideparen{U}(\mathscr{L}))$ the category of co-admissible $\wideparen{U}(\mathscr{L})$-bimodules.
\end{defi}
In this section we will show that, locally on $X$, global sections of co-admissible bimodules are co-admissible over the complete bi-enveloping algebra of $\mathscr{L}$.\newline
The contents of this section rely heavily on techniques from $p$-adic functional analysis. For this reason, it will be convenient to work with locally convex or bornological spaces instead of Ind-Banach spaces. We remark that all the spaces we will be interested in are (bornologifications of) Fréchet spaces which satisfy some nuclearity condition. Thus, not much will be lost with this change (\emph{cf.} \cite[Lemma 4.3]{bode2021operations}).\\
We denote the category of locally convex $K$-vector spaces by $LCS_K$, and let $\widehat{\otimes}_K^{LCS}$ be the completed projective tensor product in  $LCS_K$ (cf. \cite[Chapter IV 17.B]{schneider2013nonarchimedean}). As before, we let $\widehat{\mathcal{B}}c_K$ be the category of complete bornological spaces, and let $\widehat{\otimes}_K$ denote the completed projective bornological tensor product.\bigskip

Assume that $X=\Sp(A)$ is a smooth affinoid variety, and let $\mathfrak{X}=\Spf(\mathcal{A})$ be an affine formal model of $X$. Additionally, we will assume that $L:=\mathscr{L}(X)$ admits a free $(\mathcal{R},\mathcal{A})$-Lie lattice $\mathcal{L}$. As shown in Corollary \ref{coro frechet-stein presentation of bi-enveloping algebra}, for each non-negative integer $n,m\geq 0$ we have a Fréchet-Stein presentation:
\begin{equation*}
    \wideparen{E}(L)=\varprojlim_s\left(\widehat{U}(\pi^{n+s}\mathcal{L})_K\widehat{\otimes}_K\widehat{U}(\pi^{m+s}\mathcal{L})_K^{\op}\right)=\wideparen{U}(L)\widehat{\otimes}_K\wideparen{U}(L)^{\op}.
\end{equation*}
In order to keep thing simple, for each non-negative integer $n,m\geq 0$ we will write:
\begin{equation*}
    \widehat{E}_{n,m}:=\widehat{U}(\pi^{n}\mathcal{L})_K\widehat{\otimes}_K\widehat{U}(\pi^{m}\mathcal{L})_K^{\op}, \quad \widehat{U}_{n}=\widehat{U}(\pi^{n}\mathcal{L})_K.
\end{equation*}
Similarly, we write $\wideparen{E}=\wideparen{E}(L)$ and $\wideparen{U}=\wideparen{U}(L)$. Let $\mathcal{M}\in\mathcal{C}(\wideparen{U})$ be a co-admissible $\wideparen{U}$-module. In particular, writing $\mathcal{M}_n:=\widehat{U}_n\otimes_{\wideparen{U}}\mathcal{M}$, we have: $\mathcal{M}=\varprojlim_n \mathcal{M}_n$.
\begin{defi}\label{defi different definitions of co-admissible bimodules}
We define the following full subcategories of $\Mod_{\widehat{\mathcal{B}}c}(\wideparen{E})$:
\begin{enumerate}[label=(\roman*)]
    \item Let  $\mathcal{C}_L^{\operatorname{Bi}}(\wideparen{U})$
    be given by the objects $\mathcal{M}$ which are co-admissible $\wideparen{U}$-modules.
    \item Let  $\mathcal{C}_R^{\operatorname{Bi}}(\wideparen{U})$
    be given by the objects $\mathcal{M}$ which are co-admissible $\wideparen{U}^{\op}$-modules.
    \item Let  $\mathcal{C}_B^{\operatorname{Bi}}(\wideparen{U})$
    be the intersection of the two previous categories.
\end{enumerate}
\end{defi}
Let $\mathcal{M}\in \mathcal{C}_L^{\operatorname{Bi}}(\wideparen{U}(L))$, and $\mathcal{M}=\varprojlim_n \mathcal{M}_n$ be a co-admissible $\wideparen{U}$-module presentation of $\mathcal{M}$. As shown in \cite[\textnormal{Section 7.2}]{ardakov2019},
there is a canonical $\wideparen{U}$-linear isomorphism:
\begin{equation*}
    \operatorname{End}_{\wideparen{U}}(\mathcal{M})=\varprojlim_n \operatorname{End}_{\widehat{U}_n}(\mathcal{M}_n).
\end{equation*}
Each $\mathcal{M}_n$ is a finite module over a Banach $K$-algebra. Hence, it is a $K$-Banach space. Thus, we may regard $\operatorname{End}_{K}(\mathcal{M}_n)$ as a Banach space with respect to the operator norm. The canonical inclusion $\operatorname{End}_{\widehat{U}_n}(\mathcal{M}_n)\subset \operatorname{End}_{K}(\mathcal{M}_n)$ has closed image. Thus, $\operatorname{End}_{\widehat{U}_n}(\mathcal{M}_n)$ is also a Banach space.
Hence, we may regard $\operatorname{End}_{\wideparen{U}}(\mathcal{M})$ as a topological $K$-vector space with respect to the inverse limit topology. This topology turns $\operatorname{End}_{\wideparen{U}}(\mathcal{M})$  into a Fréchet space. From now on, we will always regard $\operatorname{End}_{\wideparen{U}}(\mathcal{M})$ as a Fréchet space with respect to this topology.\bigskip

By definition of $\mathcal{M}\in\mathcal{C}_L^{\operatorname{Bi}}(\wideparen{U}(L))$, we have a bounded action:
\begin{equation*}
\mathcal{M}\widehat{\otimes}_K\left(\wideparen{U}\widehat{\otimes}_K\wideparen{U}^{\op}\right)\rightarrow \mathcal{M}.
\end{equation*}
Furthermore, by \cite[Proposition 5.5]{bode2021operations}, the spaces 
$\mathcal{M}$, $\wideparen{U}$ and $\wideparen{U}^{\op}$ are $A$-nuclear Fréchet spaces. Hence, we may apply  \cite[Corollary 5.20]{bode2021operations} to show that this map arises from a continuous action:
\begin{equation*}
\mathcal{M}\widehat{\otimes}^{LCS}_K\left(\wideparen{U}\widehat{\otimes}^{LCS}_K\wideparen{U}^{\op}\right)\rightarrow \mathcal{M},
\end{equation*}
between the associated locally convex spaces. The properties of the projective tensor product in $LCS_K$ show that this is equivalent to a continuous map:
\begin{equation*}
    \wideparen{U}^{\op}\rightarrow \operatorname{End}_{\wideparen{U}}(\mathcal{M}).
\end{equation*}
As both spaces are Fréchet, it follows that this map is continuous if and only if for every $n\geq 0$ there is some $m\geq 0$ such that we have a commutative diagram:
\begin{equation*}
\begin{tikzcd}
\wideparen{U}^{\op} \arrow[r] \arrow[d] & \operatorname{End}_{\wideparen{U}}(\mathcal{M}) \arrow[d] \\
\widehat{U}_m^{\op} \arrow[r]      & \operatorname{End}_{\widehat{U}_n}(\mathcal{M}_n)   
\end{tikzcd}
\end{equation*}
where all maps in the diagram are continuous maps of Banach $K$-algebras.
\begin{prop}\label{prop topology of co-admissible bimodules}
    Let $\mathcal{M}\in \mathcal{C}^{\Bi}_L(\wideparen{U})$, and assume that $\mathcal{M}$ is co-admissible as an abstract $\wideparen{U}^{\op}$-module. Then the canonical topologies of $\mathcal{M}$ as a co-admissible $\wideparen{U}$-module and as a co-admissible $\wideparen{U}^{\op}$-module agree. In particular, $\mathcal{M}\in \mathcal{C}^{\Bi}_B(\wideparen{U})$.
\end{prop}
\begin{proof}
Define the following families of finite $\widehat{U}_n$-modules, resp. finite $\widehat{U}^{\op}_m$-modules:
\begin{equation*}
    \mathcal{M}_n=\widehat{U}_n\otimes_{\wideparen{U}}\mathcal{M}, \quad \mathcal{M}^r_m=\widehat{U}^{\op}_m\otimes_{\wideparen{U}^{\op}}\mathcal{M}=\mathcal{M}\otimes_{\wideparen{U}}\widehat{U}_m.
\end{equation*}
Then we have the following inverse limits:
\begin{equation*}
    \mathcal{M} =\varprojlim \mathcal{M}_n \textnormal{,  } \mathcal{M} =\varprojlim \mathcal{M}^r_m,
\end{equation*}
where the first is an isomorphism of topological vector spaces, and the second one is an isomorphism of abstract $\wideparen{U}^{\op}$-modules. We let 
$\mathcal{M}^r$ be $\mathcal{M}$ regarded as a Fréchet space with respect to the topology given by the inverse limit $\mathcal{M} =\varprojlim \mathcal{M}^r_m$.\bigskip

By the open mapping theorem, $\mathcal{M}=\mathcal{M}^r$ as topological $K$-vector spaces if and only if the identity map $\mathcal{M}\rightarrow \mathcal{M}^r$ is continuous. As shown above, the fact that $\mathcal{M}\in \mathcal{C}^{\Bi}_L(\wideparen{U})$ implies that the map $\Phi:\wideparen{U}^{\op}\rightarrow \operatorname{End}_{\wideparen{U}}(\mathcal{M})$ is continuous. Hence, for each $n\geq 0$ there is a minimal $m\geq 0$ such that $\Phi$ factors as a continuous map of Banach $K$-algebras:
\begin{equation*}
  \Phi_{n}:\widehat{U}_m^{\op}\rightarrow \operatorname{End}_{\widehat{U}_n}\left(\mathcal{M}_n\right).  
\end{equation*}
Fix  $n,m\geq 0$ as above. We have the following isomorphisms of $\widehat{U}_n\otimes_K\widehat{U}_m$-modules:
\begin{equation}\label{equation 1 frechet topology coadmissible bimodules}
       \widehat{U}_n\otimes_{\wideparen{U}}\mathcal{M}^r_m \rightarrow \widehat{U}_n\otimes_{\wideparen{U}}\mathcal{M}\otimes_{\wideparen{U}}\widehat{U}_m\rightarrow \mathcal{M}_n\otimes_{\wideparen{U}}\widehat{U}_m.      
\end{equation}
Additionally, we have a pair of $\widehat{U}_m^{\op}$-linear maps:
\begin{equation}\label{equation 2 frechet topology coadmissible bimodules}
\mathcal{M}^r_m\rightarrow   \widehat{U}_n\otimes_{\wideparen{U}} \mathcal{M}^r_m, \quad \mathcal{M}_n\otimes_{\wideparen{U}}\widehat{U}_m\rightarrow \mathcal{M}_n,
\end{equation}
where the first one is the canonical inclusion $x\mapsto x\otimes 1$, and the second one is given by the $\wideparen{U}$-linear action of $\widehat{U}^{\op}_m$ on $\mathcal{M}_n$ induced by $\Phi_n$.\\
Combining the maps in $(\ref{equation 1 frechet topology coadmissible bimodules})$ and $(\ref{equation 2 frechet topology coadmissible bimodules})$, we arrive at a $\widehat{U}_m^{\op}$-linear map:
\begin{equation*}
    \psi_{n,m}: \mathcal{M}_m^r\rightarrow \mathcal{M}_n,
\end{equation*}
which fits into the following commutative diagram:
\begin{equation*}
\begin{tikzcd}
\mathcal{M}^r \arrow[r] \arrow[d] & \mathcal{M} \arrow[d] \\
\mathcal{M}_m^r \arrow[r]         & \mathcal{M}_n        
\end{tikzcd}
\end{equation*}
where the upper horizontal map is the identity. As the topologies on $\mathcal{M}$ and $\mathcal{M}^r$ are given by inverse limits, it is enough to show that $\psi_{n,m}$ is continuous. However, as $\mathcal{M}_m^r$ is a finite $ \widehat{U}_m^{\op}$-module, we can choose an epimorphism $ \left(\widehat{U}_m^{\op}\right)^s\rightarrow \mathcal{M}^r_m$, and it suffices to show continuity of the composition: 
\begin{equation*}
    \left(\widehat{U}_m^{\op}\right)^s\rightarrow\mathcal{M}^r_m\rightarrow \mathcal{M}_n.
\end{equation*}
Again, we may reduce it to showing that for any $x\in \mathcal{M}^r_m$, the map:
\begin{equation*}
     \widehat{U}_m^{\op}\rightarrow \mathcal{M}_n\textnormal{, } f\mapsto xf,
\end{equation*}
is continuous. And this follows by $\Phi_n:\widehat{U}_m^{\op}\rightarrow \operatorname{End}_{\widehat{U}_n}(\mathcal{M}_n)$ being continuous.
\end{proof}
We may use this to obtain a characterization of the elements of $\mathcal{C}^{\Bi}_B(\wideparen{U})$:
\begin{coro}\label{coro topological equivalences to co-admissible bi-modules}
Let $\mathcal{M}\in \Mod_{\widehat{\mathcal{B}}c_K}(\wideparen{E})$. The following are equivalent:
\begin{enumerate}[label=(\roman*)]
    \item $\mathcal{M}\in \mathcal{C}^{\Bi}_L(\wideparen{U})$ and $\mathcal{M}$ is co-admissible as an abstract $\wideparen{U}^{\op}$-module.
    \item  $\mathcal{M}\in \mathcal{C}^{\Bi}_R(\wideparen{U})$ and $\mathcal{M}$ is co-admissible as an abstract $\wideparen{U}$-module.
    \item $\mathcal{M}\in \mathcal{C}^{\Bi}_B(\wideparen{U})$.
\end{enumerate}
\end{coro}
Our next goal is finding the relation between these categories and the category of co-admissible $\wideparen{E}$-modules. For the rest of the section, we fix some complete  bornological module $\mathcal{M}\in \mathcal{C}^{\Bi}_L(\wideparen{U})$. We will need the following technical lemmas:
\begin{Lemma}\label{lemma topology and subalgebras}
Let $B$ be a Banach $K$-algebra, and $C\subset B$ a closed $K$-subalgebra. Let $M$ be a $B$-module which is finite as an $C$-module. Then the topologies of $M$ as a finite $C$-module and as a finite $B$-module agree. 
\end{Lemma}
\begin{proof}
Denote both topologies by $M^C$ and $M^B$ respectively. Let $m_1,\cdots,m_r$ be a finite system of generators of $M$ as an $C$-module. Let $B^{\circ}$ be the closed unit ball and $C^{\circ}=B^{\circ}\cap C$.  Let $\mathcal{M}$ be the $B^{\circ}$-module generated by $m_1,\cdots,m_r$. We have a basis of open neighborhoods of zero in $M^B$ given by the family $\left(\pi^n \mathcal{M}\right)_{n\geq 0}$. Notice that $\mathcal{M}$ is an $C^{\circ}$-module which  contains the $C^{\circ}$-module generated by $m_1,\cdots,m_r$. Hence, $\mathcal{M}$ is open in $M^C$. Thus, the identity map $M^C\rightarrow M^B$ is $K$-linear, continuous and bijective. Therefore it is a homeomorphism by Banach's open mapping theorem.
\end{proof}
\begin{Lemma}\label{Lemma extension-restriction for bi-modules}
 Let $R\rightarrow B,C\rightarrow D$ be morphisms of $K$-algebras, and let $M$ be an $R\otimes_KC$-module. There is a canonical $B\otimes_KD^{\op}$-linear isomorphism:
 \begin{equation*}
     B\otimes_KD^{\op}\otimes_{R\otimes_KC^{\op}}M\rightarrow B\otimes_RM\otimes_C D.
 \end{equation*}
\end{Lemma}
\begin{proof}
By Yoneda's Lemma, it is enough to show that we have a natural isomorphism of functors:
\begin{equation*}
    \operatorname{Hom}_{B\otimes_KD^{\op}}(B\otimes_RM\otimes_C D,-)\rightarrow \operatorname{Hom}_{B\otimes_KD^{\op}}(B\otimes_KD^{\op}\otimes_{R\otimes_KC^{\op}}M,-).
\end{equation*}
Let $N$ be a $(B,D)$-bimodule and  $\varphi:M\rightarrow N$ be a morphism of $(R,C)$-bimodules. In particular, $\varphi$ is a morphism of left $R$-modules. Hence, the extension-restriction adjunction associated to the map of $K$-algebras $R\rightarrow B$ yields a unique map:
\begin{align*}
    \varphi':B&\otimes_RM\rightarrow N\\ b&\otimes m\mapsto b\varphi(m)
\end{align*}
Notice that $B\otimes_RM$ is a right $C$-module, with action $(b\otimes m)a=b\otimes (ma)$. Furthermore, as $\varphi$ is a morphism of $(R,C)$-bimodules, we have:
\begin{equation*}
    \varphi'(b\otimes ma)=b\varphi(ma)=(b\varphi(m))a.
\end{equation*}
Applying the extension-restriction adjunction with respect to  $C^{\op}\rightarrow D^{\op}$ yields:
\begin{equation*}
    \varphi'':B\otimes_RM\otimes_C D\rightarrow N.
\end{equation*}
Hence, we can define a $K$-linear map:
\begin{equation*}
  \Psi:\operatorname{Hom}_{R\otimes C^{\op}}(M,N)\rightarrow  \operatorname{Hom}_{B\otimes D^{\op}}(B\otimes_RM\otimes_C D,N), \textnormal{   } \varphi\mapsto\varphi''.
\end{equation*}
Notice that the inclusion $i:M\rightarrow B\otimes_RM\otimes_C D$, $m\mapsto 1\otimes m\otimes 1$ is $R\otimes_KC^{\op}$-linear. Hence, for any $B\otimes D^{\op}$-linear map $\psi:B\otimes_RM\otimes_C D\rightarrow N$ we get an $R\otimes_KC^{\op}$-linear map $\psi\circ i$, by pullback. Hence, we get a $K$-linear map:
\begin{equation*}
  i^*:  \operatorname{Hom}_{B\otimes_KD^{\op}}(B\otimes_RM\otimes_C D,N)\rightarrow\operatorname{Hom}_{R\otimes_KC^{\op}}(M,N), \textnormal{   } \varphi\mapsto\varphi\circ i.
\end{equation*}
Showing that $\Psi$ and $i^*$ are mutually inverse is a straightforward calculation.
\end{proof}
For simplicity, we will write $V=\wideparen{U}\otimes_K\wideparen{U}^{\op}$ until the end of this section.
\begin{Lemma}\label{Lemma 1 embedding Theorem}
For any $n,m\geq 0$, consider the following $\widehat{E}_{n,m}$-module:
\begin{equation*}   \mathcal{M}^b_{n,m}=\widehat{E}_{n,m}\otimes_{V}M.
\end{equation*}
Then $\mathcal{M}^b_{n,m}$ is a finite $\widehat{E}_{n,m}$-module, and the map $\mathcal{M}\rightarrow \mathcal{M}^b_{n,m}$ has dense image.
\end{Lemma}
\begin{proof}
Fix non-negative integers $n,m\geq 0$. In order to show that $\mathcal{M}^b_{n,m}$ is a finite $\widehat{E}_{n,m}$-module, it is enough to show that:
    \begin{equation}\label{equation 1 co-admissible over E(L)}      \left(\widehat{U}_n\otimes_K\widehat{U}_m^{\op}\right)\otimes_{V}\mathcal{M},
    \end{equation}
    is a finite $\widehat{U}_n\otimes_K\widehat{U}_m^{\op}$-module. By Lemma \ref{Lemma extension-restriction for bi-modules}, it follows that the module in $(\ref{equation 1 co-admissible over E(L)})$ is canonically isomorphic to:
    \begin{equation}\label{equation 2 co-admissible over E(L)}
\widehat{U}_n\otimes_{\wideparen{U}}\mathcal{M}\otimes_{\wideparen{U}}\widehat{U}_m.
    \end{equation}
    By definition of co-admissible $\wideparen{U}$-module, we have a  $\widehat{U}_n$-linear isomorphism:
    \begin{equation*}    
    \widehat{U}_n\otimes_{\wideparen{U}}\mathcal{M}\rightarrow \mathcal{M}_n.
    \end{equation*}
Hence, we have a $\widehat{U}_n\otimes_K\widehat{U}_m^{\op}$-linear isomorphism:
\begin{equation}
    \widehat{U}_n\otimes_{\wideparen{U}}\mathcal{M}\otimes_{\wideparen{U}}\widehat{U}_m\rightarrow \mathcal{M}_n\otimes_{\wideparen{U}}\widehat{U}_m.
\end{equation}
$\mathcal{M}_n$ is a finitely generated $\widehat{U}_n$-module by definition of co-admissibility. Hence, it follows that $\left(\widehat{U}_n\otimes_K\widehat{U}_m^{\op}\right)\otimes_{V}\mathcal{M}$ is a finite $\widehat{U}_n\otimes_K\widehat{U}_m^{\op}$-module, as wanted.\\
In order to show that $\mathcal{M}$ is dense in $\mathcal{M}^b_{n,m}$, choose a simple tensor $a\otimes m\in \mathcal{M}^b_{n,m}$. By Corollary  \ref{coro frechet-stein presentation of bi-enveloping algebra}, we have:
\begin{equation*} \widehat{V}=\wideparen{U}\widehat{\otimes}_K\wideparen{U}^{\op}=\varprojlim_{s}\widehat{E}_{n+s,m+s}.
\end{equation*}
In particular, $\widehat{E}_{n,m}$ is the completion of $V$ with respect to some norm. Hence, we may choose a sequence $\left(a_n\right)_{n\geq 0}$  in $V$ which converges to $a$ in $\widehat{E}_{n,m}$. As  $\widehat{E}_{n,m}$ is a two-sided noetherian Banach $K$-algebra, and $\mathcal{M}^b_{n,m}$ is finite module, we may regard $\mathcal{M}^b_{n,m}$  as a topological module. Furthermore, the action of $\widehat{E}_{n,m}$ on $\mathcal{M}^b_{n,m}$ is separately continuous. Hence, we have:
\begin{equation*}
    a\otimes m =\left(\varinjlim a_n\right)\otimes m=\varinjlim \left(a_n\otimes m\right)=\varinjlim \left(1\otimes a_nm\right).
\end{equation*}
Thus, every tensor can be approximated by elements in the image of $\mathcal{M}$.
\end{proof}
\begin{Lemma}\label{Lemma 2 embedding Theorem}
 Let $n,m\geq 0$, and assume we have a continuous factorization:
 \begin{equation*}
     \widehat{U}_m^{\op}\rightarrow \operatorname{End}_{\widehat{U}_n}(\mathcal{M}_n).
 \end{equation*}
 Then there is a $\widehat{E}_{n,m}$-module structure on $\mathcal{M}_n$, and a $\widehat{E}_{n,m}$-linear isomorphism:
 \begin{equation*}
\mathcal{M}^b_{n,m}:=\widehat{E}_{n,m}\otimes_{V}\mathcal{M}\rightarrow \mathcal{M}_n.
 \end{equation*}
\end{Lemma}
\begin{proof}
 Our assumptions imply that $\mathcal{M}_n$ is a topological module over $\widehat{U}_n\otimes_K\widehat{U}_m^{\op}$. As $\mathcal{M}_n$ is a Banach $K$-vector space, it is complete. Hence, the action of $\widehat{U}_n\otimes_K\widehat{U}_m^{\op}$ on $\mathcal{M}_n$ extends to an action of $\widehat{E}_{n,m}$. Furthermore, as $\mathcal{M}_n$ is a finite $\widehat{U}_n$-module, it is also a finite $\widehat{E}_{n,m}$-module.\\
 The projection $\mathcal{M}\rightarrow \mathcal{M}_n$ is $V$-linear. Thus, it induces a map of $\widehat{E}_{n,m}$-modules:
 \begin{equation*}
\psi:\mathcal{M}^b_{n,m}:=\widehat{E}_{n,m}\otimes_{V}\mathcal{M}\rightarrow \mathcal{M}_n.
 \end{equation*}
As the image of $\mathcal{M}$ is dense in $\mathcal{M}_n$, it follows that the image of $\psi$ is dense in $\mathcal{M}_n$. Since $\widehat{E}_{n,m}$ is a noetherian Banach $K$-algebra, every submodule of a finitely generated module is closed. Hence, the image of $\psi$ is both dense and closed, so $\psi$ is surjective.  The proof of Lemma \ref{Lemma 1 embedding Theorem} shows that:
 \begin{equation*}
     \mathcal{M}^b_{n,m}=\widehat{E}_{n,m}\otimes_{\left(\widehat{U}_n\otimes_K\widehat{U}_m^{\op}\right)} \left(\widehat{U}_n\otimes_{\wideparen{U}}\mathcal{M}\otimes_{\wideparen{U}}\widehat{U}_m \right).
 \end{equation*}
Thus, there is a bounded $K$-linear map $j:\mathcal{M}_n=\widehat{U}_n\otimes_{\wideparen{U}}\mathcal{M}\rightarrow \mathcal{M}^b_{n,m}$ of the form:
\begin{equation*}
    j(a\otimes m)= 1\otimes \left( a\otimes m\otimes 1\right).
\end{equation*}
In particular, we have $\psi\circ j= \operatorname{Id}_{\mathcal{M}_n}$, so we get a decomposition:
 \begin{equation}
     \mathcal{M}^b_{n,m}=\operatorname{Ker}(\psi)\oplus \mathcal{M}_n.
 \end{equation}
 Let $\overline{\mathcal{M}}$ denote the image of $\mathcal{M}$ inside $\mathcal{M}^b_{n,m}$. Our construction of $j$ shows that $\overline{\mathcal{M}}\subset \mathcal{M}_n$. However, by Lemma \ref{Lemma 1 embedding Theorem},  $\overline{\mathcal{M}}$ is dense in $\mathcal{M}^b_{n,m}$. As $\mathcal{M}_n$ is closed in $\mathcal{M}^b_{n,m}$, it follows that $\operatorname{Ker}(\psi)=0$. 
 \end{proof}
 All these technical lemmas allow us to show the following proposition:
\begin{prop}\label{prop bi-coadmissible to E-coadmissible}
 There is a canonical isomorphism of $\wideparen{E}$-modules:
 \begin{equation*}
     \mathcal{M}=\varprojlim_{n,m}\mathcal{M}^b_{n,m}.
 \end{equation*}
 In particular, $\mathcal{M}$ is a co-admissible $\wideparen{E}$-module
\end{prop}
\begin{proof}
 By Lemma \ref{Lemma 1 embedding Theorem}, each of the 
 $\widehat{E}_{n,m}$-modules $\mathcal{M}^b_{n,m}$ is finitely generated. Furthermore, by construction, we have:
\begin{equation*}
  \mathcal{M}^b_{n,m}=\widehat{E}_{n,m}\otimes_{\widehat{E}_{n+1,m}}\mathcal{M}^b_{n+1,m}, \quad\mathcal{M}^b_{n,m}=\widehat{E}_{n,m}\otimes_{\widehat{E}_{n,m+1}}\mathcal{M}^b_{n,m+1}.
\end{equation*}
Hence, the $\wideparen{E}$-module $\varprojlim_{n,m}\mathcal{M}^b_{n,m}$
is a co-admissible $\wideparen{E}$-module. We just need to show that this inverse limit is isomorphic to $\mathcal{M}$ as an $\wideparen{E}$-module.\\
As $\mathcal{M}$ is a co-admissible $\wideparen{U}$-bimodule, the discussion at the beginning of the section shows that for each $n\geq 0$ there is a minimal $m(n)\geq 0$ such that we have a continuous factorization:
\begin{equation*}
     \widehat{U}_{m(n)}^{\op}\rightarrow \operatorname{End}_{\widehat{U}_n}(\mathcal{M}_n).
 \end{equation*}
 Hence, by Lemma \ref{Lemma 2 embedding Theorem} we have an isomorphism  $\mathcal{M}^b_{n,m(n)}\rightarrow \mathcal{M}_n$. As the family $\left( \mathcal{M}^b_{n,m(n)}\right)_{n\geq 0}$ is co-initial in $\left( \mathcal{M}^b_{n,m}\right)_{n,m\geq 0}$, it follows that we have:
 \begin{equation*}
     \varprojlim_{n,m}\mathcal{M}^b_{n,m}=\varprojlim_{n}\mathcal{M}^b_{n,m(n)}=\varprojlim_{n}\mathcal{M}_{n}=\mathcal{M}.
 \end{equation*}
 Thus showing that $\mathcal{M}$ is a co-admissible $\wideparen{E}$-module.
\end{proof}
\begin{coro}\label{coro topologies of bi-coadmissible and E-coadmissible}
The topologies of $\mathcal{M}$ as a co-admissible $\wideparen{U}$-module and as a co-admissible $\wideparen{E}$-module agree.
\end{coro}
\begin{proof}
By Lemma \ref{Lemma 2 embedding Theorem}, for each $n\geq 0$, there is some $m(n)\geq 0$ such that $\mathcal{M}^b_{n,m(n)}=\mathcal{M}_n$ as $\widehat{E}_{n,m(n)}$-modules. By Lemma \ref{lemma topology and subalgebras}, the topologies on $\mathcal{M}_n$ induced by $\widehat{U}_n$ and $\widehat{E}_{n,m(n)}$ agree. By Proposition \ref{prop bi-coadmissible to E-coadmissible}, we have: $\mathcal{M}=\varprojlim_{n}\mathcal{M}^b_{n,m(n)}$. As the topologies of $\mathcal{M}$ as a co-admissible $\wideparen{E}$-module and $\wideparen{U}$-module are determined by this inverse limit, it follows that they agree. 
\end{proof}
We may condense the contents of this section into the following theorem:
\begin{teo}\label{teo embedding co-admissible bimodules}
There is an embedding of abelian categories:
\begin{equation*}
\mathcal{C}^{\operatorname{Bi}}_L(\wideparen{U}) \rightarrow \mathcal{C}(\wideparen{E}),
\end{equation*}
which satisfies the following properties:
\begin{enumerate}[label=(\roman*)]
    \item For $\mathcal{M}\in \mathcal{C}_L^{\operatorname{Bi}}(\wideparen{U})$, the canonical topologies of $\mathcal{M}$ as a co-admissible $\wideparen{U}$-module and as a co-admissible $\wideparen{E}$-module agree.
    \item $\mathcal{C}_L^{\operatorname{Bi}}(\wideparen{U})$ is a Serre subcategory of $\mathcal{C}(\wideparen{E})$.
\end{enumerate}
The corresponding results for $\mathcal{C}^{\operatorname{Bi}}_R(\wideparen{U}(L))$ hold mutatis-mutandis.
\end{teo}
\begin{proof}
As item $(i)$ is Corollary \ref{coro topologies of bi-coadmissible and E-coadmissible}, we only need to show item $(ii)$. Consider the following short exact sequence of co-admissible $\wideparen{E}$-modules:
    \begin{equation*}
        0\rightarrow \mathcal{M}_1 \rightarrow \mathcal{M}_2\rightarrow \mathcal{M}_3\rightarrow 0.
    \end{equation*}
    As morphisms of co-admissible modules are strict, it follows that $\mathcal{M}_1$ is a closed submodule of $\mathcal{M}_2$, and that it has the subspace topology. Similarly, $\mathcal{M}_3$ carries the quotient topology.\\
Assume $\mathcal{M}_2$ is in  $\mathcal{C}_L^{\operatorname{Bi}}(\wideparen{U})$. By item $(i)$ and the discussion above, $\mathcal{M}_1$ is a closed $\wideparen{U}$-submodule of $\mathcal{M}_2$ with respect to its topology as a co-admissible $\wideparen{U}$-module. Hence,  
by  \cite[Lemma 3.6]{schneider2002algebras}, it follows that $\mathcal{M}_1$ and $\mathcal{M}_3$ are co-admissible $\wideparen{U}$-modules.\\ 
Conversely, assume that $\mathcal{M}_1$ and $\mathcal{M}_3$ are objects of $\mathcal{C}^{\operatorname{Bi}}_L(\wideparen{U})$. By \cite[Remark 3.2]{schneider2002algebras}, the map $\wideparen{U}\rightarrow \widehat{U}_n$ is two- sided flat for all $n\geq 0$. Hence, tensoring by $\widehat{U}_n$ we obtain a family of short exact sequences of $\widehat{U}_n$-modules:
\begin{equation*} 
        0\rightarrow \mathcal{M}_{1,n} \rightarrow \mathcal{M}_{2,n}\rightarrow \mathcal{M}_{3,n}\rightarrow 0,
\end{equation*}
where $\mathcal{M}_{1,n}$ and $\mathcal{M}_{3,n}$ are finite $\widehat{U}_n$-modules by assumption. As $\widehat{U}_n$ is two-sided noetherian, it follows that $\mathcal{M}_{2,n}$ is also finite. Taking inverse limits, and using the fact that inverse limits associated to coherent sheaves of  $\wideparen{U}$-modules satisfy the Mittag-Leffler condition, it follows that we have a short exact sequence:
\begin{equation*}
    0\rightarrow \mathcal{M}_1 \rightarrow \varprojlim_n \mathcal{M}_{2,n}\rightarrow \mathcal{M}_3\rightarrow 0.
\end{equation*}
By the Snake Lemma, it follows that we have a $\wideparen{U}$-linear isomorphism:
\begin{equation*}
    \psi:\mathcal{M}_2\rightarrow \varprojlim_n \mathcal{M}_{2,n}.
\end{equation*}
This is a continuous isomorphism, as continuity of the map $\wideparen{U}\rightarrow \widehat{U}_n$ implies that the maps $\mathcal{M}_2\rightarrow \mathcal{M}_{2,n}$ are also continuous. Thus, $\mathcal{M}_2$ is a co-admissible $\wideparen{U}$-module, and hence an object of $\mathcal{C}^{\operatorname{Bi}}_L(\wideparen{U})$. 
\end{proof}
\subsection{Co-admissibility and the immersion functor I}
Now we will make use of the results in the preceding section to construct sheaves of co-admissible $\wideparen{E}(\mathscr{L})$-modules with support contained in the diagonal $\Delta\subset X^2$. As we will make extensive use of Kashiwara's equivalence (cf. \cite[Theorem 7.1]{ardakov2015d}, \cite[Proposition 9.5]{bode2021operations}), we will focus on the case $\mathscr{L}=\mathcal{T}_{X/K}$. In this case we have $\mathcal{T}_{X/K}^2=\mathcal{T}_{X^2/K}$, so that we have:
\begin{equation*}
    \wideparen{U}(\mathcal{T}_{X/K}^2)=\wideparen{\D}_{X^2}, \quad \wideparen{E}_X:=\wideparen{E}(\mathcal{T}_{X/K}).
\end{equation*}

As usual, we deal with the local case first. Let $X=\Sp(A)$ be a smooth affinoid variety equipped with an étale map $X\rightarrow \mathbb{A}^n$. In particular, there are some rigid functions $f_1,\cdots, f_n\in A$ such that $df_1,\cdots,df_n$ are an $A$-basis of $\Omega_{X}(X)$. Hence, there is a basis $v_1,\cdots,v_n$ of $\mathcal{T}_{X/K}$ satisfying $df_i(v_j)=\delta_{i,j}$, for $1\leq i,j\leq n$.
Notice that in this case, $\mathcal{T}_{X/K}$ satisfies the conditions of Section \ref{Section Categories of co-admissible bimodules}.\bigskip 

Let $\mathcal{I}_{\Delta}$ be the coherent sheaf of ideals of $\OX_{X^2}$ associated to the diagonal embedding $\Delta:X\rightarrow X^2$. By abuse of notation, we will often times identify $\mathcal{I}_{\Delta}$ with its global sections, and simply write $\mathcal{I}_{\Delta}=\mathcal{I}_{\Delta}(X^2)$ when no confusion is possible. We now introduce some machinery from \cite{ardakov2015d}:
\begin{defi}[{\cite[Definition 5.2]{ardakov2015d}}]
Let $I$ be an ideal in a commutative $K$-algebra $B$ and let $L$ be a $(K,B)$-Lie algebra. We call a subset $\{x_1,\cdots,x_n\}\subset L$ an $I$-standard basis if the following hold:
\begin{enumerate}[label=(\roman*)]
    \item $\{x_1,\cdots,x_n\}$ is a $B$-module basis of $L$
    \item There is a generating set $\{g_1,\cdots,g_r\}$ of $I$ with $r\leq n$.
    \item $x_i\cdot g_j=\delta_{i,j}$ for $1\leq i \leq n$, and $1\leq j\leq r$.
\end{enumerate}
\end{defi}
We can now adapt this concept to our case of interest:
\begin{Lemma}\label{Lemma support on the diagonal}
Consider the following ideals for $1\leq r\leq n$:
\begin{equation*}
    \mathcal{J}_r=\left(1\otimes f_1-f_1\otimes 1,\cdots, 1\otimes f_r-f_r\otimes 1 \right),
\end{equation*}
and let  $\mathcal{J}=\mathcal{J}_n$. The following hold:
 \begin{enumerate}[label=(\roman*)]
     \item The following subset is a $\mathcal{J}$-standard basis:
     \begin{equation*}
         \mathcal{V}:=\{1\otimes v_1,\cdots, 1\otimes v_n,1\otimes v_1 + v_1\otimes 1, \cdots, 1\otimes v_n + v_n\otimes 1\}\subset \mathcal{T}_{X^2/K}(X).
     \end{equation*}
     \item Let $B_r=A\widehat{\otimes}_KA/\mathcal{J}_r$. Then $Y_r=\Sp(B_r)$ is a smooth affinoid space, and we write $Y=Y_n$.
     \item We have a closed immersion $i:X\rightarrow Y$ making $X$ a union of connected components of $Y$.
 \end{enumerate}
\end{Lemma}
\begin{proof}
By Lemma \ref{Lemma derivations on fiber products and completions}, we know that the choice of an $A$-basis $v_1,\cdots,v_n$ of $\mathcal{T}_{X/K}(X)$ induces a basis $1\otimes v_1,\cdots,1\otimes v_n,v_1\otimes 1,\cdots, v_n\otimes 1$ of $\mathcal{T}_{X^2/K}(X^2)$.  A simple calculation then shows that $\mathcal{V}$ is also an $A\widehat{\otimes}_KA$-basis of $\mathcal{T}_{X^2/K}(X^2)$. Furthermore, for $1\leq i\leq n$, and $1\leq j\leq n$, we have:
\begin{equation*}
    1\otimes v_i(1\otimes f_j- f_j\otimes 1)=\delta_{i,j}, \textnormal{ and } (1\otimes v_i+v_i\otimes 1)(1\otimes f_j- f_j\otimes 1)=0.
\end{equation*}
Thus, $\mathcal{V}$ is a $\mathcal{J}$-adapted basis of $\mathcal{T}_{X^2/K}(X^2)$, and $(i)$ holds.\\
For statement $(ii)$, fix $1\leq r\leq n$, and consider the standard  right-exact sequence of coherent $B$-modules associated  to the closed immersion $Y_r\rightarrow X^2$:
 \begin{equation}\label{equation showing that A^2/J is smooth}
     \mathcal{J}_r/\mathcal{J}_r^2\rightarrow B_r\otimes_{A\widehat{\otimes}_KA}\Omega_{X^2/K}(X^2)\rightarrow \Omega_{Y_r/K}(Y_r)\rightarrow 0.
 \end{equation}
 In order to show that $Y_r$ is smooth, it suffices to show that $(\ref{equation showing that A^2/J is smooth})$ is split and exact. By definition, the image in $B_r\otimes_{A\widehat{\otimes}_KA}\Omega_{X^2/K}(X^2)$  of the equivalence class in $\mathcal{J}_r/\mathcal{J}_r^2$ of each of the generators $1\otimes f_j-f_j\otimes 1$ is given by the expression:
\begin{equation*}
    1\otimes \left(1\otimes d(f_j)-d(f_j)\otimes 1 \right).
\end{equation*}
Applying \ref{Lemma derivations on fiber products and completions} again, it follows that $\Omega_{X^2/K}(X^2)$ has a $A\widehat{\otimes}_KA$-basis given by  elements of the form $1\otimes d(f_j), d(f_j)\otimes 1$ for $1\leq j\leq n$. Thus, applying the same argument as above, it follows that the elements of the form:
\begin{equation*}
    1\otimes d(f_j)-d(f_j)\otimes 1, 1\otimes d(f_j) \textnormal{, where } 1\leq j\leq n,
\end{equation*}
also form a basis of $\Omega_{X^2/K}(X^2)$. Thus, we obtain a split of the sequence $(\ref{equation showing that A^2/J is smooth})$  by defining a map:
\begin{equation*}
    B_r\otimes_{A\widehat{\otimes}_KA}\Omega_{X^2/K}(X^2)\rightarrow \mathcal{J}_r/\mathcal{J}_r^2,  \quad 1\otimes \left(1\otimes d(f_j)-d(f_j)\otimes 1 \right)\mapsto 1\otimes f_j-f_j\otimes 1,
\end{equation*}
for $1\leq j\leq r$, and sending the rest of the basis to $0$. Therefore $Y_r$ is smooth for each $1\leq r\leq n$.\\ 
 In order to show $(iii)$, let $\mathcal{I}_{\Delta}$ be the sheaf of ideals associated to the diagonal $\Delta:X\rightarrow X^2$. We notice that $\mathcal{J}\subset \mathcal{I}_{\Delta}$, so we have a closed immersion $X\rightarrow Y$. As $Y$ is smooth, its dimension agrees with  the rank of the vector bundle $\Omega_{Y/K}$. By our discussion above, $\mathcal{J}/\mathcal{J}^2$ is a vector bundle of rank $n$, and by smoothness of $X^2$, $\Omega_{X^2/K}$ is a vector bundle of rank $2n$. Thus, the fact that $(\ref{equation showing that A^2/J is smooth})$ is exact implies that $\operatorname{dim}(Y)=n$. As $\operatorname{dim}(X)=n$, and $X$ is smooth (in particular reduced), it follows that $X$ is a union of irreducible components of $Y$. However, as $Y$ is smooth, its irreducible components agree with its connected components, and we are done.
\end{proof}
Let us now introduce the following objects:
\begin{defi}[{\cite[Definitions 4.2 and 5.5]{ardakov2015d}}]
Let  $C$ be an affinoid algebra, $\mathcal{M}$ a complete bornological $C$-module and  $\mathcal{I}\subset C$ an ideal. Consider the following objects:
 \begin{enumerate}[label=(\roman*)]
     \item $\mathcal{M}[\mathcal{I}]=\{ m\in \mathcal{M} \textnormal{ } \vert \textnormal{ } cm=0 \textnormal{ for all } c\in \mathcal{I}\}$.
     \item $\mathcal{M}_{\infty}(\mathcal{I})=\{ m\in \mathcal{M} \textnormal{ } \vert \textnormal{ } \varinjlim_nc^nm=0 \textnormal{ for all } c\in \mathcal{I}\}$.
 \end{enumerate}   
\end{defi}
The relevance of the previous spaces is highlighted by the following theorem:
\begin{teo}[{\cite[Theorem 6.7]{ardakov2015d}}]\label{teo support in a closed subvariety}
Let $X=\Sp(A)$ be a smooth affinoid space with a Lie algebroid $\mathscr{L}$. Let $\mathcal{I}\subset \OX$ be the ideal associated to a closed immersion $Y\rightarrow X$. Assume $\mathscr{L}(X)$ has an $\mathcal{I}(X)$-standard basis, and let $\mathcal{M}$ be a co-admissible $\wideparen{U}(\mathscr{L})$-module. The following are equivalent:
\begin{enumerate}[label=(\roman*)]
    \item $\mathcal{M}$ is supported on $Y$.
    \item  $\mathcal{M}(X)=\mathcal{M}(X)_{\infty}(\mathcal{I}(X))$.
\end{enumerate}
\end{teo}
As shown in the previous section, any bornological module $\mathcal{M}\in \mathcal{C}_L^{\operatorname{Bi}}(\wideparen{\D}_X(X))$ is canonically a co-admissible $\wideparen{E}_X(X^2)$-module. Hence, by Definition \ref{defi localization functor for co-admissible E(l)-modules}, we may consider the co-admissible $\wideparen{E}_X$-module $\operatorname{Loc}(\mathcal{M})$ on $X^2$. We are interested in finding conditions under which $\operatorname{Loc}(\mathcal{M})$ has support contained in the diagonal. As shown in Proposition \ref{prop equivalence of pullback and side-changing sheaves}, the side-changing operators $\wideparen{\mathbb{T}}^*$, and $\wideparen{\mathbb{T}}^{-1,*}$ do not change the support, and send co-admissible modules to co-admissible modules. Hence, $\operatorname{Loc}(\mathcal{M})$ is supported on the diagonal if and only if $\wideparen{\mathbb{T}}^{-1,*}(\operatorname{Loc}(\mathcal{M}))$ is a co-admissible $\wideparen{\D}_{X^2}$-module supported on the diagonal. Furthermore, the isomorphism:
\begin{equation*}
    \wideparen{\mathbb{T}}^{-1}: \wideparen{E}_X\rightarrow \wideparen{\D}_{X^2}
\end{equation*}
is the identity on $\OX_{X^2}$. In particular, we have the following identifications:
\begin{equation*}
  \operatorname{Loc}(\mathcal{M})=\wideparen{\mathbb{T}}^{-1,*}\operatorname{Loc}(\mathcal{M})=\operatorname{Loc}(\wideparen{\mathbb{T}}^{-1,*}\mathcal{M}),  
\end{equation*}
as sheaves of  Ind-Banach $\mathcal{O}_{X^2}$-modules.
\begin{prop}\label{prop co-admissible EX modules supported on the diagonal}
Consider a module $\mathcal{M}\in \mathcal{C}(\wideparen{E}_X(X^2))$. Then the co-admissible $\wideparen{E}_X$-module $\operatorname{Loc}(\mathcal{M})$ is supported on the diagonal $\Delta\subset X^2$ if and only if:
    \begin{equation*}
        \mathcal{M}=\mathcal{M}_{\infty}(\mathcal{I}_{\Delta}).
    \end{equation*}
\end{prop}
\begin{proof}
For simplicity of the notation, we will write $\mathcal{N}=\wideparen{\mathbb{T}}^{-1,*}\mathcal{M}$. 
Assume that $\mathcal{M}=\mathcal{M}_{\infty}(\mathcal{I}_{\Delta})$. As discussed above, we have $\mathcal{M}=\mathcal{N}$ as complete bornological $A\widehat{\otimes}_KA$-modules. Hence, this implies 
$\mathcal{N}=\mathcal{N}_{\infty}(\mathcal{I}_{\Delta})$. As $\mathcal{J}\subset \mathcal{I}_{\Delta}$, it follows that $\mathcal{N}=\mathcal{N}_{\infty}(\mathcal{J})$. Furthermore, as shown in Lemma \ref{Lemma support on the diagonal}, $\mathcal{T}_{X^2/K}(X^2)$ has a $\mathcal{J}$-adapted basis.  Thus, Theorem \ref{teo support in a closed subvariety} implies that $\operatorname{Loc}(\mathcal{N})$ has support contained in the closed sub variety $j:Y\rightarrow X^2$. By Kashiwara's equivalence \cite[Theorem 7.1]{ardakov2015d},
we have:
\begin{equation*}
    \operatorname{Loc}(\mathcal{N})\simeq j_+j^!\operatorname{Loc}(\mathcal{N}).
\end{equation*}
On the other hand, $j^!\operatorname{Loc}(\mathcal{N})$ is the co-admissible $\wideparen{\D}_Y$-module with global sections:
\begin{equation*}
    \Gamma(Y,j^!\operatorname{Loc}(\mathcal{N}))=\mathcal{N}[\mathcal{J}].
\end{equation*}
By statement $(iii)$ in  Lemma \ref{Lemma support on the diagonal}, $X$ is a union of connected components of $Y$.  Hence, there is an element $f\in \mathcal{I}_{\Delta}$ such that its image $\overline{f}\in A\widehat{\otimes}_KA/\mathcal{J}$ satisfies:
\begin{equation*}
    Y(1/\overline{f})=Y\setminus X, \quad X=\mathbb{V}(\overline{f}).
\end{equation*}
Let $(f)$, and $(\overline{f})$ be the ideals in $A\widehat{\otimes}_KA$ and $A\widehat{\otimes}_KA/\mathcal{J}$ generated by $f$ and  $\overline{f}$ respectively. Then the following holds:
\begin{equation*}
   \mathcal{N}[\mathcal{J}]\subset \mathcal{N}=\mathcal{N}_{\infty}(\mathcal{I}_{\Delta}). 
\end{equation*}
Thus, as $f\in \mathcal{I}_{\Delta}$, it follows that we have the following identifications:
\begin{equation*}
    \mathcal{N}[\mathcal{J}]=\mathcal{N}[\mathcal{J}]_{\infty}((f))=\mathcal{N}[\mathcal{J}]_{\infty}((\overline{f})).
\end{equation*}
By \cite[Corollary 6.6]{ardakov2015d}, it follows that $\operatorname{Loc}(\mathcal{N}[\mathcal{J}])(Y(\frac{1}{\overline{f}}))=0$. Thus, the support of $\operatorname{Loc}(\mathcal{N}[\mathcal{J}])$ is contained in $X$. Let $i:X\rightarrow Y$ denote the closed immersion. We may then apply Kashiwara's equivalence again to obtain $\operatorname{Loc}(\mathcal{N}[\mathcal{J}])\simeq i_+i^!\operatorname{Loc}(\mathcal{N}[\mathcal{J}])$. Hence, we have:
\begin{equation*}
   \operatorname{Loc}(\mathcal{N})\simeq j_+i_+i^!\operatorname{Loc}(\mathcal{N}[\mathcal{J}])\simeq\Delta_+i^!\operatorname{Loc}(\mathcal{N}[\mathcal{J}]), 
\end{equation*}
where we are using the fact that $\Delta=j\circ i$, together with \cite[Lemma 7.9]{bode2021operations}. Thus, $\operatorname{Loc}(\mathcal{M})$ has support contained in the diagonal, as we wanted to show.\\
Conversely, assume that $\operatorname{Loc}(\mathcal{M})$ has support contained in the diagonal. As the side-switching operators do not change the support, it follows that $\operatorname{Loc}(\mathcal{N})$ also has support contained in the diagonal. By \cite[Corollary 5.6]{ardakov2015d}, the submodule $\mathcal{N}_{\infty}(\mathcal{I}_{\Delta})\subset \mathcal{N}$ is also co-admissible, and has support contained in the diagonal. Hence, by Kashiwara's equivalence, we have:
\begin{multline*}
    \operatorname{Loc}(\mathcal{N}_{\infty}(\mathcal{I}_{\Delta}))\simeq \Delta_+\Delta^!\operatorname{Loc}(\mathcal{N}_{\infty}(\mathcal{I}_{\Delta}))\simeq\Delta_+\operatorname{Loc}(\mathcal{N}_{\infty}(\mathcal{I}_{\Delta})[\mathcal{I}_{\Delta}])\\
    =\Delta_+\operatorname{Loc}(\mathcal{N}[\mathcal{I}_{\Delta}])\simeq \Delta_+\Delta^!\operatorname{Loc}(\mathcal{N})\simeq\operatorname{Loc}(\mathcal{N}),
\end{multline*}
thus showing that $\mathcal{N}=\mathcal{N}_{\infty}(\mathcal{I}_{\Delta})$, as wanted. The converse follows at once from Kashiwara's equivalence.
\end{proof}
Let $\mathcal{M}$ be a co-admissible $\wideparen{E}_X$-module supported on the diagonal. By the side-switching for bimodules in Proposition \ref{prop equivalence of categories of co-admissible odules over E and U}, together with Kashiwara's equivalence, we have that:
\begin{equation*}
    \mathcal{M}\simeq\operatorname{S}\Delta_+\Delta^!\operatorname{S}^{-1}\mathcal{M}=\operatorname{S}\left(\Delta_*\left( \wideparen{\D}_{X^2\leftarrow X}\overrightarrow{\otimes}_{\wideparen{\D}_X}\Delta^{-1}\left(\operatorname{S}^{-1}\mathcal{M}\right)[\mathcal{I}_{\Delta}] \right)\right).
\end{equation*}
This will be key in our calculations with Hochschild (co)-homology bellow. 
\begin{Lemma}\label{lemma sections pullback}
Let $\mathcal{M}$ be a co-admissible $\wideparen{E}_X$-module supported on the diagonal, and $V\subset X$ be an affinoid subdomain. Then we have $\Delta^{-1}\mathcal{M}(V)=\mathcal{M}(V^2)$.
\end{Lemma}
\begin{proof}
As $\mathcal{M}$ has support contained in the diagonal, there is some sheaf of Ind-Banach spaces $\mathcal{N}$ on $X$ such that $\mathcal{M}=\Delta_*\mathcal{N}$. Thus, we have:
\begin{equation*}
    \mathcal{M}(V^2)=\Delta_*\mathcal{N}(V^2)=\mathcal{N}(V)=\Delta^{-1}\Delta_*\mathcal{N}(V)=\Delta^{-1}\mathcal{M}(V),
\end{equation*}
as we wanted to show.
\end{proof}
Recall the category of co-admissible bimodules $\mathcal{C}_B^{\Bi}(\wideparen{\D}_X)$ from Definition \ref{defi sehaves of co-admissible bimodules}.
\begin{defi}\label{defi diagonal co-admissible bimodules local case}
We define $\mathcal{C}_B^{\Bi}(\wideparen{\D}_X)_{\Delta}$ as the full subcategory of $\mathcal{C}_B^{\Bi}(\wideparen{\D}_X)$ given by objects $\mathcal{M}$ satisfying that:
\begin{equation*}
    \mathcal{M}(X)=\mathcal{M}(X)_{\infty}(\mathcal{I}_{\Delta}).
\end{equation*}
\end{defi}
Expanding on our previous results, we wish to investigate the relation between $\mathcal{C}_B^{\Bi}(\wideparen{\D}_X)$ and $\mathcal{C}(\wideparen{E}_X)$. In particular, we want to determine the image of $\mathcal{C}_B^{\Bi}(\wideparen{\D}_X)$ under the extension functor:
\begin{equation*}
    \Delta_*^E:\Mod_{\Indban}(\wideparen{\D}_X^e)\rightarrow \Mod_{\Indban}(\wideparen{E}_X).
\end{equation*}
Let us start with the following lemma:
\begin{Lemma}\label{Lemma diagonal co-admissible bimodules are abelian}
    $\mathcal{C}_B^{\Bi}(\wideparen{\D}_X)_{\Delta}$ is an abelian category.
\end{Lemma}
\begin{proof}
As  $\mathcal{C}_B^{\Bi}(\wideparen{\D}_X)_{\Delta}$ is a full subcategory of $\Mod_{\Indban}(\wideparen{\D}_X^e)$, and this category is quasi-abelian, it suffices to show that it contains the kernels and cokernels of its maps, and that every morphism is strict. As every morphism between co-admissible modules is strict, it is clear that every morphism in  $\mathcal{C}_B^{\Bi}(\wideparen{\D}_X)_{\Delta}$ is strict. Let $f:\mathcal{M}\rightarrow \mathcal{N}$ be a morphism in $\mathcal{C}_B^{\Bi}(\wideparen{\D}_X)_{\Delta}$, and consider the strict exact complex:
\begin{equation}\label{equation bimodules supported on diagonal is abelian}
    0\rightarrow \operatorname{Ker}(f)\rightarrow \mathcal{M}\rightarrow \mathcal{N}\rightarrow \operatorname{Coker}(f)\rightarrow 0.
\end{equation}
As $\mathcal{C}(\wideparen{\D}_X)$ and $\mathcal{C}(\wideparen{\D}^{\op}_X)$ are abelian categories, it is clear that:
\begin{equation*}
    \operatorname{Ker}(f),\operatorname{Coker}(f)\in \mathcal{C}_B^{\Bi}(\wideparen{\D}_X).
\end{equation*}
Furthermore, the fact that $X$ is affinoid implies that $(\ref{equation bimodules supported on diagonal is abelian})$ is exact after applying $\Gamma(X,-)$. In particular, we have a strict injection $\operatorname{Ker}(f)(X)\rightarrow \mathcal{M}(X)$, and a strict surjection $\mathcal{N}(X)\rightarrow \operatorname{Coker}(f)(X)$. By assumption, we have $\mathcal{M}(X)=\mathcal{M}(X)_{\infty}(\mathcal{I}_{\Delta})$, and $\mathcal{N}(X)=\mathcal{N}(X)_{\infty}(\mathcal{I}_{\Delta})$. Thus, we have $\operatorname{Ker}(f)(X)=\operatorname{Coker}(f)(X)_{\infty}(\mathcal{I}_{\Delta})$ and $\operatorname{Ker}(f)(X)=\operatorname{Coker}(f)(X)_{\infty}(\mathcal{I}_{\Delta})$, as we wanted to show.
\end{proof}
\begin{prop}\label{prop embedding Theorem}
We have an exact embedding of abelian categories:
\begin{equation*}
    \Delta_*^E:\mathcal{C}_B^{\Bi}(\wideparen{\D}_X)_{\Delta}\rightarrow \mathcal{C}(\wideparen{E}_X)_{\Delta}.
\end{equation*}
In particular, for $\mathcal{M}\in \mathcal{C}_B^{\Bi}(\wideparen{\D}_X)_{\Delta}$ we have $\Delta^E_*\mathcal{M}=\operatorname{Loc}(\mathcal{M}(X))$.
\end{prop}
\begin{proof}
Choose any $\mathcal{M}\in \mathcal{C}_B^{\Bi}(\wideparen{\D}_X)_{\Delta}$. By definition, we have $\mathcal{M}(X)\in \mathcal{C}_B^{\Bi}(\wideparen{\D}_X(X))$. Hence, it follows by Theorem \ref{teo embedding co-admissible bimodules}
that $\mathcal{M}(X)\in \mathcal{C}(\wideparen{E}_X(X^2))$. Thus, we may consider the co-admissible $\wideparen{E}_X$-module $\mathcal{N}=\operatorname{Loc}(\mathcal{M}(X))$. As $\mathcal{M}\in \mathcal{C}_B^{\Bi}(\wideparen{\D}_X)_{\Delta}$, it follows by  Proposition \ref{prop co-admissible EX modules supported on the diagonal}  that $\mathcal{N}$ is a co-admissible $\wideparen{E}_X$-module supported on $\Delta\subset X^2$. Consider the following presheaf in $X^2$:
\begin{equation*}
    \Tilde{\mathcal{M}}:=\wideparen{E}_X\overrightarrow{\otimes}^{\operatorname{psh}}_{p_1^{-1}\wideparen{\D}_X\overrightarrow{\otimes}_K^{\operatorname{psh}}p_2^{-1}\wideparen{\D}_X^{\op}}\Delta_*\mathcal{M},
\end{equation*}
where the operator $\overrightarrow{\otimes}^{\operatorname{\operatorname{psh}}}$ is the tensor product in the category of presheaves of Ind-Banach spaces, and all the pullbacks are also taken in the category of presheaves.\\
With this definition, it follows that we have:
\begin{equation*}
    \Gamma(X^2,\Tilde{\mathcal{M}})=\mathcal{M}(X).
\end{equation*}
Thus, we may follow the arguments in the proof of Proposition \ref{prop properties of Loc} to show that $\operatorname{Id}_{\mathcal{M}(X)}$ induces a unique $\wideparen{E}_X$-linear map of presheaves of Ind-Banach modules:
 \begin{equation*}
    \mathcal{N}\rightarrow \Tilde{\mathcal{M}}.
 \end{equation*}
Consider the following morphism of presheaves:
\begin{equation*}
    \varphi:\mathcal{N}\rightarrow \Tilde{\mathcal{M}}\rightarrow \Delta_*\Delta^{-1}\Tilde{\mathcal{M}}=\Delta_*\left(\Delta^{-1}\wideparen{E}_X\overrightarrow{\otimes}^{\operatorname{psh}}_{\wideparen{\D}_X^e}\mathcal{M}\right).
\end{equation*}
Notice that $\Tilde{\mathcal{M}}$ has support contained in $\Delta$. Hence, the map $\Tilde{\mathcal{M}}\rightarrow \Delta_*\Delta^{-1}\Tilde{\mathcal{M}}$ induces an isomorphism on sheafifications. As the sheafification of $\Tilde{\mathcal{M}}$ is $\Delta_*^E\mathcal{M}$, it is enough to show that $\varphi$ induces an isomorphism:
\begin{equation*}
    \Delta^{-1}\mathcal{N}\rightarrow \Delta^{-1}\wideparen{E}_X\overrightarrow{\otimes}^{\operatorname{psh}}_{\wideparen{\D}_X^e}\mathcal{M}=\Delta^{-1}\OX_{X^2}\overrightarrow{\otimes}^{\operatorname{psh}}_{\OX_X\overrightarrow{\otimes}^{\operatorname{psh}}_K\OX_X}\mathcal{M}=:\mathcal{H}.
\end{equation*}
Choose an affinoid subdomain $V\subset X$. By Lemma \ref{lemma sections pullback} we have the following:
 \begin{align}\label{equation immersion Theorem}
   \Delta^{-1}\mathcal{N}(V)=\mathcal{N}(V^2)=&\wideparen{E}_X(V^2)\overrightarrow{\otimes}_{\wideparen{E}_X(X^2)}\mathcal{M}(X)\\ =&\wideparen{\D}_X(V)\overrightarrow{\otimes}_{\wideparen{\D}_X(X)}\mathcal{M}\overrightarrow{\otimes}_{\wideparen{\D}_X(X)}\wideparen{\D}_X(V)\nonumber,
\end{align}
where the last identity follows by the definition of $\wideparen{E}_X$, and Proposition \ref{prop enveloping of product as product of envelopings}. By assumption, $\mathcal{M}$ is a co-admissible module over $\wideparen{\D}_X$ and $\wideparen{\D}_X^{\op}$. Thus, by definition of co-admissibility we have identities:
\begin{equation*}
\wideparen{\D}_X(V)\overrightarrow{\otimes}_{\wideparen{\D}_X(X)}\mathcal{M}(X)=\mathcal{M}(V)=\mathcal{M}(X)\overrightarrow{\otimes}_{\wideparen{\D}_X(X)}\wideparen{\D}_X(V).
\end{equation*}
Hence, the identity (\ref{equation immersion Theorem}) may be reformulated into the following chain of identities:
\begin{align*}
\Delta^{-1}\mathcal{N}(V)=&\mathcal{M}(V)\overrightarrow{\otimes}_{\wideparen{\D}_X(X)}\wideparen{\D}_X(V)\\
=&\left(\mathcal{M}(X)  \overrightarrow{\otimes}_{\wideparen{\D}_X(X)}\wideparen{\D}_X(V)\right)\overrightarrow{\otimes}_{\wideparen{\D}_X(X)}\wideparen{\D}_X(V)\\=&
\mathcal{M}(X)  \overrightarrow{\otimes}_{\wideparen{\D}_X(X)}\left(\wideparen{\D}_X(V)\overrightarrow{\otimes}_{\wideparen{\D}_X(X)}\wideparen{\D}_X(V)\right)\\=&\mathcal{M}(X)\overrightarrow{\otimes}_{\wideparen{\D}_X(X)}\wideparen{\D}_X(V)\\=&\mathcal{M}(V).  
\end{align*}
On the other hand, by definition of the pullback in the category of presheaves:
\begin{equation*}
    \mathcal{H}(V)=\varinjlim_{U}\left(\OX_{X^2}(U)\overrightarrow{\otimes}_{\OX_{X^2}(V^2)}\mathcal{M}(V)\right).
\end{equation*}
where the $U$ range over all the affinoid subdomains $U\subset V^2$ such that $V=\Delta(X)\cap U$. Hence, it is enough to show that we have an isomorphism:
\begin{equation*}
    \OX_{X^2}(U)\overrightarrow{\otimes}_{\OX_{X^2}(V^2)}\mathcal{M}(V)=\mathcal{M}(V),
\end{equation*}
for each $U$ as above. However, by definition we have that $\mathcal{N}$ is a co-admissible $\wideparen{E}_X$-module with support contained in the diagonal. Hence, as $V^2\cap \Delta(X)=U\cap \Delta(X)$ by our choice of $U$, it follows that we have:
\begin{align*}
    \mathcal{M}(V)=&\mathcal{N}(V^2)=\mathcal{N}(U)=\OX_{X^2}(U)\overrightarrow{\otimes}_{\OX_{X^2}(V^2)}\mathcal{M}(V),
\end{align*}
where the last identity follows by the definition of the $\operatorname{Loc}(-)$ functor (\emph{cf.} Proposition \ref{prop properties of Loc}), and the fact that in the current conditions there is an isomorphism $\wideparen{E}_X\cong \OX_{X^2}\overrightarrow{\otimes}_KK\{x_1,\cdots,x_n\}$. Hence, we have obtained an isomorphism:
\begin{equation*}
  \operatorname{Loc}(\mathcal{M}(X))\rightarrow \Delta^E_*\mathcal{M},
\end{equation*}
as we wanted to show. This shows that the essential image of $\mathcal{C}_B^{\Bi}(\wideparen{\D}_X)_{\Delta}$ under the extension functor $\Delta_*^E$ is indeed contained in $\mathcal{C}(\wideparen{E}_X)_{\Delta}$. Furthermore, as $X$ is affinoid with free tangent sheaf, exactness can be checked at the global sections, and thus it follows that $\Delta_*^E$ is exact. Similarly, as morphisms in both categories are determined at the level of global sections, it follows that the functor is also fully faithful.
\end{proof}
We can draw a plethora of conclusions from this proposition:
\begin{coro}
The immersion functor induces an exact embedding:
    \begin{equation*}
        \Delta_*^{\operatorname{S}}:\mathcal{C}_B^{\Bi}(\wideparen{\D}_X)_{\Delta}\rightarrow \mathcal{C}(\wideparen{\D}_{X^2})_{\Delta}.
    \end{equation*}
\end{coro}
\begin{coro}\label{coro khiels Theorem for diagonal co-admissible bi-modules}
Let $\mathcal{M}\in\Mod_{\Indban}(\wideparen{\D}_X^e)$. The following are equivalent:
\begin{enumerate}[label=(\roman*)]
    \item $\mathcal{M}\in \mathcal{C}_B^{\Bi}(\wideparen{\D}_X)_{\Delta}$.
    \item There is an affinoid cover $(V_i)_{i=1}^n$ of $X$ such that $\mathcal{M}_{\vert V_i}\in \mathcal{C}_B^{\Bi}(\wideparen{\D}_{V_i})_{\Delta}$.
    \item For any affinoid cover $(V_i)_{i=1}^n$ of $X$ we have $\mathcal{M}_{\vert V_i}\in \mathcal{C}_B^{\Bi}(\wideparen{\D}_{V_i})_{\Delta}$.
\end{enumerate}
\end{coro}
\begin{proof}
Let $\mathcal{M}\in \mathcal{C}_B^{\Bi}(\wideparen{\D}_X)_{\Delta}$, and let $(V_i)_{i=1}^n$ be an affinoid cover of $X$. By Kiehl's Theorem for co-admissible modules, it follows that $\mathcal{M}_{\vert V_i}\in\mathcal{C}_B^{\Bi}(\wideparen{\D}_{V_i})$ for $1\leq i\leq n$. Thus, we only need to show that:
\begin{equation*}
    \mathcal{M}(V_i)=\mathcal{M}(V_i)_{\infty}(\mathcal{I}_{\Delta}(V_i)).
\end{equation*}
By Proposition \ref{prop embedding Theorem} we have $\Delta^E_*\mathcal{M}=\operatorname{Loc}(\mathcal{M}(X))$. Thus, for $1\leq i\leq n$ we have:
\begin{equation*}
(\Delta_{\vert V_i})^E_*\left(\mathcal{M}_{\vert V_i}\right)=\left(\Delta^E_*\mathcal{M}\right)_{\vert V_i^2}=\operatorname{Loc}(\mathcal{M}(X))_{\vert V_i^2}= \operatorname{Loc}(\mathcal{M}(V_i)).   
\end{equation*}
Thus, the co-admissible $\wideparen{E}_{V_i}$-module associated to $\mathcal{M}(V_i)$ has support contained in the diagonal. Hence, by Proposition \ref{prop co-admissible EX modules supported on the diagonal}, we have $\mathcal{M}(V_i)=\mathcal{M}(V_i)_{\infty}(\mathcal{I}_{\Delta}(V_i))$.
Thus, $(i)$ implies $(ii)$ and $(iii)$. The converse is clear.
\end{proof}
As a consequence of our results thus far, it follows that the category of co-admissible diagonal bimodules $\mathcal{C}_B^{\Bi}(\wideparen{\D}_X)_{\Delta}$ presents similar behavior to 
the usual category of co-admissible modules $\mathcal{C}(\wideparen{\D}_X)$. In our current setting, the sheaves in $\mathcal{C}(\wideparen{\D}_X)$ are completely determined by their global sections, so one may expect that this phenomena carries over to $\mathcal{C}_B^{\Bi}(\wideparen{\D}_X)_{\Delta}$. We conclude this section by confirming this intuition. Let us start by defining the relevant categories:
\begin{defi}
Let $*=L,R$, or $B$. We define $\mathcal{C}_*^{\operatorname{Bi}}(\wideparen{\D}_X(X))_{\Delta}$ as the full subcategory of $\mathcal{C}_*^{\operatorname{Bi}}(\wideparen{\D}_X(X))$, given by objects $\mathcal{M}$ such that $ \mathcal{M}=\mathcal{M}_{\infty}(\mathcal{I}_{\Delta})$. 
\end{defi}
\begin{coro}\label{coro co-admissible EX-modules supported on the diagonal}
 Let $\mathcal{M}\in \mathcal{C}_*^{\operatorname{Bi}}(\wideparen{\D}_X(X))$. Then $\mathcal{M}\in \mathcal{C}(\wideparen{E}_X(X^2))$. Furthermore, $\operatorname{Loc}(\mathcal{M})$ is supported on $\Delta\subset X^2$ if and only if $ \mathcal{M}\in \mathcal{C}_*^{\operatorname{Bi}}(\wideparen{\D}_X(X))_{\Delta}$.
\end{coro}
\begin{proof}
This is a consequence of Theorem  \ref{teo embedding co-admissible bimodules}, and Proposition \ref{prop co-admissible EX modules supported on the diagonal}.
\end{proof}
Our next goal is constructing a localization functor:
\begin{equation*}
    \operatorname{Loc}(-):\mathcal{C}_B^{\operatorname{Bi}}(\wideparen{\D}_X(X))_{\Delta}\rightarrow \mathcal{C}_B^{\operatorname{Bi}}(\wideparen{\D}_X)_{\Delta},
\end{equation*}
which shows that co-admissible bimodules supported on the diagonal are completely determined by their global sections.
\begin{prop}\label{prop localization bimodules with support contained in the diagonal}
There is an equivalence of abelian categories:
\begin{equation*}
    \operatorname{Loc}(-):\mathcal{C}_B^{\operatorname{Bi}}(\wideparen{\D}_X(X))_{\Delta}\rightarrow \mathcal{C}_B^{\operatorname{Bi}}(\wideparen{\D}_X)_{\Delta},
\end{equation*}
given by sending $\mathcal{M}\in \mathcal{C}_B^{\operatorname{Bi}}(\wideparen{\D}_X(X))_{\Delta}$ to the unique $\wideparen{\D}_X^e$-module defined on affinoid subdomains $V\subset X$ by:
\begin{equation*}
    \operatorname{Loc}(\mathcal{M})(V)= \wideparen{\D}_X(V)\overrightarrow{\otimes}_{\wideparen{\D}_X(X)}\mathcal{M}\overrightarrow{\otimes}_{\wideparen{\D}_X(X)}\wideparen{\D}_X(V).
\end{equation*}
\end{prop}
\begin{proof}
Consider a module $\mathcal{M}\in \mathcal{C}_B^{\operatorname{Bi}}(\wideparen{\D}_X(X))_{\Delta}$. Let $\mathcal{M}^L=\operatorname{Loc}(\mathcal{M})$ be the associated sheaf of co-admissible $\wideparen{\D}_X$-modules,  $\mathcal{M}^R=\operatorname{Loc}(\mathcal{M})$ be the analogous sheaf of co-admissible $\wideparen{\D}_X^{\op}$-modules, and  $\mathcal{M}^B:= \operatorname{Loc}(\mathcal{M})$ be the corresponding co-admissible $\wideparen{E}_X$-module. By Corollary \ref{coro co-admissible EX-modules supported on the diagonal}, $\mathcal{M}^B$ is supported on the diagonal. Hence, for any affinoid subdomain $V\subset X$ we have:
\begin{equation*}
    \mathcal{M}^B(V\times X)=\mathcal{M}^B(V\times V)=\mathcal{M}^B(X\times V).
\end{equation*}
Thus, we have the following identities of complete bornological $\wideparen{E}_X(V\times V)$-modules:
\begin{equation*}
 \wideparen{\D}_X(V)\overrightarrow{\otimes}_{\wideparen{\D}_X(X)}\mathcal{M}= \wideparen{\D}_X(V)\overrightarrow{\otimes}_{\wideparen{\D}_X(X)}\mathcal{M}\overrightarrow{\otimes}_{\wideparen{\D}_X(X)}\wideparen{\D}_X(V)=\mathcal{M}\overrightarrow{\otimes}_{\wideparen{\D}_X(X)}\wideparen{\D}_X(V).
\end{equation*}
 In particular, $\operatorname{Loc}(\mathcal{M})=\mathcal{M}^L=\mathcal{M}^R$. Hence, it follows that $\operatorname{Loc}(\mathcal{M})\in \mathcal{C}_B^{\operatorname{Bi}}(\wideparen{\D}_X)_{\Delta}$. The fact that this is an equivalence follows by Proposition \ref{prop embedding Theorem}.
\end{proof}
\begin{coro}\label{coro immersion Theorem local case}
    We have the following commutative diagram of categories:
    \begin{equation*}
\begin{tikzcd}
\mathcal{C}_B^{\operatorname{Bi}}(\wideparen{\D}_X(X))_{\Delta} \arrow[d, "\operatorname{Loc}"'] \arrow[r] & \mathcal{C}(\wideparen{E}_X(X^2))_{\Delta} \arrow[d, "\operatorname{Loc}"] \\
\mathcal{C}_B^{\operatorname{Bi}}(\wideparen{\D}_X)_{\Delta} \arrow[r, "\Delta_*"]                         & \mathcal{C}(\wideparen{E}_X))_{\Delta}                                    
\end{tikzcd}
    \end{equation*}
    satisfying the following properties:
    \begin{enumerate}[label=(\roman*)]
        \item The upper horizontal map is the restriction to $\mathcal{C}_B^{\operatorname{Bi}}(\wideparen{\D}_X(X))_{\Delta}$ of the map constructed in Theorem \ref{teo embedding co-admissible bimodules}.
        \item The vertical maps are equivalences of abelian categories.
        \item The horizontal maps are embeddings of abelian categories. Furthermore, their essential images are Serre subcategories.
    \end{enumerate}
\end{coro}
\begin{proof}
    This is a summary of the contents of previous sections. Namely, of the diagram follows by Propositions \ref{prop embedding Theorem} and \ref{prop localization bimodules with support contained in the diagonal}. Claim  $(ii)$ is showed in Proposition \ref{prop localization bimodules with support contained in the diagonal} and Definition \ref{defi localization functor for co-admissible E(l)-modules}, and $(iii)$ was part of Theorem \ref{teo embedding co-admissible bimodules}.
\end{proof}
\subsection{Co-admissibility and the immersion functor II}
In this section, we generalize the results of the previous sections to the non-affinoid setting. As before, Kashiwara's equivalence will play an important role. Thus, we will only deal with the case $\mathscr{L}=\mathcal{T}_{X/K}$. For the rest of this section, we fix a separated smooth rigid analytic space $X$. Let us start with the following definition:
\begin{defi}\label{defi co-admissible diagonal bimodule}
Let $\mathcal{C}_B^{\operatorname{Bi}}(\wideparen{\D}_X)_{\Delta}$ be the full subcategory of $\Mod_{\Indban}(\wideparen{\D}_X^e)$ given by the modules $\mathcal{M}$ such that there is an admissible affinoid cover $(V_i)_{i\in I}$ of $X$ satisfying the following:
\begin{enumerate}[label=(\roman*)]
    \item For each $i\in I$ there is an étale map $V_i\rightarrow \mathbb{A}^n_K$.
    \item For each $i\in I$, we have $\mathcal{M}_{\vert V_i}\in \mathcal{C}_B^{\operatorname{Bi}}(\wideparen{\D}_{V_i})_{\Delta}$.
\end{enumerate}
We call $\mathcal{C}_B^{\operatorname{Bi}}(\wideparen{\D}_X)_{\Delta}$ the category of  co-admissible diagonal $\wideparen{\D}_X$-bimodules.
\end{defi}
Notice that if $X$ is affinoid and satisfies that there is an étale map $X\rightarrow \mathbb{A}^n_K$, then by Corollary \ref{coro khiels Theorem for diagonal co-admissible bi-modules} this definition agrees with the one given in Definition \ref{defi diagonal co-admissible bimodules local case}. Furthermore, we can show the following:
\begin{Lemma}\label{Lemma local properties of diagonal co-admissible bimodules}
Let $\mathcal{M}\in\Mod_{\Indban}(\wideparen{\D}_X^e)$. The following are equivalent:
\begin{enumerate}[label=(\roman*)]
    \item $\mathcal{M}\in \mathcal{C}_B^{\operatorname{Bi}}(\wideparen{\D}_X)_{\Delta}$.
    \item For every admissible affinoid cover $(V_i)_{i\in I}$ of $X$ satisfying that for each $i\in I$ there is an étale map $V_i\rightarrow \mathbb{A}_K^n$, we have $\mathcal{M}_{\vert V_i}\in \mathcal{C}_B^{\operatorname{Bi}}(\wideparen{\D}_{V_i})_{\Delta}$.
\end{enumerate}
\end{Lemma}
\begin{proof}
Clearly statement $(ii)$ implies $(i)$. Thus, let $\mathcal{M}\in \mathcal{C}_B^{\operatorname{Bi}}(\wideparen{\D}_X)_{\Delta}$  and let $(V_i)_{i\in I}$ be an admissible affinoid cover of $X$ as in the statement. By definition of $\mathcal{C}_B^{\operatorname{Bi}}(\wideparen{\D}_X)_{\Delta}$, there is an affinoid cover $(W_j)_{j\in J}$ of $X$ such that for each $j\in J$ there is an étale map $W_j\rightarrow \mathbb{A}^n_K$, and $\mathcal{M}_{\vert W_j}\in \mathcal{C}_B^{\operatorname{Bi}}(\wideparen{\D}_{W_j})_{\Delta}$. As $X$ is separated, for each $i\in I$ the family $(V_i\cap W_j)_{j\in J}$ is an admissible affinoid cover of $V_i$. Furthermore, as $V_i$ is affinoid, it is quasi-compact. Hence, we may reduce this to a finite affinoid cover. But then by Corollary \ref{coro khiels Theorem for diagonal co-admissible bi-modules}, it follows that $\mathcal{M}_{\vert V_i\cap W_j}\in \mathcal{C}_B^{\operatorname{Bi}}(\wideparen{\D}_{V_i\cap W_j})_{\Delta}$ for each $i\in I$ and each $j\in J$. A second application of Corollary \ref{coro khiels Theorem for diagonal co-admissible bi-modules}, shows that $\mathcal{M}_{\vert V_i}\in \mathcal{C}_B^{\operatorname{Bi}}(\wideparen{\D}_{V_i})_{\Delta}$.
\end{proof}
\begin{prop}
Let $\mathcal{M}\in \Mod_{\Indban}(\wideparen{\D}_X^e)$. The following are equivalent:
\begin{enumerate}[label=(\roman*)]
    \item $\mathcal{M}\in \mathcal{C}_B^{\operatorname{Bi}}(\wideparen{\D}_X)_{\Delta}$.
    \item There is an admissible cover $(V_i)_{i\in I}$ of $X$ such that $\mathcal{M}_{\vert V_i}\in \mathcal{C}_B^{\operatorname{Bi}}(\wideparen{\D}_{V_i})_{\Delta}$ for each $i\in I$.
\end{enumerate}
\end{prop}
\begin{proof}
This follows at once from  Lemma \ref{Lemma local properties of diagonal co-admissible bimodules}, as given an admissible cover $(V_i)_{i\in I}$ of $X$, we may refine it by an admissible affinoid cover and, conversely, given admissible affinoid covers of each of the $U_i$, the union of all such covers is an admissible affinoid cover of $X$.
\end{proof}
Now that we have shown that $\mathcal{C}_B^{\operatorname{Bi}}(\wideparen{\D}_X)_{\Delta}$ may be determined locally, we can start studying some of its algebraic properties:
\begin{prop}\label{prop co-admissible diag is abelian}
 $\mathcal{C}_B^{\operatorname{Bi}}(\wideparen{\D}_X)_{\Delta}$ is an abelian category.  
\end{prop}
\begin{proof}
As $\mathcal{C}_B^{\operatorname{Bi}}(\wideparen{\D}_X)_{\Delta}$ is a subcategory of $\Mod_{\Indban}(\wideparen{\D}_X^e)$, it suffices to show that it contains kernels and cokernels of maps in $\mathcal{C}_B^{\operatorname{Bi}}(\wideparen{\D}_X)_{\Delta}$, and that every morphism in $\mathcal{C}_B^{\operatorname{Bi}}(\wideparen{\D}_X)_{\Delta}$ is strict. This may be shown locally, where we can use Lemma \ref{Lemma diagonal co-admissible bimodules are abelian}.    
\end{proof}
\begin{Lemma}\label{Lemma immersion Theorem}
Let $X$ be a smooth affinoid space and $f:U\rightarrow X$ be the inclusion of an affinoid subspace. Let $h:Y\rightarrow X$ be a closed immersion, and assume we have a factorization  $g:Y\rightarrow U$, $h=fg$. Then the pushforward of $\wideparen{\D}$-modules:
\begin{equation*}
    f_+:\operatorname{D}_{\mathcal{C}^Y}(\wideparen{\D}_{U})\rightarrow \operatorname{D}_{\mathcal{C}^Y}(\wideparen{\D}_{X}),
\end{equation*}
satisfies the identity $f_+=f_*$.
\end{Lemma}
\begin{proof}
By assumption, we have $h=fg$. As $f$ is an open immersion, it is smooth. Thus, by  \cite[Lemma 7.8]{bode2021operations}, we have $h_+=f_+g_+$. On the other hand,
we always have $h^!=g^!f^!$. By Kashiwara's equivalence, we have mutually inverse equivalences of triangulated categories: 
\begin{equation*}
    h_+:\operatorname{D}_\mathcal{C}(\wideparen{\D}_{Y})\leftrightarrows \operatorname{D}_{\mathcal{C}^Y}(\wideparen{\D}_{X}):h^!, \quad g_+:\operatorname{D}_\mathcal{C}(\wideparen{\D}_{Y})\leftrightarrows \operatorname{D}_{\mathcal{C}^Y}(\wideparen{\D}_{U}):g^!.
\end{equation*}
Hence, the composition rules described above imply that we have mutually inverse equivalences of triangulated categories: 
\begin{equation*}
    f_+:\operatorname{D}_{\mathcal{C}^Y}(\wideparen{\D}_{U})\leftrightarrows \operatorname{D}_{\mathcal{C}^Y}(\wideparen{\D}_{X}):f^!.
\end{equation*}
As $f:U\rightarrow X$ is an open immersion, for every $\mathcal{M}\in \operatorname{D}(\wideparen{\D}_X)$ we have:
\begin{equation*}
    f^!\mathcal{M}:=\OX_U\overrightarrow{\otimes}_{f^{-1}\OX_X}^{\mathbb{L}}f^{-1}\mathcal{M}[\operatorname{dim}(U)-\operatorname{dim}(X)]=f^{-1}\mathcal{M}=\mathcal{M}_{\vert U}.
\end{equation*}
Arguing as above, and using Lemma \ref{Lemma properties of inclusion of derived category of modules supported on a closed subspace}, we have mutually inverse equivalences:  $f_*:\operatorname{D}(\wideparen{\D}_U)_Y\leftrightarrows \operatorname{D}(\wideparen{\D}_X)_Y:f^{-1}$. The fact that $f^!=f^{-1}$ implies that $f_+=f_*$.
\end{proof}
\begin{teo}\label{teo immersion Theorem of sheaves of diagonal co-admissible bimodules}
Let $X$ be a smooth and separated rigid space. The extension functor induces an exact embedding of abelian categories:
\begin{equation*}
    \Delta^E_*:\mathcal{C}_B^{\operatorname{Bi}}(\wideparen{\D}_X)_{\Delta}\rightarrow \mathcal{C}(\wideparen{E}_X)_{\Delta}.
\end{equation*}
Furthermore, its essential image is a Serre subcategory of $\mathcal{C}(\wideparen{E}_X)$. 
\end{teo}
\begin{proof}
Let $(V_i)_{i\in I}$ be an admissible affinoid cover of $X$ such that for each $i\in I$ there is an étale map $V_i\rightarrow \mathbb{A}^{n}_K$. Then the family $(V_i\times V_j)_{i,j\in I}$ yields an admissible affinoid cover of $X^2$, and for each $i,j\in I$ we have an étale map $V_{i}\times V_j\rightarrow \mathbb{A}^{2n}_K$. Let $\mathcal{M}\in \mathcal{C}_B^{\operatorname{Bi}}(\wideparen{\D}_X)_{\Delta}$. By Lemma \ref{Lemma kiehl's Theorem for E-modules on general rigid spaces},  it is enough to show that $\left(\Delta_{*}^E\mathcal{M}\right)_{\vert V_i\times V_j}$ is a sheaf of co-admissible $\wideparen{E}_{X\vert V_i\times V_j}$-modules for each $i,j\in I$. Choose a pair $i,j\in I$. If $V_i\cap V_j$ is empty, then $\left(\Delta_{*}^E\mathcal{M}\right)_{\vert V_i\times V_j}=0$. In particular, it is co-admissible. Assume then that $W=V_i\cap V_j$ is non-empty. Then we have a pullback diagram:
\begin{equation*}
\begin{tikzcd}
X \arrow[rr, "\Delta"]                                           &                          & X^2                     \\
W \arrow[u] \arrow[r, "f"'] \arrow[rr, "\Delta", bend left] & W^2 \arrow[r, "g"'] & V_i\times V_j \arrow[u]
\end{tikzcd}
\end{equation*}
where the vertical arrows and $g$ are open immersions, and $\Delta$ and $f$ are closed immersions. By functoriality of the pushforward of sheaves, we know that:
\begin{equation*}
    \left(\Delta_*^E\mathcal{M}\right)_{\vert V_i\times V_j}=\Delta^E_*(\mathcal{M}_{\vert W})=g_*f_*^E(\mathcal{M}_{\vert W}),
\end{equation*}
where $f_*^E:=\OX_{W^2}\overrightarrow{\otimes}_{p_1^{-1}\OX_W\overrightarrow{\otimes}_Kp_2^{-1}\OX_W}f_*\mathcal{M}_{\vert W}$.
By Proposition \ref{prop embedding Theorem}, it follows that $f_*^E(\mathcal{M}_{\vert W})=\operatorname{Loc}(\mathcal{M}(W))$ is a co-admissible $\wideparen{E}_{W}$-module supported on the diagonal. Equivalently, the induced complete bornological $\wideparen{\D}_{W^2}$-module:
\begin{equation*}
  \operatorname{S}^{-1}f_*^E(\mathcal{M}_{\vert W})=p_2^{-1}\Omega_{W}^{-1}\overrightarrow{\otimes}_{p_2^{-1}\OX_W}f_*^E(\mathcal{M}_{\vert W}),  
\end{equation*}
is a co-admissible $\wideparen{\D}_{W^2}$-module supported on the diagonal: $\Delta(W)\subset W^2\subset V_i\times V_j$. Thus, there is a co-admissible $\wideparen{\D}_W$-module $\mathcal{N}$ such that $\operatorname{S}^{-1}f_*^E(\mathcal{M}_{\vert W})=f_+\mathcal{N}$. As $g$ is an open immersion, it is smooth. Hence, by \cite[Lemma 7.8]{bode2021operations} we have $\Delta_+=g_+\circ f_+$. Thus, by Lemma \ref{Lemma immersion Theorem}, we have:
\begin{align*}
\Delta_+\mathcal{N}=g_+f_+\mathcal{N}=g_+\left(\operatorname{S}^{-1}f_*^E(\mathcal{M}_{\vert W}) \right)=&g_*\left(\operatorname{S}^{-1}f_*^E(\mathcal{M}_{\vert W}) \right)\\
 =&g_*\left(p_2^{-1}\Omega_{W}^{-1}\overrightarrow{\otimes}_{p_2^{-1}\OX_W}f_*^E(\mathcal{M}_{\vert W})\right)\\=&
 p_2^{-1}\Omega_{V_j}^{-1}\overrightarrow{\otimes}_{p_2^{-1}\OX_{V_j}}g_*\left(f_*^E(\mathcal{M}_{\vert W}) \right)\\=&\operatorname{S}^{-1}\Delta_*^E(\mathcal{M}_{\vert W}),
\end{align*}
where the fourth identity follows by the fact that the fourth and fifth terms have supports contained in the diagonal. As the pushforward of $\wideparen{\D}$-modules along a closed immersion preserves co-admissibility, it follows that $\operatorname{S}^{-1}\Delta_*^E(\mathcal{M}_{\vert W})=\Delta_+\mathcal{N}$
is a co-admissible $\wideparen{\D}_{V_i\times V_j}$-module. Hence, $\Delta_*(\mathcal{M}_{\vert W})$  is a co-admissible $\wideparen{E}_{X\vert V_i\times V_j}$-module, as we wanted to show. As $\mathcal{C}(\wideparen{E}_X)_{\Delta}$ is a Serre subcategory of $\mathcal{C}(\wideparen{E}_X)$, it suffices to show that the essential image of $\Delta_*:\mathcal{C}_B^{\operatorname{Bi}}(\wideparen{\D}_X)_{\Delta}\rightarrow \mathcal{C}(\wideparen{E}_X)_{\Delta}$ is a Serre subcategory. As this can be done locally, it follows by $(iii)$ in Corollary \ref{coro immersion Theorem local case}.
\end{proof}
As always, we get the corresponding result for the immersion functor:
\begin{coro}\label{coro immersion Theorem of sheaves of diagonal co-admissible bimodules}
There is an exact embedding of abelian categories:
\begin{equation*}
    \Delta^{\operatorname{S}}_*:\mathcal{C}_B^{\operatorname{Bi}}(\wideparen{\D}_X)_{\Delta}\rightarrow \mathcal{C}(\wideparen{\D}_{X^2})_{\Delta}.
\end{equation*}
Furthermore, the essential image of this functor is a Serre subcategory of $\mathcal{C}(\wideparen{\D}_{X^2})_{\Delta}$.
\end{coro}
\subsection{\texorpdfstring{C-complexes and the immersion functor}{}}
We will now extend the results of the previous sections to $\mathcal{C}$-complexes. As before, we fix a smooth and separated rigid space $X$. We start by defining $\mathcal{C}$-complexes in the category of bimodules:
\begin{defi}
Choose $C^{\bullet}\in\operatorname{D}(\wideparen{\D}^e_X)$. We make the following definitions:
\begin{enumerate}[label=(\roman*)]
    \item We say $C^{\bullet}$ is a bi-$\mathcal{C}$-complex if it is a $\mathcal{C}$-complex with respect to its structure as a complex of $\wideparen{\D}_X$-modules and $\wideparen{\D}^{\op}_X$-modules. We denote the full subcategory of $\operatorname{D}(\wideparen{\D}^e_X)$ given by  bi-$\mathcal{C}$-complexes by $\operatorname{D}^{\operatorname{Bi}}_{\mathcal{C}}(\wideparen{\D}_X)$.
    \item We say $C^{\bullet}$ is a diagonal $\mathcal{C}$-complex if it is a bi-$\mathcal{C}$-complex such that its cohomology sheaves are  co-admissible diagonal $\wideparen{\D}_X$-bimodules. We denote the full subcategory of $\operatorname{D}(\wideparen{\D}^e_X)$ given by  diagonal $\mathcal{C}$-complexes by $\operatorname{D}^{\operatorname{Bi}}_{\mathcal{C}}(\wideparen{\D}_X)_{\Delta}$.
\end{enumerate}
\end{defi}
We start by showing some basic properties of $\operatorname{D}^{\operatorname{Bi}}_{\mathcal{C}}(\wideparen{\D}_X)$ and 
$\operatorname{D}^{\operatorname{Bi}}_{\mathcal{C}}(\wideparen{\D}_X)_{\Delta}$:
\begin{Lemma}\label{Lemma diagonal bimodules Serre subcategory}
The category $\mathcal{C}_B^{\operatorname{Bi}}(\wideparen{\D}_X)_{\Delta}$ is a Serre subcategory of
  $\mathcal{C}_B^{\operatorname{Bi}}(\wideparen{\D}_X)$.    
\end{Lemma}
\begin{proof}
By definition of $\mathcal{C}_B^{\operatorname{Bi}}(\wideparen{\D}_X)_{\Delta}$, we may assume that $X$ is affinoid and admits an étale map $X\rightarrow \mathbb{A}^n_K$. Consider a short exact sequence in $\mathcal{C}_B^{\operatorname{Bi}}(\wideparen{\D}_X)$:
\begin{equation*}
    0\rightarrow \mathcal{M}_1\rightarrow \mathcal{M}_2\rightarrow \mathcal{M}_3\rightarrow 0,
\end{equation*}
and assume that $\mathcal{M}_1, \mathcal{M}_3\in \mathcal{C}_B^{\operatorname{Bi}}(\wideparen{\D}_X)_{\Delta}$. We need to show $\mathcal{M}_2\in \mathcal{C}_B^{\operatorname{Bi}}(\wideparen{\D}_X)_{\Delta}$. As $X$ is affinoid, we get a strict short exact sequence on global sections:
\begin{equation*}
     0\rightarrow \mathcal{M}_1(X)\rightarrow \mathcal{M}_2(X)\rightarrow \mathcal{M}_3(X)\rightarrow 0.
\end{equation*}
By Theorem \ref{teo embedding co-admissible bimodules}, we can regard this as a short exact sequence of co-admissible $\wideparen{E}_X(X^2)$-modules. Hence, there is a short exact sequence of $\wideparen{E}_X$-modules:
\begin{equation*}
    0\rightarrow \operatorname{Loc}(\mathcal{M}_1(X))\rightarrow \operatorname{Loc}(\mathcal{M}_2(X))\rightarrow \operatorname{Loc}(\mathcal{M}_3(X))\rightarrow 0.
\end{equation*}
As the first and third terms of the sequence have support contained in the diagonal $\Delta(X)\subset X^2$, it follows that $\operatorname{Loc}(\mathcal{M}_2(X))$ also has support contained in the diagonal. However, by  Proposition \ref{prop co-admissible EX modules supported on the diagonal}, this is equivalent to:
\begin{equation*}
    \mathcal{M}_2(X)=\mathcal{M}_2(X)_{\infty}(\mathcal{I}_{\Delta}(X^2)),
\end{equation*}
which shows that $\mathcal{M}_2\in \mathcal{C}_B^{\operatorname{Bi}}(\wideparen{\D}_X)_{\Delta}$, as we wanted to show.
\end{proof}
\begin{prop}\label{prop diagonal c comp are triangulated}
The categories $\operatorname{D}^{\operatorname{Bi}}_{\mathcal{C}}(\wideparen{\D}_X)$ and 
$\operatorname{D}^{\operatorname{Bi}}_{\mathcal{C}}(\wideparen{\D}_X)_{\Delta}$ are triangulated subcategories of the derived category  $\operatorname{D}(\wideparen{\D}^e_X)$.    
\end{prop}
\begin{proof}
Notice that the strict monomorphisms:
\begin{equation*}
    \wideparen{\D}_X\rightarrow \wideparen{\D}_X^{e}, \textnormal{ }\wideparen{\D}_X^{\op}\rightarrow \wideparen{\D}_X^{e},
\end{equation*}
induce strongly exact forgetful functors:
\begin{align*}
    &\operatorname{Forget}:\Mod_{\Indban}(\wideparen{\D}_X^{e})\rightarrow \Mod_{\Indban}(\wideparen{\D}_X),\\
    &\operatorname{Forget}:\Mod_{\Indban}(\wideparen{\D}_X^{e})\rightarrow \Mod_{\Indban}(\wideparen{\D}^{\op}_X).
\end{align*}
These lift to triangulated functors:
\begin{equation*}
\operatorname{Forget}:\operatorname{D}(\wideparen{\D}_X^{e})\rightarrow \operatorname{D}(\wideparen{\D}_X),\textnormal{ }\operatorname{Forget}:\operatorname{D}(\wideparen{\D}_X^{e})\rightarrow \operatorname{D}(\wideparen{\D}^{\op}_X).
\end{equation*}
By \cite[Theorem 8.9]{bode2021operations}, the categories $\operatorname{D}_{\mathcal{C}}(\wideparen{\D}_X)$ and $\operatorname{D}_{\mathcal{C}}(\wideparen{\D}^{\op}_X)$ are triangulated subcategories of $\operatorname{D}(\wideparen{\D}_X)$ and $\operatorname{D}_{\mathcal{C}}(\wideparen{\D}^{\op}_X)$ respectively. Thus, $\operatorname{D}^{\operatorname{Bi}}_{\mathcal{C}}(\wideparen{\D}_X)$ is a triangulated subcategory of $\operatorname{D}(\wideparen{\D}_X^{e})$.\\
Now consider a distinguished triangle:
\begin{equation*}
    C_1^{\bullet}\rightarrow C_2^{\bullet}\rightarrow C_3^{\bullet}\rightarrow C_1^{\bullet}[1],
\end{equation*}
such that $ C_1^{\bullet},C_3^{\bullet}\in \operatorname{D}^{\operatorname{Bi}}_{\mathcal{C}}(\wideparen{\D}_X)_{\Delta}$ and $C_2^{\bullet}\in \operatorname{D}^{\operatorname{Bi}}_{\mathcal{C}}(\wideparen{\D}_X)$. We need to show that $C_2^{\bullet}$ is an object of $  \operatorname{D}^{\operatorname{Bi}}_{\mathcal{C}}(\wideparen{\D}_X)_{\Delta}$. It is enough to show that $\operatorname{H}^i(C_2^{\bullet})\in \mathcal{C}_B^{\operatorname{Bi}}(\wideparen{\D}_X)_{\Delta}$ for each $i\in\mathbb{Z}$. For each $i\in\mathbb{Z}$, we have the following exact sequence in $LH(\Mod_{\Indban}(\wideparen{\D}_X)^{e})$:
\begin{equation*}
    \operatorname{H}^{i-1}(C^{\bullet}_3)\rightarrow\operatorname{H}^{i}(C^{\bullet}_1)\rightarrow \operatorname{H}^{i}(C^{\bullet}_2)\rightarrow \operatorname{H}^{i}(C^{\bullet}_3)\rightarrow\operatorname{H}^{i+1}(C^{\bullet}_1).
\end{equation*}
By assumption, all elements in this sequence are bi-co-admissible $\wideparen{\D}_X$-modules. Hence, we may regard this as an exact sequence in the abelian category $\mathcal{C}_B^{\operatorname{Bi}}(\wideparen{\D}_X)$. Furthermore, as a consequence of Lemma \ref{Lemma diagonal bimodules Serre subcategory},  $\mathcal{C}_B^{\operatorname{Bi}}(\wideparen{\D}_X)_{\Delta}$ is a Serre subcategory of
  $\mathcal{C}_B^{\operatorname{Bi}}(\wideparen{\D}_X)$. Thus, $\operatorname{H}^{i}(C^{\bullet}_2)\in\mathcal{C}_B^{\operatorname{Bi}}(\wideparen{\D}_X)_{\Delta}$ for each $i\in \mathbb{Z}$, and we are done.
\end{proof}
Next, we want to understand the behavior of $\mathcal{C}$-complexes with respect to the derived extension functor $\EDelta$. Let us start with the following lemma:
\begin{Lemma}\label{lemma aciclycity of diagonal co-ad for extension fucnt}
Let $\mathcal{M}\in \mathcal{C}_B^{\operatorname{Bi}}(\wideparen{\D}_X)_{\Delta}$. Then there is an isomorphism in $\operatorname{D}(\wideparen{E}_X)$:
\begin{equation*}
   \EDelta\mathcal{M} \xrightarrow[]{\cong}\Delta_*^E\mathcal{M}.
\end{equation*}
\end{Lemma}
\begin{proof}
By definition of the derived extension functor we have:
\begin{equation*}
    \EDelta\mathcal{M}= I(\wideparen{E}_X)\Tilde{\otimes}^{\mathbb{L}}_{p_1^{-1}I(\wideparen{\D}_X)\Tilde{\otimes}_Kp_2^{-1}I(\wideparen{\D}_X)^{\op}}\Delta_*I(\mathcal{M}),
\end{equation*}
where $\Tilde{\otimes}$ denotes the tensor product in the left heart $\operatorname{Shv}(X^2,LH(\widehat{\mathcal{B}}c_K))$. Thus, we need to show that the canonical map:
\begin{equation*}
    \EDelta\mathcal{M}\rightarrow I(\wideparen{E}_X)\Tilde{\otimes}_{p_1^{-1}I(\wideparen{\D}_X)\Tilde{\otimes}_Kp_2^{-1}I(\wideparen{\D}_X)^{\op}}\Delta_*I(\mathcal{M}),
\end{equation*}
is a quasi-isomorphism, and that the  canonical map:
\begin{equation*}
I(\wideparen{E}_X)\Tilde{\otimes}_{p_1^{-1}I(\wideparen{\D}_X)\Tilde{\otimes}_Kp_2^{-1}I(\wideparen{\D}_X)^{\op}}\Delta_*I(\mathcal{M}) \rightarrow I(\Delta_*^E\mathcal{M}), 
\end{equation*}
is an isomorphism as well. As these maps are already defined globally, we may as well assume that $X$ is affinoid with an étale map $X\rightarrow \mathbb{A}^n_K$. By the properties of the tensor product of sheaves, for each $i\in \mathbb{Z}$  the cohomology sheaf $\operatorname{H}^i(\EDelta\mathcal{M})$ is the sheafification of the presheaf sending an affinoid subdomain $V\subset X^2$ to:
\begin{equation}
    \mathcal{F}^i(V):=\operatorname{H}^i\left(I(\wideparen{E}_X(V))\Tilde{\otimes}^{\mathbb{L}}_{I(\wideparen{\D}_X(p_1(V))\Tilde{\otimes}_KI(\wideparen{\D}_X(p_2(V))^{\op}}I(\mathcal{M})(V\cap \Delta(X))\right).
\end{equation}
First, notice that $\mathcal{F}^i(V)=0$ whenever $V\cap \Delta(X)$ is empty. In particular, the $\mathcal{F}^i$ are supported in $\Delta(X)$, so  their sheafifications $\Tilde{\mathcal{F}}^i$ will be supported at $\Delta(X)$ as well. Let us first show that $\Tilde{\mathcal{F}}^i=0$ for all $i< 0$. Let $U\subset X$ be an affinoid subdomain, then by definition of the pullback in the category of presheaves we have:
\begin{equation}\label{equation tor}
    \Delta^{-1}\mathcal{F}^i(U)=\varinjlim_V \operatorname{H}^i\left(I(\wideparen{E}_X(V))\Tilde{\otimes}^{\mathbb{L}}_{I(\wideparen{\D}_X(p_1(V))\Tilde{\otimes}_KI(\wideparen{\D}_X(p_2(V))^{\op}}I(\mathcal{M})(U)\right),
\end{equation}
where the $V\subset X^2$ are affinoid subdomains satisfying $V\subset U^2$ and $V\cap \Delta(X)=U$. In this case, notice that for $i=1,2$ we have:
\begin{equation*}
    U=p_i(\Delta(U))\subseteq p_i(\Delta(V))\subseteq p_i(\Delta(U^2))=U.
\end{equation*}
Hence, we may rewrite the expression in (\ref{equation tor}) as follows:
\begin{equation*}
    \Delta^{-1}\mathcal{F}^i(U)=\varinjlim_V \operatorname{H}^i\left(I(\wideparen{E}_X(V))\Tilde{\otimes}^{\mathbb{L}}_{I(\wideparen{E}_X(U^2))}I(\mathcal{M})(U)\right).
\end{equation*}
However, by Theorem \ref{teo immersion Theorem of sheaves of diagonal co-admissible bimodules}  $\mathcal{M}(U)$, is a co-admissible $\wideparen{E}_X(U^2)$-module, and the transition map $\wideparen{E}_X(U^2)\rightarrow \wideparen{E}_X(V)$ is $c$-flat. Hence, it follows by \cite[Proposition 5.33]{bode2021operations} that we have:
\begin{equation*}
    \Delta^{-1}\mathcal{F}^0(U)=\varinjlim_V I(\wideparen{E}_X(V)\overrightarrow{\otimes}_{\wideparen{E}_X(U^2)}\mathcal{M}(U)) =I(\mathcal{M}(U))=\Delta^{-1}I(\Delta_*^E\mathcal{M})(U),
\end{equation*}
 and $\Delta^{-1}\mathcal{F}^i(U)=0$ for all $i< -1$, as we wanted to show.
\end{proof}
We now extend this to the category of diagonal $\mathcal{C}$-complexes:
\begin{prop}\label{prop immersion functor and C-complexes}
The derived extension functor induces a triangulated functor:
\begin{equation*}    \EDelta:\operatorname{D}^{\operatorname{Bi}}_{\mathcal{C}}(\wideparen{\D}_X)_{\Delta}\rightarrow \operatorname{D}_{\mathcal{C}^{\Delta}}(\wideparen{E}_X).
\end{equation*}
Furthermore, $\EDelta$ is an equivalence onto its essential image, with inverse:
\begin{equation*}
    \Delta^{-1}:\operatorname{D}_{\mathcal{C}^{\Delta}}(\wideparen{E}_X)\rightarrow \operatorname{D}(\wideparen{\D}_X^e).
\end{equation*}
Thus, for every diagonal $\mathcal{C}$-complex $\mathcal{M}^{\bullet}$ there is a canonical isomorphism in $\operatorname{D}(\wideparen{\D}_X^e)$:
\begin{equation*}
    \mathcal{M}^{\bullet}\xrightarrow[]{\cong} \Delta^{-1}\EDelta \mathcal{M}^{\bullet}.
\end{equation*}
\end{prop}
\begin{proof}
 Let $C^{\bullet}$ be a diagonal $\mathcal{C}$-complex. By construction, we know that $\EDelta(C^{\bullet})$ is an object of $\operatorname{D}(\wideparen{E}_X)$ with cohomology supported on the diagonal. We need  to show that  it is a $\mathcal{C}$-complex. We know that $\EDelta(C^{\bullet})$ is a $\mathcal{C}$-complex over $\wideparen{E}_X$ if and only if $\Ifunct (C^{\bullet})$ is a $\mathcal{C}$-complex over $\wideparen{\D}_{X^2}$. Hence, we may use Lemma \ref{Lemma immersion Theorem}, and the same argument as in the proof of Theorem \ref{teo immersion Theorem of sheaves of diagonal co-admissible bimodules}
 to reduce the proof to the case where $X=\Sp(A)$, and has an étale map $X\rightarrow \mathbb{A}^n_K$.\\
First, recall that, by definition, we have that $\EDelta$  is the left derived functor of:
\begin{equation*}
    \Delta_*^{I(E)}(-):=I(\wideparen{E}_X)\Tilde{\otimes}^{\mathbb{L}}_{p_1^{-1}I(\wideparen{\D}_X)\Tilde{\otimes}_Kp_2^{-1}I(\wideparen{\D}_X)^{\op}}\Delta_*(-).
\end{equation*}
By definition of a diagonal $\mathcal{C}$-complex, all cohomology groups of $C^{\bullet}$ are co-admissible diagonal $\wideparen{\D}_X$-bimodules. In particular, it follows by Lemma \ref{lemma aciclycity of diagonal co-ad for extension fucnt} that all the cohomology groups of $C^{\bullet}$ are $\EDelta$-acyclic. Hence, we can use \cite[Tag 06XW]{stacks-project} to conclude that for all $i\in \mathbb{Z}$ we have:
\begin{equation*}
    \operatorname{H}^i\left(\EDelta(C^{\bullet}) \right)=\EDelta\left(\operatorname{H}^i(C^{\bullet}) \right).
\end{equation*}
In particular, the cohomology groups of $\EDelta(C^{\bullet})$ are co-admissible modules supported in $\Delta(X)$. Choose a Fréchet-Stein presentation $\wideparen{\D}_X(X)=\varprojlim_n \widehat{D}_n$  such that we have a Fréchet-Stein presentation $\wideparen{\D}_X(X)^e=\varprojlim_n \widehat{D}_n^e$. It suffices to show that for all $n\geq 0$ we have:
\begin{equation*} \widehat{D}_n^e\overrightarrow{\otimes}_{\wideparen{\D}_X(X)^e}\Delta_*^E\operatorname{H}^i(C^{\bullet})(X^2)=\widehat{D}_n^e\overrightarrow{\otimes}_{\wideparen{\D}_X(X)^e}\operatorname{H}^i(C^{\bullet})(X)=0, 
\end{equation*}
 for all but finitely many $i\in \mathbb{Z}$. Notice that we have:
 \begin{equation*} 
 \widehat{D}_n^e\overrightarrow{\otimes}_{\wideparen{\D}_X(X)^e}\operatorname{H}^i(C^{\bullet})(X)=\widehat{\D}_n\overrightarrow{\otimes}_{\wideparen{\D}_X(X)}\operatorname{H}^i(C^{\bullet})(X)\overrightarrow{\otimes}_{\wideparen{\D}_X(X)}\widehat{\D}_n.   
 \end{equation*}
 As $C^{\bullet}$ is a diagonal $\mathcal{C}$-complex, it is in particular a $\mathcal{C}$-complex over $\wideparen{\D}_X$. Hence, the right hand side of the above equation is zero for all but finitely many $i\in \mathbb{Z}$.\bigskip
 
 For the last part, we first notice that there is a canonical morphism of sheaves of Ind-Banach algebras $\wideparen{\D}_X^e\rightarrow \Delta^{-1}\wideparen{E}_X$, so that the functor is well-defined. Additionally, for any $\mathcal{M}^{\bullet}\in\operatorname{D}^{\operatorname{Bi}}_{\mathcal{C}}(\wideparen{\D}_X)_{\Delta}$  there is a canonical $\wideparen{\D}_X^e$-linear map:
 \begin{equation*}
     \mathcal{M}^{\bullet}\rightarrow \Delta^{-1}\EDelta \mathcal{M}^{\bullet}=\Delta^{-1}\wideparen{E}_X\overrightarrow{\otimes}_{\wideparen{\D}_X^e}^{\mathbb{L}}\mathcal{M}^{\bullet}.
 \end{equation*}
 We need to show that this is a quasi-isomorphism. However, as $\Delta^{-1}$ is an exact functor, if suffices to show that for every $i\in \mathbb{Z}$ the following map is an isomorphism:
 \begin{equation*}
     \operatorname{H}^i(\mathcal{M}^{\bullet})\rightarrow \Delta^{-1}\operatorname{H}^i(\EDelta(\mathcal{M}^{\bullet}))=\Delta^{-1}\operatorname{H}^i(\EDelta(\mathcal{M}^{\bullet}))=\Delta^{-1}\EDelta(\operatorname{H}^i(\mathcal{M}^{\bullet})),
 \end{equation*}
 where the last identity follows because $\mathcal{M}^{\bullet}$ is a diagonal $\mathcal{C}$-complex, so its cohomology sheaves are co-admissible diagonal $\wideparen{\D}_X$-bimodules. Thus, we are reduced to showing $\mathcal{M}=\Delta^{-1}\Delta_*^E\mathcal{M}$ for $\mathcal{M}\in \mathcal{C}_B^{\operatorname{Bi}}(\wideparen{\D}_X)_{\Delta}$. This follows from Theorem \ref{teo immersion Theorem of sheaves of diagonal co-admissible bimodules}.
\end{proof}
\begin{coro}\label{coro immersion functor and C-complexes}
The derived immersion functor induces a triangulated functor:
\begin{equation*}
\Ifunct:\operatorname{D}^{\operatorname{Bi}}_{\mathcal{C}}(\wideparen{\D}_X)_{\Delta}\rightarrow \operatorname{D}_{\mathcal{C}^{\Delta}}(\wideparen{\D}_X^2).
\end{equation*}    
\end{coro}
\begin{coro}\label{coro Hh as derived tensor product}
There is a natural transformation of functors:
\begin{equation*}
    \mathcal{HH}_{\bullet}(\wideparen{\D}_X,-)\xrightarrow[]{\cong}\wideparen{\D}_X\overrightarrow{\otimes}_{\wideparen{\D}_X^e}^{\mathbb{L}}-.
\end{equation*}
\end{coro}
\begin{proof}
Follows from Proposition \ref{prop bi-enveloping algebra and HHo} and the second part of Proposition \ref{prop immersion functor and C-complexes}.
\end{proof}
\subsection{\texorpdfstring{D-cap as a co-admissible diagonal bimodule}{}}
In this section we will show that  $\wideparen{\D}_X$ is an object of $\mathcal{C}_B^{\operatorname{Bi}}(\wideparen{\D}_X)_{\Delta}$, and determine its image under the immersion functor. This will be instrumental for our calculations of Hochschild (co)-homology of diagonal $C$-complexes in the final chapter.
\begin{prop}\label{prop sheaf of completed differential operators is diagonal bimodule}
Let $X$ be smooth and separated. Then $\wideparen{\D}_X\in\mathcal{C}_B^{\operatorname{Bi}}(\wideparen{\D}_X)_{\Delta}$.
\end{prop}
\begin{proof}
As this is a local property, we may assume that $X=\Sp(A)$ is affinoid and there is an étale map $f:X\rightarrow \mathbb{A}^n_K$. In this case, we have that $\wideparen{\D}_X(X)$ is the completion of $\D_X(X)$ with respect to the induced bornology. Furthermore, $\D_X(X)$ has a canonical filtration $\Phi_{\bullet}\D_X(X)$ known as the PBW filtration, or filtration by order of differential operators. The ideal $\mathcal{I}_{\Delta}$ is given by elements of the form $c_f= f\otimes 1 -1\otimes f$, for $f\in A$. In particular, for any $\alpha\in \D_X(X)$, we have $c_f \alpha=[f,\alpha]$. By definition of the order of a differential operator, it follows that $\alpha\in\Phi_n\D_X(X)$ if and only if $c_f^n\alpha=0$ for all $f\in A$. As the PBW filtration is exhaustive, it follows that $\D_X(X)\subset \wideparen{\D}_X(X)_{\infty}(\mathcal{I}_{\Delta})$. However, as $\D_X(X)$ is dense in $\wideparen{\D}_X(X)$, it follows that
$\wideparen{\D}_X(X)=\wideparen{\D}_X(X)_{\infty}(\mathcal{I}_{\Delta})$. The fact that $\wideparen{\D}_X$ is co-admissible as a left and right module over itself is clear.
\end{proof}
Thus, $\Delta_*^{\operatorname{S}}(\wideparen{\D}_X)$ is a co-admissible $\wideparen{\D}_{X^2}$-module supported on the diagonal.
\begin{prop}\label{prop image of D under side changing}
There is an isomorphism of co-admissible $\wideparen{D}_{X^2}$-modules:  
 \begin{equation*}
     \Delta_*^{\operatorname{S}}(\wideparen{\D}_X)=\Delta_+\OX_X.
 \end{equation*}   
\end{prop}
\begin{proof}
By Kashiwara's equivalence, it is enough to show that:
\begin{equation*}
    \Delta^{!}\Delta_*^{\operatorname{S}}(\wideparen{\D}_X)=\OX_X.
\end{equation*}
We can assume that $X=\Sp(A)$ is affinoid, and that there is an étale map $X\rightarrow \mathbb{A}^n_K$, induced by local coordinates $x_1,\cdots,x_n\in A$. These local coordinates induce a basis $\partial_{x_1},\cdots, \partial_{x_n}$ of $\mathcal{T}_{X/K}$. This basis satisfies the following properties:
\begin{equation*}    \mathcal{T}_{X/K}=\bigoplus_{i=1}^n\OX_X\partial_{x_i}, \textnormal{ } [\partial_{x_i},\partial_{x_j}]=0 \textnormal{, } \partial_{x_i}(x_j)=\delta_{i,j}\textnormal{, for all } 1\leq i,j\leq n.
\end{equation*}
Notice that, under these conditions, we have the following identity in $\wideparen{\D}_X(X)$ for each $j\geq 0$, and each $1\leq i\leq n$:
\begin{equation*}
    x_i\partial_{x_i}^{j+1}-\partial_{x_i}^{j+1}x_i=(j+1)\partial_{x_i}^j. 
\end{equation*}
Next, notice that, as co-admissible modules are determined by their global sections, we need to show that the following identity of co-admissible $\wideparen{\D}_X(X)$-modules holds:
\begin{equation*}
    \left(\Omega_{X}^{-1}(X)\overrightarrow{\otimes}_{\OX_X(X)}\wideparen{\D}_X(X)\right)[\mathcal{I}_{\Delta}(X^2)]=\OX_X(X).
\end{equation*}
Our choice of an étale map $X\rightarrow \mathbb{A}^n_K$ induces a trivialization of $\Omega_X$ as a line bundle. In particular, we  can freely take $\Omega_{X}^{-1}\overrightarrow{\otimes}_{\OX_X}\wideparen{\D}_X=\wideparen{\D}_X$, and every $a\otimes b\in A\widehat{\otimes}_KA$ acts on each $\alpha \in \wideparen{\D}_X(X)$ by the formula $(a\otimes b)\alpha=a\alpha b$.\\
Recall from Lemma \ref{Lemma support on the diagonal} the ideals $\mathcal{J}_r$ in $A\widehat{\otimes}_KA$ generated by $(1\otimes x_j-x_j\otimes 1)_{j=1}^r$, and let $X_r=\Sp(A\widehat{\otimes}_KA/\mathcal{J}_r)$. We have an ascending chain of ideals:
\begin{equation*}
    \mathcal{J}_1\subset \mathcal{J}_2\subset \cdots \subset \mathcal{J}_n\subset \mathcal{I}_{\Delta}.
\end{equation*}
Hence, we get an associated family of closed immersions:
\begin{equation*}
    X\xrightarrow[]{i_{n+1}} X_n\rightarrow \cdots \rightarrow X_1 \xrightarrow[]{i_1} X^2,
\end{equation*}
such that the composition of all these maps is $\Delta:X\rightarrow X^2$.\\
By \cite[Lemma 9.3]{bode2021operations},  $i_1^{!}\wideparen{\D}_X$ is the co-admissible $\wideparen{\D}_{X_1}$-module associated to the kernel of the following short exact sequence:
\begin{equation*}
    0\rightarrow \wideparen{\D}_X(X)[\mathcal{J}_1] \rightarrow \wideparen{\D}_X(X)\xrightarrow[]{1\otimes x_1-x_1\otimes 1}\wideparen{\D}_X(X)\rightarrow 0.
\end{equation*}
Let $\wideparen{\D}_X(X)'$ be the subalgebra of $\wideparen{\D}_X(X)$ which is topologically generated by $A$, and $\partial_{x_2},\cdots, \partial_{x_n}$.
Notice that $x_1$ commutes with every element in $\wideparen{\D}_X(X)'$. Hence, we have $\wideparen{\D}_X(X)'\subset \wideparen{\D}_X(X)[\mathcal{J}_1]$. By construction of $\wideparen{\D}_X(X)$, we may express every element $f\in \wideparen{\D}_X(X)$ as the power series:
\begin{equation*}
    f=\sum_{i\geq 0}P_i\partial_{x_1}^i, \textnormal{ where } P_i\in \wideparen{\D}_X(X)'.
\end{equation*}
Assume that $f\in \wideparen{\D}_X(X)[\mathcal{J}_1]$. Then we have:
\begin{equation*}
    0=x_1f-fx_1=\sum_{i\geq 0}P_i(x_1\partial_{x_1}^i-\partial_{x_1}^ix_1)=\sum_{i\geq 1}iP_i\partial_{x_1}^{i-1}.
\end{equation*}
Hence, $P_i=0$ for $i\geq 1$, so $f\in \wideparen{\D}_X(X)'$. It follows that $ \wideparen{\D}_X(X)[\mathcal{J}_1]=\wideparen{\D}_X(X)'$.\\
We may inductively repeat the argument above, and it follows that we have:
\begin{equation*}
    \wideparen{\D}_X(X)[\mathcal{J}_n]=\OX_X(X),
\end{equation*}
and the result follows from the fact that the action of $\mathcal{I}_{\Delta}(X)$ on $\OX_X(X)$ is trivial.
\end{proof}
\begin{coro}\label{coro image of DX under right immersion functor}
There is an isomorphism of co-admissible $\wideparen{\D}_{X^2}$-modules:
\begin{equation*}
    \Delta_*^{\operatorname{S}_r}(\wideparen{\D}_X)=\Delta^r_+\Omega_{X}.
\end{equation*}  
\end{coro}
\begin{proof}
We have the following identities:
\begin{align*}
\Delta_*^{\operatorname{S}_r}(\wideparen{\D}_X)=\operatorname{S}^{-1}_r(\operatorname{T}(\Delta_*^E\wideparen{\D}_X))=&\Omega_{X^2/K}\overrightarrow{\otimes}_{\OX_{X^2}}\operatorname{S}^{-1}(\Delta_*^E\wideparen{\D}_X)\\=&\Omega_{X^2/K}\overrightarrow{\otimes}_{\OX_{X^2}}\Delta_+\OX_X\\=&\Delta_+^r\Omega_{X}, 
\end{align*}
where the first identity is Corollary \ref{coro image of enveloping algebra under T}, the second is Proposition \ref{prop equivalence left and right E(L)-modules}, the third is Proposition \ref{prop image of D under side changing}, and the last one follows by \cite[pp. 100]{bode2021operations}.
\end{proof}
We can use a similar approach to produce multiple examples of objects in $\mathcal{C}_B^{\operatorname{Bi}}(\wideparen{\D}_X)_{\Delta}$. However, as this is not directly used in the body of the text, we will postpone it to Appendix \ref{appendix examples of co-admissible diagonal bimodules}.
\section{Hochschild (Co)-homology}\label{Chapter Hochschil (co)-homology}
Let $X$ be a separated smooth rigid space. In this final chapter, we will apply the machinery developed thus far to obtain explicit calculations of the Hochschild cohomology and homology of diagonal $\mathcal{C}$-complexes on $X$. In particular, we obtain a canonical quasi-isomorphism in $\operatorname{D}(\operatorname{Shv}(X,\Indban))$:
\begin{equation*}
    \mathcal{HH}^{\bullet}(\wideparen{\D}_X,\wideparen{\D}_X)=\Omega_{X/K}^{\bullet},
\end{equation*}
which allows us to relate the Hochschild cohomology of $\wideparen{\D}_X$ with the de Rham cohomology of $X$. Similarly, the Hochschild homology of $\wideparen{\D}_X$ can be calculated via the formula:
\begin{equation*}
    \mathcal{HH}_{\bullet}(\wideparen{\D}_X,\wideparen{\D}_X)=\Omega_{X/K}^{\bullet}[2\operatorname{dim}(X)].
\end{equation*}
Hence, we obtain an upper bound on the degrees on which de Rham cohomology is non-zero, and obtain a duality between Hochschild cohomology and homology.
\subsection{Hochschild Cohomology}\label{Section hochschild cohomology}
Let $X$ be a smooth and separated rigid analytic variety. In this section, we will study the inner Hochschild cohomology sheaves of diagonal $\mathcal{C}$-complexes on $X$.  In order to do this, we will first show a version of  Kashiwara's equivalence for sheaves of inner homomorphisms:
\begin{prop}\label{prop derived internal hom and kashiwara}
Let $\mathcal{M}\in \operatorname{D}_{\mathcal{C}}^-(\wideparen{\D}_X)$, and     $\mathcal{N}\in \operatorname{D}_{\mathcal{C}^{\Delta}}(\wideparen{\D}_{X^2})$. There is a canonical isomorphism of complexes in $\operatorname{D}(\operatorname{Shv}(X,\Indban))$:
\begin{equation*}
    R\underline{\mathcal{H}om}_{\wideparen{\D}_{X}^2}(\Delta_+\mathcal{M},\mathcal{N})=\Delta_*R\underline{\mathcal{H}om}_{\wideparen{\D}_{X}}(\mathcal{M},\Delta^!\mathcal{N}).
\end{equation*}
\end{prop}
\begin{proof}
This proof follows \cite[Proposition 1.5.25]{hotta2007d}. We will start by showing that the transfer bimodule $\wideparen{\D}_{X^2\leftarrow X}$ is a $(\Delta^{-1}\wideparen{\D}_{X^2},\wideparen{\D}_{X})$-bimodule satisfying the conditions of Proposition
\ref{prop derived tensor-hom adjunction wrt a flat bimodule}. As this may be done locally, we may assume that $X$ is affinoid and admits an étale map $X\rightarrow \mathbb{A}^n_K$.\newline Let $Y\rightarrow X^2$ be the smooth rigid space from Lemma \ref{Lemma support on the diagonal}. By the discussion after Definition \ref{defi integrable connections}, we know that we have an isomorphism of right
$\wideparen{\D}_{Y}$-modules:
\begin{equation*}
    \wideparen{\D}_{X^2\leftarrow Y}\cong  K\{\alpha_1,\cdots,\alpha_n \}\overrightarrow{\otimes}_K\wideparen{\D}_{Y}.
\end{equation*}
As the canonical closed immersion $X\rightarrow Y$ identifies $X$ with a connected component of $Y$, it follows that $\wideparen{\D}_{X^2\leftarrow X}\cong  K\{\alpha_1,\cdots,\alpha_n \}\overrightarrow{\otimes}_K\wideparen{\D}_{X}$. Hence, for any $\mathcal{M}\in \Mod_{\Indban}(\wideparen{\D}_{X})$ we have:
\begin{align*}
    \wideparen{\D}_{X^2\leftarrow X}\overrightarrow{\otimes}^{\mathbb{L}}_{\wideparen{\D}_{X}}\mathcal{M}=&I\left(K\{\alpha_1,\cdots,\alpha_n \}\widetilde{\otimes}_K\wideparen{\D}_{X}\right)\widetilde{\otimes}^{\mathbb{L}}_{I(\wideparen{\D}_{X})}I(\mathcal{M})\\=&I(K\{\alpha_1,\cdots,\alpha_n \})\widetilde{\otimes}_KI(\mathcal{M})\\
    =&I(K\{\alpha_1,\cdots,\alpha_n \}\widetilde{\otimes}_K\mathcal{M})\\=&I(\wideparen{\D}_{X^2\leftarrow X}\overrightarrow{\otimes}_{\wideparen{\D}_{X}}\mathcal{M}),
\end{align*}
as we wanted to show. We then have the following chain of identities:
\begin{align*}
R\underline{\mathcal{H}om}_{\wideparen{\D}_{X^2}}(\Delta_+\mathcal{M},\mathcal{N})
=&R\underline{\mathcal{H}om}_{\wideparen{\D}_{X^2}}(\Delta_*\left(\wideparen{\D}_{X^2\leftarrow X}\overrightarrow{\otimes}_{\wideparen{\D}_X}\mathcal{M}\right),\Delta_*\Delta^{-1}\mathcal{N})\\=&
\Delta_*R\underline{\mathcal{H}om}_{\Delta^{-1}\wideparen{\D}_{X^2}}(\wideparen{\D}_{X^2\leftarrow X}\overrightarrow{\otimes}_{\wideparen{\D}_X}\mathcal{M},\Delta^{-1}\mathcal{N})\\
=&\Delta_*R\underline{\mathcal{H}om}_{\wideparen{\D}_{X}}(\mathcal{M},R\underline{\mathcal{H}om}_{\Delta^{-1}\wideparen{\D}_{X^2}}(\wideparen{\D}_{X^2\leftarrow X},\Delta^{-1}\mathcal{N}))\\  
=&\Delta_*R\underline{\mathcal{H}om}_{\wideparen{\D}_{X}}(\mathcal{M},\Delta^{!}\mathcal{N})),
\end{align*}
where the first identity follows by Lemma \ref{Lemma properties of inclusion of derived category of modules supported on a closed subspace}, the second by Proposition \ref{Proposition cohomology and closed immersions}, the third is Proposition \ref{prop derived tensor-hom adjunction wrt a flat bimodule}, and the last one is \cite[Lemma 7.4]{bode2021operations}.
\end{proof}
\begin{coro}\label{coro expression of Hochschild cohomology sheaves}
Let $\mathcal{M}^{\bullet}$ be a diagonal $\mathcal{C}$-complex on $X$. Then we have:
\begin{equation*}
    \mathcal{HH}^{\bullet}(\wideparen{\D}_X,\mathcal{M}^{\bullet})=R\underline{\mathcal{H}om}_{\wideparen{\D}_X}(\OX_X,\Delta^!\Ifunct\mathcal{M}^{\bullet}).
\end{equation*}
\end{coro}
\begin{proof}
This is a consequence of  Corollary \ref{coro hochschild cohomology and the immersion functor}, Corollary \ref{coro immersion functor and C-complexes}, Proposition \ref{prop image of D under side changing}, and Proposition \ref{prop derived internal hom and kashiwara}. 
\end{proof}
Recall from \cite[Section 6.3]{bode2021operations} that for any smooth rigid analytic variety $X$ we have a canonical resolution of $\OX_X$ by locally free $\wideparen{\D}_X$-modules. Namely, we have the Spencer resolution:
\begin{equation}\label{equation Spencer resolution}
    0\rightarrow \wideparen{\D}_X\overrightarrow{\otimes}_{\OX_X}\wedge^n\mathcal{T}_{X/K}\rightarrow \cdots \rightarrow \wideparen{\D}_X\rightarrow \OX_X.
\end{equation}
For $1\leq i\leq n$, the chain maps are the completions of the following $\mathcal{\D}_X$-linear maps:
\begin{multline*}
    \quad\quad\quad\quad\quad\quad\quad\quad\quad\quad\quad\quad\mathcal{D}_X\otimes_{\OX_X}\wedge^i\mathcal{T}_{X/K}\rightarrow \mathcal{D}_X\otimes_{\OX_X}\wedge^{i-1}\mathcal{T}_{X/K}\\
       P\otimes (\wedge \partial_j)\mapsto \sum_j (-1)^{j+1}P\partial_j\otimes \partial_1\wedge \cdots \wedge \widehat{\partial_j}\wedge\cdots \wedge \partial_i \\+ \sum_{j<j'} (-1)^{j+j'}P[\partial_j,\partial_{j'}]\otimes \partial_1\wedge \cdots \wedge \widehat{\partial_j}\wedge\cdots \wedge \widehat{\partial_j'}\wedge\cdots \wedge\partial_i
\end{multline*}
We will denote this resolution by:
\begin{equation*}
    S^{\bullet}:=\left(0\rightarrow \wideparen{\D}_X\overrightarrow{\otimes}_{\OX_X}\wedge^n\mathcal{T}_{X/K}\rightarrow \cdots \rightarrow \wideparen{\D}_X \rightarrow 0\right).
\end{equation*}
Applying side-changing, we get a resolution of $\Omega_X$ by finite locally-free $\wideparen{\D}_X^{\op}$-modules:
\begin{equation}\label{equation side-changing of the Spencer resolution}
\overline{S}^{\bullet}:=\left(0\rightarrow \wideparen{\D}_X\rightarrow \cdots \rightarrow \Omega_{X}\overrightarrow{\otimes}_{\OX_X}\wideparen{\D}_X \rightarrow 0\right).
\end{equation}
We may use this resolution together with Corollary \ref{coro expression of Hochschild cohomology sheaves} to give a precise description of the inner Hochschild complex:
\begin{defi}
For any $\mathcal{M}^{\bullet}\in\operatorname{D}(\wideparen{\D}_X)$, the de Rham complex of $\mathcal{M}^{\bullet}$ is defined as the following complex in $\operatorname{D}(\operatorname{Shv}(X,\Indban))$:
\begin{equation*}
    \operatorname{dR}(\mathcal{M}^{\bullet}):= R\underline{\mathcal{H}om}_{\wideparen{\D}_X}(\OX_{X},\mathcal{M}^{\bullet}).
\end{equation*}
\end{defi}
The idea is relating the Hochschild cohomology complex with the de Rham complex:
\begin{Lemma}
Let $\mathcal{M}^{\bullet}\in\operatorname{D}(\wideparen{\D}_X)$. Then we have an identification:
\begin{equation*}
    R\underline{\mathcal{H}om}_{\wideparen{\D}_X}(\OX_{X},\mathcal{M}^{\bullet})=\Omega_X\overrightarrow{\otimes}^{\mathbb{L}}_{\wideparen{\D}_X}\mathcal{M}^{\bullet}[-n].
\end{equation*}
\end{Lemma}
\begin{proof}
The proof of \cite[Proposition 2.6.14]{hotta2007d} works in this setting.    
\end{proof}
\begin{obs}\label{remark de rham complex of a variety as the de rham complex of its structure sheaf}
 Notice that we have:
 \begin{equation*}
  \operatorname{dR}(\OX_X)=R\underline{\mathcal{H}om}_{\wideparen{\D}_X}(\OX_{X},\OX_X)=\Omega_X\overrightarrow{\otimes}^{\mathbb{L}}_{\wideparen{\D}_X}\OX_X[-n]=\overline{S}^{\bullet}\overrightarrow{\otimes}_{\wideparen{\D}_X}\OX_X[-n]=\Omega_{X}^{\bullet}.  
 \end{equation*}
Hence, we recover $\Omega_X^{\bullet}$ as the de Rham complex of the structure sheaf $\OX_X$. 
\end{obs}
 We may condense our results into the following theorem:
\begin{teo}\label{teo explicit computation of hochschild cohomology}
Let $X$ be a separated and smooth rigid space. For any diagonal $\mathcal{C}$-complex 
$\mathcal{M}^{\bullet}\in \operatorname{D}^{\operatorname{Bi}}_{\mathcal{C}}(\wideparen{\D}_X)_{\Delta}$ we have the following identifications:
\begin{align*}
    \mathcal{HH}^{\bullet}(\wideparen{\D}_X,\mathcal{M}^{\bullet})=\operatorname{dR}\left(\Delta^!\Ifunct\mathcal{M}^{\bullet}\right)=&\operatorname{dR}\left(\Omega_X^{-1}\overrightarrow{\otimes}_{\OX_X}\mathcal{M}^{\bullet}[\mathcal{I}_{\Delta}]\right)\\=&\Omega_X\overrightarrow{\otimes}^{\mathbb{L}}_{\wideparen{\D}_X}\left(\Omega_X^{-1}\overrightarrow{\otimes}_{\OX_X}\mathcal{M}^{\bullet}[\mathcal{I}_{\Delta}]\right)[-\operatorname{dim}(X)].
\end{align*}
In particular, we have:
 \begin{enumerate}[label=(\roman*)]
        \item $\mathcal{HH}^{\bullet}(\wideparen{\D}_X,\wideparen{\D}_X)=\Omega_{X}^{\bullet}$ as objects in $\operatorname{D}(\operatorname{Shv}(X,\Indban))$.
        \item $\operatorname{HH}^{\bullet}(\wideparen{\D}_X,\wideparen{\D}_X)=R\Gamma(X,\Omega_{X}^{\bullet})$ as objects in $\operatorname{D}(\widehat{\mathcal{B}}c_K)$. 
    \end{enumerate}
And there is a Hodge-de Rham spectral sequence in $LH(\widehat{\mathcal{B}}c_K)$:
\begin{equation*}   E^{p,q}_1:=\operatorname{H}^q(X,\Omega_{X}^p)\Rightarrow \operatorname{HH}^{p+q}(\wideparen{\D}_X).
\end{equation*}
\end{teo}
\begin{proof}
The first part was shown in Corollary \ref{coro expression of Hochschild cohomology sheaves}. For the second part, we have:
\begin{equation*}
    \mathcal{HH}^{\bullet}(\wideparen{\D}_X,\wideparen{\D}_X)=\operatorname{dR}\left(\Delta^!\Delta_*^{\operatorname{S}}\wideparen{\D}_X\right)=\operatorname{dR}\left(\OX_X\right)=\Omega_{X}^{\bullet},
\end{equation*}
where the second identity is Proposition \ref{prop image of D under side changing}, and the third one is Remark \ref{remark de rham complex of a variety as the de rham complex of its structure sheaf}.
\end{proof}
Notice that the Hochschild cohomology spaces:
\begin{equation*}
    \operatorname{HH}^{n}(\wideparen{\D}_X):=\operatorname{H}^n(R\Gamma(X,\Omega_{X}^{\bullet})),
\end{equation*}
do not agree with the usual de Rham cohomology groups. Indeed, by forgetting the bornological structure, one can regard $\Omega_{X}^{\bullet}$ as an object of $\operatorname{D}(\operatorname{Shv}(X,\operatorname{Vect}_K))$. Then we have:
\begin{equation*}
    \operatorname{H}^{n}_{\operatorname{dR}}(X):=\mathbb{H}^n(X,\Omega_{X}^{\bullet}),
\end{equation*}
In particular, $\operatorname{H}^{n}_{\operatorname{dR}}(X)$ is a $K$-vector space. On the other hand, our definition of Hochschild cohomology spaces takes into account the canonical bornological structure of 
$\Omega_{X}^{\bullet}$. In particular, as the complex $\Omega_{X}^{\bullet}$ is not necessarily strict, the spaces $\operatorname{HH}^{n}(\wideparen{\D}_X)$ are elements in $LH(\widehat{\mathcal{B}}c_K)$, which need not be complete bornological spaces. However, we can use the material of \cite{prosmans2000homological} to recover the de Rham cohomology groups from the Hochschild cohomology groups.\\

Indeed, recall that we have an adjunction:
\begin{equation*}
    J:\widehat{\mathcal{B}}c_K\leftrightarrows \mathcal{B}c_K:\widehat{\operatorname{Cpl}},
\end{equation*}
where $J:\widehat{\mathcal{B}}c_K\rightarrow \mathcal{B}c_K$ denotes the inclusion, and $\widehat{\operatorname{Cpl}}:\mathcal{B}c_K\rightarrow \widehat{\mathcal{B}}c_K$ denotes the completion of bornological spaces. This adjunction satisfies the following properties:
\begin{prop}[{\cite[Proposition 5.14]{prosmans2000homological}}]\label{prop properties of J}
 The functor $J:\widehat{\mathcal{B}}c_K\rightarrow \mathcal{B}c_K$ is exact, and gives rise to a functor:
 \begin{equation*}
     J:\operatorname{D}(\widehat{\mathcal{B}}c_K)\rightarrow \operatorname{D}(\mathcal{B}c_K).
 \end{equation*}
 The functor $\widehat{\operatorname{Cpl}}:\mathcal{B}c_K\rightarrow \widehat{\mathcal{B}}c_K$ is cokernel preserving, and has a left derived functor:
 \begin{equation*}
     L\widehat{\operatorname{Cpl}}:\operatorname{D}(\mathcal{B}c_K)\rightarrow \operatorname{D}(\widehat{\mathcal{B}}c_K).
 \end{equation*}
 These derived functors fit into a derived adjunction:
 \begin{equation*}
     J:\operatorname{D}(\widehat{\mathcal{B}}c_K)\leftrightarrows \operatorname{D}(\mathcal{B}c_K):L\widehat{\operatorname{Cpl}}.
 \end{equation*}
 Furthermore, we have an identification of derived functors:
 \begin{equation*}
     L\widehat{\operatorname{Cpl}}\circ J=\operatorname{Id}.
 \end{equation*}
 In particular, $\operatorname{D}(\widehat{\mathcal{B}}c_K)$ is equivalent to a full triangulated subcategory of $\operatorname{D}(\mathcal{B}c_K)$, and $LH(\widehat{\mathcal{B}}c_K)$ is equivalent to a reflexive abelian subcategory of $LH(\mathcal{B}c_K)$.
\end{prop}
We will denote the restriction of $J:\operatorname{D}(\widehat{\mathcal{B}}c_K)\rightarrow \operatorname{D}(\mathcal{B}c_K)$ to $LH(\widehat{\mathcal{B}}c_K)$ by:
\begin{equation*}
    J:LH(\widehat{\mathcal{B}}c_K)\rightarrow LH(\mathcal{B}c_K), \quad \mathcal{V}\mapsto J(\mathcal{V})=\operatorname{H}^0(J(\mathcal{V})).
\end{equation*}
By the previous proposition, this is an exact functor. Furthermore, its extension to the derived categories is given by $J$ (\emph{cf}. \cite[Proposition 1.3.10]{schneiders1999quasi}).\\
On the other hand, we have the forgetful functor:
\begin{equation*}
    \operatorname{Forget}(-):\mathcal{B}c_K\rightarrow \operatorname{Vect}_K.
\end{equation*}
By construction of the quasi-abelian structure on bornological spaces, this functor is strongly exact. In particular, it preserves arbitrary kernels and cokernels. As $\operatorname{Vect}_K$ is an abelian category, it is canonically isomorphic to its left heart. Thus, we get an exact functor:
\begin{equation*}  \widetilde{\operatorname{Forget}}(-):LH(\mathcal{B}c_K)\rightarrow\operatorname{Vect}_K.
\end{equation*}
Both of these functors extend to functors between the corresponding derived categories, and these extensions get identified via the isomorphism:
\begin{equation*}
    \operatorname{D}(\mathcal{B}c_K) =\operatorname{D}(LH(\mathcal{B}c_K)).
\end{equation*}
Thus, in order to simplify notation, we will denote $\widetilde{\operatorname{Forget}}(-)$ simply by $\operatorname{Forget}(-)$.\\

In order to apply these functors to our setting, we need to extend them to the corresponding categories of sheaves. However, as $\mathcal{B}c_K$ and $\widehat{\mathcal{B}}c_K$ are not elementary quasi-abelian categories, their associated categories of sheaves are not well-behaved (\emph{i.e.} no sheafification functor). However, as $\mathcal{B}c_K$ and $\widehat{\mathcal{B}}c_K$ are quasi-elementary, their left hearts are Grothendieck abelian categories. Hence, their categories of sheaves are also Grothendieck, and thus are suited for homological algebra.
\begin{defi}\label{defi J and Forget for sheaves}
We define the following functor:
    \begin{equation*}
        J:\operatorname{Shv}(X,LH(\widehat{\mathcal{B}}c_K))\rightarrow \operatorname{Shv}(X,LH(\mathcal{B}c_K)), \quad \mathcal{F}\mapsto J(\mathcal{F}),
    \end{equation*}
    where $J(\mathcal{F})$ satisfies $J(\mathcal{F})(U)=J(\mathcal{F}(U))$ for every open $U\subset X$. Analogously:
    \begin{equation*}
       \operatorname{Forget}(-):\operatorname{Shv}(X,LH(\mathcal{B}c_K))\rightarrow \operatorname{Shv}(X,\operatorname{Vect}_K), \quad \mathcal{G}\mapsto \operatorname{Forget}(\mathcal{G}),
    \end{equation*}
where $\operatorname{Forget}(\mathcal{G})$ satisfies $\operatorname{Forget}(\mathcal{G})(U)=\operatorname{Forget}(\mathcal{G}(U))$ for every open $U\subset X$.
\end{defi}
Notice that as both of the following functors are exact:
\begin{equation*}
    J:LH(\widehat{\mathcal{B}}c_K)\rightarrow LH(\mathcal{B}c_K), \textnormal{ } \operatorname{Forget}(-):LH(\mathcal{B}c_K)\rightarrow\operatorname{Vect}_K.
\end{equation*}
Hence, the definitions above make sense. Furthermore, we have the following:
\begin{Lemma}\label{lemma J commutes with filtered colim}
The functor $J:LH(\widehat{\mathcal{B}}c_K)\rightarrow LH(\mathcal{B}c_K)$ commutes with filtered colimits.
\end{Lemma}
\begin{proof}
We start by giving a slightly different interpretation of the functor $J:LH(\widehat{\mathcal{B}}c_K)\rightarrow LH(\mathcal{B}c_K)$. Indeed, let $\operatorname{Sns}_K$ be the category of semi-normed spaces over $K$, and let $\operatorname{Ind}(\operatorname{Sns}_K)$ be its associated Ind-category. Then by \cite[Proposition 3.10]{prosmans2000homological}, $\operatorname{Ind}(\operatorname{Sns}_K)$ is an elementary quasi-abelian category. Furthermore, there is a canonical identification:
\begin{equation*}
    LH(\mathcal{B}c_K)=LH(\operatorname{Ind}(\operatorname{Sns}_K)).
\end{equation*}
In virtue of  \cite[Section 5]{prosmans2000homological}, we can use the fact that Banach space is a semi normed space, to obtain a canonical inclusion:
\begin{equation*}
    F:\Indban\rightarrow \operatorname{Ind}(\operatorname{Sns}_K).
\end{equation*}
Using the identification $LH(\Indban)=LH(\widehat{\mathcal{B}}c_K)$, we arrive at the following diagram:
\begin{equation*}
\begin{tikzcd}
LH(\widehat{\mathcal{B}}c_K) \arrow[r, "J"] & LH(\mathcal{B}c_K)                                      \\
\Indban \arrow[r, "F"] \arrow[u, "I"]       & \operatorname{Ind}(\operatorname{Sns}_K) \arrow[u, "I"]
\end{tikzcd}
\end{equation*}
where all functors except for maybe $J$ commute with filtered colimits. Let $\mathcal{I}$ be a small filtered category.  We let $\Indban^{\mathcal{I}}$
be the category of functors $\mathcal{I}\rightarrow\Indban$. As $\Indban$ is an elementary quasi-abelian category, it follows by  \cite[Proposition 1.4.15]{schneiders1999quasi}
that $\Indban^{\mathcal{I}}$ is a quasi-abelian category, and that there is a canonical isomorphism:
\begin{equation}\label{equation 1 J commutes with filtcolim}
    LH(\Indban^{\mathcal{I}})=LH(\widehat{\mathcal{B}}c_K)^{\mathcal{I}}.
\end{equation}
Similarly, $\operatorname{Ind}(\operatorname{Sns}_K)^{\mathcal{I}}$, is also a quasi-abelian category, and there is a canonical equivalence:
\begin{equation}\label{equation 2 J commutes with filtcolim}
    LH(\operatorname{Ind}(\operatorname{Sns}_K)^{\mathcal{I}})=LH(\mathcal{B}c_K)^{\mathcal{I}}.
\end{equation}
As both $\Indban$, and $\operatorname{Ind}(\operatorname{Sns}_K)$ are co-complete, we have the following additive functors:
\begin{equation*}
    \varinjlim_{\mathcal{I}}:\Indban^{\mathcal{I}}\rightarrow \Indban, \quad \varinjlim_{\mathcal{I}}:\operatorname{Ind}(\operatorname{Sns}_K)^{\mathcal{I}}\rightarrow \operatorname{Ind}(\operatorname{Sns}_K),
\end{equation*}
where the first one is given by sending a functor $\mathcal{I} \rightarrow \Indban$  to its colimit, and the second one is defined analogously. Notice that $F:\Indban\rightarrow \operatorname{Ind}(\operatorname{Sns}_K)$ commutes with filtered colimits if and only if the following diagram is commutative:
\begin{equation*}
\begin{tikzcd}
\Indban^{\mathcal{I}} \arrow[d, "\varinjlim_{\mathcal{I}}"] \arrow[r, "F_*"] & \operatorname{Ind}(\operatorname{Sns}_K)^{\mathcal{I}} \arrow[d, "\varinjlim_{\mathcal{I}}"] \\
\Indban \arrow[r, "F"]                                                       & \operatorname{Ind}(\operatorname{Sns}_K)                                                    
\end{tikzcd}
\end{equation*}
Passing onto the left hearts, and using the identities (\ref{equation 1 J commutes with filtcolim}), (\ref{equation 2 J commutes with filtcolim}), and the fact that:
\begin{equation*}
    I:\Indban \rightarrow LH(\widehat{\mathcal{B}}c_K), \quad I:\operatorname{Ind}(\operatorname{Sns}_K)\rightarrow LH(\mathcal{B}c_K),
\end{equation*}
commute with filtered colimits, we arrive at the following diagram:
\begin{equation*}
\begin{tikzcd}
LH(\widehat{\mathcal{B}}c_K)^{\mathcal{I}} \arrow[d, "\varinjlim_{\mathcal{I}}"] \arrow[r, "J_*"] & LH(\mathcal{B}c_K)^{\mathcal{I}} \arrow[d, "\varinjlim_{\mathcal{I}}"] \\
LH(\widehat{\mathcal{B}}c_K) \arrow[r, "J"]                                                       & LH(\mathcal{B}c_K)                                                    
\end{tikzcd}
\end{equation*}
We are done if we show that this diagram commutes. However, $LH(\widehat{\mathcal{B}}c_K)$
is a Grothendieck abelian category. In particular, filtered colimits are exact. Hence, the vertical maps in the diagram are exact. As mentioned above, $J:LH(\widehat{\mathcal{B}}c_K)\rightarrow LH(\mathcal{B}c_K)$ is also exact. Hence, $J_*$ is exact as well. Thus, it follows by \cite[Proposition 1.2.34]{schneiders1999quasi} that the diagram is commutative, as we wanted to show.
\end{proof}
\begin{Lemma}\label{lemma exactness of extension to sheaves}
The following hold:
\begin{enumerate}[label=(\roman*)]
    \item The functor $J:\operatorname{Shv}(X,LH(\widehat{\mathcal{B}}c_K))\rightarrow \operatorname{Shv}(X,LH(\mathcal{B}c_K))$ is exact.
    \item The functor $\operatorname{Forget}(-):\operatorname{Shv}(X,LH(\mathcal{B}c_K))\rightarrow \operatorname{Shv}(X,\operatorname{Vect}_K)$ is exact.
\end{enumerate}
\end{Lemma}
\begin{proof}
 As kernels of morphisms of sheaves can be calculated in the category of presheaves, it follows from left exactness of $J:LH(\widehat{\mathcal{B}}c_K)\rightarrow LH(\mathcal{B}c_K)$, and $\operatorname{Forget}(-):LH(\mathcal{B}c_K)\rightarrow\operatorname{Vect}_K$
 that both functors are left exact. In order to show that they are also right exact, we need to show that they commute with sheafification. As the functors are left exact, it suffices to show that they commute with filtered colimits. The first functor was dealt with in Lemma \ref{lemma J commutes with filtered colim}.\\
 For the second functor, notice that $\operatorname{Forget}(-)$ can be expressed as the following composition:
 \begin{equation*}
LH(\mathcal{B}c_K) \xrightarrow[]{C} \mathcal{B}c_K \xrightarrow[]{\operatorname{Forget}(-)}\operatorname{Vect}_K,
 \end{equation*}
where $C:LH(\mathcal{B}c_K)\rightarrow \mathcal{B}c_K$ is right adjoint to $I:\mathcal{B}c_K\rightarrow LH(\mathcal{B}c_K)$. As $C$ is a right adjoint, it commutes with filtered colimits. Furthermore, the definition of colimit
of bornological spaces shows that $\operatorname{Forget}(-):\mathcal{B}c_K\rightarrow\operatorname{Vect}_K$ also commutes with filtered colimits. Thus, the functor $\operatorname{Forget}(-):LH(\mathcal{B}c_K)\rightarrow\operatorname{Vect}_K$ commutes with filtered colimits, as we wanted to show.
\end{proof}
We can use these lemmas to show the following proposition:
\begin{prop}\label{prop comparison of ext functors}
We have the following identities of derived functors:
\begin{enumerate}[label=(\roman*)]
    \item $\operatorname{Forget}( R\Gamma(X,-))=R\Gamma(X,\operatorname{Forget}(-)).$
    \item $R\Gamma(X,J(-))=JR\Gamma(X,-).$
\end{enumerate}
\end{prop}
\begin{proof}
For the first identity, notice that as $\operatorname{Forget}(-)$ is exact, it suffices to show that it sends injective objects to $\Gamma(X,-)$-acyclic objects. In particular, it suffices to show that the image under 
$\operatorname{Forget}(-)$ of any injective object in $\operatorname{Shv}(X,LH(\mathcal{B}c_K))$ is \v{C}ech-acyclic. However, this follows by the fact that $\operatorname{Forget}(-)$ is exact, and every injective object in $\operatorname{Shv}(X,LH(\mathcal{B}c_K))$ is \v{C}ech-acyclic. The second identity is completely analogous.
\end{proof}
As a consequence, we obtain an interpretation of de Rham cohomology of $X$ in terms of the Hochschild cohomology groups of $\wideparen{\D}_X$:
\begin{coro}\label{coro recovering de Rham cohomology (without topology) from Hochschild cohomology}
We have the following identity in $\operatorname{D}(\operatorname{Vect}_K)$:
\begin{equation*}
    \operatorname{H}_{\operatorname{dR}}^{\bullet}(X)=\operatorname{Forget}(J\operatorname{HH}^{\bullet}(\wideparen{\D}_X)).
\end{equation*}
\end{coro}
\begin{proof}
This follows by Theorem \ref{teo explicit computation of hochschild cohomology}, and Proposition \ref{prop comparison of ext functors}, as we have:
\begin{align*}
    \operatorname{Forget}(J\operatorname{HH}^{\bullet}(\wideparen{\D}_X))=\operatorname{Forget}(JR\Gamma(X,\Omega^{\bullet}_{X/K}))=&R\Gamma(X,\operatorname{Forget}(J\Omega^{\bullet}_{X/K}))\\=&\operatorname{H}_{\operatorname{dR}}^{\bullet}(X).
\end{align*}
\end{proof}
\begin{obs}
Notice that we could use the previous proposition to endow the de Rham cohomology groups $\operatorname{H}_{dR}^{n}(X)$ with a bornology. However, this bornology is usually not interesting, and presents pathological behavior. However, in \cite{HochschildDmodules2} we will study a class of spaces for which this induced bornology makes the de Rham cohomology groups into nuclear Fréchet spaces.
\end{obs}
\subsection{Hochschild Homology}
Let $X$ be a separated smooth rigid analytic variety. In this section, we will use the machinery developed thus far to obtain an explicit description of the inner Hochschild homology complex  of diagonal $\mathcal{C}$-complexes on $X$.  We start with the following proposition:
\begin{prop}\label{prop derived complete tensor product and Kashiwara}
Let $\mathcal{M}\in \operatorname{D}_{\mathcal{C}}(\wideparen{\D}_X)$, and     $\mathcal{N}\in \operatorname{D}_{\mathcal{C}^{\Delta}}(\wideparen{\D}_{X^2})$. There is a canonical isomorphism in $\operatorname{D}(\operatorname{Shv}(X,\Indban))$:
\begin{equation*}
   \Delta^r_+\mathcal{M}\overrightarrow{\otimes}^{\mathbb{L}}_{\wideparen{\D}_{X^2}}\mathcal{N}=\Delta_*\left(\mathcal{M}\overrightarrow{\otimes}^{\mathbb{L}}_{\wideparen{\D}_{X}}\Delta^!\mathcal{N}\right)[\operatorname{dim}(X)].
\end{equation*}
\end{prop}
\begin{proof}
We have the following identities:
\begin{multline*}
\Delta^{-1}\left(\Delta^r_+\mathcal{M}\overrightarrow{\otimes}^{\mathbb{L}}_{\wideparen{\D}_{X^2}}\mathcal{N}\right)=\Delta^{-1}\left( \Delta_*\left(\mathcal{M}\overrightarrow{\otimes}^{\mathbb{L}}_{\wideparen{\D}_X}\wideparen{\D}_{X\rightarrow X^2}   \right)\overrightarrow{\otimes}^{\mathbb{L}}_{\wideparen{\D}_{X^2}}\mathcal{N}\right)\\= \left(\mathcal{M}\overrightarrow{\otimes}^{\mathbb{L}}_{\wideparen{\D}_X}\wideparen{\D}_{X\rightarrow X^2}   \right)\overrightarrow{\otimes}^{\mathbb{L}}_{\Delta^{-1}\wideparen{\D}_{X^2}}\Delta^{-1}\mathcal{N},    
\end{multline*}
where the second identity follows by Proposition \ref{prop derived tensor products and closed immersions}. Hence, we have:
\begin{multline*}
\Delta^{-1}\left(\Delta^r_+\mathcal{M}\overrightarrow{\otimes}^{\mathbb{L}}_{\wideparen{\D}_{X^2}}\mathcal{N}\right)= \mathcal{M}\overrightarrow{\otimes}^{\mathbb{L}}_{\wideparen{\D}_X}\left(\wideparen{\D}_{X\rightarrow X^2}   \overrightarrow{\otimes}^{\mathbb{L}}_{\Delta^{-1}\wideparen{\D}_{X^2}}\Delta^{-1}\mathcal{N}\right)\\ = \mathcal{M}\overrightarrow{\otimes}^{\mathbb{L}}_{\wideparen{\D}_{X}}\Delta^!\mathcal{N}[\operatorname{dim}(X)].
\end{multline*}
As $\Delta^r_+\mathcal{M}\overrightarrow{\otimes}^{\mathbb{L}}_{\wideparen{\D}_{X^2}}\mathcal{N}$ has support contained in the diagonal, we are done.
\end{proof}
This allows us to calculate the Hochschild homology of diagonal $\mathcal{C}$-complexes:
\begin{teo}\label{teo expression of Hochschild homology sheaves}
Let $X$ be a  smooth and separated rigid space. For any diagonal $\mathcal{C}$-complex 
$\mathcal{M}^{\bullet}\in \operatorname{D}^{\operatorname{Bi}}_{\mathcal{C}}(\wideparen{\D}_X)_{\Delta}$ we have:
\begin{equation*}
    \mathcal{HH}_{\bullet}(\wideparen{\D}_X,\mathcal{M}^{\bullet})=\wideparen{\D}_X\overrightarrow{\otimes}^{\mathbb{L}}_{\wideparen{\D}_X^e}\mathcal{M}^{\bullet}=\Omega_X\overrightarrow{\otimes}^{\mathbb{L}}_{\wideparen{\D}_X}\left(\Omega_X\overrightarrow{\otimes}_{\OX_X}\mathcal{M}^{\bullet}[\mathcal{I}_{\Delta}]\right)[\operatorname{dim}(X)].
\end{equation*}
In particular the following identities hold:
\begin{enumerate}[label=(\roman*)]
\item $\mathcal{HH}_{\bullet}(\wideparen{\D}_X,\mathcal{M}^{\bullet})=\mathcal{HH}^{\bullet}(\wideparen{\D}_X,\mathcal{M}^{\bullet})[2\operatorname{dim}(X)]$.
\item $\operatorname{HH}_{\bullet}(\wideparen{\D}_X,\mathcal{M}^{\bullet})=\operatorname{HH}^{\bullet}(\wideparen{\D}_X,\mathcal{M}^{\bullet})[2\operatorname{dim}(X)]$.
\item  $\operatorname{HH}_{n}(\wideparen{\D}_X,\mathcal{M}^{\bullet})=\operatorname{HH}^{2\operatorname{dim}(X)-n}(\wideparen{\D}_X,\mathcal{M}^{\bullet})$ for all $n\in\mathbb{Z}$.
\end{enumerate}
Furthermore, for $\mathcal{M}=\wideparen{\D}_X$, we have:
\begin{equation*}
    \mathcal{HH}_{\bullet}(\wideparen{\D}_X,\wideparen{\D}_X)=\Omega_{X}\overrightarrow{\otimes}^{\mathbb{L}}_{\wideparen{\D}_X}\OX_X[\operatorname{dim}(X)]=\Omega_{X}^{\bullet}[2\operatorname{dim}(X)].
\end{equation*}
\end{teo}
\begin{proof}
The first identity is Corollary \ref{coro Hh as derived tensor product}, the second one is a consequence of Corollaries \ref{coro hochschild cohomology and the immersion functor}, \ref{coro image of DX under right immersion functor}, and Proposition \ref{prop derived complete tensor product and Kashiwara}. The rest follows by the description of the inner Hochschild cohomology complex from Theorem \ref{teo explicit computation of hochschild cohomology}.
\end{proof}
Thus, we obtain a duality between Hochschild cohomology and homology for $\wideparen{\D}$-modules on separated smooth rigid analytic spaces. This can be regarded as a $p$-adic version of the classical duality between Hochschild cohomology and homology of  $\D$-modules on smooth algebraic varieties and smooth complex analytic spaces.\\  
Furthermore, for each $n\geq 0$ we have the following identity:
\begin{equation*}
 \operatorname{HH}_n(\wideparen{\D}_X)=\operatorname{H}^{2\operatorname{dim}(X)-n}\left(R\Gamma(X,\Omega_{X}^{\bullet})\right).
\end{equation*}
This was obtained for affine algebraic varieties over a  characteristic zero field, Stein complex manifolds, and compact $C^{\infty}$-manifolds by Wodzicki in $\textnormal{\cite[Theorem 2]{wodzicki1987cyclic}}$. The  result for affine algebraic varieties was also  shown by Kassel-Mitschi (see \cite[pp. 19]{pseudodiff}),  and for smooth real manifolds by Brylinski-Getzler in \cite[pp. 18]{pseudodiff}. As a consequence of Theorem \ref{teo expression of Hochschild homology sheaves}, we recover the following classical result:
\begin{coro}
 Let $X$ be a separated and smooth rigid analytic space. We have:
 \begin{equation*}
     \operatorname{HH}^{n}(\wideparen{\D}_X)=0 \textnormal{ for all }n> 2\operatorname{dim}(X).
 \end{equation*}
 In particular, $\operatorname{H}^n_{\operatorname{dR}}(X)=0$ for all $n> 2\operatorname{dim}(X)$.   
\end{coro}
\begin{proof}
Choose some $n> 2\operatorname{dim}(X)$. By Theorem \ref{teo expression of Hochschild homology sheaves} we have:
\begin{equation*}
   \operatorname{HH}^{n}(\wideparen{\D}_X)=\operatorname{HH}_{n-2\operatorname{dim}(X)}(\wideparen{\D}_X)= \operatorname{H}^{n-2\operatorname{dim}(X)}\left(\wideparen{\D}_X\overrightarrow{\otimes}^{\mathbb{L}}_{\wideparen{\D}_X^e}\wideparen{\D}_X \right).
\end{equation*}
 As $\wideparen{\D}_X\overrightarrow{\otimes}^{\mathbb{L}}_{\wideparen{\D}_X^e}\wideparen{\D}_X$ is trivial on positive degrees, this shows $\operatorname{HH}^{n}(\wideparen{\D}_X)=0$.
\end{proof}
\appendix
\section{}
\subsection{Examples of co-admissible diagonal bimodules}\label{appendix examples of co-admissible diagonal bimodules}
The goal of this appendix is producing some examples of objects in $\mathcal{C}_B^{\operatorname{Bi}}(\wideparen{\D}_X)_{\Delta}$. We start by recalling the following well-known theorem:
\begin{teo}\label{teo characterization of integrable connections}
Let $\mathcal{V}$ be a co-admissible $\wideparen{\D}_X$-module which is also a coherent $\OX_X$-module, 
 then $\mathcal{V}$ is a vector bundle.    
\end{teo}
\begin{proof}
It suffices to show the case where $X=\Sp(A)$ is affinoid, with an étale map $X\rightarrow \mathbb{A}^n_k$. As $\mathcal{V}$ is a coherent module, for $x\in X$ we have $\mathcal{V}_x=\mathcal{V}(X)\otimes_{\OX_X(X)}\OX_{X,x}$. Then, the proof of \cite[1.4.10]{hotta2007d} shows that $\mathcal{V}_x$ is a finite-free module over $\OX_{X,x}$. Thus, $\widehat{\mathcal{V}}_{x}=\mathcal{V}(X)\otimes_{\OX_X(X)}\widehat{\OX}_{X,x}$ is a finite-free module over the maximal adic completion $\widehat{\OX}_{X,x}$. Let $\mathfrak{m}_x\subset A$ be the maximal ideal associated to the point $x\in X$. Then, by \cite[4.1.2]{bosch2014lectures}, we have that $\widehat{A}_{\mathfrak{m}_x}=\widehat{\OX}_{X,x}$. As $A$ is noetherian, the map $A_{\mathfrak{m}_x}\rightarrow\widehat{A}_{\mathfrak{m}_x}$ is faithfully flat. In particular, $\widehat{\mathcal{V}}_{x}$ is finite-free if and only if $\mathcal{V}(X)\otimes_{A}A_{\mathfrak{m}_x}$ is finite-free. Applying the discussion above, it follows that $\mathcal{V}(X)\otimes_{A}A_{\mathfrak{m}_x}$ is finite free for every maximal ideal of $A$. Thus, $\mathcal{V}(X)$ is finite and locally free. Equivalently, $\mathcal{V}$ is a vector bundle.
\end{proof}
\begin{defi}\label{defi integrable connections}
We define $\operatorname{Conn}(X)$ as the full subcategory of $\mathcal{C}(\wideparen{\D}_X)$ given by the $\OX_X$-coherent modules. We call elements $\mathcal{V}\in \operatorname{Conn}(X)$ integrable connections.
\end{defi}
By \cite[Theorem 7.3]{ardakov2015d} there is an equivalence of categories between the categories of $\OX_X$-coherent co-admissible $\wideparen{\D}_X$-modules, and the category of $\OX_X$-coherent $\D_X$-modules. In particular, and integrable connection is equivalent to a vector bundle $\mathcal{V}$, equipped with a bounded $K$-linear map: 
\begin{equation*}
    \nabla:\mathcal{V}\rightarrow \mathcal{V}\otimes_{\OX_X}\Omega^1_{X/K},
\end{equation*}
called an integrable (flat) connection on $\mathcal{V}$. This map needs to satisfy the following conditions: 
\begin{enumerate}[label=(\roman*)]
\item  $\nabla(fv)=v\otimes d(f)+f\nabla(v)$,  for an admissible open $U\subset X$, and sections $f\in \OX_X(U),v\in\mathcal{V}(U)$.
    \item The bounded morphism: $\nabla^2:\mathcal{V}\rightarrow \mathcal{V}\otimes_{\OX_X}\Omega^2_{X/K}$, defined by:
    \begin{equation*}
    \nabla^2:=(\nabla\otimes 1)\circ\nabla + \nabla\otimes d,    
    \end{equation*}
    satisfies $\nabla^2=0$. This is called the integrability condition.
    
\end{enumerate}
Notice that Theorem \ref{teo characterization of integrable connections} implies that $\operatorname{Conn}(X)$ is an abelian category. We will now show that every integrable connection on $X$ induces a co-admissible diagonal bimodule in a natural way. In order to do so, we will need an explicit description of the pushforward along $\Delta:X\rightarrow X^2$ in the local case. We borrow this description from \cite[Example 1.3.5]{hotta2007d}.\\
Let $X=\Sp(A)$ be an affinoid variety with an étale map $X\rightarrow \mathbb{A}^n_K$ induced by sections $f_1,\cdots,f_n$. These sections yield a basis of $\Omega_{X}^1(X)$, and we let
$\partial_1,\cdots,\partial_n\subset \mathcal{T}_{X/K}(X)$ be its dual basis. Define the following vector fields in $\mathcal{T}_{X^2/K}(X^2)$:
\begin{equation*}
  \alpha_i=\partial_i\otimes 1 -1\otimes\partial_i, \textnormal{ }  \beta_i=\partial_i\otimes1+1\otimes\partial_i, \textnormal{ for }1\leq i \leq n.
\end{equation*}
Then $ \mathcal{B}:=\{ \alpha_1,\cdots,\alpha_n,\beta_1,\cdots,\beta_n \}$ is a $\OX_{X^2}(X^2)$-basis of $\mathcal{T}_{X^2/K}(X^2)$.
Furthermore, this is the dual basis of the differentials of
the following family of rigid functions:
\begin{equation*}
    \{ f_1\otimes 1-1\otimes f_1,\cdots,f_n\otimes 1-1\otimes f_n,f_1\otimes 1+1\otimes f_1, \cdots,f_n\otimes 1 + 1\otimes f_n\}\subset \OX_{X^2}(X^2).
\end{equation*}

Recall the smooth rigid variety $Y$ from Lemma \ref{Lemma support on the diagonal}, and consider the canonical closed immersion $i:Y\rightarrow X^2$. By definition, $Y$ is the zero locus of the sections $f_1\otimes 1-1\otimes f_1,\cdots, f_n\otimes 1 - 1\otimes f_n$. Hence, as shown in  \cite[Lemma 7.9]{bode2021operations} we have an isomorphism of right $\wideparen{\D}_Y$-modules:
\begin{equation}\label{equation decomposition transfer bimodule}
    \wideparen{\D}_{X^2\leftarrow Y}\cong K\{\alpha_1,\cdots,\alpha_n \}\overrightarrow{\otimes}_K\wideparen{\D}_{Y}.
\end{equation}
We can use this isomorphism to give an explicit description of the left action of $i^{-1}\wideparen{\D}_{X^2}$ on $\wideparen{\D}_{X^2\leftarrow Y}$. Indeed, we have a decomposition of sheaves of Ind-Banach $K$-vector spaces:
\begin{equation}\label{equation decomposition of DX2}
\wideparen{\D}_{X^2}\cong K\{\alpha_1,\cdots,\alpha_n,\beta_1,\cdots,\beta_n \}\overrightarrow{\otimes}_K\OX_{X^2}.   
\end{equation}
Regard $K\{\alpha_1,\cdots,\alpha_n,\beta_1,\cdots,\beta_n \}\overrightarrow{\otimes}_K\OX_{X^2}$ as a sheaf of Ind-Banach $K$-algebras with respect to the unique structure making the previous identification an algebra isomorphism.\bigskip

Identify $K\{\alpha_1,\cdots,\alpha_n\}$ with the sheaf
defined on any affinoid subdomain $U\subset X^2$ by:
\begin{equation*}
    K\{\alpha_1,\cdots,\alpha_n\}(U)=K\{\alpha_1,\cdots,\alpha_n\}\overrightarrow{\otimes}_K1\subset \wideparen{\D}_{X^2}(U).
\end{equation*}
By construction of $\mathcal{B}$, we have $[\alpha_i,\alpha_j]=0$ for all $1\leq i,j\leq n$. Thus, $K\{\alpha_1,\cdots,\alpha_n\}$, is a sheaf of subalgebras of $\wideparen{\D}_{X^2}$. Furthermore, 
$K\{\alpha_1,\cdots,\alpha_n\}$ is a sheaf of commutative Ind-Banach $K$-algebras. Analogously, we may regard  $K\{\beta_1,\cdots,\beta_n \}$ as a sheaf of commutative subalgebras of $\wideparen{\D}_{X^2}$.\bigskip

Define $\wideparen{\D}':=K\{\beta_1,\cdots,\beta_n \}\overrightarrow{\otimes}_K\OX_{X^2}$. Notice that this is a subsheaf of $\wideparen{\D}_{X^2}$ which contains the unit  of the sheaf of algebras $\wideparen{\D}_{X^2}$. Furthermore, $\wideparen{\D}'$ is closed under the product in $\wideparen{\D}_{X^2}$. Hence, $\wideparen{\D}'$ is canonically a sheaf of Ind-Banach $K$-algebras, and the inclusion $\wideparen{\D}'\rightarrow \wideparen{\D}_{X^2}$ is a morphism of sheaves of Ind-Banach $K$-algebras.\bigskip

We may rewrite $(\ref{equation decomposition of DX2})$  as 
$\wideparen{\D}_{X^2}\cong K\{\alpha_1,\cdots,\alpha_n\}\overrightarrow{\otimes}_K\wideparen{\D}'$, and this is an identity of right $\wideparen{\D}'$-modules. Therefore,  we have an isomorphism of right $i^{-1}\wideparen{\D}'$-modules:
\begin{multline}\label{equation decomposition pullback od differentials on product}
    i^{-1}\wideparen{\D}_{X^2}=i^{-1}\left( K\{\alpha_1,\cdots,\alpha_n,\beta_1,\cdots,\beta_n \}\overrightarrow{\otimes}_K\OX_{X^2}\right)\\
    =K\{\alpha_1,\cdots,\alpha_n\}\overrightarrow{\otimes}_Ki^{-1}\wideparen{\D}'.
\end{multline}
Hence, in order to define the action of $i^{-1}\wideparen{\D}_{X^2}$ on $\wideparen{\D}_{X^2\leftarrow Y}$, it suffices to give actions of 
$K\{\alpha_1,\cdots,\alpha_n\}$ and $i^{-1}\wideparen{\D}$. The action of $K\{\alpha_1,\cdots,\alpha_n\}$ on $\wideparen{\D}_{X^2\leftarrow Y}$ is given by left multiplication using $(\ref{equation decomposition transfer bimodule})$. For the action of $i^{-1}\wideparen{\D}'$, notice that we have:
\begin{equation*}  i^{-1}\wideparen{\D}'\overrightarrow{\otimes}_{i^{-1}\OX_{X^2}}\OX_Y=K\{\beta_1,\cdots,\beta_n \}\overrightarrow{\otimes}_K\OX_Y\cong \wideparen{\D}_Y.
\end{equation*}
Hence, we have a canonical action of $i^{-1}\wideparen{\D}'$ on $\wideparen{\D}_Y$ given by left multiplication. The action of $i^{-1}\wideparen{\D}'$ on $\wideparen{\D}_{X^2\leftarrow Y}$ is defined via the following composition:
\begin{multline*}
 i^{-1}\wideparen{\D}'\overrightarrow{\otimes}_K\wideparen{\D}_{X^2\leftarrow Y}\cong i^{-1}\wideparen{\D}'\overrightarrow{\otimes}_K K\{\alpha_1,\cdots,\alpha_n \}\overrightarrow{\otimes}_K\wideparen{\D}_{Y} \rightarrow i^{-1}\wideparen{\D}_{X^2}\overrightarrow{\otimes}_K\wideparen{\D}_{Y} \\ 
 \cong K\{\alpha_1,\cdots,\alpha_n \}\overrightarrow{\otimes}_K i^{-1}\wideparen{\D}'\overrightarrow{\otimes}_K\wideparen{\D}_{Y}\rightarrow K\{\alpha_1,\cdots,\alpha_n \}\overrightarrow{\otimes}_K\wideparen{\D}_{Y},  
\end{multline*}
where the first map is $(\ref{equation decomposition transfer bimodule})$, the second map is induced by the multiplication in $i^{-1}\wideparen{\D}_{X^2}$, the third is $(\ref{equation decomposition pullback od differentials on product})$, and the last is induced by the action of $i^{-1}\wideparen{\D}'$ on $\wideparen{\D}_{Y}$.\\ 
By Proposition \ref{prop enveloping of product as product of envelopings}, we have a decomposition $\wideparen{\D}_{X^2}=p_1^{-1}\wideparen{\D}_{X}\overrightarrow{\otimes}_Kp_2^{-1}\wideparen{\D}_{X}$,
as sheaves of Ind-Banach $K$-vector spaces. Under the isomorphism $(\ref{equation decomposition of DX2})$, we get isomorphisms:
\begin{align}\label{equation decomposition of the two DX factors in DX2}
    p_1^{-1}\wideparen{\D}_{X}\cong K\{\partial_1\otimes 1,\cdots,\partial_n\otimes 1\}\overrightarrow{\otimes}_Kp_1^{-1}\OX_{X},
\end{align}
\begin{equation}  
    p_2^{-1}\wideparen{\D}_{X}\cong K\{1\otimes\partial_1,\cdots,1\otimes\partial_n\}\overrightarrow{\otimes}_Kp_2^{-1}\OX_{X}.
\end{equation}
As $X$ is affinoid we have $\wideparen{\D}_{X^2}(X^2)\cong K\{\alpha_1,\cdots,\alpha_n,\beta_1,\cdots,\beta_n \}\overrightarrow{\otimes}_K \OX_{X^2}(X^2)$. Similarly, we get two different identifications of subalgebras of $\wideparen{\D}_{X^2}(X^2)$:
\begin{multline*}
  \wideparen{\D}_{X}(X)\cong \wideparen{\D}_1:= K\{\partial_1\otimes 1,\cdots,\partial_n\otimes 1\}\overrightarrow{\otimes}_K(A\overrightarrow{\otimes}_K1),\\ \wideparen{\D}_{X}(X)\cong \wideparen{\D}_2:=K\{1\otimes\partial_1,\cdots,1\otimes\partial_n\}\overrightarrow{\otimes}_K(1\otimes A),  
\end{multline*}
 We will now describe the actions of $\wideparen{\D}_1$ and $\wideparen{\D}_2$ on $\wideparen{\D}_{X^2\leftarrow Y}(X)$. Notice that for any $1\leq i\leq n$ we have the following identities in $\mathcal{T}_{X^2/K}(X^2)$: 
 \begin{equation*}
    \partial_i\otimes 1=\frac{\alpha_i+\beta_i}{2}, \textnormal{ } 1\otimes\partial_i=\frac{\beta_i-\alpha_i}{2}. 
 \end{equation*}
 Hence, for $f\in K\{\alpha_1,\cdots,\alpha_n \}$, and $d\in \wideparen{\D}_{Y}(Y)$
we have:
\begin{multline}\label{equation explicit action of operators on the pushforward}
\partial_i\otimes 1\left(f\otimes d \right)=\frac{1}{2}\left(\alpha_if\otimes d + f\otimes \beta_id  \right),\\  1\otimes \partial_i\left(f\otimes d \right)=\frac{1}{2}\left(f\otimes \beta_id-\alpha_if\otimes d\right), 
\end{multline}
where we have used the fact that $K\{\alpha_1,\cdots,\alpha_n,\beta_1,\cdots,\beta_n \}$ is a commutative subalgebra of $\wideparen{\D}_{X^2}(X^2)$.
\begin{Lemma}\label{Lemma integrable connections to diagonal bimodules}
Let $X=\Sp(A)$ be a smooth affinoid space with an étale map $X\rightarrow \mathbb{A}^n_K$, and let $\mathcal{V}$ be an integrable connection on $X$. We have the following identification:
\begin{equation*}
    \Delta_+\mathcal{V}(X^2)=K\{\alpha_1,\cdots,\alpha_n \}\overrightarrow{\otimes}_K\mathcal{V}(X).
\end{equation*}
\end{Lemma}
\begin{proof}
 Consider the closed immersion $j:X\rightarrow Y$. Then we have $\Delta= i\circ j$ and, by Lemma \ref{Lemma immersion Theorem}, we have
$\Delta_+\mathcal{V}=i_+ j_+\mathcal{V}=i_+j_*\mathcal{V}$.
Thus, the discussion above shows that $\Delta_+\mathcal{V}$ is the co-admissible  $\wideparen{\D}_{X^2}$-module with global sections: 
\begin{equation*}
    \Delta_+\mathcal{V}(X^2)=i_*\left(\wideparen{\D}_{X^2\leftarrow Y}\overrightarrow{\otimes}_{\wideparen{\D}_Y}j_*\mathcal{V} \right)(X^2) =K\{\alpha_1,\cdots,\alpha_n \}\widehat{\otimes}_K\mathcal{V}(X),
\end{equation*}
as we wanted to show.
\end{proof}
\begin{Lemma}\label{Lemma 2 integrable connections to diagonal bimodules}
Let $X=\Sp(A)$ be a smooth affinoid space as in Lemma \ref{Lemma integrable connections to diagonal bimodules}. Assume there is an affine formal model $\mathcal{A}\subset A$ such that $\partial_i(\mathcal{A})\subset \mathcal{A}$ for all $1\leq i\leq n$. Then $\mathcal{T}_{\mathcal{A}}=\bigoplus_{i=1}^n \mathcal{A}\partial_i$ is a finite-free $\mathcal{A}$-Lie lattice of $\mathcal{T}_{X/K}(X)$, and we have:
\begin{enumerate}[label=(\roman*)]
    \item $\mathcal{A}\widehat{\otimes}_{\mathcal{R}}\mathcal{A}\subset A\overrightarrow{\otimes}_KA$ is an affine formal model.
    \item $\mathcal{T}=\bigoplus_{i=1}^n\mathcal{A}\widehat{\otimes}_{\mathcal{R}}\mathcal{A}\alpha_i\oplus \bigoplus_{i=1}^n\mathcal{A}\widehat{\otimes}_{\mathcal{R}}\mathcal{A}\beta_i$ is a finite free $\mathcal{A}\widehat{\otimes}_{\mathcal{R}}\mathcal{A}$-Lie lattice of $\mathcal{T}_{X^2/K}(X^2)$.
    \item $\widehat{U}(\pi^m\mathcal{T})_K\otimes_{\wideparen{\D}_{X^2}(X^2)}\Delta_+\mathcal{V}(X^2)=K\langle \pi^m\alpha_1,\cdots,\pi^m\alpha_n\rangle\overrightarrow{\otimes}_K\mathcal{V}(X)$.  
\end{enumerate}
\end{Lemma}
\begin{proof}
We only show statement $(iii)$. As the statement does not need sheaf theory, we will work with the underlying bornological spaces. Let $B=A\widehat{\otimes}_KA$. Notice that, by statement $(ii)$, we have a Fréchet-Stein presentation:
\begin{equation*}
    \wideparen{\D}_{X^2}(X^2)=\wideparen{U}(\mathcal{T}_{X^2/K}(X^2))=\varprojlim_m\widehat{U}(\pi^m\mathcal{T})_K.
\end{equation*}
Similarly, we have $\wideparen{\D}_{X}(X)=\varprojlim_m\widehat{U}(\pi^m\mathcal{T}_{\mathcal{A}})_K$. As $\mathcal{V}$ is an integrable connection, we may apply \cite[Section 7.2]{ardakov2015d} to show that for every $m\geq 0$ we have the following identities of $\wideparen{\D}_{X}(X)$-modules:
\begin{equation*}
\mathcal{V}(X)= \wideparen{\D}_{X}(X)\widehat{\otimes}_{\D_{X}(X)}\mathcal{V}(X)= \widehat{U}(\pi^m\mathcal{T}_{\mathcal{A}})_K\widehat{\otimes}_{\D_{X}(X)}\mathcal{V}(X). 
\end{equation*}
In particular, we have a presentation of $\mathcal{V}(X)$ as a co-admissible $\wideparen{\D}_{X}(X)$-module given by the trivial inverse limit $\mathcal{V}(X)=\varprojlim_m \mathcal{V}(X)$. The version for left modules  of \cite[5.4]{ardakov2015d}, shows that we have the following identities of $\wideparen{\D}_{X^2}(X^2)$-modules:
\begin{align*}
 &\wideparen{\D}_{X^2\leftarrow Y}(Y)\widehat{\otimes}_{\wideparen{\D}_Y(Y)}\mathcal{V}(X)=\wideparen{\D}_{X^2\leftarrow Y}(Y)\wideparen{\otimes}_{\wideparen{\D}_Y(Y)}\mathcal{V}(X)\\
 &=\varprojlim_m\left(\left(K\langle\pi^m\alpha_1,\cdots,\pi^m\alpha_n,\pi^m\beta_1,\cdots,\pi^m\beta_n\rangle\widehat{\otimes}_KB\right)\widehat{\otimes}_{\wideparen{\D}_Y(Y)}\mathcal{V}(X) \right)\\&=
\varprojlim_m\left(\left(K\langle\pi^m\alpha_1,\cdots,\pi^m\alpha_n,\pi^m\beta_1,\cdots,\pi^m\beta_n\rangle\widehat{\otimes}_KB\right)\widehat{\otimes}_{\widehat{U}(\pi^m\mathcal{T}_{\mathcal{A}})_K}\mathcal{V}(X)\right)\\&=
 \varprojlim_m\left(B\langle\pi^m\alpha_1,\cdots,\pi^m\alpha_n,\pi^m\beta_1,\cdots,\pi^m\beta_n\rangle\widehat{\otimes}_{B\langle\pi^m\beta_1,\cdots,\pi^m\beta_n\rangle}\mathcal{V}(X) \right)\\&=
\varprojlim_m\left(K\langle\pi^m\alpha_1,\cdots,\pi^m\alpha_n\rangle\widehat{\otimes}_K\mathcal{V}(X) \right).
\end{align*}
Hence, we have the following Fréchet-Stein presentation of $\Delta_+\mathcal{V}(X^2)$:
\begin{equation*}
    K\{\alpha_1,\cdots,\alpha_n\}\widehat{\otimes}_K\mathcal{V}(X)= \varprojlim_m\left(K\langle\pi^m\alpha_1,\cdots,\pi^m\alpha_n\rangle\widehat{\otimes}_K\mathcal{V}(X) \right),
\end{equation*}
with respect to $\wideparen{\D}_{X^2}(X^2)=\varprojlim_m\widehat{U}(\pi^m\mathcal{T})_K$. In particular, we have:
\begin{multline*}
\widehat{U}(\pi^m\mathcal{T})_K\otimes_{\wideparen{\D}_{X^2}(X^2)}\varprojlim_m\left(K\langle\pi^m\alpha_1,\cdots,\pi^m\alpha_n\rangle\widehat{\otimes}_K\mathcal{V}(X) \right)\\=K\langle \pi^m\alpha_1,\cdots,\pi^m\alpha_n\rangle\overrightarrow{\otimes}_K\mathcal{V}(X),    
\end{multline*}
as we wanted to show.
\end{proof}
\begin{prop}
There is an exact embedding of abelian categories:
\begin{equation*}
  \operatorname{Conn}(X)\rightarrow \mathcal{C}^{\Bi}_B(\wideparen{\D}_X)_{\Delta},  
\end{equation*}
given by the following expression:
\begin{equation}\label{equation integrable connections to diagonal bimodules}
    \mathcal{V}\mapsto \wideparen{\D}_{\mathcal{V}}=\Delta^{-1}\left(\operatorname{S}\left(\Delta_+\mathcal{V} \right)\right)=\Omega_X\overrightarrow{\otimes}_{\OX_X}\left(\wideparen{\D}_{X^2\leftarrow X}\overrightarrow{\otimes}_{\wideparen{\D}_{X}}\mathcal{V}\right).
\end{equation}
In particular, we have $\wideparen{\D}_{\OX_X}=\wideparen{\D}_X$, and $\Delta^!\Delta^{\operatorname{S}}(\wideparen{\D}_{\mathcal{V}})=\mathcal{V}$.
\end{prop}
\begin{proof}
Let $\mathcal{V}$ be an integrable connection on $X$. By construction, $\wideparen{\D}_{\mathcal{V}}$ is an Ind-Banach $\wideparen{\D}_X^e$-module. Furthermore, as $\Delta_+\mathcal{V}$ has support contained in the diagonal, by Theorem \ref{teo equivalence of co-admissible modules under side-changing} and Proposition \ref{prop co-admissible EX modules supported on the diagonal},
we have $\operatorname{S}\Delta_+\mathcal{V}(X)_{\infty}(\mathcal{I}_{\Delta})=\operatorname{S}\Delta_+\mathcal{V}(X)$. Thus, we have:
\begin{equation*}
    \wideparen{\D}_{\mathcal{V}}(X)_{\infty}(\mathcal{I}_{\Delta})=\wideparen{\D}_{\mathcal{V}}(X).
\end{equation*}
Hence, we just need to show that $\wideparen{\D}_{\mathcal{V}}$ is co-admissible as a $\wideparen{\D}_X$ and $\wideparen{\D}^{\op}_X$ module. As this is a purely local property, we may assume that $X=\Sp(A)$ is affinoid and satisfies the conditions of Lemma \ref{Lemma 2 integrable connections to diagonal bimodules}. We are also free to assume that $\mathcal{V}\cong \bigoplus^r_{i=1} \OX_X$ as a coherent $\OX_X$-module. For simplicity, we will assume $r=1$. The general case is analogous. Also, as we do not need to work with sheaves, we will work with the underlying complete bornological spaces.\\
Let $\OX_{\mathcal{V}}$ denote $\OX_X$ with the $\wideparen{\D}_X$-module structure induced by the isomorphism $\mathcal{V}\cong \OX_X$, and  let $A_{\mathcal{V}}=\OX_{\mathcal{V}}(X)$.
By Lemma \ref{Lemma integrable connections to diagonal bimodules}, we have:
\begin{equation*}
    \Delta_+\OX_{\mathcal{V}}(X^2)=K\{\alpha_1,\cdots,\alpha_n \}\widehat{\otimes}_KA_{\mathcal{V}}=: \widehat{B}_{\mathcal{V}}.
\end{equation*}
The element $1\otimes 1\in \widehat{B}_{\mathcal{V}}$ induces a canonical $\wideparen{\D}_{X^2}(X^2)$-linear bounded map:
\begin{equation*}
    \psi:\wideparen{\D}_{X^2}(X^2)\rightarrow \wideparen{B}_{\mathcal{V}}.
\end{equation*}
Recall the discussion from the beginning of the appendix, and in particular the definition  of the algebras $\wideparen{\D}_1$  and $\wideparen{\D}_2$. By restriction, we get two bounded maps:
\begin{equation*}
 \varphi_1:\wideparen{\D}_1\rightarrow \wideparen{B}_{\mathcal{V}}, \quad  \varphi_2:\wideparen{\D}_2\rightarrow \wideparen{B}_{\mathcal{V}}, 
\end{equation*}
which endow $\wideparen{B}_{\mathcal{V}}$ with two different bornological $\wideparen{\D}_{X}(X)$-module structures. We need to show that $\widehat{B}_{\mathcal{V}}$ is co-admissible with respect to both structures. We do this only for the first one, as the second is analogous. By Lemma \ref{Lemma 2 integrable connections to diagonal bimodules}, for each $m\geq 0$, the map $\psi:\wideparen{\D}_{X^2}(X^2)\rightarrow \wideparen{B}_{\mathcal{V}}$ induces a map: 
\begin{equation*}
\psi^m:\widehat{U}(\pi^m\mathcal{T})_K\rightarrow K\langle \pi^m\alpha_1,\cdots,\pi^m\alpha_n\rangle\widehat{\otimes}_K\mathcal{V}(X)=:\widehat{B}^m_{\mathcal{V}},    
\end{equation*}
where $\wideparen{B}_{\mathcal{V}}=\varprojlim_m\widehat{B}^m_{\mathcal{V}}$. Similarly, the composition $\wideparen{\D}_1\rightarrow \wideparen{\D}_{X^2}(X^2)\rightarrow \widehat{U}(\pi^m\mathcal{T})_K$ factors as follows:
\begin{equation*}
 \wideparen{\D}_1\rightarrow \widehat{\D}^m_1:=K\langle\pi^m\partial_1\otimes 1,\cdots,\pi^m\partial_n\otimes 1\rangle\widehat{\otimes}_K(A\widehat{\otimes}_K1)\rightarrow  \widehat{U}(\pi^m\mathcal{T})_K, 
\end{equation*}
such that $\wideparen{\D}_1=\varprojlim_m\widehat{\D}^m_1$ is a Fréchet-Stein presentation. In particular, we have a family of maps:
\begin{equation*}
    \widehat{\varphi}_1^m:\widehat{\D}^m_1 \rightarrow \widehat{U}(\pi^m\mathcal{T})_K\rightarrow \widehat{B}^m_{\mathcal{V}},
\end{equation*}
such that $\varphi_1=\varprojlim \widehat{\varphi}_1^m$. It is enough to show that $\widehat{\varphi}_1^m$ is an isomorphism for $m\geq 0$.\\
It suffices to show this for  $m=0$, as the other cases follow from this case multiplying the sections inducing the map $X\rightarrow \mathbb{A}^n_K$ by an appropriate power of $\pi$. Thus,  we will drop the $m$ from the notation for the rest of the proof. Consider the following $K$-vector spaces:
\begin{equation*}
  \D_1\cong (A\widehat{\otimes}_K 1)\otimes_K K[\partial_1\otimes 1,\cdots,  \partial_n\otimes 1], \textnormal{ } B_{\mathcal{V}}:=K[\alpha_1,\cdots,\alpha_n ]\otimes_KA_{\mathcal{V}}. 
\end{equation*}
Notice that $\D_1$ is a dense subalgebra of $\widehat{\D}_1$, and 
$B_{\mathcal{V}}$, is a dense subspace of $\widehat{B}_{\mathcal{V}}$. Hence, we regard them as bornological $K$-vector spaces with respect to the subspace bornology.\\
We will first show that $\widehat{\varphi}_1$ restricts to an isomorphism:
$\widetilde{\varphi}_1:\D_1\rightarrow B_{\mathcal{V}}$. Any polynomial ring $K[x_1,\cdots,x_n ]$ has a canonical filtration given by the total degree of the polynomials. Hence, it induces filtrations $F_{\bullet}\D_1$ and 
$F_{\bullet}B_{\mathcal{V}}$. We first show that $\widetilde{\varphi}_1$ respects those filtrations. Choose some non-negative integer $r\geq 0$. Elements in $F_{r}\D_1$ are $K$-linear combinations of operators of the form:
\begin{equation*}
   \gamma=(1\otimes \partial_1)^{m_1}\cdots (1\otimes \partial_n)^{m_n}f, \textnormal{ where } f\in A, \textnormal{ and }\sum_im_i\leq r.
\end{equation*}
We may assume for simplicity that
$\sum_im_i= r$. Then we have:
\begin{equation}\label{equation image of map in proof of integrable connections to diagonal bimodules}
    \widetilde{\varphi}_1(\gamma)=\frac{1}{2^r}(\alpha_1+\beta_1)^{m_1}\cdots (\alpha_n+\beta_n)^{m_n}\left(1\otimes f\right)=\frac{1}{2^r}\alpha_1^{m_1}\cdots\alpha_n^{m_n}\otimes f + R(\gamma).
\end{equation}
Notice that $R(\gamma)$ has total degree in the $\alpha_i$ at most $r-1$. Thus, it is contained in $F_{r-1}B_{\mathcal{V}}$. Hence, $\widetilde{\varphi}_1$ is a map of filtered $K$-vector spaces, and we get a map between the associated graded spaces:
\begin{equation*}
  \gr\widetilde{\varphi}_1:\gr\D_1\rightarrow \gr B_{\mathcal{V}}.  
\end{equation*}
Given an element $x\in \D_1$, we denote its image in $\gr\D_1$ by $[x]$, and we adopt the same convention for $B_{\mathcal{V}}$. As shown in equation $(\ref{equation image of map in proof of integrable connections to diagonal bimodules})$, we have:
\begin{equation*}
    \gr\widetilde{\varphi}_1([\gamma])=[\alpha_1^{m_1}\cdots\alpha_n^{m_n}\otimes f].
\end{equation*}
In particular, $\gr\widetilde{\varphi}_1$ is an isomorphism of graded $K$-vector spaces, so $\widetilde{\varphi}_1:\D_1\rightarrow B_{\mathcal{V}}$ is a bounded and bijective map of filtered bornological $K$-vector spaces.\\
Notice that for each $r\geq 0$, $F_r\D_1$ and $F_rB_{\mathcal{V}}$
are finite direct sums of Banach spaces. In particular, $F_r\D_1$ is a finite sum of copies of $A$, and $F_rB_{\mathcal{V}}$ is a finite sum of copies of $A_{\mathcal{V}}$. Hence, the restriction 
$\widetilde{\varphi}_1:F_r\D_1\rightarrow F_rB_{\mathcal{V}}$ is a bijective $K$-linear bounded map between Banach spaces. Hence, it is an isomorphism of Banach spaces. We have identities   $\D_1=\varinjlim_r F_r\D_1$, and $B_{\mathcal{V}}=\varinjlim_rF_rB_{\mathcal{V}}$ of bornological $K$-vector spaces. Therefore, $\widetilde{\varphi}_1:\D_1\rightarrow B_{\mathcal{V}}$ is an isomorphism of bornological $K$-vector spaces. Thus, it induces an isomorphism between the completions, and we get an isomorphism $\widehat{\varphi}_1:\widehat{\D}_1\rightarrow \widehat{B}_{\mathcal{V}}$. By taking inverse limits, it follows that $\varphi_1:\wideparen{\D}_1\rightarrow \wideparen{B}_{\mathcal{V}}$ is an isomorphism of complete bornological $K$-vector spaces, thus showing that $\wideparen{B}_{\mathcal{V}}$ is a co-admissible $\wideparen{\D}_X(X)\cong \wideparen{\D}_1$-module. Equivalently, $\wideparen{B}_{\mathcal{V}}$ is a co-admissible $\wideparen{\D}_X(X)\cong\wideparen{\D}_2$-module. Using the side-changing equivalence \cite[Theorem 3.4]{ardakov2015d}, it follows that $S(\Delta_+\mathcal{V})(X^2)=\Omega_{X}(X)\widehat{\otimes}_{A}B_{\mathcal{V}}$ is a co-admissible $\wideparen{\D}_X(X)^{\op}$-module.\bigskip

Furthermore, for every affinoid subdomain  $V\subset X$ we have:
\begin{align*}
    \wideparen{\D}_{\mathcal{V}}(V)=\operatorname{S}(\Delta_+\mathcal{V})(V\times X)=\wideparen{\D}_X(V)\widehat{\otimes}_{\wideparen{\D}_X(X)}\wideparen{\D}_{\mathcal{V}}(X)=&\operatorname{S}(\Delta_+\mathcal{V})(X\times V)\\=&\wideparen{\D}_{\mathcal{V}}(X)\widehat{\otimes}_{\wideparen{\D}_X(X)}\wideparen{\D}_X(V).
\end{align*}
Hence, $\wideparen{\D}_{\mathcal{V}}$ is a co-admissible $\wideparen{\D}_{X}$-module and a co-admissible $\wideparen{\D}_X^{\op}$-module.
\end{proof}
\begin{coro}
 Let $C^{\bullet}\in\operatorname{D}_{\mathcal{C}}(\wideparen{\D}_X)^-$ be a bounded above $\mathcal{C}$-complex such that:
 \begin{equation*}
     \operatorname{H}^i(C^{\bullet})\in \operatorname{Conn}(X), \textnormal{ for each } i\in\mathbb{Z}.
 \end{equation*}
Then $\Delta^{-1}\left(\operatorname{S}\left(\Delta_+C^{\bullet} \right)\right)$  is a bounded above diagonal $\mathcal{C}$-complex on $X$.
\end{coro}

Mathematisch-Naturwissenschaftliche Fakult\"at der Humboldt-Universit\"at zu Berlin, Rudower Chaussee 25, 12489 Berlin, Germany \newline
\textit{Email address: fpvmath@gmail.com}
\end{document}